\documentclass[12pt,oneside,reqno,english]{amsart}
\usepackage{lmodern}
\usepackage{lmodern}
\usepackage[T1]{fontenc}
\usepackage[latin9]{luainputenc}
\usepackage{geometry}
\usepackage[active]{srcltx}
\usepackage{color}
\usepackage{babel}
\usepackage{amstext}
\usepackage{amsthm}
\usepackage{amssymb}
\usepackage{xcolor}
\usepackage{stmaryrd}
\usepackage{graphicx}
\usepackage{setspace}
\usepackage[unicode=true,
 bookmarks=false,
 breaklinks=false,pdfborder={0 0 1},backref=section,colorlinks=true]
 {hyperref}
\hypersetup{
 linkcolor=blue,urlcolor=blue,citecolor=blue,anchorcolor=blue}

\makeatletter
\numberwithin{equation}{section}
\numberwithin{figure}{section}
\theoremstyle{plain}
\newtheorem{thm}{\protect\theoremname}[section]
\theoremstyle{plain}
\newtheorem{assumption}[thm]{\protect\assumptionname}
\theoremstyle{remark}
\newtheorem{rem}[thm]{\protect\remarkname}
\theoremstyle{remark}
\newtheorem{notation}[thm]{\protect\notationname}
\theoremstyle{plain}
\newtheorem{prop}[thm]{\protect\propositionname}
\theoremstyle{plain}
\newtheorem{lem}[thm]{\protect\lemmaname}
\theoremstyle{plain}
\newtheorem{cor}[thm]{\protect\corollaryname}

\@ifundefined{date}{}{\date{}}
\usepackage{babel}

\makeatother

\newcommand{\R}{\mathbb{R}} 
\newcommand{\M}{\mathcal{M}}

\providecommand{\assumptionname}{Assumption}
\providecommand{\lemmaname}{Lemma}
\providecommand{\corollaryname}{Corollary}
\providecommand{\notationname}{Notation}
\providecommand{\propositionname}{Proposition}
\providecommand{\remarkname}{Remark}
\providecommand{\theoremname}{Theorem}

\begin{document}
\title{Eyring-Kramers Law for the Underdamped Langevin Process}
\author{Seungwoo Lee, Mouad Ramil, Insuk Seo}

\address{Seungwoo Lee: Department of Mathematical Sciences, Seoul National University \\
e-mail: \texttt{swlee017@snu.ac.kr} }

\address{Mouad Ramil: Univ Rennes, INRIA, IRMAR [(Institut de Recherche Mathematique de Rennes)] - UMR 6625, F-35000 Rennes, France\\
e-mail: \texttt{mouad.ramil@inria.fr} }

\address{Insuk Seo: Department of Mathematical Sciences and Research Institute of Mathematics, Seoul National University, Republic of Korea. \\
e-mail: \texttt{insuk.seo@snu.ac.kr} }  

\begin{abstract}
In this article, we analyze the metastable behavior of the underdamped Langevin process $(q(t),p(t))_{t\geq0} \in \mathbb R^d \times \mathbb R^d$ described by the stochastic differential equation 
\begin{equation} 
  \left\{
    \begin{aligned}
        &\mathrm{d}q(t)= p(t) \mathrm{d}t , \\
        &\mathrm{d}p(t)=-(\nabla U(q(t))  + \gamma p(t))
        \mathrm{d}t +\sqrt{2\gamma\epsilon} \mathrm{d}B(t),
    \end{aligned}
\right.  
\end{equation}
where $\epsilon>0$ and $\gamma>0$ represent the temperature and the friction coefficient, respectively, and where $U:\mathbb R^d \rightarrow \mathbb R $ is a double-well potential. For this model, the Eyring-Kramers law, namely the sharp asymptotic of the transition time between metastable sets in the low temperature regime $\epsilon \to 0$ is a long-standing open problem because of the combined irreversibility of the model and the hypo-ellipticity of the corresponding generator. In this article, we develop a novel approach to overcome such difficulties and thus derive the Eyring-Kramers law for the underdamped Langevin process. 
\end{abstract}

\maketitle 

\section{Introduction}
The \emph{underdamped Langevin process} $(q^\epsilon(t),p^\epsilon(t))_{t\geq0}$ is the diffusion process described by the following stochastic differential equation (SDE): 
\begin{equation}\label{eq:Langevin intro}
  \left\{
    \begin{aligned}
        &\mathrm{d}q^\epsilon(t)=M^{-1} p^\epsilon(t) \mathrm{d}t , \\
        &\mathrm{d}p^\epsilon(t)=-\nabla U(q^\epsilon(t)) \mathrm{d}t -\gamma M^{-1}  p^\epsilon(t)
        \mathrm{d}t +\sqrt{2\gamma\epsilon} \mathrm{d}B(t),
    \end{aligned}
\right.  
\end{equation}
where, for $d=3N$, $N$ being the number of particles, $q^\epsilon(t)\in\R^{d}$ and $p^\epsilon(t)\in\R^{d}$ represent their position and momentum vector respectively, $U:\mathbb{R}^{d}\rightarrow \mathbb{R}$ is the interaction potential between the particles, $\gamma>0$ is the friction coefficient, $M\in \mathbb{R}^{d\times d}$
is the mass matrix, and $(B(t))_{t\geq0}$ is a $d$-dimensional Brownian motion accounting for the random thermal fluctuations. 

This process is a fundamental model in mathematical physics that describes the movement of particles and is also often used to describe the evolution of complex molecular systems evolving in a medium at a fixed temperature $\epsilon>0$. Given its importance, this process has been studied extensively in statistical mechanics and molecular dynamics~\cite{LeimMatt,LelSto16,Tuck}. This process has also many well-known applications: it is used for instance in material science to model the phase transition \textit{liquid $\to$ solid} of a given metal~\cite{AllenTild,LelRouSto10}. Simulations have shown that, when the temperature $\epsilon$ is low, this phase transition can take a very long time. Meanwhile, the process slightly oscillates around a local equilibrium. This phenomenon is called \emph{Metastability} and comes from the fact that the process needs to overcome an energetic barrier of the potential function $U$ in order to make a transition from one phase to another, which can take a very long amount of time in the low-temperature regime. 

To provide a more rigorous and intuitive mathematical explanation, consider, for example, a simple double-well potential function $U$ where we denote by $m$ (resp. $s$) the local (resp. global) minimum of $U$. These two minima are separated by a saddle point denoted by $\sigma$, see Figure~\ref{fig:doub} for a representation. 

\begin{figure}[h]
\begin{centering}
\includegraphics[scale=0.16]{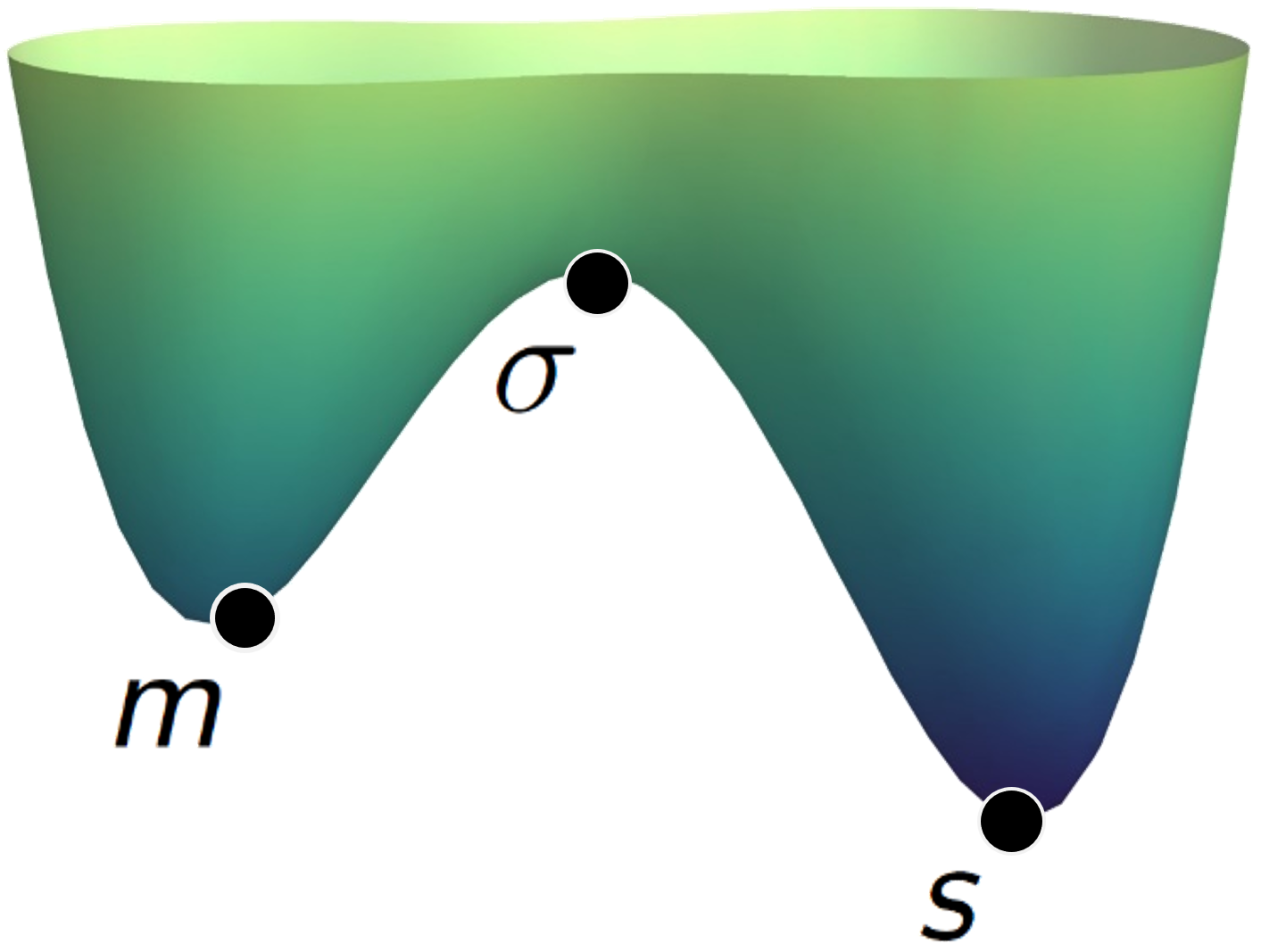}
\par\end{centering}
\caption{\label{fig:doub}An illustration of a double-well potential $U$.}
\end{figure}

As we will examine in detail later, in a stochastic dynamical system exhibiting metastability, if the process starts at $m$ or near it, it will spend a significant amount of time in the vicinity of $m$ because the energy barrier corresponding to $U(\sigma) - U(m)$ takes time to overcome. However, as time progresses and noise accumulates, the process will eventually surpass this energy barrier and transition to the valley where $s$ is located. This corresponds to the metastability phenomenon mentioned previously. In the case of the process $(q^\epsilon(t),p^\epsilon(t))_{t\geq0}$, it is expected to exhibit such a metastable transition from $(m,0)$ to $(s,0)$.

Such metastable behavior typically occurs in the low-temperature regime. Therefore, our main interest here is computing the low-temperature asymptotic (i.e., the $\epsilon\to 0$ regime in our model~\eqref{eq:Langevin intro}) of the mean hitting time of a small neighborhood of $(s,\,0)$ when starting from $(m,\,0)$ (or its neighborhood). This question has been investigated in the literature for various elliptic diffusion processes using diverse sets of tools. We refer the reader to the review~\cite{berglund2011kramers} for more details.

A sharp estimate of such mean transition time is referred to as the \textit{Eyring-Kramers law}. Despite being proven for numerous important models following the seminal work~\cite{BEGK}, the Eyring-Kramers law for the underdamped Langevin process described earlier in \eqref{eq:Langevin intro} remains an open problem. This is due to two challenging characteristics of the model: \textit{irreversibility and degeneracy}. In particular, while the methodology for the analysis of irreversible processes has been established in works such as~\cite{LS,LMS} authored by one of the authors of the current article, the issue of degeneracy of the underdamped Langevin process has remained a significant difficulty, preventing us to derive the Eyring-Kramers law for the mean transition time. \textit{The main achievement of the current article is to solve this long-standing open problem.}

Let us also mention previous advances made in the literature regarding the study of the underdamped Langevin process in the low-temperature regime. The sharp asymptotic of the low-lying spectrum of the associated kinetic Fokker-Planck operator was obtained in works like~\cite{herau2008tunnel,bony2024eyring} using analytic tools. Heuristically, we observe that the absolute value of the first eigenvalue behaves as the inverse of the mean first transition time when the temperature goes to zero. More recently, the authors also described the speed of convergence towards its invariant measure through the so-called cutoff phenomenon~\cite{cutoff} for general interaction forces.

To get a better grasp on metastability for this model, recent works by one of the authors have also focused on the study of the underdamped Langevin process when absorbed at the boundary of a domain~\cite{Hill,LelRamRey}. The question of the existence and convergence to a long-time equilibrium for this absorbed process has also been very active recently, see~\cite{champagnat2024quasi,benaim2025degenerate,guillin2024quasi,QSDLelRamRey}.

\subsection{Metastability of the overdamped Langevin process}
In the literature so far, the focus has mainly been on elliptic diffusion processes $(X^\epsilon(t))_{t\geq0}$ in $\R^d$ solutions to a SDE of the form:
\begin{equation}\label{eq:ov}
\mathrm{d}X^\epsilon(t)=-F(X^\epsilon(t))\mathrm{d}t+\sqrt{2\epsilon}\mathrm{d}B(t)\;,
\end{equation}
where $F:\mathbb{R}^d\rightarrow\mathbb{R}^d$ is a smooth vector field and $(B(t))_{t\geq0}$ is the standard $d$-dimensional Brownian motion. 
This process is called the \textit{overdamped Langevin process} and its infinitesimal generator is given by
$$\mathcal{L}_\epsilon:=-F(x)\cdot\nabla +\epsilon\,\Delta\;.$$
In particular, the process $(X^\epsilon(t))_{t\geq0}$ exhibits the metastability phenomenon when the zero-noise dynamics $\mathrm{d}X(t)=-F(X(t))\mathrm{d}t$ has multiple equilibria. 
The simplest case is the gradient case, where $F = \nabla U$ for a double-well potential $U$ as shown in Figure~\ref{fig:doub} which we also assume to be a Morse function. In this case, as discussed above, the dynamics $(X^\epsilon(t))_{t\geq0}$ exhibits metastability in the low-temperature regime where $\epsilon>0$ is small. That is, the dynamics starting from $m$ remains in a neighborhood of $m$ for a long time before transitioning to $s$. The corresponding Eyring-Kramers law has been established in~\cite{BEGK}. 

Let us recall here some of the important results in the literature. Let us define the open balls in $\R^d$ 
$$\M_\epsilon:=\mathrm{B}(m,\epsilon),\qquad \mathcal{S}_\epsilon:=\mathrm{B}(s,\epsilon).$$
Let $\tau^\epsilon_{\mathcal{S}_\epsilon}$ be the first hitting time of $\mathcal{S}_\epsilon$ for the process $(X^\epsilon(t))_{t\geq0}$ and we denote by $\mathbb{E}_m[\tau^\epsilon_{\mathcal{S}_\epsilon}]$ the expectation of $\tau^\epsilon_{\mathcal{S}_\epsilon}$ under the event $X_0^\epsilon=m$. We are interested here in the sharp asymptotic of $\mathbb{E}_m[\tau^\epsilon_{\mathcal{S}_\epsilon}]$ when $\epsilon$ goes to zero, which corresponds to the Eyring-Kramers law.

In their pioneered \emph{Large deviation} techniques, Freidlin and Ventsel have shown, see~\cite{ventsel1970small}, that 
\begin{equation}\label{eq:large dev}
\epsilon\log\mathbb{E}_m[\tau^\epsilon_{\mathcal{S}_\epsilon}]\underset{\epsilon\rightarrow0}{\longrightarrow}U(\sigma)-U(m).
\end{equation}
In particular, this ensures that  
\begin{equation}\label{eq:EK estimate}
\mathbb{E}_m[\tau^\epsilon_{\mathcal{S}_\epsilon}]=[1+o_{\epsilon}(1)]\,\kappa_\epsilon\,\mathrm{e}^{(U(\sigma)-U(m))/\epsilon}.
\end{equation}
To fully establish the Eyring-Kramers law, we need to establish the sharp low-temperature asymptotic of the sub-exponential prefactor $\kappa_\epsilon$. However, the computation of $\kappa_\epsilon$ is a non-trivial question and was only achieved years later in~\cite{BEGK} using \emph{Potential theory} techniques. Their approach namely relies on analyzing the low-temperature behavior of the following quantities:
\begin{enumerate}
    \item the function $h_{\mathcal{M}_{\epsilon},\,\mathcal{S_{\epsilon}}}(x):=\mathbb{P}_{x}(\tau_{\mathcal{M}_{\epsilon}}^{\epsilon}<\tau_{\mathcal{S}_{\epsilon}}^{\epsilon})$, called \emph{equilibrium potential function}, 
    \item the boundary integral $\int_{\partial\mathcal{M}_\epsilon}\mathrm{e}^{-U(x)/\epsilon}\eta_\epsilon(\mathrm{d}x)$, called \emph{capacity}, where $\eta_\epsilon$ is called the \emph{equilibrium measure}.
\end{enumerate} 
With a sufficient understanding of the low-temperature behavior of the two objects above, the Eyring-Kramers law then follows from a well-known identity in Potential Theory, which links the expectation $\mathbb{E}_m[\tau^\epsilon_{\mathcal{S}_\epsilon}]$ with the two previous quantities. In particular, for the overdamped Langevin process discussed above, one obtains that
$$\kappa_\epsilon=\kappa =\frac{2\pi}{\lambda^{\sigma}}\sqrt{\frac{-\det\mathbb{H}_{U}^{\sigma}}{\det\mathbb{H}_{U}^{m}}},$$
(which is independent of $\epsilon$) where $\mathbb{H}_{U}^{x}=\nabla^2 U(x)$ corresponds to the Hessian matrix of the potential $U$ at $x$ and $-\lambda^{\sigma}$ is the unique negative eigenvalue of the matrix $\mathbb{H}_{U}^{\sigma}$ which exists since we assume that $U$ is a Morse function.

The most crucial aspect of the above approach is obtaining the precise low-temperature asymptotic of the capacity. In this context, the potential theoretic approach developed in \cite{BEGK} proposes the use of variational principles related to the capacity, such as the \emph{Dirichlet principle} and the \emph{Thomson principle}. This method has been remarkably successful and has been applied not only to the overdamped Langevin process but also to many other models. These cases are systematically summarized in the monograph~\cite{Bovier}.

An important point to mention here is that this potential theoretic approach was originally developed for \textit{reversible} Markov processes, since the Dirichlet and Thomson principles are valid only under the reversibility of the model. In particular, the above overdamped Langevin process when $F=\nabla U$ is reversible with respect to its invariant Gibbs measure, which allowed this methodology to be successfully applied. However, the underdamped model~\eqref{eq:Langevin intro} is inherently irreversible.

Fortunately, in recent years, methods for applying Potential theory techniques to irreversible processes have been developed. For instance, in~\cite{gaudilliere2014dirichlet}, the Dirichlet principle and Thomson principle for irreversible Markov processes on discrete spaces were established, and these results were extended to the diffusion setting in \cite{LMS}.
The method for deriving the Eyring-Kramers law using these results has been progressively developed in~\cite{LS,LMS}, and has been applied to various models. 

Recently, a new alternative technique for estimating the capacity (without relying on variational principles) was developed in~\cite{JS}. It relies on an alternative expression of the capacity which involves a clever use of test functions, building on the inherent definition of the equilibrium measure as $\eta_\epsilon=-\mathcal{L}_\epsilon h_{\mathcal{M}_{\epsilon},\,\mathcal{S_{\epsilon}}}$ in the sense of distributions, see~\cite{Bovier}. Using these results, the Eyring-Kramers law has been extensively proven for more general cases beyond the reversible setting where $F = \nabla U$ in \eqref{eq:ov}.  

We note that, all these cases are restricted to situations where the invariant measure of the process $(X^\epsilon(t))_{t\geq0}$ defined in \eqref{eq:ov} is the Gibbs measure. When the invariant measure of the process $(X^\epsilon(t))_{t\geq0}$ is no longer a Gibbs measure then very little is known in the literature, although some advances have been made recently in~\cite{bouchet2016generalisation,le2024exit}.

 Let us also mention that in recent years, other techniques have been developed to deduce the Eyring-Kramers law, in particular involving geometric arguments, see for instance~\cite{avelin2023geometric}.

\subsection{Metastability of the underdamped Langevin process}

With regard to the underdamped Langevin process $(q^\epsilon(t),p^\epsilon(t))_{t\geq0}$ defined in~\eqref{eq:Langevin intro}, although the invariant measure is a well-known Gibbs measure, the study is very tricky. In fact, as mentioned earlier, this model not only faces the challenge of being irreversible but also suffers from degeneracy. That is, the dimension of the Brownian noise acting on the model is $d$, which is smaller than the model's dimension which is $2d$. As a result, the generator associated with this diffusion process is not elliptic but only hypo-elliptic.  

This makes it extremely difficult to apply existing methodologies to derive the Eyring-Kramers law of the underdamped Langevin process. The difficulty is not merely in estimating the capacity but arises fundamentally from the very definition of the capacity itself. Indeed, the capacity is ill-defined in this case since the boundary-regularity of the gradient of the equilibrium potential function is not known.   Additionally, the divergence theorem which allows us to compute the capacity as the $\mathrm{L}^2$ norm of the gradient of the equilibrium potential function cannot be justified. As a result, a rigorous definition of the capacity as a boundary integral and the corresponding Dirichlet and Thomson principle for the underdamped Langevin process are open questions and are actually the subject of an ongoing work. 

In this paper, we aim to overcome these challenges.  In order to do that, we recall the definition of the equilibrium measure, in a weak sense, which was mentioned in~\cite[Section 7.2.4]{Bovier} and from which the capacity can be derived. Namely, the equilibrium measure can be defined as $-\mathcal{L}_\epsilon h_{\mathcal{M}_\epsilon,\,\mathcal{S}_\epsilon}$ in the distributional sense, where $\mathcal{L}_\epsilon$ is the generator of the underdamped process. Then, the capacity corresponds to the boundary integral on $\partial\mathcal{M}_\epsilon$ of the equilibrium measure. Our complete method of deriving the Eyring-Kramers law then follows the scheme of proof below: 
\begin{enumerate}
    \item a deep study of the trajectory of the process~\eqref{eq:Langevin intro} based on previous results in~\cite{cutoff,LelRamRey} to deduce sharp bounds on the equilibrium potential,   
    \item derive an expression linking the expectation of the transition time and the capacity and estimate the capacity using well-thought test functions following the method developed in~\cite{JS}, while also avoiding boundary-regularity assumption on the equilibrium potential. 
\end{enumerate}
All of these approaches are novel methodologies that have never been attempted in the study of the underdamped Langevin process. Additionally, the arguments and computation involved in this methodology require a deep understanding on recent results on the underdamped Langevin process absorbed at the boundary of a domain developed in works such as~\cite{cutoff,Hill,LelRamRey,QSDLelRamRey}. All these points increase the complexity of the paper and deepen its argument, although the authors make every effort to present it as clearly as possible.

\section{Model and main results}

\subsection{Underdamped Langevin process}

In this paper, our main interest is to analyze the metastability phenomenon in the $\epsilon\to 0$ regime of \eqref{eq:Langevin intro} and establish the Eyring-Kramers law. Since the proof is long and complex, we aim to present it in the simplest possible manner by minimizing unnecessary notational complexity. In particular, we assume that $M$ in \eqref{eq:Langevin intro} is the identity matrix and consider a model in $d=3N$ dimensions.\footnote{Of course, our results can be immediately extended to the case where $M$ is not the identity matrix.} 

To describe this model concretely, we repeat the definition here. Let $d\geq1$, and we shall focus on the study of the underdamped Langevin process $(X^{\epsilon}(t)=(q^{\epsilon}(t),\,p^{\epsilon}(t)))_{t\geq0}$
in $\mathbb{R}^{d}\times\mathbb{R}^{d}$ which is solution to the following SDE: 
\begin{equation}
\left\{ \begin{aligned} & \mathrm{d}q^{\epsilon}(t)=p^{\epsilon}(t)\mathrm{d}t\;,\\
 & \mathrm{d}p^{\epsilon}(t)=-\nabla U(q^{\epsilon}(t))\mathrm{d}t-\gamma p^{\epsilon}(t)\mathrm{d}t+\sqrt{2\gamma\epsilon}\mathrm{d}B(t)\;,
\end{aligned}
\right.\label{eq:sde}
\end{equation}
where $U\in C^{2}(\mathbb{R}^{d},\mathbb{\,R})$ is the potential function,
$\gamma>0$ and $\epsilon>0$ denote the friction coefficient and
the temperature, respectively. In order to prove the Eyring-Kramers law for this process, we will investigate its behaviour when the temperature
$\epsilon>0$ is small.

We denote by $\mathcal{L}_{\epsilon}$ the infinitesimal generator associated with the underdamped Langevin process $(X^\epsilon (t))_{t\geq0}$. This generator is also called
kinetic Fokker-Planck operator, and writes for $(q,\;p)\in\mathbb{R}^{d}\times\mathbb{R}^{d}$,
\begin{equation}
\mathcal{L}_{\epsilon}=\langle p,\,\nabla_{q}\rangle-\langle\nabla U(q),\,\nabla_{p}\rangle-\gamma\langle p,\,\nabla_{p}\rangle+\gamma\epsilon\Delta_{p}\;.\label{eq:gen}
\end{equation}

\subsubsection*{Potential function $U$}

We make the following assumptions on the potential function $U$. 
\begin{assumption}[Double-well structure]
\label{ass:U} The function $U$ is a double-well potential and Morse
function which admits two local minima denoted by $m,\,s\in\mathbb{R}^{d}$. The unique saddle point linking these two minima is denoted by $\sigma\in\mathbb{R}^{d}$. Notice that we do not assume here any order between $U(m)$ and $U(s)$. In particular, one can have $U(m)<U(s)$ unlike Figure~\ref{fig:doub}.
\end{assumption} 

\begin{rem}
Assumption~\ref{ass:U} is a routine one. From the perspective of our argument, it is not crucial, as a similar discussion can be carried out even in the case of multiple saddle points, see~\cite[Section 3]{JS}. Nevertheless, since the notation becomes significantly more complex in such cases, we adopt this assumption to simplify the framework.
\end{rem}

\begin{assumption}[Non-negativity and growth of $U$]
 \label{ass:non-negU} Since $U$ is bounded from below by Assumption~\ref{ass:U} and since the coefficients in~\eqref{eq:sde} only depend on $\nabla U$, by adding
a suitable constant to $U$, we can assume that $U$ is non-negative. Thus, we shall
henceforth assume that $U$ is non-negative. Furthermore, we assume that $U$ satisfies the following growth conditions:  
\label{ass:nabla U} 
\begin{equation}
\liminf_{|q|\longrightarrow\infty}\frac{\langle q,\,\nabla U(q)\rangle}{|q|^{2}+U(q)}>0\;\;\;\;\text{and}\;\;\;\;\liminf_{|q|\rightarrow\infty}\left(|\nabla U(q)|-\beta\Delta U(q)\right)>0\;\label{eq:nablaU_cond}
\end{equation}
for some constant $\beta>0$. 
\end{assumption}

\begin{rem}
We note that the denominator of the first limit at \eqref{eq:nablaU_cond}
is non-zero thanks to the non-negativity assumption on $U$. We also remark
that the growth conditions of Assumption \ref{ass:nabla U} are essential for the tightness of
the Hamiltonian function defined below in \eqref{eq:ham} (cf. Proposition
\ref{prop:tight}) which implies the positive recurrence of the process
$(X^{\epsilon}(t))_{t\ge0}$ (cf. Proposition \ref{prop:pos_rec}). In particular, similar assumptions were made in many previous works for the study of the overdamped Langevin process. Note also that these conditions follow immediately if $U$ grows quadratically.
The use of these assumptions will be further discussed in Section \ref{sec3:recurrence}. 
\end{rem}

\subsubsection*{Invariant measure}

Define the Hamiltonian function associated with the potential $U$
as 
\begin{equation}
V(q,\,p)=U(q)+\frac{|p|^{2}}{2}\;.\label{eq:ham}
\end{equation}
Then, we can write the invariant measure of the process $(X^{\epsilon}(t))_{t\ge0}$
defined in \eqref{eq:sde} as
\begin{equation}
\mu_{\epsilon}(q,\,p)\mathrm{d}q\mathrm{d}p=\frac{1}{Z_{\epsilon}}e^{-V(q,\,p)/\epsilon}\mathrm{d}q\mathrm{d}p\;,\label{eq:invm}
\end{equation}
where 
\begin{equation}
Z_{\epsilon}=\int_{\mathbb{R}^{d}\times\mathbb{R}^{d}}e^{-V(q,\,p)/\epsilon}\mathrm{d}q\mathrm{d}p\label{eq:partition}
\end{equation}
is the normalization constant also called \textit{partition function}. Indeed, we prove
in Proposition \ref{prop:tight} that $Z_{\epsilon}<\infty$ for all
$\epsilon\in(0,1)$ under Assumption \ref{ass:nabla U}. 

We denote by $\mathcal{L}_{\epsilon}^{*}$ the adjoint generator of
$\mathcal{L}_{\epsilon}$ on the space $\mathrm{L}^{2}(\mu_{\epsilon}(q,\,p)\mathrm{d}q\mathrm{d}p)$,
i.e., 
\begin{equation}
\mathcal{L}_{\epsilon}^{*}=-\langle p,\,\nabla_{q}\rangle+\langle\nabla U(q),\,\nabla_{p}\rangle-\gamma\langle p,\,\nabla_{p}\rangle+\gamma\epsilon\Delta_{p}\;.\label{eq:adj_gen2}
\end{equation}

\subsection{Notation}

Let us now introduce several notations used throughout this work. 

\begin{notation}
\label{not:main}The following notations will be frequently used in
this article. 
\begin{itemize}
\item (Distance) For a given non-empty measurable set $\mathcal{C}\subset\mathbb{R}^{2d}$,
we denote by $\mathrm{d}(x,\,\mathcal{C})$ the Euclidean distance
of $x\in\mathbb{R}^{2d}$ to the set $\mathcal{C}$, namely, 
\[
\mathrm{d}(x,\,\mathcal{C})=\inf_{y\in\mathcal{C}}|y-x|\;.
\]
\item (Hitting time) For a given non-empty measurable set $\mathcal{C}\subset\mathbb{R}^{2d}$,
we denote by $\tau_{\mathcal{C}}^{\epsilon}$ the first hitting time
of $\mathcal{C}$ by the process \eqref{eq:sde}, i.e. 
\begin{equation}
\tau_{\mathcal{C}}^{\epsilon}:=\inf\{t\geq0:(q^{\epsilon}(t),\,p^{\epsilon}(t))\in\mathcal{C}\}\;.\label{eq:hitting time def}
\end{equation}
\item (Vector notation) We denote by 
\begin{equation}
x=(q,\,p)\in\mathbb{R}^{2d}\label{eq:xqp}
\end{equation}
 a $2d$-dimensional vector where $q,\,p\in\mathbb{R}^{d}$. 
\item (Ball) For $r>0$, we denote by $\textup{B}(x,\,r)$ or
$\mathrm{B}((q,\,p),\,r)$ the open ball centered at $x=(q,\,p)\in\mathbb{R}^{2d}$
with radius $r>0$. We simply write, for $q\in\mathbb{R}^{d}$ and
$r>0$, 
\begin{equation}
\textup{B}(q,\,r):=\textup{B}((q,\,0),\,r)\;.\label{eq:ball2}
\end{equation}
\item (Matrix) For a square matrix $A$, we denote by $A^{\dagger}$ the transpose
matrix of $A$, and denote by $\Vert A\Vert$ (resp. $\Vert A\Vert_{\textup{F}}$)
the triple norm (resp. Frobenius norm) of the matrix $A$,
i.e. 
\[
\Vert A\Vert=\sup_{|z|\neq0}\frac{|Az|}{|z|}\;,\qquad\Vert A\Vert_{\textup{F}}=\sqrt{\sum_{i,\,j}A_{i,\,j}^{2}}\;,
\]
where $A_{i,\,j}$ is the $(i,j)$-th coordinate of the matrix $A$. 
\item (Hessian) For $z\in\mathbb{R}^{n}$ and $f\in C^{2}(\mathbb{R}^{n},\,\mathbb{R})$,
we denote by $\mathbb{H}_{f}^{z}\in\mathbb{R}^{n\times n}$ the Hessian
matrix of $f$ at $z$, i.e., $\mathbb{H}_{f}^{z}=\nabla^2 f (z)$.
\end{itemize}
\end{notation}

\subsection{Main result: Eyring-Kramers law}

Our main goal in this work is to derive the Eyring-Kramers law
for the underdamped Langevin process $(X^\epsilon (t))_{t\geq0}$ defined by \eqref{eq:sde}. For $x\in\R^{2d}$, we shall denote here by $\mathbb{P}_x = \mathbb{P}^\epsilon_x$ the law of the process $(X^\epsilon (t))_{t\geq0}$
under which $X_0^\epsilon=x$. The expectation under $\mathbb{P}_x$ is denoted by $\mathbb{E}_x= \mathbb{E}^\epsilon_x$. We will abbreviate the dependency on $\epsilon$ to simplify the notation. 

Now, let
us define $\mathcal{M}_{\epsilon}$ and $\mathcal{S}_{\epsilon}$ as the following
balls around the minima $(m,\,0)$ and $(s,\,0)$ of $V$ (cf. Assumption \ref{ass:U} and \eqref{eq:ball2}), i.e.,
\begin{equation}
\mathcal{M}_{\epsilon}=\mathrm{B}(m,\,\epsilon),\qquad \mathcal{S}_{\epsilon}=\mathrm{B}(s,\,\epsilon)\;.\label{eq:neigh}
\end{equation}

The Eyring-Kramers law corresponds to the asymptotic of $\mathbb{E}_{(m,\,0)}(\tau_{\mathcal{S}_{\epsilon}}^{\epsilon})$
when $\epsilon$ goes to zero. In this section we shall state such
estimate for the underdamped Langevin process \eqref{eq:sde}. Additionally, we write from now on $a_\epsilon=o_{\epsilon}(1)$ whenever $a_\epsilon\underset{\epsilon\rightarrow0}{\longrightarrow}0$.

\begin{thm}[Eyring-Kramers law]
\label{thm:ek}We have that 
\begin{equation}
\mathbb{E}_{(m,\,0)}(\tau_{\mathcal{S}_{\epsilon}}^{\epsilon})=[1+o_{\epsilon}(1)]\frac{2\pi}{\mu^{\sigma}}\sqrt{\frac{-\det\mathbb{H}_{U}^{\sigma}}{\det\mathbb{H}_{U}^{m}}}\mathrm{e}^{(U(\sigma)-U(m))/\epsilon}
\end{equation}
where $-\lambda^{\sigma}$ is the unique negative eigenvalue of $\mathbb{H}_{U}^{\sigma}$ and $\mu^{\sigma}=\frac{-\gamma+\sqrt{\gamma^{2}+4\lambda^{\sigma}}}{2}>0$. 
\end{thm}

\begin{rem}
Indeed, since $U$ is a Morse function by Assumption~\ref{ass:U}, its Hessian
matrix at its saddle point $\sigma$ admits a unique negative eigenvalue
$-\lambda^{\sigma}$. 
\end{rem}

The proof of this theorem follows the outline
of the potential theoretic approach of~\cite{BEGK,LS, LMS}, but
many of these steps fail due to the non-ellipticity of the infinitesimal generator $\mathcal{L}_\epsilon$ defined in~\eqref{eq:gen}.

\subsection{Sketch of proof}

Let us define the function $h_{\mathcal{M}_{\epsilon},\mathcal{\,S}_{\epsilon}}:\mathbb{R}^{2d}\rightarrow[0,\,1]$
such that 
\[
h_{\mathcal{M}_{\epsilon},\,\mathcal{S_{\epsilon}}}(x):=\mathbb{P}_{x}(\tau_{\mathcal{M}_{\epsilon}}^{\epsilon}<\tau_{\mathcal{S}_{\epsilon}}^{\epsilon})\;\;\;\;;\;x\in\mathbb{R}^{2d}\;.
\]
This function is called equilibrium potential function and played a primary role
in the proof of the Eyring-Kramers law for the overdamped Langevin process
in~\cite{BEGK,LMS,JS}. Note that the equilibrium potential function trivially satisfies 
\[
h_{\mathcal{M}_{\epsilon},\,\mathcal{S_{\epsilon}}}\equiv1\;\text{on }\mathcal{M}_{\epsilon}\;\;\;\text{and}\;\;\;h_{\mathcal{M}_{\epsilon},\,\mathcal{S_{\epsilon}}}\equiv0\;\text{on }\mathcal{S}_{\epsilon}\;.
\] However, given that $\mathcal{L}_\epsilon$ is only hypoelliptic, the regularity of $\nabla h_{\mathcal{M}_{\epsilon},\,\mathcal{S_{\epsilon}}}$ at $\partial(\mathcal{M}_{\epsilon}\cup\mathcal{S}_{\epsilon})$ is unknown.  In particular, one cannot define the capacity as a boundary integral involving the normal derivative of the equilibrium potential. Nor can we use the divergence theorem to state the following equivalency (which is well-known in the elliptic setting) between the capacity and the electrostatic energy 
\begin{equation}\label{eq:def capacity Langevin}
\mathrm{Cap}_\epsilon(\mathcal{M}_{\epsilon},\,\mathcal{S_{\epsilon}})=\epsilon\int_{(\overline{\mathcal{M}_{\epsilon}}\cup\overline{\mathcal{S}_{\epsilon}})^{c}}\left|\nabla_{p}h_{\mathcal{M}_{\epsilon},\,\mathcal{S}_{\epsilon}}(x)\right|^{2}\mu_{\epsilon}(x)\mathrm{d}x,
\end{equation}
All of these questions remain however the subject of an ongoing work.

In the meantime, one can define the equilibrium measure $\eta_\epsilon$ in a weak sense following the definition provided in~\cite{Bovier}. In order to do that let us first define the following function for $(q,\,p)\in\R^{2d}$,
\begin{equation}
h_{\mathcal{M}_{\epsilon},\,\mathcal{S}_{\epsilon}}^{*}(q,p):=h_{\mathcal{M}_{\epsilon},\,\mathcal{S}_{\epsilon}}(q,\,-p)\;.\label{eq:def h^*}
\end{equation}

\begin{prop}[Capacity]
\label{prop:cap} There exists a unique measure $\eta_\epsilon$ on $\partial\mathcal{M}_\epsilon$ such that, for any function $f\in C_{c}^{\infty}(\overline{\mathcal{S}_{\epsilon}}^{c})$,
\begin{equation}\label{eq:def eq measure}
\int_{\partial\mathcal{M}_\epsilon}f(x)\,\eta_\epsilon(\mathrm{d}x)=\int_{\mathbb{R}^{2d}}h_{\mathcal{M}_{\epsilon},\,\mathcal{S}_{\epsilon}}^{*}(x)(-\mathcal{L}_{\epsilon}f(x))\mu_{\epsilon}(x)\mathrm{d}x\;.
\end{equation}

In particular, the capacity $\mathrm{Cap}_\epsilon(\mathcal{M}_{\epsilon},\,\mathcal{S_{\epsilon}})$ is given by, for any function $f\in C_{c}^{\infty}(\overline{\mathcal{S}_{\epsilon}}^{c})$
satisfying $f=1$ on $\partial\mathcal{M}_{\epsilon}$,
\begin{equation}
\mathrm{Cap}_\epsilon(\mathcal{M}_{\epsilon},\,\mathcal{S_{\epsilon}}):=\int_{\partial\mathcal{M}_\epsilon}\eta_\epsilon(\mathrm{d}x)=\int_{\mathbb{R}^{2d}}h_{\mathcal{M}_{\epsilon},\,\mathcal{S}_{\epsilon}}^{*}(x)(-\mathcal{L}_{\epsilon}f(x))\mu_{\epsilon}(x)\mathrm{d}x\;.\label{eq:cap}
\end{equation}

\end{prop}

The following identity links the expectation of the first transition time to $\mathcal{S}_\epsilon$, starting from a neighborhood of $(m,\,0)$, with the capacity and the equilibrium potential. It is often used in Potential theory to derive the Eyring-Kramers law, see the seminal work~\cite{BEGK}. We prove here this identity in the highly non-trivial setting of the underdamped Langevin process~\eqref{eq:sde}. This identity along with Proposition~\ref{prop:cap} constitute a key step in the proof of Theorem \ref{thm:ek}.

\begin{prop}
\label{prop:harm} There exists a probability measure $\nu_{\epsilon}$
on $\partial\mathcal{M}_{\epsilon}$ such that
\begin{equation}
\int_{\partial\mathcal{M}_{\epsilon}}\mathbb{E}_{x}(\tau_{\mathcal{S}_{\epsilon}}^{\epsilon})\nu_{\epsilon}(\mathrm{d}x)=\frac{1}{\mathrm{Cap}_\epsilon(\mathcal{M}_{\epsilon},\,\mathcal{S}_{\epsilon})}\int_{\mathbb{R}^{2d}}h_{\mathcal{M}_{\epsilon},\,\mathcal{S}_{\epsilon}}^{*}(x)\mu_{\epsilon}(x)\mathrm{d}x\;.\label{eq:magicf}
\end{equation}
\end{prop}

Propositions \ref{prop:cap} and \ref{prop:harm} are proven in Section
\ref{sec4:magicf}. They are first obtained for a regularized version of the differential operator $\mathcal{L}_\epsilon$ and then extended to the current setting. Additionally, in order to deduce the Eyring-Kramers law, we approximate the left side of
\eqref{eq:magicf} by $\mathbb{E}_{(m,\,0)}(\tau_{\mathcal{S}_{\epsilon}}^{\epsilon})$ using the following proposition which will be proven in Section
\ref{sec:Proof of prop:control w}. 

\begin{prop}
\label{prop:control v}For each $\beta\in (1/2,1]$, there exist constants $c=c(\beta),\,\delta=\delta(\beta),\,\epsilon_{0}=\epsilon_{0}(\beta)>0$
such that
\[
\mathbb{E}_{(m,\,0)}(\tau_{\mathcal{S}_{\epsilon}}^{\epsilon})(1-c\epsilon^{\delta})\leq\mathbb{E}_{x}(\tau_{\mathcal{S}_{\epsilon}}^{\epsilon})\leq\mathbb{E}_{(m,\,0)}(\tau_{\mathcal{S}_{\epsilon}}^{\epsilon})(1+c\epsilon^{\delta})
\]
for all $x\in\mathrm{B}(m,\,\epsilon^{\beta})$ and $\epsilon\in(0,\epsilon_{0})$. 
\end{prop}

We next introduce the key propositions providing sharp asymptotics when $\epsilon\rightarrow0$
of the two terms at the right-hand side of \eqref{eq:magicf}. 
Recall from \eqref{eq:partition} that $Z_{\epsilon}$ represents the
partition function and from Notation \ref{not:main} that $\mathbb{H}_{U}^{q}$ denotes the Hessian matrix of $U$ at $q\in\R^d$. 
\begin{prop}
\label{prop:main1} We have that 
\begin{equation}
\int_{\mathbb{R}^{2d}}h_{\mathcal{M}_{\epsilon},\,\mathcal{S}_{\epsilon}}^{*}(x)\,\mu_{\epsilon}(x)\mathrm{d}x=[1+o_{\epsilon}(1)]\frac{1}{Z_{\epsilon}}(2\pi\epsilon)^{d}\frac{1}{\sqrt{\det\mathbb{H}_{U}^{m}}}e^{-U(m)/\epsilon}\;.
\end{equation}
\end{prop}

Next recall that $\mu^{\sigma}$ is the constant introduced in Theorem~\ref{thm:ek}. 
\begin{prop}
\label{prop:main2} We have that 
\begin{equation}
\mathrm{Cap}_{\epsilon}(\mathcal{M}_{\epsilon},\,\mathcal{S}_{\epsilon})=[1+o_{\epsilon}(1)]\frac{1}{Z_{\epsilon}}\frac{(2\pi\epsilon)^{d}}{2\pi}\frac{\mu^{\sigma}}{\sqrt{-\det\mathbb{H}_{U}^{\sigma}}}e^{-U(\sigma)/\epsilon}\;.
\end{equation}
 
\end{prop}
The proofs of Propositions \ref{prop:main1} and \ref{prop:main2} are
given in Section~\ref{sec:Estimate of numerator} and Section~\ref{sec:Esitmate of dominator}, respectively. At this moment, we complete the proof of Theorem~\ref{thm:ek} by assuming these propositions.  

\begin{proof}[Proof of Theorem~\ref{thm:ek}]
By Proposition~\ref{prop:harm} along with Propositions~\ref{prop:main1}
and~\ref{prop:main2}, 
\[
\int_{\partial\mathcal{M}_{\epsilon}}\mathbb{E}_{x}(\tau_{\mathcal{S}_{\epsilon}}^{\epsilon})\,\nu_{\epsilon}(\mathrm{d}x)=[1+o_{\epsilon}(1)]\frac{2\pi}{\mu^{\sigma}}\sqrt{\frac{-\det\mathbb{H}_{U}^{\sigma}}{\det\mathbb{H}_{U}^{m}}}\mathrm{e}^{(U(\sigma)-U(m))/\epsilon}\;.
\]
In order to conclude the proof it is sufficient to show that 
\begin{equation}
\int_{\partial\mathcal{M}_{\epsilon}}\mathbb{E}_{x}(\tau_{\mathcal{S}_{\epsilon}}^{\epsilon})\,\nu_{\epsilon}(\mathrm{d}x)=\mathbb{E}_{(m,\,0)}(\tau_{\mathcal{S}_{\epsilon}}^{\epsilon})[1+o_{\epsilon}(1)]\;.\label{eq:apu}
\end{equation}
To this end, take any $\beta\in(1/2,\,1]$ so that $\partial\mathcal{M}_{\epsilon}\subset\mathrm{B}(m,\,\epsilon^{\beta})$
for $\epsilon\in(0,\,1)$. By Proposition~\ref{prop:control v}, there
exist constants $c,\,\delta,\,\epsilon_{0}>0$ such that for all $x\in\partial\mathcal{M}_{\epsilon}$
and $\epsilon\in(0,\,\epsilon_{0})$, 
\[
\mathbb{E}_{(m,\,0)}(\tau_{\mathcal{S}_{\epsilon}}^{\epsilon})(1-c\epsilon^{\delta})\leq\mathbb{E}_{x}(\tau_{\mathcal{S}_{\epsilon}}^{\epsilon})\leq\mathbb{E}_{(m,\,0)}(\tau_{\mathcal{S}_{\epsilon}}^{\epsilon})(1+c\epsilon^{\delta})\;.
\]
As a result, 
\begin{align*}
\left|\int_{\partial\mathcal{M}_{\epsilon}}\mathbb{E}_{x}(\tau_{\mathcal{S}_{\epsilon}}^{\epsilon})\nu_{\epsilon}(\mathrm{d}x)-\mathbb{E}_{(m,\,0)}(\tau_{\mathcal{S}_{\epsilon}}^{\epsilon})\right| & \leq\mathbb{E}_{(m,\,0)}(\tau_{\mathcal{S}_{\epsilon}}^{\epsilon})\int_{\partial\mathcal{M}_{\epsilon}}\left|\frac{\mathbb{E}_{x}(\tau_{\mathcal{S}_{\epsilon}}^{\epsilon})}{\mathbb{E}_{(m,\,0)}(\tau_{\mathcal{S}_{\epsilon}}^{\epsilon})}-1\right|\nu_{\epsilon}(\mathrm{d}x)\\
 & \leq c\epsilon^{\delta}\mathbb{E}_{(m,\,0)}(\tau_{\mathcal{S}_{\epsilon}}^{\epsilon})\;,
\end{align*}
which concludes the proof of \eqref{eq:apu}. 
\end{proof}
\textbf{Outline of the article.} This
work is divided in several sections, each one focusing on the proof
of a main result. 

\begin{itemize}
 
   \item Section~\ref{sec3:recurrence} provides some key
properties satisfied by the process~\eqref{eq:sde} such as the positive
recurrence. 

   \item Sections~\ref{sec4:magicf} and \ref{sec5:lem_unifc} are dedicated to the proof of Propositions~\ref{prop:cap} and~\ref{prop:harm}. 

   \item Section~\ref{sec:Proof of prop:control w}
is devoted to the proof of Proposition~\ref{prop:control v}.  

   \item Sections~\ref{sec:Estimate of numerator}, \ref{sec:density estimate}, \ref{sec:Lyapunov H} and \ref{sec:proof op Prop bound exponential exit time} are dedicated to the proof of Proposition~\ref{prop:main1}  

\item Section~\ref{sec:Esitmate of dominator} focuses on the
proof of Proposition~\ref{prop:main2}.
\end{itemize}

\section{Preliminary analysis\label{sec3:recurrence}}

We prove here several key preliminary properties satisfied by the
underdamped Langevin process \eqref{eq:sde}. To simplify the discussion,
we henceforth always assume that $\epsilon\in(0,\;1)$, unless otherwise
specified. 

\subsection{Tightness of the invariant measure }

We first establish the tightness of the invariant measure \eqref{eq:invm}
which implies in particular the finiteness of the partition function. We recall from \eqref{eq:ham} the definition of the function $V$. 
 
\begin{prop}
\label{prop:tight}For any $a\geq0$, there exists a constant $C_{a}>0$
independent of $\epsilon\in(0,\,1)$ such that 
\begin{equation}
\int_{x\in\mathbb{R}^{2d}:V(x)\ge a}e^{-V(x)/\epsilon}\mathrm{d}x\le C_{a}e^{-a/\epsilon}\;.\label{eq:tightness stationary distrib}
\end{equation}
In particular, by the non-negativity of $U$ (cf. Assumption \ref{ass:non-negU}) which immediately implies the non-negativity of $V$, the partition function
$Z_{\epsilon}$ defined in \eqref{eq:partition} is finite. 
\end{prop}

\begin{proof}
Let us fix $a\ge0$. Since $U$ is non-negative by Assumption \ref{ass:non-negU},
it follows from \eqref{eq:nablaU_cond} that there exist constants
$c,\,M>0$ such that, for all $|q|\geq M$, 
\begin{equation}
\langle q,\,\nabla U(q)\rangle\geq c|q|^{2}\;.\label{eq:311}
\end{equation}
Hence, for $|q|\geq M$, we have 
\begin{align*}
U(q) & =U(0)+\int_{0}^{1}\langle q,\,\nabla U(sq)\rangle\mathrm{d}s\\
 & =U(0)+\int_{0}^{M/|q|}\langle q,\,\nabla U(sq)\rangle\mathrm{d}s+\int_{M/|q|}^{1}\langle q,\,\nabla U(sq)\rangle\mathrm{d}s\\
 & \geq U(0)-M\sup_{|z|\leq M}|\nabla U(z)|+\int_{M/|q|}^{1}\frac{c|sq|^{2}}{s}\mathrm{d}s\\
 & =U(0)-M\sup_{|z|\leq M}|\nabla U(z)|+\frac{c}{2}|q|^{2}-\frac{c}{2}M^{2}\;,
\end{align*}
where the inequality follows from the trivial bound 
\[
\langle q,\,\nabla U(sq)\rangle\ge-|q|\,|\nabla U(sq)|
\]
along with \eqref{eq:311}. In particular, this ensures that there
exists a constant $M_{a}\geq M$ such that 
\begin{equation}
U(q)\geq a+\frac{c}{4}|q|^{2}\;\;\;\text{for all }|q|\ge M_{a}\;.\label{eq:lbUq}
\end{equation}
Now, we have {[}recall the notation $x=(q,\,p)$ from \eqref{eq:xqp}{]}
\begin{equation}
\int_{x\in\mathbb{R}^{2d}:V(x)\ge a}e^{-V(x)/\epsilon}\mathrm{d}x\le\left[\int_{x\in\mathbb{R}^{2d}:V(x)\ge a,\,|q|<M_{a}}+\int_{x\in\mathbb{R}^{2d}:|q|\geq M_{a}}\right]e^{-V(x)/\epsilon}\mathrm{d}x\;.\label{eq:312}
\end{equation}
It follows from~\eqref{eq:lbUq} and the definition of $V$  
that, for all $|q|\geq M_{a}$ and $p\in\mathbb{R}^{d}$, 
\[
V(x)\geq a+\frac{c}{4}|q|^{2}+\frac{1}{2}|p|^{2}\;.
\]
Therefore, a direct computation reveals that the second integral in
the right-hand side of \eqref{eq:312} is bounded by $C\mathrm{e}^{-a/\epsilon}\epsilon^{d}$
for some constant $C>0$ independent of $\epsilon$. 

Now it remains to bound the first integral in the right-hand side
of \eqref{eq:312}. To that end, we start by decomposing this integral
into 
\[
 \left[\int_{x\in\mathbb{R}^{2d}:V(x)\ge a,\,|q|<M_{a},\,|p|<2\sqrt{a}}+\int_{x\in\mathbb{R}^{2d}:|q|<M_{a},\,|p|\geq2\sqrt{a}}\right]e^{-V(x)/\epsilon}\mathrm{d}x\;.
\]
We can bound the first integral at the right-hand side by $C_{a}e^{-a/\epsilon}$
for some constant $C_{a}>0$ independent of $\epsilon$ as the domain of integration is bounded. For the second integral at the right-hand
side, using a trivial bound 
\[
V(x)\geq\frac{1}{2}|p|^{2}\geq a+\frac{1}{4}|p|^{2}\;,
\]
for $|p|\geq2\sqrt{a}$, we get
\[
\int_{x\in\mathbb{R}^{2d}:|q|<M_{a},\,|p|\geq2\sqrt{a}}\leq\mathrm{e}^{-a/\epsilon}|\mathrm{B}_d (0,\,M_{a})|\int_{\mathbb{R}^{d}}\mathrm{e}^{-|p|^{2}/4\epsilon}\mathrm{d}p\leq C_{a}\mathrm{e}^{-a/\epsilon}\epsilon^{d/2}
\]
for some constant $C_{a}>0$ also independent of $\epsilon$, where $\mathrm{B}_d (0,\,M_{a})$ denotes the $d$-dimensional ball centered at the origin with radius $M_a$.  This completes the proof. 

The last assertion of the statement is immediate from the first one
since we can take $a=0$ in \eqref{eq:tightness stationary distrib}
and recall Assumption \ref{ass:non-negU}. 
\end{proof}

\subsection{Non-explosion and positive recurrence }

The argument developed here relies on the use of a suitable Lyapunov function for the underdamped Langevin process~\eqref{eq:sde} which was established in~\cite[Definition 4.1]{cutoff}. We now introduce this function for our specific setting. 
\begin{lem}
\label{lem:lyap}There exist a constant $\lambda>0$, a compact set $\mathcal{K}\subset\mathbb{R}^{2d}$,
and a $C^{2}$-function $H:\mathbb{R}^{2d}\rightarrow[0,\,\infty)$
such that
\begin{equation}
\mathcal{L}_{\epsilon}H(x)+\lambda H(x)\leq0\;\;\;\text{ for all }x\in\mathbb{R}^{2d}\setminus\mathcal{K}\;.\label{eq:lyapunov H}
\end{equation}
\end{lem}

\begin{proof}
By Assumption \eqref{ass:nabla U}, there exist constants $c,\,M_1 >0$
such that, for all $|q|\geq M_1$, 
\begin{equation}
\langle q,\,\nabla U(q)\rangle\geq c(|q|^{2}+U(q))\;.\label{eq:lower bound grad U}
\end{equation}
Let us now take a constant $\lambda\in(0,\,\gamma)$ small enough
(where $\gamma>0$ is the friction coefficient appeared in~\eqref{eq:sde})
so that 
\begin{equation}
\frac{\lambda(\gamma-\lambda)}{2}<c,\qquad\frac{2\lambda}{\gamma-\lambda}<c\;.\label{eq:lambda}
\end{equation}
Following the work in~\cite[Definition 4.1]{cutoff}, we define 
\begin{equation}
H(x):=\frac{1}{2}|p|^{2}+\frac{\gamma-\lambda}{2}\langle q,\,p\rangle+\frac{(\gamma-\lambda)^{2}}{4}|q|^{2}+U(q)\;.\label{def:H}
\end{equation}
Let us take a constant $M\in(M_{1},\,\infty)$ large enough so that
\begin{equation}
\frac{\gamma-\lambda}{2}\left(c-\lambda\frac{(\gamma-\lambda)}{2}\right)M^{2}\geq d\gamma\;.\label{eq:def M'}
\end{equation}
Then, let us define a constant $R>0$ large enough so that 
\begin{equation}
\frac{\gamma}{2}R^{2}\geq\frac{\gamma-\lambda}{2}\sup_{|q|\leq M}\left|\langle\nabla U(q),\,q\rangle-\frac{2\lambda}{\gamma-\lambda}U(q)-\lambda\frac{(\gamma-\lambda)}{2}|q|^{2}\right|+d\gamma\epsilon\;.\label{eq:lb_R}
\end{equation}
Let us now define the compact set 
\[
\mathcal{K}:=\{x=(q,\,p)\in\mathbb{R}^{2d}:|q|\leq M,\,|p|\leq R\}\;.
\]
A direct computation (see \cite[Lemma 4.3]{cutoff} for more details)
shows that for all $x\in\mathbb{R}^{2d}$, 
\[
\mathcal{L}_{\epsilon}H(x)+\lambda H(x)=-\frac{\gamma}{2}|p|^{2}-\frac{\gamma-\lambda}{2}\left(\langle\nabla U(q),\,q\rangle-\frac{2\lambda}{\gamma-\lambda}U(q)-\lambda\frac{(\gamma-\lambda)}{2}|q|^{2}\right)+d\gamma\epsilon\;.
\]

Now let $x\in\mathbb{R}^{2d}\setminus\mathcal{K}$ so that either
$|p|>R$, $|q|\leq M$ or $|q|>M$. In the former case, it follows
from \eqref{eq:lb_R} that 
\[
\mathcal{L}_{\epsilon}H(x)+\lambda H(x)\leq0\;.
\]
On the other hand, for the case $|q|>M$, by~\eqref{eq:lower bound grad U},
\eqref{eq:lambda}, and the assumption that $U$ is non-negative,
we get
\begin{align*}
\mathcal{L}_{\epsilon}H(x)+\lambda H(x) & \leq-\frac{\gamma-\lambda}{2}\left(\left(c-\frac{2\lambda}{\gamma-\lambda}\right)U(q)+\left(c-\lambda\frac{(\gamma-\lambda)}{2}\right)|q|^{2}\right)+d\gamma\epsilon\\
 & \leq-\frac{\gamma-\lambda}{2}\left(c-\lambda\frac{(\gamma-\lambda)}{2}\right)|q|^{2}+d\gamma\epsilon<0
\end{align*}
where the last inequality follows from \eqref{eq:def M'} (and the
assumption that $\epsilon\in(0,\,1)$).
\end{proof}
Henceforth, the function $H$ and the compact set $\mathcal{K}$ always
refer to the one appeared in the previous lemma. In the next proposition,
we verify that the underdamped Langevin process does not explode in
finite time and is positive recurrent under Assumption \ref{ass:nabla U}.
\begin{prop}
\label{prop:pos_rec} For all $\epsilon\in(0,\,1)$, the process $(X^{\epsilon}(t))_{t\geq0}$
defined in \eqref{eq:sde} does not explode in finite time almost-surely and is
positive recurrent. 
\end{prop}

\begin{proof}
We fix $\epsilon\in(0,\,1)$ throughout the proof. Let 
\[
C_{0}:=C_{0}(\epsilon)=\max\left\{ \sup_{x\in\mathcal{K}}\mathcal{L}_{\epsilon}H(x),\,0\right\} 
\]
and let $\widetilde{H}(x)=H(x)+C_{0}$. Since $H\ge0$ and since $\mathcal{L}_{\epsilon}H\le-\lambda H\le0$
on $\mathbb{R}^{2d}\setminus\mathcal{K}$ by Lemma \ref{lem:lyap},
we get 
\[
\mathcal{L}_{\epsilon}\widetilde{H}(x)=\mathcal{L}_{\epsilon}H(x)\leq C_{0}\leq\widetilde{H}(x)
\]
for all $x\in\mathbb{R}^{2d}$. Consequently, by \cite[page 197]{varadhan1980lectures},
the process \eqref{eq:sde} does not explode in finite time almost-surely.
In addition, one can then also deduce from the Lyapunov property in
\eqref{eq:lyapunov H} that the process $(X^{\epsilon}(t))_{t\geq0}$ is
positive recurrent by \cite[Theorem 6.1.3]{pinsky1995positive}. 
\end{proof}

\section{Proof of Proposition~\ref{prop:harm}\label{sec4:magicf} }

The main purpose of this section is to prove Proposition~\ref{prop:harm}. The first step is to consider a perturbation of the process~\eqref{eq:sde}.

\subsection{Perturbed underdamped Langevin process}

For $\alpha>0$, we define the perturbed underdamped Langevin process
$(X^{\alpha,\,\epsilon}(t)=(q^{\alpha,\,\epsilon}(t),\,p^{\alpha,\,\epsilon}(t)))_{t\ge0}$
by
\begin{equation}
\left\{ \begin{aligned} & \mathrm{d}q^{\alpha,\,\epsilon}(t)=p^{\alpha,\,\epsilon}(t)\mathrm{d}t-\alpha\nabla U(q^{\alpha,\,\epsilon}(t))\mathrm{d}t+\sqrt{2\alpha\epsilon}\mathrm{d}\widetilde{B}(t)\;,\\
 & \mathrm{d}p^{\alpha,\,\epsilon}(t)=-\nabla U(q^{\alpha,\,\epsilon}(t))\mathrm{d}t-\gamma p^{\alpha,\,\epsilon}(t)\mathrm{d}t+\sqrt{2\gamma\epsilon}\mathrm{d}B(t)\;,
\end{aligned}
\right.\label{eq:sde_pur}
\end{equation}
where $(\widetilde{B}(t))_{t\geq0}$ is a $d$-dimensional Brownian
motion in $\mathbb{R}^{d}$ independent of $(B(t))_{t\geq0}$. We
summarize the relevant information and notation regarding the process
$(X^{\alpha,\,\epsilon}(t))_{t\ge0}$.
\begin{itemize}
\item The infinitesimal generator associated with the process $(X^{\alpha,\,\epsilon}(t))_{t\geq0}$
can be written as 
\begin{equation}
\mathcal{L}_{\alpha,\,\epsilon}=\mathcal{L}_{\epsilon}+\alpha\left(-\langle\nabla U(q),\,\nabla_{q}\rangle+\epsilon\Delta_{q}\right)\;.\label{eq:gen_pur}
\end{equation}
We can rewrite the generator $\mathcal{L}_{\alpha,\,\epsilon}$ as
\[
\mathcal{L}_{\alpha,\,\epsilon}=\epsilon\mathrm{e}^{V(x)/\epsilon}\mathrm{div}\left(\mathrm{e}^{-V(x)/\epsilon}\mathbb{M}_{\alpha}\nabla\right)
\]
where 
\[
\mathbb{M}_{\alpha}=\begin{pmatrix}\alpha\mathbb{I}_{d} & \mathbb{I}_{d}\\
-\mathbb{I}_{d} & \gamma\mathbb{I}_{d}
\end{pmatrix}\in\mathbb{R}^{2d\times2d}\;.
\]
Therefore, the process $X^{\alpha,\,\epsilon}(t)$ is an overdamped
Langevin process studied in~\cite{LMS} and its generator $\mathcal{L}_{\alpha,\,\epsilon}$ is elliptic.
\item Following the same argument as in Proposition~\ref{prop:pos_rec},
we can readily show that the process $(X^{\alpha,\,\epsilon}(t))_{t\ge0}$
does not explode in finite time and is positive recurrent provided
that $\alpha$ and $\epsilon$ are sufficiently small. Indeed, the
Lyapunov function $H$ introduced in the proof of Proposition~\ref{prop:pos_rec}
also satisfies 
\[
\mathcal{L}_{\alpha,\,\epsilon}H+\lambda H\le0
\]
for small enough $\alpha$ and $\epsilon$, and thus we can repeat
the argument of Proposition~\ref{prop:pos_rec} to the process
$(X^{\alpha,\,\epsilon}(t))_{t\geq0}$. Henceforth, we always assume without
further mention that $\alpha$ and $\epsilon$ are sufficiently small
so that the non-explosion and positive recurrence of the process $(X^{\alpha,\,\epsilon}(t))_{t\ge0}$
are satisfied. Indeed, we only consider the regime $\alpha,\,\epsilon\rightarrow0$ in the rest of this article.
\item Recall the invariant measure $\mu_{\epsilon}$ defined in \eqref{eq:invm}.
Then, the adjoint operator $\mathcal{L}_{\alpha,\,\epsilon}^{*}$
of $\mathcal{L}_{\alpha,\,\epsilon}$ in $\mathrm{L}^{2}(\mu_{\epsilon}(x)\mathrm{d}x)$
is given by 
\begin{equation}
\mathcal{L}_{\alpha,\,\epsilon}^{*}=\epsilon\mathrm{e}^{V(x)/\epsilon}\mathrm{div}\left(\mathrm{e}^{-V(x)/\epsilon}\mathbb{M}_{\alpha}^{\dagger}\nabla\right)\;.\label{eq:gen_pur_adj}
\end{equation}
In particular one can check easily that for $x\in\mathbb{R}^{2d}$,
\[
\mathcal{L}_{\alpha,\,\epsilon}^{*}\big(\mathbf{1}_{\mathbb{R}^{2d}}\big)(x)=0
\]
such that the invariance of the distribution $\mu_{\epsilon}$ is preserved for the
process \eqref{eq:sde_pur} as well. 
\item Note that, if the initial position $x=(q,\,p)\in\mathbb{R}^{2d}$
is given, the laws of the processes $(X^{\epsilon}(t))_{t\ge0}$ and
$(X^{\alpha,\,\epsilon}(t))_{t\ge0}$ are totally determined by the
Brownian motions $(B(t))_{t\ge0}$ and $(\widetilde{B}(t))_{t\ge0}$.
We denote by $\mathbb{P}_{x}$, the law of coupling (by sharing $(B_{t})_{t\ge0}$
and $(\widetilde{B}(t))_{t\ge0}$) of all the processes $(X^{\epsilon}(t))_{t\ge0}$
and $(X^{\alpha,\,\epsilon}(t))_{t\ge0}$ starting from $x$ for any
$\alpha,\,\epsilon>0$. 
\item In addition, we henceforth always assume that, for all $\alpha,\,\epsilon>0$, the processes $(X^{\epsilon}(t))_{t\ge0}$
and $(X^{\alpha,\,\epsilon}(t))_{t\ge0}$ 
under consideration share the same Brownian motions $(B(t))_{t\ge0}$
and $(\widetilde{B}(t))_{t\ge0}$.
\end{itemize}
 
\subsection{Distance between $X^{\epsilon}(t)$ and $X^{\alpha,\,\epsilon}(t)$}

We can heuristically guess that the process $(X^{\alpha,\,\epsilon}(t))_{t\ge0}$
is close to $(X^{\epsilon}(t))_{t\ge0}$ when $\alpha$ is small.
In this subsection we quantify this observation. We start by a lemma
regarding the small time behaviour of the process $(X^{\epsilon}(t))_{t\ge0}$.
Recall the coupling $\mathbb{P}_{x}$ from the previous subsection. 
\begin{lem}
\label{lem:small time asymptotics} Let $E>0$ and let $x=(q,\,p)\in\mathbb{R}^{2d}$
satisfy $p\neq0$. 
\begin{enumerate}
\item Suppose that $V(x)=E$. Then, for all $\delta>0$, 
\[
\mathbb{P}_{x}\left[\sup_{t\in[0,\,\delta]}V(X^{\epsilon}(t))>E\quad\text{and}\quad\inf_{t\in[0,\,\delta]}V(X^{\epsilon}(t))<E\right]=1\;.
\]
\item Suppose that $z\in\mathbb{R}^{d}$ satisfies $|q-z|^{2}+|p|^{2}=E$.
Then, for all $\delta>0$, 
\[
\mathbb{P}_{x}\left[\sup_{t\in[0,\,\delta]}\left(|q^{\epsilon}(t)-z|^{2}+|p^{\epsilon}(t)|^{2}\right)>E\quad\text{and}\quad\inf_{t\in[0,\,\delta]}\left(|q^{\epsilon}(t)-z|^{2}+|p^{\epsilon}(t)|^{2}\right)<E\right]=1\;.
\]
\end{enumerate}
\end{lem}

\begin{proof}
(1) By It\^{o}'s formula, for all $t\geq0$, 
\begin{align*}
V(X^{\epsilon}(t))-E & =U(q^{\epsilon}(t))+\frac{1}{2}|p^{\epsilon}(t)|^{2}-E\\
 & =-\gamma\int_{0}^{t}|p^{\epsilon}(s)|^{2}\mathrm{d}s+\sqrt{2\gamma\epsilon}\int_{0}^{t}\langle p^{\epsilon}(s),\,\mathrm{d}B(s)\rangle+d\gamma\epsilon t\;.
\end{align*}
Therefore, for small times $t>0$, the dominant term
at the first order in time is $\sqrt{2\gamma\epsilon}\langle p,\,B(t)\rangle$
since $p\neq0$. In addition, due to its Brownian nature, this process
will visit infinitely often positive and negative values in any time
interval $[0,\,\delta]$ for $\delta>0$. This proves part (1).

\noindent (2) Since
\begin{align*}
\frac{1}{2}\mathrm{d}\left[|q^{\epsilon}(t)-z|^{2}+|p^{\epsilon}(t)|^{2}\right]
&=\left[\langle q^{\epsilon}(t)-z,\,p^{\epsilon}(t)\rangle+\langle p^{\epsilon}(t),\,-\gamma p^{\epsilon}(t)-\nabla U(q^{\epsilon}(t))\rangle+d\gamma\epsilon\right]\mathrm{d}t\\
&
+\sqrt{2\gamma\epsilon}\langle p^{\epsilon}(t),\mathrm{\,d}B(t)\rangle\;,
\end{align*}
by It\^{o}'s formula, we can write 
\[
|q^{\epsilon}(t)-z|^{2}+|p^{\epsilon}(t)|^{2}-E=\int_{0}^{t}F^{\epsilon}(s)\mathrm{d}s+2\sqrt{2\gamma\epsilon}\int_{0}^{t}\langle p^{\epsilon}(s),\,\mathrm{d}B(s)\rangle\;,
\]
where $t\mapsto F^{\epsilon}(t)$ is an almost-surely continuous function.
Again, the dominant term at the first order in time
is $2\sqrt{2\gamma\epsilon}\langle p,\,B(t)\rangle$ since $p\neq0$
which concludes the proof as in part (1). 
\end{proof}
Note that the result of the previous lemma is not trivial given the degeneracy of the 
process $(X^{\epsilon}(t))_{t\ge0}$. The condition
$p\neq0$ is crucial there to guarantee this result. 

Now let us provide an estimate on the distance between the processes $(X^{\epsilon}(t))_{t\ge0}$
and $(X^{\alpha,\,\epsilon}(t))_{t\ge0}$. The following notation will be used
frequently from now on. 
\begin{notation}
\label{not:hitting}For any measurable set $\mathcal{C}\subset\mathbb{R}^{2d}$
and $\alpha,\,\epsilon>0$, we denote by $\tau_{\mathcal{C}}^{\epsilon}$
and $\tau_{\mathcal{C}}^{\alpha,\,\epsilon}$ the hitting time of
$\mathcal{C}$ by the processes $(X^{\epsilon}(t))_{t\ge0}$
and $(X^{\alpha,\,\epsilon}(t))_{t\ge0}$, respectively. 
\end{notation}

\begin{lem}
\label{lem:gronw}Let $\mathcal{D}\subset\mathbb{R}^{2d}$ be a measurable set such that $\mathbb{R}^{2d}\setminus\mathcal{D}$ is bounded. Then,
there exists a constant $C=C(\mathcal{D})>0$ such that, for all $\epsilon,\,\alpha>0$
and $x\in\mathbb{R}^{2d}$, under the coupling $\mathbb{P}_{x}$,
namely under the assumption that the two processes $(X^{\epsilon}(t))_{t\ge0}$
and $(X^{\alpha,\,\epsilon}(t))_{t\ge0}$ simultaneously start at
$x\in\mathbb{R}^{2d}$ and share the Brownian motions $(B(t))_{t\ge0}$
and $(\widetilde{B}(t))_{t\ge0}$, it holds that 
\begin{equation}\label{eq:ineq gronwall process alpha eps}
    \left|X^{\alpha,\,\epsilon}(t)-X^{\epsilon}(t)\right|\leq\mathrm{e}^{Ct}\left[\alpha\int_{0}^{t}|\nabla U(q^{\epsilon}(r))|\mathrm{d}r+\sqrt{2\alpha\epsilon}\sup_{r\in[0,\,t]}|\widetilde{B}(r)|\right]
\end{equation}
 for all $t\in[0,\,\tau_{\mathcal{D}}^{\alpha,\;\epsilon}\land\tau_{\mathcal{D}}^{\epsilon}]$.
\end{lem}

\begin{proof}
Consider the process $(\widehat{X}^{\alpha,\,\epsilon}(t))_{t\geq0}=(\widehat{q}^{\alpha,\,\epsilon}(t),\widehat{p}^{\alpha,\,\epsilon}(t))_{t\geq0}$ where $\widehat{q}^{\alpha,\,\epsilon}(t)=q^{\alpha,\,\epsilon}(t)-q^{\epsilon}(t)$
and $\widehat{p}^{\alpha,\,\epsilon}(t)=p^{\alpha,\,\epsilon}(t)-p^{\epsilon}(t)$
so that by \eqref{eq:sde} and \eqref{eq:sde_pur}, we have
\begin{equation}
\left\{ \begin{aligned} & \mathrm{d}\widehat{q}^{\alpha,\,\epsilon}(t)=\widehat{p}^{\alpha,\,\epsilon}(t)\mathrm{d}t-\alpha\nabla U(q^{\alpha,\,\epsilon}(t))\mathrm{d}t+\sqrt{2\alpha\epsilon}\mathrm{d}\widetilde{B}(t)\;,\\
 & \mathrm{d}\widehat{p}^{\alpha,\,\epsilon}(t)=-\left[\nabla U(q^{\alpha,\,\epsilon}(t))-\nabla U(q^{\epsilon}(t))\right]\mathrm{d}t-\gamma\widehat{p}^{\alpha,\,\epsilon}(t)\mathrm{d}t\;.
\end{aligned}
\right.\label{eq:sde_pur-1}
\end{equation}
Let $t\in[0,\,\tau_{\mathcal{D}}^{\alpha,\;\epsilon}\land\tau_{\mathcal{D}}^{\epsilon}]$. Then, it is easy to see that there exists a constant $C>0$ independent of $\alpha$ and $\epsilon$ such that
\begin{align*}
|\widehat{X}^{\alpha,\,\epsilon}(t)|&\leq C\int_0^t|\widehat{X}^{\alpha,\,\epsilon}(r)|\mathrm{d}r+ \alpha\int_{0}^{t}|\nabla U(q^{\alpha,\,\epsilon}(r))|\mathrm{d}r+\int_0^t|\nabla U(q^{\alpha,\,\epsilon}(r))-\nabla U(q^{\epsilon}(r))|\mathrm{d}r\\
&+\sqrt{2\alpha\epsilon}\sup_{r\in[0,\,t]}|\widetilde{B}(r)|.
\end{align*}

Furthermore, given that $U\in C^2(\R^d,\R)$ and that $\mathbb{R}^{2d}\setminus\mathcal{D}$ is a bounded set, there exists $L=L(\mathcal{D})>0$ such that for all $x=(q,p),x'=(q',p')\in\mathbb{R}^{2d}\setminus\mathcal{D}$,
\[
\left|\nabla U(q)-\nabla U(q')\right|\le L|q-q'|\;.
\]
Injecting into the previous inequality involving $|\widehat{X}^{\alpha,\,\epsilon}(t)|$, one easily deduces~\eqref{eq:ineq gronwall process alpha eps} using Gr\"onwall's inequality.

\end{proof}

\subsection{Approximation of $h_{\mathcal{M}_{\epsilon},\,\mathcal{S}_{\epsilon}}^{*}$}

For a constant $N>0$ large enough so that 
\begin{equation}
\mathcal{M}_{\epsilon},\,\mathcal{S}_{\epsilon}\subset\textup{B}(0,\,N)\;\text{ for all }\epsilon\in(0,\,1)\;,\label{eq:condN}
\end{equation}
we define  
\begin{align}
\mathcal{O}_{N,\,\epsilon} & :=\mathcal{S}_{\epsilon}\cup\{x\in\mathbb{R}^{2d}:|x|>N\}\;\;\text{and}\label{eq:definition O_N}\\
\mathcal{D}_{N,\,\epsilon} & :=\textup{B}(0,\,N)\setminus(\overline{\mathcal{M}_{\epsilon}}\cup\overline{\mathcal{S}_{\epsilon}})\;.\nonumber 
\end{align}
As $\mathcal{L}_{\alpha,\,\epsilon}^{*}$ is an elliptic operator,
by~\cite[Theorems 8.3, 8.8 and 9.19]{GT} there exists a unique solution
$h_{\alpha,\,N,\,\epsilon}^{*}$ in $C^{2}(\mathcal{D}_{N,\,\epsilon})\cap C_{b}(\overline{\mathcal{D}_{N,\,\epsilon}})$
to the boundary-value problem 
\[
\left\{ \begin{aligned}\mathcal{L}_{\alpha,\,\epsilon}^{*}h_{\alpha,\,N,\,\epsilon}^{*}(x) & =0, &  & x\in\mathcal{D}_{N,\,\epsilon}\;,\\
h_{\alpha,\,N,\,\epsilon}^{*}(x) & =1, &  & x\in\partial\mathcal{M}_{\epsilon}\;,\\
h_{\alpha,\,N,\,\epsilon}^{*}(x) & =0, &  & x\in\partial\mathcal{O}_{N,\,\epsilon}\;,
\end{aligned}
\right.
\]
which corresponds to the equilibrium potential between $\mathcal{M}_{\epsilon}$
and $\mathcal{O}_{N,\,\epsilon}$ with respect to the diffusion process
associated with the generator $\mathcal{L}_{\alpha,\,\epsilon}^{*}$
defined in \eqref{eq:gen_pur_adj}. Furthermore, $h_{\alpha,\,N,\,\epsilon}^{*}$ can also be written as stated in the following lemma.
\begin{lem}\label{lem:expr h*} For all $x=(q,p)\in\mathcal{D}_{N,\,\epsilon}$,
\begin{equation}\label{eq:expr h alpha N eps}
h_{\alpha,\,N,\,\epsilon}^{*}(x)=\mathbb{P}_{(q,-p)}(\tau_{\mathcal{M}_{\epsilon}}^{\alpha,\,\epsilon}<\tau_{\mathcal{O}_{N,\,\epsilon}}^{\alpha,\,\epsilon})\,.
\end{equation} 
\end{lem}
\begin{proof}
Similarly, by~\cite[Theorems 8.3, 8.8 and 9.19]{GT} there also exists a unique solution
$h_{\alpha,\,N,\,\epsilon}$ in $C^{2}(\mathcal{D}_{N,\,\epsilon})\cap C_{b}(\overline{\mathcal{D}_{N,\,\epsilon}})$
to the boundary-value problem 
\begin{equation}\label{eq:boundary val probl}
    \left\{ \begin{aligned}\mathcal{L}_{\alpha,\,\epsilon}h_{\alpha,\,N,\,\epsilon}(x) & =0, &  & x\in\mathcal{D}_{N,\,\epsilon}\;,\\
h_{\alpha,\,N,\,\epsilon}(x) & =1, &  & x\in\partial\mathcal{M}_{\epsilon}\;,\\
h_{\alpha,\,N,\,\epsilon}(x) & =0, &  & x\in\partial\mathcal{O}_{N,\,\epsilon}\;.
\end{aligned}
\right.
\end{equation}
Moreover, by~\cite[Theorem 6.5.1]{Friedman}, $h_{\alpha,\,N,\,\epsilon}$ is given by 
$$h_{\alpha,\,N,\,\epsilon}(x)=\mathbb{P}_{x}(\tau_{\mathcal{M}_{\epsilon}}^{\alpha,\,\epsilon}<\tau_{\mathcal{O}_{N,\,\epsilon}}^{\alpha,\,\epsilon})\;.$$ 
Furthermore, using the definition of $h_{\alpha,\,N,\,\epsilon}^{*}$, one can easily check that the function 
$$(q,p)\in\R^{2d}\mapsto h_{\alpha,\,N,\,\epsilon}^{*}(q,-p)$$ is also a solution to the boundary-value problem~\eqref{eq:boundary val probl}. Therefore, by the uniqueness of such solution, the equality~\eqref{eq:expr h alpha N eps} immediately follows. 
\end{proof}
From now on, the function $h_{\alpha,\,N,\,\epsilon}^{*}$ is extended to $\R^{2d}$ under the equality~\eqref{eq:expr h alpha N eps}.
Now recall the equilibrium potential $h_{\mathcal{M}_{\epsilon},\,\mathcal{S}_{\epsilon}}^{*}$
which is defined in~\eqref{eq:def h^*}. The next lemma shows that $h_{\alpha,\,N,\,\epsilon}^{*}$
indeed approximates this function in the regime $\alpha\rightarrow0$
and $N\rightarrow\infty$. 
\begin{lem}
\label{lem:conv h^*} For all $x\in\R^{2d}$, we have that
\[
\lim_{N\rightarrow\infty}\lim_{\alpha\rightarrow0}h_{\alpha,\,N,\,\epsilon}^{*}(x)=h_{\mathcal{M}_{\epsilon},\,\mathcal{S}_{\epsilon}}^{*}(x)\;.
\]
\end{lem}

\begin{proof}
By \eqref{eq:def h^*} and Lemma~\ref{lem:expr h*}, it is enough
to prove that for all $x\in\R^{2d}$,
\begin{equation}
\lim_{N\rightarrow\infty}\lim_{\alpha\rightarrow0}\mathbb{P}_{x}(\tau_{\mathcal{M}_{\epsilon}}^{\alpha,\,\epsilon}<\tau_{\mathcal{O}_{N,\,\epsilon}}^{\alpha,\,\epsilon})=\mathbb{P}_{x}(\tau_{\mathcal{\mathcal{M}}_{\epsilon}}^{\epsilon}<\tau_{\mathcal{S}_{\epsilon}}^{\epsilon})\;.\label{eq:dlim}
\end{equation}
Since this is immediate when $x\in\mathcal{M}_{\epsilon}$, we fix
$x\in\mathbb{R}^{2d}\setminus\mathcal{M}_{\epsilon}.$ We divide the
proof of \eqref{eq:dlim} into 
\begin{align}
 & \lim_{\alpha\rightarrow0}\mathbb{P}_{x}(\tau_{\mathcal{\mathcal{M}}_{\epsilon}}^{\alpha,\,\epsilon}<\tau_{\mathcal{O}_{N,\,\epsilon}}^{\alpha,\,\epsilon})=\mathbb{P}_{x}(\tau_{\mathcal{\mathcal{M}}_{\epsilon}}^{\epsilon}<\tau_{\mathcal{O}_{N,\,\epsilon}}^{\epsilon})\;\;\;\text{and}\label{eq:conv1}\\
 & \lim_{N\rightarrow\infty}\mathbb{P}_{x}(\tau_{\mathcal{M}_{\epsilon}}^{\epsilon}<\tau_{\mathcal{O}_{N,\,\epsilon}}^{\epsilon})=\mathbb{P}_{x}(\tau_{\mathcal{M}_{\epsilon}}^{\epsilon}<\tau_{\mathcal{\mathcal{S}}_{\epsilon}}^{\epsilon})\;.\label{eq:conv2}
\end{align}
We note that the definition of $\mathcal{O}_{N,\,\epsilon}$ \eqref{eq:definition O_N}
and the positive recurrence of $(X^{\epsilon}(t))_{t\geq0}$ established in
Proposition \ref{prop:pos_rec} ensure that $\mathbf{1}\{\tau_{\mathcal{M}_{\epsilon}}^{\epsilon}<\tau_{\mathcal{O}_{N,\,\epsilon}}^{\epsilon}\}$
converges to $\mathbf{1}\{\tau_{\mathcal{M}_{\epsilon}}^{\epsilon}<\tau_{\mathcal{\mathcal{S}}_{\epsilon}}^{\epsilon}\}$
almost surely as $N\rightarrow\infty$, and hence \eqref{eq:conv2}
follows from the dominated convergence theorem. 

We now turn to \eqref{eq:conv1}. We fix $N>0$. Then, by Lemma \ref{lem:gronw},
there exists $C=C(N)$ such that 
\begin{equation}
\left|X^{\alpha,\,\epsilon}(t)-X^{\epsilon}(t)\right|\leq\mathrm{e}^{Ct}\left[\alpha\int_{0}^{t}|\nabla U(q^{\epsilon}(r))|\mathrm{d}r+\sqrt{2\alpha\epsilon}\sup_{r\in[0,\,t]}|\widetilde{B}(r)|\right]\;.\label{eq:grb}
\end{equation}
for all $t\in[0,\,\tau_{\mathcal{O}_{N+1,\,\epsilon}}^{\alpha,\,\epsilon}\land\tau_{\mathcal{O}_{N+1,\,\epsilon}}^{\epsilon}]$. 

First, under the event $\{\tau_{\mathcal{M}_{\epsilon}}^{\epsilon}<\tau_{\mathcal{O}_{N,\,\epsilon}}^{\epsilon}\}$,
let us take $\tau=\frac{1}{2}(\tau_{\mathcal{M}_{\epsilon}}^{\epsilon}+\tau_{\mathcal{O}_{N,\,\epsilon}}^{\epsilon})\in(\tau_{\mathcal{M}_{\epsilon}}^{\epsilon},\,\tau_{\mathcal{O}_{N,\,\epsilon}}^{\epsilon})$.
By \cite[Lemma 3.1]{Hill}, the process \eqref{eq:sde} attains the
boundary $\partial\mathcal{M}_{\epsilon}$ with non-zero velocity.
Therefore, by Lemma \ref{lem:small time asymptotics}-(2), it visits
the interior of $\mathcal{M}_{\epsilon}$ after hitting the boundary
$\partial\mathcal{M}_{\epsilon}$ (notice here that $\tau_{\mathcal{M}_{\epsilon}}^{\epsilon}=\tau_{\partial\mathcal{M}_{\epsilon}}^{\epsilon}$
since we assume that $x\in\mathbb{R}^{2d}\setminus\mathcal{M}_{\epsilon}$
). Combining this observation with the fact that $\tau<\tau_{\mathcal{O}_{N,\,\epsilon}}^{\epsilon}$,
and the continuity of the trajectories of $(X^{\epsilon}(t))_{t\geq0}$, we
can conclude that 
\begin{equation}
\sup_{t\in[0,\,\tau]}\mathrm{d}(X^{\epsilon}(t),\,\mathcal{M}_{\epsilon}^{c})>0\;\;\;\text{and}\;\;\;\inf_{t\in[0,\,\tau]}\mathrm{d}(X^{\epsilon}(t),\,\mathcal{O}_{N,\,\epsilon})>0\;.\label{eq:grb1}
\end{equation}
Hence, by \eqref{eq:grb} and the $1$-Lipschitz continuity of the distance, for $\alpha>0$ small enough, we get 
\[
\sup_{t\in[0,\,\tau]}\mathrm{d}(X^{\alpha,\,\epsilon}(t),\,\mathcal{M}_{\epsilon}^{c})>0\;\;\;\text{and}\;\;\;\inf_{t\in[0,\,\tau]}\mathrm{d}(X^{\alpha,\,\epsilon}(t),\,\mathcal{O}_{N,\,\epsilon})>0\;.
\]
This ensures that $\tau_{\mathcal{M}_{\epsilon}}^{\alpha,\,\epsilon}<\tau<\tau_{\mathcal{O}_{N,\,\epsilon}}^{\alpha,\,\epsilon}$
holds for small enough $\alpha$. Hence, we can conclude that 
\begin{equation}
\mathbf{1}\{\tau_{\mathcal{M}_{\epsilon}}^{\epsilon}<\tau_{\mathcal{O}_{N,\,\epsilon}}^{\epsilon}\}\le\liminf_{\alpha\rightarrow0}\mathbf{1}\{\tau_{\mathcal{M}_{\epsilon}}^{\alpha,\,\epsilon}<\tau_{\mathcal{O}_{N,\,\epsilon}}^{\alpha,\,\epsilon}\}\;.\label{eq:indl1}
\end{equation}

On the other hand, under the event $\{\tau_{\mathcal{O}_{N,\,\epsilon}}^{\epsilon}<\tau_{\mathcal{M}_{\epsilon}}^{\epsilon}\}$,
let 
\[
\widehat{\tau}=\frac{1}{2}\left(\tau_{\mathcal{O}_{N,\,\epsilon}}^{\epsilon}+\tau_{\mathcal{O}_{N+1,\,\epsilon}}^{\epsilon}\land\tau_{\mathcal{M}_{\epsilon}}^{\epsilon}\right)\in\left(\tau_{\mathcal{O}_{N,\,\epsilon}}^{\epsilon},\;\tau_{\mathcal{O}_{N+1,\,\epsilon}}^{\epsilon}\land\tau_{\mathcal{M}_{\epsilon}}^{\epsilon}\right)\;.
\]
Then, by the same logic with the derivation of \eqref{eq:grb1}, we
have 
\[
\sup_{t\in[0,\,\widehat{\tau}]}\mathrm{d}(X^{\epsilon}(t),\,\mathcal{O}_{N,\,\epsilon}^{c})>0\;\;\;\text{and}\;\;\;\inf_{t\in[0,\,\widehat{\tau}]}\mathrm{d}(X^{\epsilon}(t),\,\mathcal{M}_{\epsilon})>0\;.
\]
Then, applying \eqref{eq:grb} as before, for $\alpha>0$ small enough,
we get 
\[
\sup_{t\in[0,\,\widehat{\tau}]}\mathrm{d}(X^{\alpha,\,\epsilon}(t),\,\mathcal{O}_{N,\,\epsilon}^{c})>0\;\;\;\text{and}\;\;\;\inf_{t\in[0,\,\widehat{\tau}]}\mathrm{d}(X^{\alpha,\,\epsilon}(t),\,\mathcal{M}_{\epsilon})>0\;,
\]
and consequently $\tau_{\mathcal{O}_{N,\,\epsilon}}^{\alpha,\,\epsilon}<\widehat{\tau}<\tau_{\mathcal{M}_{\epsilon}}^{\alpha,\,\epsilon}$.
This concludes that 
\begin{equation}
\mathbf{1}\{\tau_{\mathcal{O}_{N,\,\epsilon}}^{\epsilon}<\tau_{\mathcal{M}_{\epsilon}}^{\epsilon}\}\le\liminf_{\alpha\rightarrow0}\mathbf{1}\{\tau_{\mathcal{O}_{N,\,\epsilon}}^{\alpha,\,\epsilon}<\tau_{\mathcal{M}_{\epsilon}}^{\alpha,\,\epsilon}\}\;.\label{eq:indl2}
\end{equation}
Combining \eqref{eq:indl1} and \eqref{eq:indl2} implies that 
\[
\mathbf{1}\{\tau_{\mathcal{M}_{\epsilon}}^{\epsilon}<\tau_{\mathcal{O}_{N,\,\epsilon}}^{\epsilon}\}=\lim_{\alpha\rightarrow0}\mathbf{1}\{\tau_{\mathcal{M}_{\epsilon}}^{\alpha,\,\epsilon}<\tau_{\mathcal{O}_{N,\,\epsilon}}^{\alpha,\,\epsilon}\}
\]
and hence \eqref{eq:conv1} follows from the dominated convergence
theorem. 
\end{proof}
Let us now prove Proposition \ref{prop:average time perturbed Langevin}
based on the lemmas above.

\subsection{Perturbed version of Proposition~\ref{prop:harm}}
In this subsection, we prove Proposition~\ref{prop:cap} and an analogous of Proposition~\ref{prop:harm}
for the perturbed process~\eqref{eq:sde_pur} on the domains $\mathcal{M}_{\epsilon}$
and $\mathcal{O}_{N,\,\epsilon}$, which is stated in Proposition~\ref{prop:average time perturbed Langevin}. Let us start with the following result which implies Proposition~\ref{prop:cap} in particular.
\begin{prop}\label{prop:cv eq meas perturbed}
For all $\epsilon,\,\alpha>0$
and for all $N>0$ large enough so that \eqref{eq:condN} holds, there
exists a non-negative surface measure $\eta_{\alpha,\,N,\,\epsilon}$
on $\partial\mathcal{M}_{\epsilon}$ such that for any non-negative bounded measurable function $g$ in $\R^{2d}$,
\begin{equation}
\int_{\partial\mathcal{M}_{\epsilon}}\mathbb{E}_{x}\left(\int_{0}^{\tau_{\mathcal{O}_{N,\,\epsilon}}^{\alpha,\,\epsilon}}g(X^{\alpha,\,\epsilon}(r))\mathrm{d}r\right)\eta_{\alpha,\,N,\,\epsilon}(\mathrm{d}x)=\int_{\mathbb{R}^{2d}}h_{\alpha,\,N,\,\epsilon}^{*}(x)g(x)\mu_{\epsilon}(x)\mathrm{d}x\;.\label{eq:ipp elliptic 2}
\end{equation}
Furthermore, there exists a non-negative surface measure $\eta_\epsilon$ on $\partial\mathcal{M}_{\epsilon}$ such that
\begin{equation}\label{eq:weak cv measures}
\lim_{N\rightarrow\infty}\lim_{\alpha\rightarrow0}\eta_{\alpha,\,N,\,\epsilon}=\eta_\epsilon\;\;\text{[weakly]}\,,
\end{equation}
and $\eta_\epsilon$ is the unique measure satisfying~\eqref{eq:def eq measure}.
\end{prop}
\begin{proof}
We fix $\epsilon>0$ in the proof. It was shown in~\cite[Equation 3.16]{LMS} that for elliptic processes like~\eqref{eq:sde_pur}, there exists a non-negative surface measure $\eta_{\alpha,\,N,\,\epsilon}$ on $\partial\mathcal{M}_{\epsilon}$ such that for any smooth function
$f\in C_{c}^{\infty}(\overline{\mathcal{S}_{\epsilon}}^{c})$, 
\begin{equation}
\int_{\partial\mathcal{M}_{\epsilon}}f(x)\eta_{\alpha,\,N,\,\epsilon}(\mathrm{d}x)=\int_{\mathbb{R}^{2d}}h_{\alpha,\,N,\,\epsilon}^{*}(x)(-\mathcal{L}_{\alpha,\,\epsilon}f(x))\mu_\epsilon(x)\mathrm{d}x\;.\label{eq:ipp elliptic}
\end{equation}
The identity above can also be written in the form of~\eqref{eq:ipp elliptic 2}, see~\cite[Proposition 3.3]{LMS}. However, the formula stated there is only valid when $g$ is a H\"{o}lderian-continuous function on $\mathbb{R}^{2d}$. Nonetheless, using the monotone convergence theorem, this identity can immediately be extended to any non-negative bounded measurable function $g$ in $\mathbb{R}^{2d}$, hence~\eqref{eq:ipp elliptic 2}. To conclude the proof, it remains therefore to show the weak convergence~\eqref{eq:weak cv measures} as well as~\eqref{eq:def eq measure}.\medskip{}

\noindent \textbf{Step 1:} There exists a constant $K_{\epsilon}>0$
such that $\eta_{\alpha,\,N,\,\epsilon}(\partial\mathcal{M}_{\epsilon})<K_{\epsilon}$
for all $\alpha\in(0,\,1)$ and $N>0$.

\noindent Let us take $f\in C_{c}^{\infty}(\overline{\mathcal{S}_{\epsilon}}^{c})$
independent of $\alpha$, $N$ such that $f\equiv1$ on $\partial\mathcal{M}_{\epsilon}$.
It follows from~\eqref{eq:ipp elliptic} and the trivial bound $0\leq h_{\alpha,\,N,\,\epsilon}^{*}\leq1$
which comes immediately from Lemma~\ref{lem:expr h*} that 
\[
\eta_{\alpha,\,N,\,\epsilon}(\partial\mathcal{M}_{\epsilon})\leq\Vert\mathcal{L}_{\alpha,\,\epsilon}f\Vert_{\infty}\;.
\]
This step is concluded since 
\[
\Vert\mathcal{L}_{\alpha,\,\epsilon}f\Vert_{\infty}\leq\Vert\mathcal{L}_{\epsilon}f\Vert_{\infty}+\alpha\Vert-\nabla U(q)\cdot\nabla_{q}f+\epsilon\Delta_{q}f\Vert_{\infty}\;.
\]  

\noindent \textbf{Step 2:} Let us now prove that for any $f\in C_b(\partial\M_\epsilon)$, the following limit exists: 
\[\lim_{N\rightarrow\infty}\lim_{\alpha\rightarrow0}\int_{\partial\M_\epsilon}f(x)\eta_{\alpha,N,\epsilon}(\mathrm{d}x)\;.\]
Assume that there exists some $f\in C_b(\partial\M_\epsilon)$ such that $\int_{\partial\M_\epsilon}f(x)\eta_{\alpha,N,\epsilon}(\mathrm{d}x)$ does not converge when $\alpha\rightarrow0$ and $N\rightarrow\infty$. Since the integral is bounded independently of $\alpha,N$, it admits at least two accumulation points $l_{1,\epsilon}\neq l_{2,\epsilon}$. By Tietze-Urysohn's extension theorem, the function $f$ can be extended by a compactly supported continuous function on $\overline{\mathcal{S}_{\epsilon}}^{c}$.

Let $\delta\in(0,|l_{1,\epsilon}-l_{2,\epsilon}|/4K_\epsilon)$. Since $f$ is continuous with a compact support in $\overline{\mathcal{S}_{\epsilon}}^{c}$, there exists a function $\phi\in C_c^\infty(\overline{\mathcal{S}_{\epsilon}}^{c})$ such that $||f-\phi||_\infty\leq\delta$. Notice that  
$$c_\epsilon:=\lim_{N\rightarrow\infty}\lim_{\alpha\rightarrow0}\int_{\partial\M_\epsilon}\phi(x)\eta_{\alpha,N,\epsilon}(\mathrm{d}x)$$ 
exists by Lemma~\ref{lem:conv h^*} and~\eqref{eq:ipp elliptic} since $\phi\in C^\infty_c(\overline{\mathcal{S}_{\epsilon}}^{c})$. As a result, 
$$\limsup_{N\rightarrow\infty,\alpha\rightarrow0}\left|\int_{\partial\M_\epsilon}f(x)\eta_{\alpha,N,\epsilon}(\mathrm{d}x)-c_\epsilon\right|\leq\delta K_\epsilon\leq|l_{1,\epsilon}-l_{2,\epsilon}|/4\;.$$ 
Taking appropriate subsequences when $N\rightarrow\infty,\alpha\rightarrow0$ one has that
$$\left|l_{1,\epsilon}-c_\epsilon\right|\leq|l_{1,\epsilon}-l_{2,\epsilon}|/4,\quad\text{and}\quad\left|l_{2,\epsilon}-c_\epsilon\right|\leq|l_{1,\epsilon}-l_{2,\epsilon}|/4,$$
which implies that $|l_{1,\epsilon}-l_{2,\epsilon}|\leq|l_{1,\epsilon}-l_{2,\epsilon}|/2$ contradicting the fact that $l_{1,\epsilon}\neq l_{2,\epsilon}$. Therefore, there exists a unique accumulation point when $\alpha\rightarrow0,N\rightarrow\infty$ for the integral $\int_{\partial\M_\epsilon}f(x)\eta_{\alpha,N,\epsilon}(\mathrm{d}x)$, which therefore converges.

\medskip{}

\noindent \textbf{Step 3:} The family of measures $(\eta_{\alpha,\,N,\,\epsilon})_{\alpha\in(0,\,1),\,N>0}$
converges weakly to a bounded non-negative surface measure $\eta_{\epsilon}$
on $\partial\mathcal{M}_{\epsilon}$ as $\alpha\rightarrow0$ and
$N\rightarrow\infty$.

\noindent By Step 2, we can define a bounded (by Step 1) linear operator
\[
T:f\in C_{b}(\partial\mathcal{M}_{\epsilon})\mapsto\lim_{N\rightarrow\infty}\lim_{\alpha\rightarrow0}\int_{\partial\mathcal{M}_{\epsilon}}f(x)\eta_{\alpha,\,N,\,\epsilon}(\mathrm{d}x)\;.
\]
Then, by the Riesz--Markov--Kakutani representation theorem, there exists
a measure $\eta_{\epsilon}$ on $\partial\mathcal{M}_{\epsilon}$
such that, for all $f\in C_{b}(\partial\mathcal{M}_{\epsilon})$,
\[
T(f)=\int_{\partial\mathcal{M}_{\epsilon}}f(x)\eta_{\epsilon}(\mathrm{d}x)\;.
\]
In particular, the limit~\eqref{eq:weak cv measures} holds.
\medskip 

\noindent\textbf{Step 4}: $\eta_\epsilon$ is the unique measure satisfying~\eqref{eq:def eq measure}. In fact, the uniqueness is obvious and the existence immediately follows by taking the limit of~\eqref{eq:ipp elliptic} when $N\rightarrow\infty,\,\alpha\rightarrow0$.
\end{proof}

\begin{prop}
\label{prop:average time perturbed Langevin}For all $\epsilon,\,\alpha>0$
and for all $N>0$ large enough so that \eqref{eq:condN} holds, there
exists a surface probability measure $\nu_{\alpha,\,N,\,\epsilon}$
on $\partial\mathcal{M}_{\epsilon}$ and a constant $C_{\alpha,\,N,\,\epsilon}>0$
such that, for all $M>0$, 
\begin{equation}
\int_{\partial\mathcal{M}_{\epsilon}}\mathbb{E}_{x}(\tau_{\mathcal{O}_{N,\,\epsilon}}^{\alpha,\,\epsilon}\land M)\nu_{\alpha,\,N,\,\epsilon}(\mathrm{d}x)=\frac{1}{C_{\alpha,\,N,\,\epsilon}}\int_{\mathbb{R}^{2d}}h_{\alpha,\,N,\,\epsilon}^{*}(x)\mathbb{P}_{x}(\tau_{\mathcal{O}_{N,\,\epsilon}}^{\alpha,\,\epsilon}\leq M)\mu_{\epsilon}(x)\mathrm{d}x\;.\label{eq:magic_pur}
\end{equation}
Furthermore, it holds that 
\begin{align*}
 & \lim_{N\rightarrow\infty}\lim_{\alpha\rightarrow0}C_{\alpha,\,N,\,\epsilon}=\mathrm{Cap}_{\epsilon}(\mathcal{M}_{\epsilon},\,\mathcal{S}_{\epsilon})\;\;\;\text{and}\\
 & \lim_{N\rightarrow\infty}\lim_{\alpha\rightarrow0}\nu_{\alpha,\,N,\,\epsilon}=\nu_{\epsilon}\;\;\textrm{[weakly]}
\end{align*}
where $\mathrm{Cap}_\epsilon(\mathcal{M}_{\epsilon},\,\mathcal{S}_{\epsilon})$
is the capacity defined in \eqref{eq:cap}, $\nu_{\epsilon}=\eta_\epsilon/\mathrm{Cap}_\epsilon(\mathcal{M}_{\epsilon},\,\mathcal{S}_{\epsilon})$ where $\eta_\epsilon$ is
the limiting surface measure appearing in~\eqref{eq:weak cv measures}.
\end{prop}

\begin{proof}
Let us define $C_{\alpha,\,N,\,\epsilon}:=\eta_{\alpha,\,N,\,\epsilon}(\partial\mathcal{M}_{\epsilon})$
and define the probability measure $\nu_{\alpha,\,N,\,\epsilon}$
on $\partial\mathcal{M}_{\epsilon}$ as 
\[
\nu_{\alpha,\,N,\,\epsilon}(\cdot)=\frac{1}{C_{\alpha,\,N,\,\epsilon}}\eta_{\alpha,\,N,\,\epsilon}(\cdot)\;.
\]
The limits of $C_{\alpha,\,N,\,\epsilon}$ and $\nu_{\alpha,\,N,\,\epsilon}$ immediately follow from Proposition~\ref{prop:cv eq meas perturbed}. It only remains to prove the equality~\eqref{eq:magic_pur}. In order to do that, let us fix
$M>0$ and apply the identity~\eqref{eq:ipp elliptic 2} to the function $g(x)=\mathbb{P}_{x}(\tau_{\mathcal{O}_{N,\,\epsilon}}^{\alpha,\,\epsilon}\leq M)$. Then, by the Markov property and the Fubini theorem,
we have 
\begin{align*}
\mathbb{E}_{x}\left[\int_{0}^{\tau_{\mathcal{O}_{N,\,\epsilon}}^{\alpha,\,\epsilon}}g(X^{\alpha,\,\epsilon}(r))\mathrm{d}r\right] & =\int_{0}^{\infty}\mathbb{E}_{x}\left[\mathbf{1}_{\tau_{\mathcal{O}_{N,\,\epsilon}}^{\alpha,\,\epsilon}>r}\mathbb{P}_{X^{\alpha,\,\epsilon}(r)}\left(\tau_{\mathcal{O}_{N,\,\epsilon}}^{\alpha,\,\epsilon}\leq M\right)\right]\mathrm{d}r\\
 & =\int_{0}^{\infty}\mathbb{P}_{x}\left(r<\tau_{\mathcal{O}_{N,\,\epsilon}}^{\alpha,\,\epsilon}\leq r+M\right)\mathrm{d}r=\mathbb{E}_{x}\left[\tau_{\mathcal{O}_{N,\,\epsilon}}^{\alpha,\,\epsilon}\land M\right]\;.
\end{align*}
Reinjecting into \eqref{eq:ipp elliptic 2} and dividing both sides by $C_{\alpha,\,N,\,\epsilon}$ yields~\eqref{eq:magic_pur}.
\end{proof}

\subsection{Proof of Proposition~\ref{prop:harm}}

In this subsection, we prove Proposition~\ref{prop:harm} based on
Proposition~\ref{prop:average time perturbed Langevin}. We first
introduce several lemmas. 
\begin{lem}
\label{lem:unif_c} For all $\epsilon\in(0,\,1)$, $r>0$ and $N>0$
large enough, we have that
\[
\lim_{\alpha\rightarrow0}\sup_{x\in\partial\mathcal{M}_{\epsilon}}\left|\mathbb{P}_{x}(\tau_{\mathcal{O}_{N,\,\epsilon}}^{\alpha,\,\epsilon}>r)-\mathbb{P}_{x}(\tau_{\mathcal{O}_{N,\,\epsilon}}^{\epsilon}>r)\right|=0\;.
\]
\end{lem}

The proof of Lemma~\ref{lem:unif_c} is postponed to Section~\ref{sec5:lem_unifc}. 
\begin{notation}
For $x\in\mathbb{R}^{2d}$, we denote here by $(X^{\epsilon}(t;\,x)=(q^{\epsilon}(t;\,x),\,p^{\epsilon}(t;\,x))_{t\geq0}$
the process satisfying \eqref{eq:sde} and starting from $x$ (i.e.,
$X^{\epsilon}(0;\,x)=x$). Then, we can couple all the processes $\{(X^{\epsilon}(t;\,x))_{t\geq0}:x\in\mathbb{R}^{2d}\}$
together by using (sharing) the same Brownian motion $(B(t))_{t\ge0}$.
We shall always assume that these processes are coupled in this manner.
Finally, for $x\in\mathbb{R}^{2d}$ and a measurable set $\mathcal{C}\subset\mathbb{R}^{2d}$,
we denote by $\tau_{\mathcal{C}}^{\epsilon}(x)$ the hitting time
of $\mathcal{C}$ by the process $(X^{\epsilon}(t;\,x))_{t\geq0}$.
\end{notation}

\begin{lem}
\label{lem:gr2}Let $\mathcal{D}\subset\mathbb{R}^{2d}$ be a measurable set such that $\mathbb{R}^{2d}\setminus\mathcal{D}$
is bounded. Then, there exists a constant $C(\mathcal{D})>0$ such that for all $x,\,y\in\mathbb{R}^{2d}\setminus\mathcal{D}$,
\[
\left|X^{\epsilon}(t;\,x)-X^{\epsilon}(t;\,y)\right|\leq|x-y|\mathrm{e}^{C(\mathcal{D})t}\;\;\;\;\text{for }t\in[0,\,\tau_{\mathcal{D}}^{\epsilon}(x)\land\tau_{\mathcal{D}}^{\epsilon}(y)]
\]
for all $\epsilon>0$. 
\end{lem}

\begin{proof}
Fix $x,\,y\in\mathbb{R}^{2d}\setminus\mathcal{D}$. Writing $\widehat{q}^{\epsilon}(t)=q^{\epsilon}(t;\,x)-q^{\epsilon}(t;\,y)$
and $\widehat{p}^{\epsilon}(t)=p^{\epsilon}(t;\,x)-p^{\epsilon}(t;\,y)$,
we derive from \eqref{eq:sde} that 
\begin{equation}
\left\{ \begin{aligned} & \mathrm{d}\widehat{q}^{\epsilon}(t)=\widehat{p}^{\epsilon}(t)\mathrm{d}t\;,\\
 & \mathrm{d}\widehat{p}^{\epsilon}(t)=-\left[\nabla U(q^{\epsilon}(t;\,x))-\nabla U(q^{\epsilon}(t;\,y))\right]\mathrm{d}t-\gamma\widehat{p}^{\epsilon}(t)\mathrm{d}t\;.
\end{aligned}
\right.\label{eq:gr21}
\end{equation}
Furthermore, given that $U\in C^2(\R^d,\R)$ and that $\mathbb{R}^{2d}\setminus\mathcal{D}$ is a bounded set, there exists $L=L(\mathcal{D})>0$ such that for all $x=(q,p),x'=(q',p')\in\mathbb{R}^{2d}\setminus\mathcal{D}$,
\begin{equation}\label{eq:gr22}
    \left|\nabla U(q)-\nabla U(q')\right|\le L|q-q'|\;.
\end{equation}
By \eqref{eq:gr21} and \eqref{eq:gr22}, and the fact that $X^{\epsilon}(t;\,x),\,X^{\epsilon}(t;\,y)\in\mathbb{R}^{2d}\setminus\mathcal{D}$
for $t\in[0,\,\tau_{\mathcal{D}}^{\epsilon}(x)\land\tau_{\mathcal{D}}^{\epsilon}(y)]$,
the statement of this lemma immediately follows from Gr\"onwall's inequality.
\end{proof}
\begin{lem}
\label{lem:continuity probability} For all $\epsilon\in(0,\,1)$,
$r>0$ and $N>0$ large enough, the functions $F_{1},\,F_{2}:\mathbb{R}^{2d}\rightarrow[0,\,1]$
defined by 
\[
F_{1}(x)=\mathbb{P}_{x}(\tau_{\mathcal{O}_{N,\,\epsilon}}^{\epsilon}>r)\;\;\;\text{and\;\;\;}F_{2}(x)=\mathbb{P}_{x}(\tau_{\mathcal{S}_{\epsilon}}^{\epsilon}>r)
\]
are continuous. 
\end{lem}

\begin{proof}
Let us fix $\epsilon\in(0,\,1)$, $r>0$ and $N>0$. We first look
at the function $F_{1}.$ Since $F_{1}\equiv0$ on $\mathcal{O}_{N,\,\epsilon}$,
it suffices to consider the continuity on $\mathbb{R}^{2d}\setminus\mathcal{O}_{N,\,\epsilon}.$
By Lemma \ref{lem:gr2}, there exists a constant $C=C(N)$ such that
for all $x,\,y\in\mathbb{R}^{2d}\setminus\mathcal{O}_{N,\,\epsilon}$,
we have
\begin{equation}
\left|X^{\epsilon}(t;\,x)-X^{\epsilon}(t;\,y)\right|\leq|x-y|\mathrm{e}^{Ct}\;\;\;\text{for }t\in[0,\,\tau_{\mathcal{O}_{N+1,\,\epsilon}}^{\epsilon}(x)\land\,\tau_{\mathcal{O}_{N+1},\,\epsilon}^{\epsilon}(y)]\;.\label{eq:grxy}
\end{equation}
Notice that for all $x\in\mathbb{R}^{2d}\setminus\mathcal{O}_{N,\,\epsilon}$,
$\mathbb{P}_{x}(\tau_{\mathcal{O}_{N,\,\epsilon}}^{\epsilon}=r)=0$
since $\partial\mathcal{O}_{N,\,\epsilon}$ has zero Lebesgue measure
and the process~\eqref{eq:sde} admits a transition density (see~\cite{Chaudru2022_preprint}).
Moreover, the process~\eqref{eq:sde} does not attain the boundary
$\partial\mathcal{O}_{N,\,\epsilon}$ with zero velocity by~\cite[Lemma 3.1]{Hill},
i.e. $\big|p^{\epsilon}(\tau_{\mathcal{O}_{N,\,\epsilon}}^{\epsilon})\big|>0$
almost-surely. Therefore, by Lemma~\ref{lem:small time asymptotics}
it visits the interior of the domain $\mathcal{O}_{N,\,\epsilon}$
in the interval $[\tau_{\mathcal{O}_{N,\,\epsilon}}^{\epsilon},\,\tau_{\mathcal{O}_{N,\,\epsilon}}^{\epsilon}+\delta]$
for any $\delta>0$.  

Hence, by \eqref{eq:grxy}, if $\tau_{\mathcal{O}_{N,\,\epsilon}}^{\epsilon}(y)<r$
then taking $|x-y|$ small enough ensures that $\tau_{\mathcal{O}_{N,\,\epsilon}}^{\epsilon}(x)<r$.
Thus, we have 
\[
\liminf_{x\rightarrow y}\mathbf{1}\{\tau_{\mathcal{O}_{N,\,\epsilon}}^{\epsilon}(x)<r\}\ge\mathbf{1}\{\tau_{\mathcal{O}_{N,\,\epsilon}}^{\epsilon}(y)<r\}
\]
and therefore, as $\mathbb{P}_{x}(\tau_{\mathcal{O}_{N,\,\epsilon}}^{\epsilon}=r)=0$
for all $r>0$ and $x\in\mathbb{R}^{2d}$, we get 
\begin{equation}
\limsup_{x\rightarrow y}\mathbb{P}_{x}(\tau_{\mathcal{O}_{N,\,\epsilon}}^{\epsilon}>r)\le\mathbb{P}_{y}(\tau_{\mathcal{O}_{N,\,\epsilon}}^{\epsilon}>r)\;.\label{eq:ellb}
\end{equation}
Similarly, if $\tau_{\mathcal{O}_{N,\,\epsilon}}^{\epsilon}(y)>r$
then taking $|x-y|$ small enough ensures that $\tau_{\mathcal{O}_{N,\,\epsilon}}^{\epsilon}(x)>r$.
Thus, we have 
\[
\liminf_{x\rightarrow y}\mathbf{1}\{\tau_{\mathcal{O}_{N,\,\epsilon}}^{\epsilon}(x)>r\}\ge\mathbf{1}\{\tau_{\mathcal{O}_{N,\,\epsilon}}^{\epsilon}(y)>r\}
\]
and therefore 
\begin{equation}
\liminf_{x\rightarrow y}\mathbb{P}_{x}(\tau_{\mathcal{O}_{N,\,\epsilon}}^{\epsilon}>r)\ge\mathbb{P}_{y}(\tau_{\mathcal{O}_{N,\,\epsilon}}^{\epsilon}>r)\;.\label{eq:ellb2}
\end{equation}
The continuity of $F_{1}$ follows from \eqref{eq:ellb} and \eqref{eq:ellb2}.
The proof of the continuity of the function $F_{2}$ is identical. 
\end{proof}
\begin{lem}
\label{lem:convergence proba until N}Let $x\in\mathbb{R}^{2d}$ and
$M>0$. Then, for all $\epsilon\in(0,\,1)$ and $N>0$, we have 
\begin{equation}
\lim_{\alpha\rightarrow0}\mathbb{P}_{x}(\tau_{\mathcal{O}_{N,\,\epsilon}}^{\alpha,\,\epsilon}\leq M)=\mathbb{P}_{x}(\tau_{\mathcal{O}_{N,\,\epsilon}}^{\epsilon}\leq M)\;.\label{eq:first limit alpha}
\end{equation}
Moreover, for all $\epsilon\in(0,\,1)$, we have 
\begin{equation}
\lim_{N\rightarrow\infty}\mathbb{P}_{x}(\tau_{\mathcal{O}_{N,\,\epsilon}}^{\epsilon}\leq M)=\mathbb{P}_{x}(\tau_{\mathcal{S}_{\epsilon}}^{\epsilon}\leq M)\;.\label{eq:convergence proba until M}
\end{equation}
\end{lem}

\begin{proof}
Let us first prove~\eqref{eq:first limit alpha} for $x\in\mathbb{R}^{2d}$
and $M>0$. The proof is similar to the proof of Lemma~\ref{lem:continuity probability}.
It relies on the application of the Gr\"onwall inequality obtained in Lemma~\ref{lem:gronw} which
provides the existence of a constant $C>0$ independent of $\alpha$
such that for all $t\in[0,\,\tau_{\mathcal{O}_{N,\,\epsilon}}^{\alpha,\,\epsilon}\land~\tau_{\mathcal{O}_{N,\,\epsilon}}^{\epsilon}]$,
\[
\left|X^{\alpha,\,\epsilon}(t)-X^{\epsilon}(t)\right|\leq\left(\alpha\int_{0}^{t}|\nabla U(q^{\epsilon}(r))|\mathrm{d}r+\sqrt{2\alpha\epsilon}\sup_{r\in[0,\,t]}|\widetilde{B}(r)|\right)\mathrm{e}^{Ct}\;.
\]
Then with the same reasoning as in Lemma~\ref{lem:continuity probability},
if $\tau_{\mathcal{O}_{N,\,\epsilon}}^{\epsilon}<M$ then taking $\alpha$
small enough ensures that $\tau_{\mathcal{O}_{N,\,\epsilon}}^{\alpha,\,\epsilon}<M$.
Also, if $\tau_{\mathcal{O}_{N,\,\epsilon}}^{\epsilon}>M$ then taking
$\alpha$ small enough ensures that $\tau_{\mathcal{O}_{N,\,\epsilon}}^{\alpha,\,\epsilon}>M$.
Notice that the events $\{\tau_{\mathcal{O}_{N,\,\epsilon}}^{\alpha,\,\epsilon}=M\}$,
$\{\tau_{\mathcal{O}_{N,\,\epsilon}}^{\epsilon}=M\}$ have zero probability
since the boundary has zero Lebesgue measure and the processes~\eqref{eq:sde},~\eqref{eq:sde_pur}
admit a transition density, hence the proof of~\eqref{eq:first limit alpha}.

Now consider the limit~\eqref{eq:convergence proba until M} of $\mathbb{P}_{x}(\tau_{\mathcal{O}_{N,\,\epsilon}}^{\epsilon}\leq M)$
for $M>0$ when $N\rightarrow\infty$. Let us define the first
exit time $\tau_{\mathrm{B}(0,\,N)^{c}}^{\epsilon}$ for the process~\eqref{eq:sde}
from the ball $\mathrm{B}(0,\,N)$. By definition of $\mathcal{O}_{N,\,\epsilon}$
in~\eqref{eq:definition O_N}, one has that 
\[
\tau_{\mathcal{O}_{N,\,\epsilon}}^{\epsilon}=\tau_{\mathcal{S}_{\epsilon}}^{\epsilon}\land\tau_{\mathrm{B}(0,\,N)^{c}}^{\epsilon}\;.
\]
Furthermore, by the non-explosion of the process~\eqref{eq:sde}
stated in Proposition~\ref{prop:pos_rec}, $\tau_{\mathrm{B}(0,\,N)^{c}}^{\epsilon}\underset{N\rightarrow\infty}{\longrightarrow}\infty$.
As a result, $\mathbb{P}_{x}(\tau_{\mathrm{B}(0,\,N)^{c}}^{\epsilon}\leq M)\underset{N\rightarrow\infty}{\longrightarrow}0$,
which ensures the limit~\eqref{eq:convergence proba until M}. 
\end{proof}
Now we are ready to prove Proposition~\ref{prop:harm}.
\begin{proof}[Proof of Proposition~\ref{prop:harm}]
By Proposition~\ref{prop:average time perturbed Langevin}, one
has for all $N,\,M>0$ large enough, 
\begin{equation}
\int_{\partial\mathcal{M}_{\epsilon}}\mathbb{E}_{x}(\tau_{\mathcal{O}_{N,\,\epsilon}}^{\alpha,\,\epsilon}\land M)\nu_{\alpha,\,N,\,\epsilon}(\mathrm{d}x)=\frac{1}{C_{\alpha,\,N,\,\epsilon}}\int_{\mathbb{R}^{2d}}h_{\alpha,\,N,\,\epsilon}^{*}(x)\mathbb{P}_{x}(\tau_{\mathcal{O}_{N,\,\epsilon}}^{\alpha,\,\epsilon}\leq M)\mu_{\epsilon}(x)\mathrm{d}x\;.\label{eq:average first exit time 2}
\end{equation}
Since 
\[
\int_{\partial\mathcal{M}_{\epsilon}}\mathbb{E}_{x}(\tau_{\mathcal{O}_{N,\,\epsilon}}^{\alpha,\,\epsilon}\land M)\nu_{\alpha,N,\epsilon}(\mathrm{d}x)=\int_{\partial\mathcal{M}_{\epsilon}}\int_{0}^{M}\mathbb{P}_{x}(\tau_{\mathcal{O}_{N,\,\epsilon}}^{\alpha,\,\epsilon}>r)\mathrm{d}r\,\nu_{\alpha,\,N,\,\epsilon}(\mathrm{d}x)\;,
\]
then 
\begin{align*}
 & \left|\int_{\partial\mathcal{M}_{\epsilon}}\mathbb{E}_{x}(\tau_{\mathcal{O}_{N,\,\epsilon}}^{\alpha,\,\epsilon}\land M)\nu_{\alpha,\,N,\,\epsilon}(\mathrm{d}x)-\int_{\partial\mathcal{M}_{\epsilon}}\mathbb{E}_{x}(\tau_{\mathcal{S}_{\epsilon}}^{\epsilon}\land M)\nu_{\epsilon}(\mathrm{d}x)\right|\\
 & \leq\int_{0}^{M}\sup_{x\in\partial\mathcal{M}_{\epsilon}}\left|\mathbb{P}_{x}(\tau_{\mathcal{O}_{N,\,\epsilon}}^{\alpha,\,\epsilon}>r)-\mathbb{P}_{x}(\tau_{\mathcal{O}_{N,\,\epsilon}}^{\epsilon}>r)\right|\mathrm{d}r\\
 & +\int_{0}^{M}\sup_{x\in\partial\mathcal{M}_{\epsilon}}\left|\mathbb{P}_{x}(\tau_{\mathcal{O}_{N,\,\epsilon}}^{\epsilon}>r)-\mathbb{P}_{x}(\tau_{\mathcal{S}_{\epsilon}}^{\epsilon}>r)\right|\mathrm{d}r\\
 & +\left|\int_{\partial\mathcal{M}_{\epsilon}}\int_{0}^{M}\mathbb{P}_{x}(\tau_{\mathcal{S}_{\epsilon}}^{\epsilon}>r)\mathrm{d}r\,\nu_{\alpha,\,N,\,\epsilon}(\mathrm{d}x)-\int_{\partial\mathcal{M}_{\epsilon}}\int_{0}^{M}\mathbb{P}_{x}(\tau_{\mathcal{S}_{\epsilon}}^{\epsilon}>r)\mathrm{d}r\,\nu_{\epsilon}(\mathrm{d}x)\right|\;.
\end{align*}
By Lemma~\ref{lem:unif_c}, 
\begin{align*}
 & \limsup_{\alpha\rightarrow0}\left|\int_{\partial\mathcal{M}_{\epsilon}}\mathbb{E}_{x}(\tau_{\mathcal{O}_{N,\,\epsilon}}^{\alpha,\,\epsilon}\land M)\nu_{\alpha,\,N,\,\epsilon}(\mathrm{d}x)-\int_{\partial\mathcal{M}_{\epsilon}}\mathbb{E}_{x}(\tau_{\mathcal{S}_{\epsilon}}^{\epsilon}\land M)\nu_{\epsilon}(\mathrm{d}x)\right|\\
 & \leq\int_{0}^{M}\sup_{x\in\partial\mathcal{M}_{\epsilon}}\left|\mathbb{P}_{x}(\tau_{\mathcal{O}_{N,\,\epsilon}}^{\epsilon}>r)-\mathbb{P}_{x}(\tau_{\mathcal{S}_{\epsilon}}^{\epsilon}>r)\right|\mathrm{d}r\\
 & +\limsup_{\alpha\rightarrow0}\left|\int_{\partial\mathcal{M}_{\epsilon}}\int_{0}^{M}\mathbb{P}_{x}(\tau_{\mathcal{S}_{\epsilon}}^{\epsilon}>r)\mathrm{d}r\,\nu_{\alpha,\,N,\,\epsilon}(\mathrm{d}x)-\int_{\partial\mathcal{M}_{\epsilon}}\int_{0}^{M}\mathbb{P}_{x}(\tau_{\mathcal{S}_{\epsilon}}^{\epsilon}>r)\mathrm{d}r\,\nu_{\epsilon}(\mathrm{d}x)\right|\;.
\end{align*}
Let us now take the limitsup above when $N\rightarrow\infty$. To
this extent, let us first notice that for all $x\in\partial\mathcal{M}_{\epsilon}$
and any $r>0$, $\mathbb{P}_{x}(\tau_{\mathcal{O}_{N,\,\epsilon}}^{\epsilon}>r)\underset{N\rightarrow\infty}{\longrightarrow}\mathbb{P}_{x}(\tau_{\mathcal{S}_{\epsilon}}^{\epsilon}>r)$
by the positive recurrence of the process~\eqref{eq:sde} in Proposition~\ref{prop:pos_rec}.
Additionally, since the function $x\in\partial\mathcal{M}_{\epsilon}\to\mathbb{P}_{x}(\tau_{\mathcal{O}_{N,\,\epsilon}}^{\epsilon}>r)$
is continuous on the compact set $\partial\mathcal{M}_{\epsilon}$
by Lemma~\ref{lem:continuity probability} and the probability $\mathbb{P}_{x}(\tau_{\mathcal{O}_{N,\,\epsilon}}^{\epsilon}>r)$
is an increasing function of $N\geq1$ then by Dini's theorem the
convergence is uniform in $\partial\mathcal{M}_{\epsilon}$, i.e.
\[
\sup_{x\in\partial\mathcal{M}_{\epsilon}}\left|\mathbb{P}_{x}(\tau_{\mathcal{O}_{N,\,\epsilon}}^{\epsilon}>r)-\mathbb{P}_{x}(\tau_{\mathcal{S}_{\epsilon}}^{\epsilon}>r)\right|\underset{N\rightarrow\infty}{\longrightarrow}0\;.
\]

Additionally, by the weak-convergence of $\nu_{\alpha,\,N,\,\epsilon}$
stated in Proposition~\ref{prop:average time perturbed Langevin}
and the continuity of the second probability in Lemma~\ref{lem:continuity probability}
one has that 
\[
\lim_{N\rightarrow\infty}\lim_{\alpha\rightarrow0}\int_{\partial\mathcal{M}_{\epsilon}}\mathbb{P}_{x}(\tau_{\mathcal{S}_{\epsilon}}^{\epsilon}>r)\,\nu_{\alpha,\,N,\,\epsilon}(\mathrm{d}x)=\int_{\partial\mathcal{M}_{\epsilon}}\mathbb{P}_{x}(\tau_{\mathcal{S}_{\epsilon}}^{\epsilon}>r)\,\nu_{\epsilon}(\mathrm{d}x)\;.
\]
Hence, 
\[
\lim_{N\rightarrow\infty}\lim_{\alpha\rightarrow0}\int_{\partial\mathcal{M}_{\epsilon}}\mathbb{E}_{x}(\tau_{\mathcal{O}_{N,\,\epsilon}}^{\alpha,\,\epsilon}\land M)\nu_{\alpha,\,N,\,\epsilon}(\mathrm{d}x)=\int_{\partial\mathcal{M}_{\epsilon}}\mathbb{E}_{x}(\tau_{\mathcal{S}_{\epsilon}}^{\epsilon}\land M)\nu_{\epsilon}(\mathrm{d}x)\;.
\]
As a result, taking the limit $N\rightarrow\infty$, $\alpha\rightarrow0$
in~\eqref{eq:average first exit time 2} ensures by Proposition~\ref{prop:average time perturbed Langevin},
Lemmas~\ref{lem:conv h^*} and~\ref{lem:convergence proba until N}
that 
\[
\int_{\partial\mathcal{M}_{\epsilon}}\mathbb{E}_{x}(\tau_{\mathcal{S}_{\epsilon}}^{\epsilon}\land M)\nu_{\epsilon}(\mathrm{d}x)=\frac{1}{C_{\epsilon}(\mathcal{M}_{\epsilon},\,\mathcal{S}_{\epsilon})}\int_{\mathbb{R}^{2d}}h_{\mathcal{M}_{\epsilon},\,\mathcal{S}_{\epsilon}}^{*}(x)\mathbb{P}_{x}(\tau_{\mathcal{S}_{\epsilon}}^{\epsilon}\leq M)\mu_{\epsilon}(x)\mathrm{d}x\;.
\]
Moreover, the positive recurrence of the process~\eqref{eq:sde}
in Proposition~\ref{prop:pos_rec} ensures that $\tau_{\mathcal{S}_{\epsilon}}^{\epsilon}<\infty$
almost surely. Therefore, taking the limit above when $M\rightarrow\infty$
by the monotone convergence theorem ensures~\eqref{eq:magicf}. 
\end{proof}

\section{Proof of Lemma~\ref{lem:unif_c}\label{sec5:lem_unifc}}

 For the sake of notation we shall denote here for $t\geq0$
the vectors $X^{\alpha,\,\epsilon}(t)=(q^{\alpha,\,\epsilon}(t),\,p^{\alpha,\,\epsilon}(t))$, $X^{\epsilon}(t)=(q^{\epsilon}(t),\,p^{\epsilon}(t))$
where the process $(X^{\alpha,\,\epsilon}(t))_{t\geq0}$ satisfies~\eqref{eq:sde_pur}
and $(X^{\epsilon}(t))_{t\geq0}$ satisfies~\eqref{eq:sde}. Let
us also fix $r>0$ and denote for $N>0$, the open set $\mathcal{O}_{N,\,\epsilon}$
defined in~\eqref{eq:definition O_N}. 

\begin{notation}\label{not:multiple cste}
From this moment on, different appearances of a constant may stand for different constants unless
otherwise mentioned. For instance, the constant $C(N)$ represents a constant depending on the parameter $N>0$, but different appearances of $C(N)$ may refer to different quantities.
\end{notation}

\subsection{Preliminary results} In this subsection, we state and prove some preliminary results related to the short-time variations of the processes $(X^{\epsilon}(t))_{t\geq0}$, $(X^{\alpha,\,\epsilon}(t))_{t\geq0}$ and to the distance between both processes when $\alpha\rightarrow0$.

\begin{lem}
\label{lem:variations process} For any $\beta\in(0,\,1/2)$, 
\[
\sup_{x\in\partial\mathcal{M}_{\epsilon}}\mathbb{P}_{x}\left(\sup_{|t_{1}-t_{2}|\leq\delta,\,t_{1},\,t_{2}\in\left[0,\,\tau_{\mathcal{O}_{N,\,\epsilon}}^{\epsilon}\land r\right]}|X^{\epsilon}(t_{1})-X^{\epsilon}(t_{2})|>\delta^{\beta}\right)\underset{\delta\rightarrow0}{\longrightarrow}0
\]
and
\[
\sup_{x\in\partial\mathcal{M}_{\epsilon},\,\alpha\in(0,\,1)}\mathbb{P}_{x}\left(\sup_{|t_{1}-t_{2}|\leq\delta,\,t_{1},\,t_{2}\in\left[0,\,\tau_{\mathcal{O}_{N,\,\epsilon}}^{\alpha,\,\epsilon}\land r\right]}|X^{\alpha,\,\epsilon}(t_{1})-X^{\alpha,\,\epsilon}(t_{2})|>\delta^{\beta}\right)\underset{\delta\rightarrow0}{\longrightarrow}0\;.
\]
\end{lem}

\begin{lem}
\label{lem:difference process alpha} For any $\beta\in(0,\,1/2)$,
\[
\sup_{x\in\partial\mathcal{M}_{\epsilon}}\mathbb{P}_{x}\left(\sup_{t\in\left[0,\,\tau_{\mathcal{O}_{N,\,\epsilon}}^{\alpha,\,\epsilon}\land\tau_{\mathcal{O}_{N,\,\epsilon}}^{\epsilon}\land r\right]}|X^{\alpha,\,\epsilon}(t)-X^{\epsilon}(t)|>\alpha^{\beta}\right)\underset{\alpha\rightarrow0}{\longrightarrow}0\;.
\]
\end{lem} 
\begin{proof}[Proof of Lemma~\ref{lem:variations process}]
Let us fix $\beta\in(0,1/2)$. Recall that the sample paths of a Brownian motion are, almost-surely, $\lambda$-H\"olderian for any $\lambda\in(0,\,1/2)$. Therefore, taking $\lambda\in(\beta,1/2)$ ensures the existence of a finite random variable $\mathcal{Z}>0$ such that 
\[
\sup_{|t_{1}-t_{2}|\leq\delta,\,t_{1},\,t_{2}\in\left[0,\,r\right]}|B(t_{1})-B(t_{2})|\leq\mathcal{Z}\delta^{\lambda}\;.
\]

Therefore, using It\^o's formula for~\eqref{eq:sde} and the boundedness of $\R^{2d}\setminus\mathcal{O}_{N,\,\epsilon}$, one has the existence of a constant $C=C(N)>0$ independent of $x\in\partial\M_\epsilon$ such that, $\mathbb{P}_x$ almost-surely,
for all $t_{1},\,t_{2}$ satisfying $|t_{1}-t_{2}|\leq\delta$  and
$t_{1},\,t_{2}\in\left[0,\,\tau_{\mathcal{O}_{N,\,\epsilon}}^{\epsilon}\land r\right]$,
\begin{align*}
|X^{\epsilon}(t_{1})-X^{\epsilon}(t_{2})|\leq C\delta+C\mathcal{Z}\delta^{\lambda}\leq C\delta^\lambda(1+\mathcal{Z})
\end{align*}
for $\delta\in[0,1)$. Therefore, 
$$\sup_{x\in\partial\M_\epsilon}\mathbb{P}_{x}\left(\sup_{|t_{1}-t_{2}|\leq\delta,\,t_{1},\,t_{2}\in\left[0,\,\tau_{\mathcal{O}_{N,\,\epsilon}}^{\epsilon}\land r\right]}|X^{\epsilon}(t_{1})-X^{\epsilon}(t_{2})|>\delta^{\beta}\right)\leq\mathbb{P}(C(1+\mathcal{Z})>\delta^{\beta-\lambda})$$
which vanishes when $\delta$ goes to zero since $\lambda>\beta$ and $\mathcal{Z}<\infty$ almost-surely. The proof of the second estimate in Lemma~\ref{lem:variations process} follows similarly by using as well the H\"olderian nature of the sample paths of the Brownian motion $(\widetilde{B}(t))_{t\geq0}$.
\end{proof} 
 
\begin{proof}[Proof of Lemma~\ref{lem:difference process alpha}]
Fix $\beta\in(0,1/2)$. By Lemma~\ref{lem:gronw}, there exists a constant $C=C(N)>0$ independent of $\alpha$ and $x\in\partial\M_\epsilon$ such that, $\mathbb{P}_x$ almost-surely, for all $t\in\left[0,\,\tau_{\mathcal{O}_{N,\,\epsilon}}^{\alpha,\,\epsilon}\land\tau_{\mathcal{O}_{N,\,\epsilon}}^{\epsilon}\land r\right]$,
\begin{align*}
\left|X^{\alpha,\,\epsilon}(t)-X^{\epsilon}(t)\right|&\leq\mathrm{e}^{Ct}\left[\alpha\int_{0}^{t}|\nabla U(q^{\epsilon}(u))|\mathrm{d}u+\sqrt{2\alpha\epsilon}\sup_{u\in[0,\,t]}|\widetilde{B}(u)|\right]\\
&\leq\mathrm{e}^{Ct}\left[C'\alpha t+\sqrt{2\alpha\epsilon}\sup_{u\in[0,\,t]}|\widetilde{B}(u)|\right]
\end{align*}
for some constant $C'=C'(N)$ independent of $\alpha$ and $x$. Therefore,
\begin{align*}
&\sup_{x\in\partial\mathcal{M}_{\epsilon}}\mathbb{P}_{x}\left(\sup_{t\in\left[0,\,\tau_{\mathcal{O}_{N,\,\epsilon}}^{\alpha,\,\epsilon}\land\tau_{\mathcal{O}_{N,\,\epsilon}}^{\epsilon}\land r\right]}|X^{\alpha,\,\epsilon}(t)-X^{\epsilon}(t)|>\alpha^{\beta}\right)\\
&\leq\mathbb{P}\left(\mathrm{e}^{Cr}\left[C'\alpha^{1-\beta}r+\alpha^{1/2-\beta}\sqrt{2\epsilon}\sup_{u\in[0,\,r]}|\widetilde{B}(u)|\right]>1\right)\underset{\alpha\rightarrow0}{\longrightarrow}0 
\end{align*}
since $\beta\in(0,1/2)$.  
\end{proof}

To conclude this subsection, the following lemma provides a $\mathrm{L}^2$ control on martingales appearing in Section~\ref{sec:proof of lemma 4.7}.

\begin{lem}\label{lem:control martingale}
 There exists a constant $C=C(N)$ independent of $\alpha$ such that for $\lambda>0$,
\begin{equation}\label{eq:first control martingale}
    \sup_{x\in\R^{2d}\setminus\mathcal{O}_{N,\,\epsilon}}\mathbb{E}_{x}\left[\sup_{t\in[0,\,\alpha^{\lambda}]}\left|\int_{0}^{t\land\tau_{\mathcal{O}_{N,\,\epsilon}}^{\alpha,\,\epsilon}}\left\langle q^{\alpha,\,\epsilon}(u)-s,\,\mathrm{d}\widetilde{B}(u)\right\rangle\right|^{2}\right]\leq C\alpha^\lambda
\end{equation}
where $s\in\R^d$ is defined in~\eqref{eq:neigh}. Also,
\begin{equation}\label{eq:second control martingale}
\sup_{x\in\R^{2d}\setminus\mathcal{O}_{N,\,\epsilon}}\mathbb{E}_{x}\left[\sup_{t\in[0,\,\alpha^{\lambda}]}\left|\int_{0}^{t\land\tau_{\mathcal{O}_{N,\,\epsilon}}^{\alpha,\,\epsilon}}\left\langle p^{\alpha,\,\epsilon}(u)-p,\,\mathrm{d}B(u)\right\rangle\right|^{2}\right]\leq C\alpha^{2\lambda}\,.
\end{equation}
\end{lem}
 
\begin{proof}
Recall the notation introduced in~\eqref{eq:xqp} and let $x=(q,p)\in\R^{2d}\setminus\mathcal{O}_{N,\,\epsilon}$. By Doob's martingale inequality and It\^o's isometry, 
\begin{align}
\mathbb{E}_{x}\left[\sup_{t\in[0,\,\alpha^{\lambda}]}\left|\int_{0}^{t\land\tau_{\mathcal{O}_{N,\,\epsilon}}^{\alpha,\,\epsilon}}\left\langle q^{\alpha,\,\epsilon}(u)-s,\,\mathrm{d}\widetilde{B}(u)\right\rangle\right|^{2}\right]\nonumber
&\leq4\int_{0}^{\alpha^{\lambda}}\mathbb{E}_{x}\left[\mathbf{1}_{\tau_{\mathcal{O}_{N,\,\epsilon}}^{\alpha,\,\epsilon}>u}\left|q^{\alpha,\,\epsilon}(u)-s\right|^{2}\right]\mathrm{d}u\nonumber\\
&\leq4\int_{0}^{\alpha^{\lambda}}\mathbb{E}_{x}\big[\big|q^{\alpha,\,\epsilon}(\tau_{\mathcal{O}_{N,\,\epsilon}}^{\alpha,\,\epsilon}\land u)-s\big|^{2}\big]\mathrm{d}u\label{eq:doob ineq}\;.
\end{align} 
Additionally, by It\^o's formula and the boundedness of $\R^{2d}\setminus\mathcal{O}_{N,\,\epsilon}$, there exists a constant $C=C(N)>0$ independent of $x$
and $\alpha\in(0,1)$  such that for all $u\in[0,\,\alpha^{\lambda}]$,
$$\big|q^{\alpha,\,\epsilon}(\tau_{\mathcal{O}_{N,\,\epsilon}}^{\alpha,\,\epsilon}\land u)-s\big|^{2}\leq \big|q-s\big|^2+Cu+2\sqrt{2\alpha\epsilon}\int_{0}^{\tau_{\mathcal{O}_{N,\,\epsilon}}^{\alpha,\,\epsilon}\land u}\left\langle q^{\alpha,\,\epsilon}(\rho)-s,\,\mathrm{d}\widetilde{B}(\rho)\right\rangle\;.$$ 
Given that the last term in the right-hand side of the
inequality above is a martingale,  
\[
\mathbb{E}_{x}\big[\big|q^{\alpha,\,\epsilon}(\tau_{\mathcal{O}_{N,\,\epsilon}}^{\alpha,\,\epsilon}\land u)-s\big|^{2}\big]\leq \big|q-s\big|^2+Cu\;.
\]
Reinjecting into~\eqref{eq:doob ineq} ensures the inequality~\eqref{eq:first control martingale}. Similarly, the inequality~\eqref{eq:second control martingale} relies on the fact that for all $u\in[0,\,\alpha^{\lambda}]$,
$$\big|p^{\alpha,\,\epsilon}(\tau_{\mathcal{O}_{N,\,\epsilon}}^{\alpha,\,\epsilon}\land u)-p\big|^{2}\leq Cu+2\sqrt{2\gamma\epsilon}\int_{0}^{\tau_{\mathcal{O}_{N,\,\epsilon}}^{\alpha,\,\epsilon}\land u}\left\langle p^{\alpha,\,\epsilon}(\rho)-p,\,\mathrm{d}B(\rho)\right\rangle\;.$$ 
Therefore,
\begin{align*}
\mathbb{E}_{x}\left[\sup_{t\in[0,\,\alpha^{\lambda}]}\left|\int_{0}^{t\land\tau_{\mathcal{O}_{N,\,\epsilon}}^{\alpha,\,\epsilon}}\left\langle p^{\alpha,\,\epsilon}(u)-p,\,\mathrm{d}B(u)\right\rangle\right|^{2}\right]
&\leq4\int_{0}^{\alpha^{\lambda}}\mathbb{E}_{x}\big[\big|p^{\alpha,\,\epsilon}(\tau_{\mathcal{O}_{N,\,\epsilon}}^{\alpha,\,\epsilon}\land u)-p\big|^{2}\big]\mathrm{d}u\\
&\leq 4C\int_{0}^{\alpha^{\lambda}}u\,\mathrm{d}u=2C\alpha^{2\lambda}\,,
\end{align*}
which concludes the proof of~\eqref{eq:second control martingale}.
\end{proof}
\subsection{Neighborhood of $\partial\mathcal{O}_{N,\,\epsilon}$} This subsection aims at providing sharp estimates for the law of the processes $(X^{\epsilon}(t))_{t\geq0}$ and $(X^{\alpha,\,\epsilon}(t))_{t\geq0}$ in a neighborhood of the boundary $\partial\mathcal{O}_{N,\,\epsilon}$. These results are used in the next subsection to prove Lemma~\ref{lem:unif_c}. 

\begin{lem}\label{lem:control density}
There exists a constant $C>0$ such that for all $t\in(0,r]$, if $\theta,\epsilon>0$ are small enough, then for all measurable set $A\subset\R^{2d}$,
$$\sup_{x\in\partial\M_\epsilon}\mathbb{P}_{x}(X^{\epsilon}(t)\in A,\,\mathrm{d}(X^{\epsilon}(t),\,\partial\mathcal{O}_{N,\,\epsilon})\leq\theta,\,\tau_{\mathcal{O}_{N,\,\epsilon}}^{\epsilon}> t)\leq C\,|A|\,,$$
where $|A|$ is the Euclidean volume of $A$.
\end{lem}
 
\begin{cor}\label{cor:weyl density}
There exists a constant $C>0$ such that for all $t\in(0,r]$, if $\eta,\,\eta',\,\epsilon>0$ are small enough, then
$$\sup_{x\in\partial\M_\epsilon}\mathbb{P}_{x}\left(\big|q^{\epsilon}(t)-s\big|\in(\epsilon-\eta,\epsilon+\eta),\,\big|p^{\epsilon}(t)\big|\leq\eta',\,\tau_{\mathcal{O}_{N,\,\epsilon}}^{\epsilon}>t\right)\leq C\eta\,\eta'.$$
\end{cor}
\begin{proof}[Proof of Lemma~\ref{lem:control density}]
Friedman's uniqueness result~\cite[Theorem 5.2.1]{Friedman}
ensures that the trajectories of $(q^{\epsilon}(t),\,p^{\epsilon}(t))_{0\le t\le\tau_{\mathcal{O}_{N,\,\epsilon}}^{\epsilon}}$
do not depend on the values of the coefficients in~\eqref{eq:sde} outside of the bounded domain $\R^{2d}\setminus\mathcal{O}_{N,\,\epsilon}$. Since the probability above involves the law of the vector $(q^{\epsilon}(t),\,p^{\epsilon}(t))$ under the event $t<\tau_{\mathcal{O}_{N,\,\epsilon}}^{\epsilon}$ then there is no loss of generality in considering the law of $(q^{\epsilon}(t),\,p^{\epsilon}(t))$ where the coefficients in~\eqref{eq:sde} are modified in $\mathcal{O}_{N,\,\epsilon}$ so that they are bounded and globally Lipschitz continuous in $\mathbb{R}^{2d}$. With this setting in mind we can then apply the result~\cite[Theorem 2.1]{Menozzi}
ensuring the existence of a Gaussian upper-bound given by 
\begin{equation}\label{eq:expr gaussian upper-bound}
f_{t}(q,\,p,\,q',\,p'):=\frac{c_{1}}{t^{2d}}\mathrm{exp}\big(-c_{2}\left[\frac{|p-p'|^2}{4t}+3\frac{|q'-q-(p+p')t/2|^2}{t^{3}}\right]\big)\;,
\end{equation}
where $c_{1},\,c_{2}>0$ are constants independent of $t$. Let us now prove that for all $(q,p),\,(q',p')\in\R^{2d}$ satisfying
$$(q,p)\in\partial\M_\epsilon,\quad\text{and}\quad \mathrm{d}((q',p'),\,\partial\mathcal{O}_{N,\,\epsilon})\leq\theta$$
where $\theta$ is sufficiently small, then $f_t(q,\,p,\,q',\,p')$ admits an upper-bound independent of $t$ and $(q,p),\,(q',p')$. This would allow to conclude the proof of Lemma~\ref{lem:control density}.

Given that $(q,p)\in\partial\M_\epsilon$, it satisfies in particular $|q-m|\leq\epsilon$ and $|p|\leq\epsilon$. Assume now that $|p'|\geq2\epsilon$ then necessarily $|p-p'|\geq\epsilon$. As a result,
$$f_t(q,\,p,\,q',\,p')\leq \frac{c_{1}}{t^{2d}}\mathrm{e}^{-c_2\epsilon^2/4t}$$
which can be bounded by a constant $C>0$ only depending on $\epsilon$ and $c_1,\,c_2$. Consider now the case $|p'|\leq2\epsilon$. Since $\mathrm{d}((q',p'),\,\partial\mathcal{O}_{N,\,\epsilon})\leq\theta$, then if $\theta$ is small enough either $|q'|\geq N/2$ or $|q'-s|\leq2\epsilon$. Assume for instance that $|q'-s|\leq2\epsilon$, then by the triangle inequality,
\begin{align*}
|q'-q-(p+p')t/2|&\geq|q'-q|-|(p+p')t/2|\\
&\geq |m-s|-|q'-s|-|q-m|-3\epsilon r/2\\
&\geq |m-s|/2
\end{align*}
when $\epsilon$ is small enough since $|q-m|\leq\epsilon$. The reasoning is similar if $|q'|\geq N/2$. As a result, one is able to deduce in both cases an upper-bound of $f_t(q,\,p,\,q',\,p')$ which is independent of $t$ and $(q,p),\,(q',p')$, hence the proof.
\end{proof}
\begin{proof}[Proof of Corollary~\ref{cor:weyl density}]
Let us take $\theta,\,\epsilon>0$ small enough as assumed in Lemma~\ref{lem:control density}. It is easy to see that under the event 
$$\big\{\big|q^{\epsilon}(t)-s\big|\in(\epsilon-\eta,\epsilon+\eta),\,\big|p^{\epsilon}(t)\big|\leq\eta'\big\}\,,$$
if $\eta,\,\eta'$ are small enough then $\mathrm{d}(X^{\epsilon}(t),\,\partial\mathcal{O}_{N,\,\epsilon})\leq\theta$. The proof then follows from Lemma~\ref{lem:control density} and Weyl's tube formula~\cite{WeylTube}. 
\end{proof}
The following lemma is crucial to control the velocity of the process~\eqref{eq:sde} entering the set $\mathcal{O}_{N,\,\epsilon}$, which is needed in the proof of Lemma~\ref{lem:unif_c} in the next subsection. It follows the outline of the proof in~\cite[Proposition 2.7]{LelRamRey}.
\begin{lem}
\label{lem:control inward velocity} For $\epsilon$ small enough
and $N$ large enough, 
\[
\sup_{x\in\partial\mathcal{M}_{\epsilon}}\mathbb{P}_{x}\left(|p^{\epsilon}(\tau_{\mathcal{O}_{N,\,\epsilon}}^{\epsilon})|\leq\delta,\tau_{\mathcal{O}_{N,\,\epsilon}}^{\epsilon}\leq r\right)\underset{\delta\rightarrow0}{\longrightarrow}0\;.
\]
\end{lem}
\begin{proof}

Let $\delta\in(0,1)$. Given the definition of $\mathcal{O}_{N,\,\epsilon}$ in~\eqref{eq:definition O_N}, either $X^{\epsilon}(\tau_{\mathcal{O}_{N,\,\epsilon}}^{\epsilon})\in\partial\mathcal{S}_{\epsilon}$ or $X^{\epsilon}(\tau_{\mathcal{O}_{N,\,\epsilon}}^{\epsilon})\in\partial\mathrm{B}(0,N)$. Therefore, let us start by considering the first case and we prove that for any $\nu>0$, if $\delta$ is small enough then 
\begin{equation}
\sup_{x\in\partial\mathcal{M}_{\epsilon}}\mathbb{P}_{x}\left(X^{\epsilon}(\tau_{\mathcal{O}_{N,\,\epsilon}}^{\epsilon})\in\partial\mathcal{S}_{\epsilon},\,|p^{\epsilon}(\tau_{\mathcal{O}_{N,\,\epsilon}}^{\epsilon})|\leq\delta,\,\tau_{\mathcal{O}_{N,\,\epsilon}}^{\epsilon}\leq r\right)\leq\nu\;.\label{eq:sup proba inward velocity proof}
\end{equation}
Let $\eta\in(0,\,1/2)$ and let
\[
\mathcal{Z}:=\sup_{0\leq t<t'\leq r}\frac{|B(t)-B(t')|}{|t-t'|^{\eta}},
\]
then $\mathcal{Z}<\infty$, almost-surely, since the Brownian motion
$(B(t))_{t\geq0}$ is $\eta$-H\"olderian.

Let $J_\delta=\lfloor1/\delta\rfloor$ and let $r_{k}=r\frac{k}{J_\delta}$ for $0\leq k\leq J_\delta$. We take $M>0$ large enough such that $\mathbb{P}(\mathcal{Z}>M)\leq\nu/2$, then 
\begin{align}
 & \mathbb{P}_{x}\left(X^{\epsilon}(\tau_{\mathcal{O}_{N,\,\epsilon}}^{\epsilon})\in\partial\mathcal{S}_{\epsilon},\,|p^{\epsilon}(\tau_{\mathcal{O}_{N,\,\epsilon}}^{\epsilon})|\leq\delta,\,\tau_{\mathcal{O}_{N,\,\epsilon}}^{\epsilon}\leq r\right)\nonumber \\
 & \leq\mathbb{P}_{x}\left(X^{\epsilon}(\tau_{\mathcal{O}_{N,\,\epsilon}}^{\epsilon})\in\partial\mathcal{S}_{\epsilon},\,|p^{\epsilon}(\tau_{\mathcal{O}_{N,\,\epsilon}}^{\epsilon})|\leq\delta,\,\tau_{\mathcal{O}_{N,\,\epsilon}}^{\epsilon}\leq r,\,\mathcal{Z}\leq M\right)+\mathbb{P}(\mathcal{Z}>M)\nonumber \\
 & \leq\sum_{k=0}^{J_\delta-1}\mathbb{P}_{x}\left(\tau_{\mathcal{O}_{N,\,\epsilon}}^{\epsilon}\in(r_{k},\,r_{k+1}],\,X^{\epsilon}(\tau_{\mathcal{O}_{N,\,\epsilon}}^{\epsilon})\in\partial\mathcal{S}_{\epsilon},\,|p^{\epsilon}(\tau_{\mathcal{O}_{N,\,\epsilon}}^{\epsilon})|\leq\delta,\,\mathcal{Z}\leq M\right)+\frac{\nu}{2}\;.\label{eq:sum interval exit time}
\end{align}

Moreover, by~\eqref{eq:sde}, for all $u\in[r_{k},\,\tau_{\mathcal{O}_{N,\,\epsilon}}^{\epsilon}]$,
\[
p^{\epsilon}(\tau_{\mathcal{O}_{N,\,\epsilon}}^{\epsilon})-p^{\epsilon}(u)=\int_{u}^{\tau_{\mathcal{O}_{N,\,\epsilon}}^{\epsilon}}\left[-\nabla U(q^{\epsilon}(\rho))-\gamma p^{\epsilon}(\rho)\right]\mathrm{d}\rho+\sqrt{2\gamma\epsilon}\left(B(\tau_{\mathcal{O}_{N,\,\epsilon}}^{\epsilon})-B(u)\right)\;.
\]
Under the event $\{\tau_{\mathcal{O}_{N,\,\epsilon}}^{\epsilon}\in(r_{k},\,r_{k+1}],\,\mathcal{Z}\leq M\}$, using the boundedness of $\R^{2d}\setminus\mathcal{O}_{N,\,\epsilon}$ there exists a constant $C=C(N)\geq1$ such that
\begin{equation}\label{eq:triangle ineq velocity boundary}
\big|p^{\epsilon}(\tau_{\mathcal{O}_{N,\,\epsilon}}^{\epsilon})-p^{\epsilon}(u)\big|\leq\frac{C}{J_\delta}+\frac{C}{J_\delta^\eta}\leq \frac{2C}{J_\delta^\eta}\;.
\end{equation}
In particular, under the event $\{\tau_{\mathcal{O}_{N,\,\epsilon}}^{\epsilon}\in(r_{k},\,r_{k+1}],\,|p^{\epsilon}(\tau_{\mathcal{O}_{N,\,\epsilon}}^{\epsilon})|\leq\delta,\,\mathcal{Z}\leq M\}$, one has by the triangle inequality,
\begin{align}
\big|p^{\epsilon}(r_k)\big|&\leq\big|p^{\epsilon}(\tau_{\mathcal{O}_{N,\,\epsilon}}^{\epsilon})-p^{\epsilon}(r_k)\big|+\big|p^{\epsilon}(\tau_{\mathcal{O}_{N,\,\epsilon}}^{\epsilon})\big|\nonumber\\
&\leq\frac{2C}{J_\delta^\eta}+\frac{1}{J_\delta}\leq\frac{3C}{J_\delta^\eta}\label{eq:ineq triangl p}\,,
\end{align}
since $\delta\leq1/J_\delta$, by definition of $J_\delta$. Moreover, under the event $\{\tau_{\mathcal{O}_{N,\,\epsilon}}^{\epsilon}\in(r_{k},\,r_{k+1}],\,|p^{\epsilon}(\tau_{\mathcal{O}_{N,\,\epsilon}}^{\epsilon})|\leq\delta,\,\mathcal{Z}\leq M\}$, the triangle inequality also yields by~\eqref{eq:triangle ineq velocity boundary},
\begin{align}
\big|q^{\epsilon}(\tau_{\mathcal{O}_{N,\,\epsilon}}^{\epsilon})-q^{\epsilon}(r_{k})\big| & =\left|\int_{r_{k}}^{\tau_{\mathcal{O}_{N,\,\epsilon}}^{\epsilon}}p^{\epsilon}(u)\mathrm{d}u\right|\nonumber\\
&\leq \frac{2C}{J_\delta^{1+\eta}}+\int_{r_{k}}^{\tau_{\mathcal{O}_{N,\,\epsilon}}^{\epsilon}}\big|p^{\epsilon}(\tau_{\mathcal{O}_{N,\,\epsilon}}^{\epsilon})\big|\mathrm{d}u\nonumber\\ 
 &\leq\frac{2C}{J_\delta^{1+\eta}}+\frac{1}{J_\delta^{2}}\leq\frac{3C}{J_\delta^{1+\eta}}\label{eq:first triangle ineq}\,.
\end{align}
Furthermore, on the event $\{X^{\epsilon}(\tau_{\mathcal{O}_{N,\,\epsilon}}^{\epsilon})\in\partial\mathcal{S}_{\epsilon}\}$, 
\begin{equation}\label{eq:equation boundary S}
\big|q^{\epsilon}(\tau_{\mathcal{O}_{N,\,\epsilon}}^{\epsilon})-s\big|^{2}+\big|p^{\epsilon}(\tau_{\mathcal{O}_{N,\,\epsilon}}^{\epsilon})\big|^{2}=\epsilon^2\;.
\end{equation}
Therefore, on the event $\{X^{\epsilon}(\tau_{\mathcal{O}_{N,\,\epsilon}}^{\epsilon})\in\partial\mathcal{S}_{\epsilon},\,\big|p^{\epsilon}(\tau_{\mathcal{O}_{N,\,\epsilon}}^{\epsilon})\big|\leq\delta\}$,
\[
\sqrt{\epsilon^2-\frac{1}{J^2_\delta}}\leq\big|q^{\epsilon}(\tau_{\mathcal{O}_{N,\,\epsilon}}^{\epsilon})-s\big|\leq \epsilon\;.
\]
Combining the inequality above with~\eqref{eq:first triangle ineq} one obtains
\[
\sqrt{\epsilon^2-\frac{1}{J^2_\delta}}-\frac{3C}{J_\delta^{1+\eta}}\leq\big|q^{\epsilon}(r_k)-s\big|\leq \epsilon+\frac{3C}{J_\delta^{1+\eta}}\;.
\]
Taking $\delta$ sufficiently small and applying the asymptotic expansion of $\sqrt{1-x}$ when $x\rightarrow0$, one obtains the existence of a constant $C=C(N,\epsilon)>0$ which is independent of $\delta$, such that
\[
\epsilon-\frac{C}{J_\delta^{1+\eta}}\leq\big|q^{\epsilon}(r_k)-s\big|\leq \epsilon+\frac{C}{J_\delta^{1+\eta}}\;.
\]

All in all, combining the inequality above with~\eqref{eq:ineq triangl p} and taking a constant $C>0$ large enough, one obtains for all $0\leq k\leq J_\delta-1$ and $x\in\partial\M_\epsilon$,
\begin{align*}
&\mathbb{P}_{x}\left(\tau_{\mathcal{O}_{N,\,\epsilon}}^{\epsilon}\in(r_{k},\,r_{k+1}],\,X^{\epsilon}(\tau_{\mathcal{O}_{N,\,\epsilon}}^{\epsilon})\in\partial\mathcal{S}_{\epsilon},\,|p^{\epsilon}(\tau_{\mathcal{O}_{N,\,\epsilon}}^{\epsilon})|\leq\delta,\,\mathcal{Z}\leq M\right)\nonumber \\
 & \leq\mathbb{P}_{x}\left(\tau_{\mathcal{O}_{N,\,\epsilon}}^{\epsilon}\in(r_{k},\,r_{k+1}],\,\big|q^{\epsilon}(r_{k})-s\big|\in\big[\epsilon-\frac{C}{J_\delta^{1+\eta}},\,\epsilon+\frac{C}{J_\delta^{1+\eta}}\big],\,\big|p^{\epsilon}(r_{k})\big|\leq\frac{C}{J^{\eta}_\delta}\right)\;.
\end{align*}
In particular, for $k=0$, since the initial condition $x\in\partial\M_\epsilon$, it is easy to see that, taking $\epsilon$ and $\delta$ sufficiently small, ensures that the probability above vanishes. Therefore, it is sufficient to consider the case $1\leq k\leq J_\delta-1$. Moreover, by Corollary~\ref{cor:weyl density}, there exists a constant $C'>0$ independent of $k,\delta$ and $x\in\partial\M_\epsilon$ such that
\begin{align*}
\mathbb{P}_{x}\left(\tau_{\mathcal{O}_{N,\,\epsilon}}^{\epsilon}\in(r_{k},\,r_{k+1}],\,\big|q^{\epsilon}(r_{k})-s\big|\in\big[\epsilon-\frac{C}{J_\delta^{1+\eta}},\,\epsilon+\frac{C}{J_\delta^{1+\eta}}\big],\,\big|p^{\epsilon}(r_{k})\big|\leq\frac{C}{J^{\eta}_\delta}\right) 
 \leq \frac{C'}{J_\delta^{1+2\eta}}\,.
\end{align*}
Reinjecting in~\eqref{eq:sum interval exit time} and summing over $1\leq k\leq J_\delta-1$, it follows that
\begin{align}
 \mathbb{P}_{x}\left(X^{\epsilon}(\tau_{\mathcal{O}_{N,\,\epsilon}}^{\epsilon})\in\partial\mathcal{S}_{\epsilon},\,|p^{\epsilon}(\tau_{\mathcal{O}_{N,\,\epsilon}}^{\epsilon})|\leq\delta,\,\tau_{\mathcal{O}_{N,\,\epsilon}}^{\epsilon}\leq r\right)\nonumber \leq \frac{C'}{J_\delta^{2\eta}}+\frac{\nu}{2}\;.
\end{align}
Taking $\delta$ sufficiently small concludes the proof. It remains however to also show that 
\[
\sup_{x\in\partial\mathcal{M}_{\epsilon}}\mathbb{P}_{x}\left(X^{\epsilon}(\tau_{\mathcal{O}_{N,\,\epsilon}}^{\epsilon})\in\partial\mathrm{B}(0,\,N),\,|p^{\epsilon}(\tau_{\mathcal{O}_{N,\,\epsilon}}^{\epsilon})|\leq\delta,\,\tau_{\mathcal{O}_{N,\,\epsilon}}^{\epsilon}\leq r\right)\underset{\delta\rightarrow0}{\longrightarrow}0\;.
\]
Nonetheless, a similar proof also applies in this case by replacing the equality~\eqref{eq:equation boundary S} by the appropriate formula on $\partial\mathrm{B}(0,\,N)$, hence the proof of this lemma. 
\end{proof}

\subsection{Proof of Lemma~\ref{lem:unif_c}}\label{sec:proof of lemma 4.7}
The lemma below investigates the probability of not entering the set $\mathcal{O}_{N,\,\epsilon}$ in a short time interval when the process~\eqref{eq:sde_pur} is close to its boundary. We then conclude this section with the proof of Lemma~\ref{lem:unif_c}. Throughout this subsection, we will regularly use the notation~\eqref{eq:xqp}.
\begin{lem}\label{lem:hitting time alpha epsilon}
Let $\lambda\in(0,1/2)$ and define
\begin{equation}\label{eq:def set C lambda}
\mathcal{C}_\lambda^{\alpha,\,\epsilon}:=\{x=(q,p)\in\R^{2d}\setminus\mathcal{O}_{N,\,\epsilon}:|p|>\alpha^{\lambda/4}/2,\,\mathrm{d}(x,\,\partial\mathcal{O}_{N,\,\epsilon})\in(0,\,\alpha^{\lambda}]\}\;.
\end{equation}
Then,
$$\sup_{x\in\mathcal{C}_\lambda^{\alpha,\,\epsilon}}\mathbb{P}_{x}\left(\tau_{\mathcal{O}_{N,\,\epsilon}}^{\alpha,\;\epsilon}>\alpha^{\lambda}\right)\underset{\alpha\rightarrow0}{\longrightarrow}0\;.$$
\end{lem}
\begin{proof}
Let $x=(q,p)\in\mathcal{C}_\lambda^{\alpha,\,\epsilon}$.
Given the definition of $\mathcal{O}_{N,\,\epsilon}$ in~\eqref{eq:definition O_N},
one has that 
\[
\mathrm{d}(x,\,\partial\mathcal{O}_{N,\,\epsilon})=\mathrm{d}(x,\,\partial\mathcal{S}_{\epsilon})\land\mathrm{d}(x,\,\partial\mathrm{B}(0,\,N))\;.
\]
Let us first assume that $\mathrm{d}(x,\,\partial\mathcal{S}_{\epsilon})\leq\alpha^{\lambda}$. Then, there exists a constant $c>0$ independent of $\alpha$ such that for all $\alpha$ small enough,
\[
|x-(s,0)|^{2}\leq(\epsilon+\alpha^\lambda)^2\leq\epsilon^2+c\alpha^\lambda\;.
\]

\noindent As a result, 
\begin{align*}
\mathbb{P}_{x}\left(\tau_{\mathcal{O}_{N,\,\epsilon}}^{\alpha,\,\epsilon}>\alpha^{\lambda}\right) \leq\mathbb{P}_{x}\left(\inf_{t\in[0,\,\alpha^{\lambda}]}|X^{\alpha,\,\epsilon}(t)-(s,0)|^{2}-|x-(s,0)|^{2}\geq-c\alpha^{\lambda}\right)\;.
\end{align*}

Furthermore, by It\^{o}'s formula applied to~\eqref{eq:sde_pur}, almost-surely
for all $t\geq0$, 
\begin{align*}
 & |X^{\alpha,\,\epsilon}(t)-(s,0)|^{2}-|x-(s,0)|^{2}\\ 
 & =2\int_{0}^{t}\left[\left\langle q^{\alpha,\,\epsilon}(u)-s,\,p^{\alpha,\,\epsilon}(u)-\alpha\nabla U(q^{\alpha,\,\epsilon}(u))\right\rangle +\left\langle p^{\alpha,\,\epsilon}(u),\,-\gamma p^{\alpha,\,\epsilon}(u)-\nabla U(q^{\alpha,\,\epsilon}(u))\right\rangle \right]\mathrm{d}u
 \\& \qquad +\,d(\gamma+\alpha)\epsilon t
  +\,2\sqrt{2\alpha\epsilon}\int_{0}^{t}\left\langle q^{\alpha,\,\epsilon}(u)-s,\,\mathrm{d}\widetilde{B}(u)\right\rangle  \\&\qquad +\,2\sqrt{2\gamma\epsilon}\int_{0}^{t}\left\langle p^{\alpha,\,\epsilon}(u)-p,\,\mathrm{d}B(u)\right\rangle +2\sqrt{2\gamma\epsilon}\left\langle p,\,B(t)\right\rangle \;.
\end{align*}
Consider now the event 
\[
\mathcal{G}_{\alpha,\,\epsilon}:=\left\{ 2\sqrt{2\alpha\epsilon}\sup_{t\in[0,\,\alpha^{\lambda}]}\int_{0}^{t}\left\langle q^{\alpha,\,\epsilon}(u)-s,\,\mathrm{d}\widetilde{B}(u)\right\rangle +2\sqrt{2\gamma\epsilon}\sup_{t\in[0,\,\alpha^{\lambda}]}\int_{0}^{t}\left\langle p^{\alpha,\,\epsilon}(u)-p,\,\mathrm{d}B(u)\right\rangle \leq4\alpha^{4\lambda/5}\right\} \;,
\]
then under the event $\{\mathcal{G}_{\alpha,\,\epsilon},\tau_{\mathcal{O}_{N,\,\epsilon}}^{\alpha,\,\epsilon}>\alpha^{\lambda}\}$
there exists a constant $C=C(N)>0$ independent of $\alpha$
such that for all $t\in[0,\,\alpha^{\lambda}]$ and $\alpha$ small enough,
\begin{align*}
 |X^{\alpha,\,\epsilon}(t)-(s,0)|^{2}-|x-(s,0)|^{2}
 &\leq C\alpha^\lambda+4\alpha^{4\lambda/5}+2\sqrt{2\gamma\epsilon}\left\langle p,\,B(t)\right\rangle\\
 & \leq C'\alpha^{4\lambda/5}+2\sqrt{2\gamma\epsilon}\left\langle p,\,B(t)\right\rangle 
\end{align*}
where $C'=C+4>0$. Consequently,  
\begin{align*}
\mathbb{P}_{x}\left(\tau_{\mathcal{O}_{N,\,\epsilon}}^{\alpha,\,\epsilon}>\alpha^{\lambda}\right)\leq\mathbb{P}\left(C'\alpha^{4\lambda/5}+2\sqrt{2\gamma\epsilon}\inf_{t\in[0,\,\alpha^{\lambda}]}\left\langle p,\,B(t)\right\rangle\geq-c\alpha^\lambda\right)+\mathbb{P}_{x}\left(\tau_{\mathcal{O}_{N,\,\epsilon}}^{\alpha,\,\epsilon}>\alpha^{\lambda},\,\mathcal{G}_{\alpha,\,\epsilon}^{c}\right)\;.
\end{align*}
Given that $(\left\langle p,\,B(t)\right\rangle/|p|)_{t\geq0}$ is a one-dimensional Brownian motion, letting $G\sim\mathcal{N}(0,1)$, one has that
\begin{align*}
\mathbb{P}\left(2\sqrt{2\gamma\epsilon}\inf_{t\in[0,\,\alpha^{\lambda}]}\left\langle p,\,B(t)\right\rangle\geq-C'\alpha^{4\lambda/5}-c\alpha^\lambda\right)
 & \leq\mathbb{P}\left(2\sqrt{2\gamma\epsilon}\,\alpha^{\lambda/2}|G|<C'\frac{\alpha^{4\lambda/5}}{|p|}+c\,\frac{\alpha^\lambda}{|p|}\right)\\
 & \leq\mathbb{P}\left(2\sqrt{2\gamma\epsilon}\,|G|<2C'\,\alpha^{\lambda/20}+2c\,\alpha^{\lambda/4}\right)\;,
\end{align*}
since $|p|>\alpha^{\lambda/4}/2$, which therefore converges to $0$
when $\alpha$ goes to zero.

It remains to show that 
\begin{equation}
\sup_{x\in\mathcal{C}_\lambda^{\alpha,\,\epsilon}}\mathbb{P}_{x}\left(\tau_{\mathcal{O}_{N,\,\epsilon}}^{\alpha,\,\epsilon}>\alpha^{\lambda},\,\mathcal{G}_{\alpha,\,\epsilon}^{c}\right)\underset{\alpha\rightarrow0}{\longrightarrow}0\;.\label{eq:convergence proba G^c}
\end{equation}
Notice that 
\begin{align*}
\mathbb{P}_{x}\left(\tau_{\mathcal{O}_{N,\,\epsilon}}^{\alpha,\,\epsilon}>\alpha^{\lambda},\,\mathcal{G}_{\alpha,\,\epsilon}^{c}\right) & \leq\mathbb{P}_{x}\left(\sqrt{2\alpha\epsilon}\sup_{t\in[0,\,\alpha^{\lambda}]}\int_{0}^{t\land\tau_{\mathcal{O}_{N,\,\epsilon}}^{\alpha,\,\epsilon}}\left\langle q^{\alpha,\,\epsilon}(u)-s,\,\mathrm{d}\widetilde{B}(u)\right\rangle >\alpha^{4\lambda/5}\right)\\
 & +\mathbb{P}_{x}\left(\sqrt{2\gamma\epsilon}\sup_{t\in[0,\,\alpha^{\lambda}]}\int_{0}^{t\land\tau_{\mathcal{O}_{N,\,\epsilon}}^{\alpha,\,\epsilon}}\left\langle p^{\alpha,\,\epsilon}(u)-p,\,\mathrm{d}B(u)\right\rangle >\alpha^{4\lambda/5}\right)\;.
\end{align*}

Using the Markov inequality for both probabilities in the right-hand side above, along with Lemma~\ref{lem:control martingale}, one deduces the existence of a constant $C>0$ independent of $\alpha$ such that
\begin{align*}
\mathbb{P}_{x}\left(\tau_{\mathcal{O}_{N,\,\epsilon}}^{\alpha,\,\epsilon}>\alpha^{\lambda},\,\mathcal{G}_{\alpha,\,\epsilon}^{c}\right) \leq C \frac{\alpha^{1+\lambda}}{\alpha^{8\lambda/5}}+ C\frac{\alpha^{2\lambda}}{\alpha^{8\lambda/5}} 
 \leq C\,\alpha^{1-3\lambda/5}+ C\,\alpha^{2\lambda/5}\underset{\alpha\rightarrow0}{\longrightarrow}0\,,
\end{align*}
since $\lambda\in(0,1/2)$. To conclude the proof of this lemma, it remains to investigate the case $\mathrm{d}(x,\,\partial\mathrm{B}(0,N))\leq\alpha^{\lambda}$. However, the proof applies similarly to this case as the argument developed previously mainly depends on the ball-shaped
boundary which is valid for both boundaries $\partial\mathrm{B}(0,\,N)$
and $\partial\mathcal{S}_{\epsilon}$.  
\end{proof}
\begin{proof}[Proof of Lemma~\ref{lem:unif_c}]
One has for any $x\in\partial\mathcal{M}_{\epsilon}$, 
\[
\mathbb{P}_{x}(\tau_{\mathcal{O}_{N,\,\epsilon}}^{\alpha,\,\epsilon}>r)-\mathbb{P}_{x}(\tau_{\mathcal{O}_{N,\,\epsilon}}^{\epsilon}>r)=\mathbb{P}_{x}(\tau_{\mathcal{O}_{N,\,\epsilon}}^{\alpha,\,\epsilon}>r,\,\tau_{\mathcal{O}_{N,\,\epsilon}}^{\epsilon}\leq r)-\mathbb{P}_{x}(\tau_{\mathcal{O}_{N,\,\epsilon}}^{\alpha,\,\epsilon}\leq r,\,\tau_{\mathcal{O}_{N,\,\epsilon}}^{\epsilon}>r)\;.
\]
Therefore, it is enough to show that 
\begin{equation}
\sup_{x\in\partial\mathcal{M}_{\epsilon}}\mathbb{P}_{x}(\tau_{\mathcal{O}_{N,\,\epsilon}}^{\alpha,\,\epsilon}>r,\,\tau_{\mathcal{O}_{N,\,\epsilon}}^{\epsilon}\leq r)\underset{\alpha\rightarrow0}{\longrightarrow}0,\quad\text{and}\quad\sup_{x\in\partial\mathcal{M}_{\epsilon}}\mathbb{P}_{x}(\tau_{\mathcal{O}_{N,\,\epsilon}}^{\alpha,\,\epsilon}\leq r,\,\tau_{\mathcal{O}_{N,\,\epsilon}}^{\epsilon}>r)\underset{\alpha\rightarrow0}{\longrightarrow}0\;.\label{eq:convergence double supremum}
\end{equation}
Let us first show the convergence of the first supremum. Let us take
$\lambda\in(0,\,1/2)$, then 
\begin{equation}
\mathbb{P}_{x}(\tau_{\mathcal{O}_{N,\,\epsilon}}^{\alpha,\,\epsilon}>r,\,\tau_{\mathcal{O}_{N,\,\epsilon}}^{\epsilon}\leq r)\leq\mathbb{P}_{x}(\tau_{\mathcal{O}_{N,\,\epsilon}}^{\epsilon}\in[r-\alpha^{\lambda},\,r])+\mathbb{P}_{x}(\tau_{\mathcal{O}_{N,\,\epsilon}}^{\alpha,\,\epsilon}>r,\,\tau_{\mathcal{O}_{N,\,\epsilon}}^{\epsilon}\leq r-\alpha^{\lambda})\;.\label{eq:decomposition proba exit}
\end{equation}
Let us first show that 
\begin{equation}
\sup_{x\in\partial\mathcal{M}_{\epsilon}}\mathbb{P}_{x}(\tau_{\mathcal{O}_{N,\,\epsilon}}^{\epsilon}\in[r-\alpha^{\lambda},\,r])\underset{\alpha\rightarrow0}{\longrightarrow}0\;.\label{eq:control hitting time interval}
\end{equation}
For any $\beta\in(0,\,\lambda/2)$, 
\begin{align*}
\mathbb{P}_{x}(\tau_{\mathcal{O}_{N,\,\epsilon}}^{\epsilon}\in[r-\alpha^{\lambda},\,r]) & \leq\mathbb{P}_{x}(\mathrm{d}(X_{r-\alpha^{\lambda}}^{\epsilon},\,\partial\mathcal{O}_{N,\,\epsilon})\leq\alpha^{\beta},\,\tau_{\mathcal{O}_{N,\,\epsilon}}^{\epsilon}> r-\alpha^\lambda)\\
 & +\mathbb{P}_{x}\left(\sup_{|t_{1}-t_{2}|\leq\alpha^{\lambda},\,t_{1},\,t_{2}\in\left[0,\,\tau_{\mathcal{O}_{N,\,\epsilon}}^{\epsilon}\land r\right]}|X^{\epsilon}(t_{1})-X^{\epsilon}(t_{2})|>\alpha^{\beta}\right)\;.
\end{align*}
Using Lemma~\ref{lem:control density} along with Weyl's tube formula~\cite{WeylTube} one obtains the existence of constants $C_1,C_2>0$ independent of $\alpha$ such that 
\begin{align*}
\sup_{x\in\partial\M_\epsilon}\mathbb{P}_{x}(\mathrm{d}(X^{\epsilon}({r-\alpha^{\lambda}}),\,\partial\mathcal{O}_{N,\,\epsilon})\leq\alpha^{\beta},\,\tau_{\mathcal{O}_{N,\,\epsilon}}^{\epsilon}> r-\alpha^\lambda)   \leq C_{1}\int_{x:\mathrm{d}(x,\,\partial\mathcal{O}_{N,\,\epsilon})\leq\alpha^{\beta}}\mathrm{d}x \leq C_{1}C_{2}\,\alpha^{\beta}\;.
\end{align*}
Additionally, by Lemma~\ref{lem:variations process}, 
\[
\sup_{x\in\partial\mathcal{M}_{\epsilon}}\mathbb{P}_{x}\left(\sup_{|t_{1}-t_{2}|\leq\alpha^{\lambda},\,t_{1},\,t_{2}\in\left[0,\,\tau_{\mathcal{O}_{N,\,\epsilon}}^{\epsilon}\land r\right]}|X^{\epsilon}(t_{1})-X^{\epsilon}(t_{2})|>\alpha^{\beta}\right)\underset{\alpha\rightarrow0}{\longrightarrow}0\;,
\]
hence the proof of~\eqref{eq:control hitting time interval}.

We consider now the second term in~\eqref{eq:decomposition proba exit}.
One has that 
\begin{align}
 & \mathbb{P}_{x}(\tau_{\mathcal{O}_{N,\,\epsilon}}^{\alpha,\,\epsilon}>r,\,\tau_{\mathcal{O}_{N,\,\epsilon}}^{\epsilon}\leq r-\alpha^{\lambda})\nonumber\\
 & \leq\mathbb{P}_{x}\left(\big|p^{\epsilon}(\tau_{\mathcal{O}_{N,\,\epsilon}}^{\epsilon})\big|\leq\alpha^{\lambda/4},\,\tau_{\mathcal{O}_{N,\,\epsilon}}^{\epsilon}\leq r-\alpha^{\lambda}\right)+\mathbb{P}_{x}\left(\big|p^{\epsilon}(\tau_{\mathcal{O}_{N,\,\epsilon}}^{\epsilon})\big|>\alpha^{\lambda/4},\,\tau_{\mathcal{O}_{N,\,\epsilon}}^{\alpha,\,\epsilon}>r,\,\tau_{\mathcal{O}_{N,\,\epsilon}}^{\epsilon}\leq r-\alpha^{\lambda}\right)\;.\label{eq:decomposition event velocity norm}
\end{align}
It follows from Lemma~\ref{lem:control inward velocity} that for
$\epsilon\in(0,\,1)$ small enough, 
\[
\sup_{x\in\partial\mathcal{M}_{\epsilon}}\mathbb{P}_{x}\left(|p^{\epsilon}(\tau_{\mathcal{O}_{N,\,\epsilon}}^{\epsilon})|\leq\alpha^{\lambda/4},\,\tau_{\mathcal{O}_{N,\,\epsilon}}^{\epsilon}\leq r-\alpha^{\lambda}\right)\underset{\alpha\rightarrow0}{\longrightarrow}0\;.
\]
It remains to consider the second probability in~\eqref{eq:decomposition event velocity norm}.
One has 
\begin{align}
 & \mathbb{P}_{x}\left(|p^{\epsilon}(\tau_{\mathcal{O}_{N,\,\epsilon}}^{\epsilon})|>\alpha^{\lambda/4},\,\tau_{\mathcal{O}_{N,\,\epsilon}}^{\alpha,\,\epsilon}>r,\,\tau_{\mathcal{O}_{N,\,\epsilon}}^{\epsilon}\leq r-\alpha^{\lambda}\right)\nonumber \\ 
 & \leq\mathbb{P}_{x}\left(|X^{\alpha,\,\epsilon}(\tau_{\mathcal{O}_{N,\,\epsilon}}^{\epsilon})-X^{\epsilon}(\tau_{\mathcal{O}_{N,\,\epsilon}}^{\epsilon})|>\alpha^{\lambda},\,\tau_{\mathcal{O}_{N,\,\epsilon}}^{\alpha,\,\epsilon}>r,\,\tau_{\mathcal{O}_{N,\,\epsilon}}^{\epsilon}\leq r-\alpha^{\lambda}\right)\nonumber \\
 & +\mathbb{P}_{x}\left(|p^{\alpha,\,\epsilon}(\tau_{\mathcal{O}_{N,\,\epsilon}}^{\epsilon})|>\alpha^{\lambda/4}/2,\,\mathrm{d}(X^{\alpha,\,\epsilon}(\tau_{\mathcal{O}_{N,\,\epsilon}}^{\epsilon}),\,\partial\mathcal{O}_{N,\,\epsilon})\leq\alpha^{\lambda},\,\tau_{\mathcal{O}_{N,\,\epsilon}}^{\alpha,\,\epsilon}>r,\,\tau_{\mathcal{O}_{N,\,\epsilon}}^{\epsilon}\leq r-\alpha^{\lambda}\right)\;.\label{eq:decomposition proba outward velocity}
\end{align}
Notice that 
\begin{align*}
 & \mathbb{P}_{x}\left(|X^{\alpha,\,\epsilon}(\tau_{\mathcal{O}_{N,\,\epsilon}}^{\epsilon})-X^{\epsilon}(\tau_{\mathcal{O}_{N,\,\epsilon}}^{\epsilon})|>\alpha^{\lambda},\,\tau_{\mathcal{O}_{N,\,\epsilon}}^{\alpha,\,\epsilon}>r,\,\tau_{\mathcal{O}_{N,\,\epsilon}}^{\epsilon}\leq r-\alpha^{\lambda}\right)\\
 & \leq\mathbb{P}_{x}\left(\sup_{t\in\left[0,\,\tau_{\mathcal{O}_{N,\,\epsilon}}^{\alpha,\,\epsilon}\land\tau_{\mathcal{O}_{N,\,\epsilon}}^{\epsilon}\land r\right]}|X^{\alpha,\,\epsilon}(t)-X^{\epsilon}(t)|\geq\alpha^{\lambda}\right)\;.
\end{align*}
Therefore, since $\lambda\in(0,\,1/2)$, it follows from Lemma~\ref{lem:difference process alpha}
that 
\[
\sup_{x\in\partial\mathcal{M}_{\epsilon}}\mathbb{P}_{x}\left(|X^{\alpha,\,\epsilon}(\tau_{\mathcal{O}_{N,\,\epsilon}}^{\epsilon})-X^{\epsilon}(\tau_{\mathcal{O}_{N,\,\epsilon}}^{\epsilon})|>\alpha^{\lambda},\,\tau_{\mathcal{O}_{N,\,\epsilon}}^{\alpha,\,\epsilon}>r,\,\tau_{\mathcal{O}_{N,\,\epsilon}}^{\epsilon}\leq r-\alpha^{\lambda}\right)\underset{\alpha\rightarrow0}{\longrightarrow}0\;,
\] 
Consider now the last term in the right-hand side of~\eqref{eq:decomposition proba outward velocity}.
One has by the strong Markov property at the stopping time $\tau_{\mathcal{O}_{N,\,\epsilon}}^{\epsilon}$,
\begin{align*}
 & \mathbb{P}_{x}\left(|p^{\alpha,\,\epsilon}(\tau_{\mathcal{O}_{N,\,\epsilon}}^{\epsilon})|>\alpha^{\lambda/4}/2,\,\mathrm{d}(X^{\alpha,\,\epsilon}(\tau_{\mathcal{O}_{N,\,\epsilon}}^{\epsilon}),\,\partial\mathcal{O}_{N,\,\epsilon})\leq\alpha^{\lambda},\,\tau_{\mathcal{O}_{N,\,\epsilon}}^{\alpha,\,\epsilon}>r,\,\tau_{\mathcal{O}_{N,\,\epsilon}}^{\epsilon}\leq r-\alpha^{\lambda}\right)\\
 & \leq\mathbb{E}_{x}\left[\mathbf{1}_{\tau_{\mathcal{O}_{N,\,\epsilon}}^{\epsilon}\leq r-\alpha^{\lambda}}\mathbf{1}_{X^{\alpha,\,\epsilon}(\tau_{\mathcal{O}_{N,\,\epsilon}}^{\epsilon})\in\mathcal{C}_\lambda^{\alpha,\,\epsilon}}\mathbb{P}_{X^{\alpha,\,\epsilon}(\tau_{\mathcal{O}_{N,\,\epsilon}}^{\epsilon})}\left(\tau_{\mathcal{O}_{N,\,\epsilon}}^{\alpha,\,\epsilon}>\alpha^{\lambda}\right)\right]
\end{align*}
where $\mathcal{C}_\lambda^{\alpha,\,\epsilon}$ is defined in~\eqref{eq:def set C lambda}. Therefore, the probability above is bounded by 
$$\sup_{x\in\mathcal{C}_\lambda^{\alpha,\,\epsilon}}\mathbb{P}_{x}\left(\tau_{\mathcal{O}_{N,\,\epsilon}}^{\alpha,\;\epsilon}>\alpha^{\lambda}\right)\,,$$
which vanishes when $\alpha\rightarrow0$ by Lemma~\ref{lem:hitting time alpha epsilon}, hence the proof.
\end{proof}

\section{Proof of Proposition~\ref{prop:control v}}

\label{sec:Proof of prop:control w}

Recall the notation introduced in~\eqref{eq:xqp}. For $\delta>0$, let us define the ball 
\[
Q_{\delta}:=\{x\in\mathbb{R}^{2d}:|q|<\delta^{3},\,|p|<\delta\}\;.
\]

Let us take in this section $\epsilon>0$ small enough such that for
all $x\in Q_{1}$, 
\[
(m+\sqrt{\epsilon}q,\,-\sqrt{\epsilon}p)\in\overline{\mathcal{S}_{\epsilon}}^{c}\;.
\]

We define $u_{\epsilon}$ as follows 
\begin{equation}
u_{\epsilon}:x\in Q_{1}\mapsto\mathbb{E}_{(m+\sqrt{\epsilon}q,\,-\sqrt{\epsilon}p)}\left[\tau_{\mathcal{S}_{\epsilon}}^{\epsilon}\right]\;.\label{eq:def u epsilon}
\end{equation}

The proof of Proposition~\ref{prop:control v} is based on the lemmas
below.
\begin{lem}
\label{lem:control v at 0} We have that 
\[
\liminf_{\epsilon\rightarrow0}u_{\epsilon}(0,\,0)>0\;.
\]
\end{lem}

\begin{lem}[H\"older regularity]
\label{lem:holder regularity of v} There exist constants $c_{1},\,c_{2},\,\lambda>0$ and $R'\in(0,1)$
independent of $\epsilon$ such that for all $\epsilon$ small enough,
for all $x,\,x'\in Q_{R'}$, 
\[
|u_{\epsilon}(x)-u_{\epsilon}(x')|\leq(c_{1}+c_{2}u_{\epsilon}(x))|x'-x|^{\lambda}\;.
\]
\end{lem}

These lemmas are proven in the rest of this section. Using these lemmas
let us now prove Proposition~\ref{prop:control v}.
\begin{proof}
Let $\beta>1/2$ and let $\epsilon$ be small enough such that $\epsilon^{\beta-1/2}\leq R'^{3}$
where $R'$ is defined in Lemma~\ref{lem:holder regularity of v}.
Therefore, by Lemma~\ref{lem:holder regularity of v}, for all $|q|,\,|p|\leq\epsilon^{\beta-1/2}$,
\[
|u_{\epsilon}(x)-u_{\epsilon}(0,\,0)|\leq(c_{1}+c_{2}u_{\epsilon}(0,\,0))\epsilon^{\lambda(\beta-1/2)}
\]
for some constants $c_{1},\,c_{2},\,\lambda>0$ independent of $\epsilon$.
By Lemma~\ref{lem:control v at 0}, we deduce that there exists $c_{3}>0$
independent of $\epsilon$ such that for $\epsilon$ small enough,
\[
|u_{\epsilon}(x)-u_{\epsilon}(0,\,0)|\leq c_{3}u_{\epsilon}(0,\,0)\epsilon^{\lambda(\beta-1/2)}
\]

Therefore, for all $|q|,\,|p|\leq\epsilon^{\beta-1/2}$, 
\[
u_{\epsilon}(q,\,p)=u_{\epsilon}(0,\,0)(1+\underset{\epsilon\rightarrow0}{O}(\epsilon^{\delta}))\;,
\]
where $\delta=\lambda(\beta-1/2)>0$. The proof of Proposition~\ref{prop:control v}
then follows from the expression of $u_{\epsilon}$ in~\eqref{eq:def u epsilon}. 
\end{proof}
Let us now prove Lemmas~\ref{lem:control v at 0} and~\ref{lem:holder regularity of v}.
\begin{proof}[Proof of Lemma~\ref{lem:control v at 0}]
Since $u_{\epsilon}\geq0$, let us assume that 
\[
\liminf_{\epsilon\rightarrow0}u_{\epsilon}(0,\,0)=0\;.
\]
Then one can build a sequence $(\epsilon_{n})_{n\geq0}$ such that
$\epsilon_{n}\underset{n\rightarrow\infty}{\longrightarrow}0$ and
$u_{\epsilon_{n}}(0,\,0)\underset{n\rightarrow\infty}{\longrightarrow}0$.
By definition of $u_{\epsilon_{n}}$, this ensures that under $\mathbb{P}_{(m,\,0)}$,
\[
\tau_{\mathcal{S}_{\epsilon_{n}}}^{\epsilon_{n}}\overset{\mathrm{L}^{1}}{\underset{n\rightarrow\infty}{\longrightarrow}}0\;.
\]
Up to taking a subsequence of $(\epsilon_{n})_{n\geq0}$ one can assume
that $\mathbb{P}_{(m,\,0)}$ almost-surely, 
\[
\tau_{\mathcal{S}_{\epsilon_{n}}}^{\epsilon_{n}}\underset{n\rightarrow\infty}{\longrightarrow}0\;.
\]
Let $K$ be a ball of radius $\theta>0$, independent of $n$, centered
around $(m,\,0)$ and let $\tau_{K^{c}}^{\epsilon_{n}}$ be the first
exit time of~\eqref{eq:sde} from $K$ then there exists a constant
$C>0$ independent of $n\geq0$ such that, $\mathbb{P}_{(m,\,0)}$
almost-surely, 
\[
0<\theta=\left|\left(q^{\epsilon_{n}}(\tau_{K^{c}}^{\epsilon_{n}}),\,p^{\epsilon_{n}}(\tau_{K^{c}}^{\epsilon_{n}})\right)-(m,\,0)\right|\leq C\tau_{K^{c}}^{\epsilon_{n}}+\sqrt{2\gamma\epsilon_{n}}\sup_{r\in[0,\,\tau_{K^{c}}^{\epsilon_{n}}]}|B(r)|\;.
\]
Furthermore, up to taking $\theta$ small enough one has by the continuity
of the trajectories of~\eqref{eq:sde} that almost-surely $\tau_{K^{c}}^{\epsilon_{n}}<\tau_{\mathcal{S}_{\epsilon_{n}}}^{\epsilon_{n}}$.
Reinjecting into the inequality above ensures that 
\[
0<\theta\leq C\tau_{\mathcal{S}_{\epsilon_{n}}}^{\epsilon_{n}}+\sqrt{2\gamma\epsilon_{n}}\sup_{r\in[0,\,\tau_{\mathcal{S}_{\epsilon_{n}}}^{\epsilon_{n}}]}|B_{r}|\underset{n\rightarrow\infty}{\longrightarrow}0\;,
\]
since $\tau_{\mathcal{S}_{\epsilon_{n}}}^{\epsilon_{n}}\underset{n\rightarrow\infty}{\longrightarrow}0$,
which is a clear contradiction. Hence the proof. 
\end{proof}
In order to prove Lemma~\ref{lem:holder regularity of v} we shall
first consider the processes~\eqref{eq:sde} and~\eqref{eq:sde_pur}
when the force $-\nabla U$ is smoothened. Namely, for $n\ge1$, for
$q\in\mathbb{R}^{d}$, let us define 
\begin{equation}
F_{n}(q)=-\int_{\mathbb{R}^{d}}\nabla U(q-q')\phi_{n}(q')\,\mathrm{d}q'\;,\label{eq:def F_n sigma_n}
\end{equation}
where $\phi_{n}(q')=n^{d}\phi(nq')$ with $\phi$ a smooth compactly
supported function on $\mathbb{R}^{d}$ integrating to $1$. Then,
$F_{n}\in C^{\infty}(\mathbb{R}^{d},\,\mathbb{R}^d)$ and for all compact
set $K\subset\mathbb{R}^{d}$, 
\begin{equation}
\sup_{q\in K}|F_{n}(q)+\nabla U(q)|\underset{n\rightarrow\infty}{\longrightarrow}0\;,\label{eq:convergence compact F_n}
\end{equation}
since $U\in C^{2}(\mathbb{R}^{d},\,\mathbb{R})$.

Let us also denote by $(X^{\alpha,\,n,\,\epsilon}(t)=(q^{\alpha,\,n,\,\epsilon}(t),\,p^{\alpha,\,n,\,\epsilon}(t)))_{t\geq0}$
the process satisfying~\eqref{eq:sde_pur} where $-\nabla U$ is
replaced by the function $F_{n}$. Similarly, denote by $(X^{n,\,\epsilon}(t)=(q^{n,\,\epsilon}(t),\,p^{n,\,\epsilon}(t)))_{t\geq0}$
the process satisfying~\eqref{eq:sde} where $-\nabla U$ is replaced
by the function $F_{n}$. Additionally, for a given set $\mathcal{C}$,
we denote by $\tau_{\mathcal{C}}^{\alpha,\,n,\,\epsilon}$ (resp.
$\tau_{\mathcal{C}}^{n,\,\epsilon}$) the first hitting time of $\mathcal{C}$
for the process $(X^{\alpha,\,n,\,\epsilon}(t))_{t\geq0}$ (resp.
$(X^{n,\,\epsilon}(t))_{t\geq0}$).

In order to prove Lemma~\ref{lem:holder regularity of v} we shall
require the following lemma which proof is provided below.
\begin{lem}
\label{lem:convergence expectation truncated} For all $x\in\mathbb{R}^{2d},$
\[
\lim_{N,\,M\rightarrow\infty}\lim_{n\rightarrow\infty}\lim_{\alpha\rightarrow0}\mathbb{E}_{x}\left[\tau_{\mathcal{O}_{N,\,\epsilon}}^{\alpha,\,n,\,\epsilon}\land M\right]=\mathbb{E}_{x}\left[\tau_{\mathcal{S}_{\epsilon}}^{\epsilon}\right]\;,
\]
where $\mathcal{O}_{N,\,\epsilon}$ is defined in~\eqref{eq:definition O_N}. 
\end{lem}

\begin{proof}
For all $\alpha,\,n,\,N,\,M>0$, by the Fubini permutation, 
\[
\mathbb{E}_{x}\left[\tau_{\mathcal{O}_{N,\,\epsilon}}^{\alpha,\,n,\,\epsilon}\land M\right]=\int_{0}^{M}\mathbb{P}_{x}(\tau_{\mathcal{O}_{N,\,\epsilon}}^{\alpha,\,n,\,\epsilon}>r)\mathrm{d}r\;.
\]
Therefore, the proof of~\eqref{eq:first limit alpha} in Lemma~\ref{lem:convergence proba until N}
which remains identical when replacing $-\nabla U$ by $F_{n}$, guarantees
by the dominated convergence theorem that 
\[
\mathbb{E}_{x}\left[\tau_{\mathcal{O}_{N,\,\epsilon}}^{\alpha,\,n,\,\epsilon}\land M\right]\underset{\alpha\rightarrow0}{\longrightarrow}\int_{0}^{M}\mathbb{P}_{x}(\tau_{\mathcal{O}_{N,\,\epsilon}}^{n,\,\epsilon}>r)\mathrm{d}r\;.
\]
Let us now look at the convergence of $\mathbb{P}_{x}(\tau_{\mathcal{O}_{N,\,\epsilon}}^{n,\,\epsilon}>r)$
when $n\rightarrow\infty$. In order to do that let us apply Gr\"onwall's
lemma which provides the existence of a constant $C>0$ depending
on $N$ but independent of $n$ such that, almost surely, for all
$t\in\left[0,\,\tau_{\mathcal{O}_{N,\,\epsilon}}^{n,\,\epsilon}\land\tau_{\mathcal{O}_{N,\,\epsilon}}^{\epsilon}\land M\right]$,
\[
\left|X^{n,\,\epsilon}(t)-X^{\epsilon}(t)\right|\leq M\sup_{(q,p)\in\mathcal{O}_{N,\,\epsilon}^{c}}|F_{n}(q)+\nabla U(q)|\mathrm{e}^{CM}\;,
\]
which goes to zero when $n\rightarrow\infty$ by~\eqref{eq:convergence compact F_n}
since $\mathcal{O}_{N,\,\epsilon}^{c}$ is a compact set of $\mathbb{R}^{2d}$.
Consequently, following the proof of~\eqref{eq:first limit alpha}
in Lemma~\ref{lem:convergence proba until N}, one also obtains the
following convergence 
\[
\mathbb{P}_{x}(\tau_{\mathcal{O}_{N,\,\epsilon}}^{n,\,\epsilon}>r)\underset{n\rightarrow\infty}{\longrightarrow}\mathbb{P}_{x}(\tau_{\mathcal{O}_{N,\,\epsilon}}^{\epsilon}>r)\;.
\]
It remains to consider the limit of $\mathbb{P}_{x}(\tau_{\mathcal{O}_{N,\,\epsilon}}^{\epsilon}>r)$
when $N\rightarrow\infty$, which follows immediately from the application
of the monotone convergence theorem since $\cup_{N\geq1}\mathcal{O}_{N,\,\epsilon}=\mathcal{S}_{\epsilon}$.
Therefore, 
\[
\int_{0}^{M}\mathbb{P}_{x}(\tau_{\mathcal{O}_{N,\,\epsilon}}^{\epsilon}>r)\mathrm{d}r\underset{N,\,M\rightarrow\infty}{\longrightarrow}\int_{0}^{\infty}\mathbb{P}_{x}(\tau_{\mathcal{S}_{\epsilon}}^{\epsilon}>r)\mathrm{d}r=\mathbb{E}_{x}\left[\tau_{\mathcal{S}_{\epsilon}}^{\epsilon}\right]\;,
\]
which concludes the proof. 
\end{proof}
Let us conclude this section with the proof of Lemma~\ref{lem:holder regularity of v}.
\begin{proof}[Proof of Lemma~\ref{lem:holder regularity of v}]
For all $x\in\R^{2d}\setminus\overline{\mathcal{O}_{N,\,\epsilon}}$,
let 
\[
v_{M,\,\alpha,\,n,\,N,\,\epsilon}(x):=\mathbb{E}_{x}\left[\tau_{\mathcal{O}_{N,\,\epsilon}}^{\alpha,\,n,\,\epsilon}\land M\right]\;.
\]
Following the same computation done in \textbf{Step 4} of the proof
of Proposition~\ref{prop:average time perturbed Langevin}, one easily
shows that 
\[
v_{M,\,\alpha,\,n,\,N,\,\epsilon}(x)=\mathbb{E}_{x}\left[\int_{0}^{\tau_{\mathcal{O}_{N,\,\epsilon}}^{\alpha,\,n,\,\epsilon}}\mathbb{P}_{X^{\alpha,\,n,\,\epsilon}(r)}(\tau_{\mathcal{O}_{N,\,\epsilon}}^{\alpha,\,n,\,\epsilon}\leq M)\mathrm{d}r\right]\;.
\]
Therefore, it follows from~\cite[Theorem 6.5.1]{Friedman} that $v_{M,\,\alpha,\,n,\,N,\,\epsilon}$
is a classical solution of the following equation, for all $x\in\R^{2d}\setminus\overline{\mathcal{O}_{N,\,\epsilon}}$,
\begin{align}
 & \langle p,\,\nabla_{q}v_{M,\,\alpha,\,n,\,N,\,\epsilon}(x)\rangle+\langle F_{n}(q)-\gamma p,\,\nabla_{p}v_{M,\,\alpha,\,n,\,N,\,\epsilon}(x)\rangle\nonumber \\
 & +\gamma\epsilon\Delta_{p}v_{M,\,\alpha,\,n,\,N,\,\epsilon}(x)+\alpha\epsilon\Delta_{q}v_{M,\,\alpha,\,n,\,N,\,\epsilon}(x)=-\mathbb{P}_{x}(\tau_{\mathcal{O}_{N,\,\epsilon}}^{\alpha,\,n,\,\epsilon}\leq M)\;.\label{eq:equation v_epsilon n M}
\end{align}
Let us now define for all $x=(q,p)\in Q_{1}$, 
\[
u_{M,\,\alpha,\,n,\,N,\,\epsilon}(x):=v_{M,\,\alpha,\,n,\,N,\,\epsilon}(m+\sqrt{\epsilon}q,\,-\sqrt{\epsilon}p)\;.
\]
It follows from~\eqref{eq:equation v_epsilon n M} that $u_{M,\,\alpha,\,n,\,N,\,\epsilon}$
satisfies the following equation 
\begin{align*}
 & -\langle p,\,\nabla_{q}u_{M,\,\alpha,\,n,\,N,\,\epsilon}(x)\rangle-\frac{1}{\sqrt{\epsilon}}\langle F_{n}(m+\sqrt{\epsilon}q),\,\nabla_{p}u_{M,\,\alpha,\,n,\,N,\,\epsilon}(x)\rangle-\gamma\langle p,\,\nabla_{p}u_{M,\,\alpha,\,n,\,N,\,\epsilon}(x)\rangle\\
 & +\gamma\Delta_{p}u_{M,\,\alpha,\,n,\,N,\,\epsilon}(x)+\alpha\Delta_{q}u_{M,\,\alpha,\,n,\,N,\,\epsilon}(x)=-\mathbb{P}_{(m+\sqrt{\epsilon}q,\,-\sqrt{\epsilon}p)}(\tau_{\mathcal{O}_{N,\,\epsilon}}^{\alpha,\,n,\,\epsilon}\leq M)\;.
\end{align*}

For all $x\in Q_{1}$, let us define its limit 
\[
u_{M,\,n,\,N,\,\epsilon}(x):=\lim_{\alpha\rightarrow0}u_{M,\,\alpha,\,n,\,N,\,\epsilon}(x)\;,
\]
which exists by Lemma~\ref{lem:convergence expectation truncated}.
Then, taking the limit of~\eqref{eq:equation v_epsilon n M} when
$\alpha\rightarrow0$ and using~\eqref{eq:first limit alpha} in
Lemma~\ref{lem:convergence proba until N}, one has that $u_{M,\,n,\,N,\,\epsilon}$
is a weak solution in $Q_{1}$, in the sense of distributions, of
\begin{align*}
 & -\langle p,\,\nabla_{q}u_{M,\,n,\,N,\,\epsilon}(x)\rangle-\frac{1}{\sqrt{\epsilon}}\langle F_{n}(m+\sqrt{\epsilon}q),\,\nabla_{p}u_{M,\,n,\,N,\,\epsilon}(x)\rangle-\gamma\langle p,\,\nabla_{p}u_{M,\,n,\,N,\,\epsilon}(x)\rangle\\
 & +\gamma\Delta_{p}u_{M,\,n,\,N,\,\epsilon}(x)=-\mathbb{P}_{(m+\sqrt{\epsilon}q,\,-\sqrt{\epsilon}p)}(\tau_{\mathcal{O}_{N,\,\epsilon}}^{n,\,\epsilon}\leq M)\;.
\end{align*}

Let 
\begin{equation}
A_{n,\,\epsilon}\geq\sup_{x\in Q_{1}}\left|\gamma p+\frac{1}{\sqrt{\epsilon}}F_{n}(m+\sqrt{\epsilon}q)\right|\;.\label{eq:def A_n epsilon}
\end{equation}

Applying the Harnack inequality in~\cite[Theorem 4]{Har} to the
function $u_{M,\,n,\,N,\,\epsilon}$ in $Q_{1}$, one has the existence
of constants $R\in(0,\,1)$, $C>0$ only depending on $A_{n,\,\epsilon}$
and $\gamma$ such that 
\begin{equation}
\sup_{x\in Q_{R}}u_{M,\,n,\,N,\,\epsilon}(x)\leq C\left(\inf_{x\in Q_{R}}u_{M,\,n,\,N,\,\epsilon}(x)+1\right)\;.\label{eq:harnack}
\end{equation}

Additionally, it follows from~\cite[Theorem 3]{Har} that there exist
constants $\lambda>0$ and $C'>0$ only depending on the dimension
$d$ and on $\gamma,\,A_{n,\,\epsilon}$ and $R$ such that for all
$x,\,x'\in Q_{R/2}$, 
\begin{equation}
|u_{M,\,n,\,N,\,\epsilon}(x')-u_{M,\,n,\,N,\,\epsilon}(x)|\leq C'\left(1+\sqrt{\int_{Q_{R}}|u_{M,\,n,\,N,\,\epsilon}(x)|^{2}\mathrm{d}x}\right)|x'-x|^{\lambda}\;.\label{eq:holder}
\end{equation}

We would like to take the limit when $n,\,M,\,N\rightarrow\infty$
in~\eqref{eq:harnack} and~\eqref{eq:holder} using Lemma~\ref{lem:convergence expectation truncated}
and consider the behaviour when $\epsilon\rightarrow0$. In this regard
we want to show that the constants $R,\,C,\,\lambda,\,C'$ are independent
of $n,\,\epsilon$. This amounts to proving that $A_{n,\,\epsilon}$
can be bounded from above by a constant independent of $n,\,\epsilon$.

In order to do that, notice that since $m$ is a critical point of
$U$, one has for all $x=(q,\,p)\in Q_{1}$ that 
\begin{align*}
\frac{1}{\sqrt{\epsilon}}\left|F_{n}(m+\sqrt{\epsilon}q)\right| & =\frac{1}{\sqrt{\epsilon}}\left|F_{n}(m+\sqrt{\epsilon}q)+\nabla U(m)\right|\\
 & \leq\frac{1}{\sqrt{\epsilon}}\int_{\mathbb{R}^{d}}\left|\nabla U(m+\sqrt{\epsilon}q+q')-\nabla U(m)\right|n^{d}\phi(nq')\,\mathrm{d}q'\\
 & =\frac{1}{\sqrt{\epsilon}}\int_{\mathbb{R}^{d}}\left|\nabla U(m+\sqrt{\epsilon}q+q'/n)-\nabla U(m)\right|\phi(q')\,\mathrm{d}q'\;.
\end{align*}
Besides, since $\phi$ is compactly supported in $\mathbb{R}^{d}$
and $U$ is $C^{2}$ in $\mathbb{R}^{d}$, there exists a constant
$c>0$ independent of $\epsilon,n$ such that for all $|q|\leq1$,
\begin{align*}
\frac{1}{\sqrt{\epsilon}}\left|F_{n}(m+\sqrt{\epsilon}q)\right| & \leq c+\frac{c}{n\sqrt{\epsilon}}\int_{\mathbb{R}^{d}}\left|q'\right|\phi(q')\,\mathrm{d}q'\;.
\end{align*}
Therefore, by taking $n$ large enough, namely $n\geq1/\sqrt{\epsilon}$,
one can has that $A_{n,\,\epsilon}$ defined in~\eqref{eq:def A_n epsilon}
can be taken independently of $n$ and $\epsilon$ for $n\geq1/\sqrt{\epsilon}$
which allows us by Lemma~\ref{lem:convergence expectation truncated}
to take the limit when $n\rightarrow\infty$ and $N,\,M\rightarrow\infty$
in~\eqref{eq:harnack} and~\eqref{eq:holder}. Consequently, by
definition of $u_{\epsilon}$ in~\eqref{eq:def u epsilon}, one obtains
\begin{equation}
\sup_{x\in Q_{R}}u_{\epsilon}(x)\leq C\left(\inf_{x\in Q_{R}}u_{\epsilon}(x)+1\right)\;.\label{eq:harnack v}
\end{equation}
Moreover, for all $x,\,x'\in Q_{R/2}$, 
\begin{equation}
|u_{\epsilon}(x')-u_{\epsilon}(x)|\leq C'\left(1+\sqrt{\int_{Q_{R}}|u_{\epsilon}(x)|^{2}\mathrm{d}x}\right)|x'-x|^{\alpha}\;.\label{eq:Holder inequality}
\end{equation}
Furthermore, we deduce from~\eqref{eq:harnack v} that for all $x\in Q_{R}$,
\begin{align*}
\sqrt{\int_{Q_{R}}|u_{\epsilon}(x')|^{2}\mathrm{d}x'} & \leq\sqrt{2C^{2}u_{\epsilon}(x)^{2}|Q_{R}|+2C^{2}|Q_{R}|}\\
 & \leq u_{\epsilon}(x)\sqrt{2C^{2}|Q_{R}|}+\sqrt{2C^{2}|Q_{R}|}\;.
\end{align*}
Therefore, for all $x,\,x'\in Q_{R/2}$, 
\begin{align*}
|u_{\epsilon}(x')-u_{\epsilon}(x)| & \leq C'\left(1+\sqrt{2C^{2}|Q_{R}|}+u_{\epsilon}(x)\sqrt{2C^{2}|Q_{R}|}\right)|x'-x|^{\alpha}\;,
\end{align*}
which concludes the proof with $R'=R/2$. 
\end{proof}

\section{Proof of Proposition~\ref{prop:main1}}

\label{sec:Estimate of numerator}

The proof of Proposition~\ref{prop:main1} relies on the important
result below which provides sharp estimates on the function $h_{\mathcal{M}_{\epsilon},\,\mathcal{S}_{\epsilon}}$. Consider
the set 
\begin{equation}
\mathcal{W}=\{x\in\mathbb{R}^{2d}:V(x)<V(\sigma,0)\}\;.\label{eq:def W}
\end{equation}
We denote by $\mathcal{W}_{m}$ (resp. $\mathcal{W}_{s}$) the connected
component of $\mathcal{W}$ containing the local minimum $(m,\,0)$
(resp. $(s,\,0)$) of $V$.
\begin{prop} 
\label{prop:ineq h_AB} There exist constants $C,\,\beta>0$ independent
of $\epsilon$ such that for all $\epsilon$ small enough, 
\begin{itemize}
\item for all $x\in\mathcal{W}_{m}$, 
\[
h_{\mathcal{M}_{\epsilon},\,\mathcal{S}_{\epsilon}}(x)\geq1-\frac{C}{\epsilon^{\beta}}\mathrm{exp}\left(\frac{V(x)-V(\sigma,\,0)}{\epsilon}\right)\;.
\]
\item for all $x\in\mathcal{W}_{s}$, 
\[
h_{\mathcal{M}_{\epsilon},\,\mathcal{S}_{\epsilon}}(x)\leq\frac{C}{\epsilon^{\beta}}\mathrm{exp}\left(\frac{V(x)-V(\sigma,\,0)}{\epsilon}\right)\;.
\]
\end{itemize}
\end{prop}

The proof of Proposition~\ref{prop:ineq h_AB} is based on the two propositions stated below. Define the process $(\widetilde{q}^{\epsilon}(t),\,\widetilde{p}^{\epsilon}(t))_{t\geq0}$
\begin{equation}
\left\{ \begin{aligned} & \mathrm{d}\widetilde{q}^{\epsilon}(t)=-\widetilde{p}^{\epsilon}(t)\mathrm{d}t,\\
 & \mathrm{d}\widetilde{p}^{\epsilon}(t)=\nabla U(\widetilde{q}^{\epsilon}(t))\mathrm{d}t+\gamma\widetilde{p}^{\epsilon}(t)\mathrm{d}t+\sqrt{2\gamma\epsilon}\mathrm{d}B(t)\;.
\end{aligned}
\right.\label{eq:adjoint langevin}
\end{equation}
Let   
\begin{equation}
\zeta^{\epsilon}:=\tau_{\partial\mathcal{W}}^{\epsilon}\land\tau_{\mathcal{M}_{\epsilon}}^{\epsilon}\land\tau_{\mathcal{S}_{\epsilon}}^{\epsilon},\quad\widetilde{\zeta}^{\epsilon}:=\widetilde{\tau}_{\partial\mathcal{W}}^{\epsilon}\land\widetilde{\tau}_{\mathcal{M_{\epsilon}}}^{\epsilon}\land\widetilde{\tau}_{\mathcal{S_{\epsilon}}}^{\epsilon}\;,\label{eq:def tau_epsilon}
\end{equation}
where $\widetilde{\tau}_{\mathcal{C}}^{\epsilon}$ (resp. $\tau_{\mathcal{C}}^{\epsilon}$)
is the first hitting time of the set $\mathcal{C}$ for the process~\eqref{eq:adjoint langevin}
(resp.~\eqref{eq:sde}).
\begin{prop}
\label{lem:first ineq h_AB} For all $x=(q,p)\in\mathcal{W}_{m}$, 
\[
h_{\mathcal{M}_{\epsilon},\,\mathcal{S}_{\epsilon}}(x)\geq1-\mathrm{exp}\left(\frac{V(x)-V(\sigma,\,0)}{\epsilon}\right)\mathbb{E}_{(q,\,-p)}\left[\mathrm{exp}(d\gamma\widetilde{\zeta}^{\epsilon})\right]\;.
\]
For all $x\in\mathcal{W}_{s}$, 
\[
h_{\mathcal{M}_{\epsilon},\,\mathcal{S}_{\epsilon}}(x)\leq\mathrm{exp}\left(\frac{V(x)-V(\sigma,\,0)}{\epsilon}\right)\mathbb{E}_{(q,\,-p)}\left[\mathrm{exp}(d\gamma\widetilde{\zeta}^{\epsilon})\right]\;.
\]
\end{prop}

\begin{prop}
\label{prop:bound exponential exit time} There exist $C,\,\beta>0$
independent of $\epsilon$ such that for all $x\in\mathcal{W}$, 
\[
\mathbb{E}_{x}\left[\mathrm{exp}(d\gamma\widetilde{\zeta}^{\epsilon})\right]\leq\frac{C}{\epsilon^{\beta}}\;.
\]
\end{prop}

It is straightforward to deduce the proof of Proposition~\ref{prop:ineq h_AB}
from the two propositions above. Therefore, the main goal of the next sections is to prove Propositions~\ref{lem:first ineq h_AB} and~\ref{prop:bound exponential exit time}. The proof
of Proposition~\ref{lem:first ineq h_AB} is given below whereas
the proof of Proposition~\ref{prop:bound exponential exit time}
is much more involved and relies on several intermediate results detailed in the next sections.
Therefore, in this section we shall only prove Proposition~\ref{lem:first ineq h_AB} and conclude the proof of Proposition~\ref{prop:main1}
based on Proposition~\ref{prop:ineq h_AB}.
\begin{proof}[Proof of Proposition~\ref{lem:first ineq h_AB}]

We only provide the proof in the case when $x\in\mathcal{W}_{m}$
as the proof when $x\in\mathcal{W}_{s}$ is exactly similar. Let us
fix $x=(q,\,p)\in\mathcal{W}_{m}$ throughout this proof. It follows
from the continuity of the trajectories of~\eqref{eq:sde} that the
process $(q^{\epsilon}(t),\,p^{\epsilon}(t))_{t\geq0}$ starting
from $x\in\mathcal{W}_{m}$ needs to cross the boundary $\partial\mathcal{W}$
in order to reach $\mathcal{S}_{\epsilon}$. As a result, 
\begin{equation}
1-h_{\mathcal{M}_{\epsilon},\,\mathcal{S}_{\epsilon}}(x)\leq\mathbb{P}_{x}(\tau_{\partial\mathcal{W}}^{\epsilon}\leq\tau_{\mathcal{M}_{\epsilon}}^{\epsilon})\;.\label{eq:h_AB and first exit}
\end{equation}

The stopping time $\zeta^{\epsilon}$ defined in~\eqref{eq:def tau_epsilon}
then satisfies $\zeta^{\epsilon}=\tau_{\partial\mathcal{W}}^{\epsilon}\land\tau_{\mathcal{M}_{\epsilon}}^{\epsilon}$ under $\mathbb{P}_x$.
Moreover, applying It\^o's formula to the process $(V(q^{\epsilon}(t),\,p^{\epsilon}(t)))_{t\geq0}$
at the time $\zeta^{\epsilon}\land t$ ensures that, $\mathbb{P}_x$ almost-surely, for
all $t\geq0$, 
\[
V(q^{\epsilon}(\zeta^{\epsilon}\land t),\,p^{\epsilon}(\zeta^{\epsilon}\land t))-V(x)=-\gamma\int_{0}^{\zeta^{\epsilon}\land t}|p^{\epsilon}(r)|^{2}\mathrm{d}r+\sqrt{2\gamma\epsilon}\int_{0}^{\zeta^{\epsilon}\land t}\langle p^{\epsilon}(r),\,\mathrm{d}B(r)\rangle+d\gamma\epsilon\left(\zeta^{\epsilon}\land t\right)\;.
\]
As a result, 
\begin{align*}
 & \mathbb{E}_{x}\left[\mathrm{exp}\left(\left(V(q^{\epsilon}(\zeta^{\epsilon}\land t),\,p^{\epsilon}(\zeta^{\epsilon}\land t))-V(x)\right)/\epsilon\right)\right]\\
 & =\mathbb{E}_{x}\left[\mathrm{exp}\left(-\frac{\gamma}{\epsilon}\int_{0}^{\zeta^{\epsilon}\land t}|p^{\epsilon}(r)|^{2}\mathrm{d}r+\sqrt{2\frac{\gamma}{\epsilon}}\int_{0}^{\zeta^{\epsilon}\land t}\langle p^{\epsilon}(r),\,\mathrm{d}B(r)\rangle\right)\mathrm{e}^{d\gamma(\zeta^{\epsilon}\land t)}\right]\;.
\end{align*}
It follows from the Girsanov argument developed in the proof of~\cite[Lemma 3.8]{LelRamRey}
that 
\[
\mathcal{Z}:t\geq0\mapsto\mathrm{exp}\left(-\frac{\gamma}{\epsilon}\int_{0}^{\zeta^{\epsilon}\land t}|p^{\epsilon}(r)|^{2}\mathrm{d}r+\sqrt{2\frac{\gamma}{\epsilon}}\int_{0}^{\zeta^{\epsilon}\land t}\langle p^{\epsilon}(r),\,\mathrm{d}B(r)\rangle\right)
\]
is a martingale and under the probability measure $\mathrm{d}\mathbb{Q}_{t}=\mathcal{Z}_{t}\mathrm{d}\mathbb{P}$
the process 
\[
\left(B(r)-\sqrt{2\frac{\gamma}{\epsilon}}\int_0^r p^{\epsilon}(u)\mathrm{d}u\right)_{r\geq0}
\]
is a Brownian motion under the filtration $\sigma(B(r),\,r\in[0,\,t])$.
Therefore, letting the process 
\[
\left\{ \begin{aligned} & \mathrm{d}\widehat{q}^{\epsilon}(t)=\widehat{p}^{\epsilon}(t)\mathrm{d}t\;,\\
 & \mathrm{d}\widehat{p}^{\epsilon}(t)=-\nabla U(\widehat{q}^{\epsilon}(t))\mathrm{d}t+\gamma\widehat{p}^{\epsilon}(t)\mathrm{d}t+\sqrt{2\gamma\epsilon}\mathrm{d}B(t)\;,
\end{aligned}
\right.
\]
and let $\widehat{\zeta}^{\epsilon}$ be its first hitting time of
the set $\partial\mathcal{W}\cup\partial\mathcal{M}_{\epsilon}$,
one deduces that for all $t\geq0$, 
\[
\mathbb{E}_{x}\left[\mathrm{exp}\left(\ \frac{V(q^{\epsilon}(\zeta^{\epsilon}\land t),\,p^{\epsilon}(\zeta^{\epsilon}\land t))-V(x)}{\epsilon}
\right)\right]=\mathbb{E}_{x}\left[\mathrm{e}^{d\gamma(\widehat{\zeta}^{\epsilon}\land t)}\right]\;.
\]

Additionally, computations done in~\cite[Section 6.1]{LelRamRey}
ensure that the process $(\widehat{q}^{\epsilon}(t),\,-\widehat{p}^{\epsilon}(t))_{t\geq0}$
and the process $(\widetilde{q}^{\epsilon}(t),\,\widetilde{p}^{\epsilon}(t))_{t\geq0}$
defined in~\eqref{eq:adjoint langevin} share the same law, thus
\[
\mathbb{E}_{x}\left[\mathrm{e}^{d\gamma(\widehat{\zeta}^{\epsilon}\land t)}\right]=\mathbb{E}_{(q,-p)}\left[\mathrm{e}^{d\gamma(\widetilde{\zeta}^{\epsilon}\land t)}\right]\;,
\]
where $\widetilde{\zeta}^{\epsilon}$ is defined in~\eqref{eq:def tau_epsilon}.
All in all this ensures that 
\[
\mathbb{E}_{x}\left[\mathrm{exp}\left(\frac{V(q^{\epsilon}(\zeta^{\epsilon}\land t),\,p^{\epsilon}(\zeta^{\epsilon}\land t))-V(x)}{\epsilon}\right)\right]=\mathbb{E}_{(q,\,-p)}\left[\mathrm{e}^{d\gamma(\widetilde{\zeta}^{\epsilon}\land t)}\right]\;.
\]
As a result, taking $t\rightarrow\infty$ by the dominated convergence
theorem in the left expectation and the monotone convergence theorem
in the right expectation yields 
\[
\mathbb{E}_{x}\left[\mathrm{exp}\left(\frac{V(q^{\epsilon}(\zeta^{\epsilon}),\,p^{\epsilon}(\zeta^{\epsilon}))-V(x)}{\epsilon}\right)\right]=\mathbb{E}_{(q,\,-p)}\left[\mathrm{e}^{d\gamma\widetilde{\zeta}^{\epsilon}}\right]\;.
\]
Moreover, since 
\[
\mathbb{E}_{x}\left[\mathrm{exp}\left(\frac{V(q^{\epsilon}(\zeta^{\epsilon}),\,p^{\epsilon}(\zeta^{\epsilon}))-V(x)}{\epsilon}\right)\right]\geq\mathrm{exp}\left( \frac{V(\sigma,\,0)-V(x)}{\epsilon}\right)\mathbb{P}_{x}\left(\tau_{\partial\mathcal{W}}^{\epsilon}<\tau_{\mathcal{M}_{\epsilon}}^{\epsilon}\right).
\]
One deduces that 
\[
\mathbb{P}_{x}\left(\tau_{\partial\mathcal{W}}^{\epsilon}<\tau_{\mathcal{M}_{\epsilon}}^{\epsilon}\right)\leq\mathrm{exp}\left(\frac{V(x)-V(\sigma,\,0)}{\epsilon}\right)\mathbb{E}_{(q,\,-p)}\left[\mathrm{e}^{d\gamma\tilde{\zeta}^{\epsilon}}\right]\;.
\]
Inequality~\eqref{eq:h_AB and first exit} then allows us to conclude
the proof of this lemma. 
\end{proof}
Assuming Proposition~\ref{prop:ineq h_AB} holds, let us now conclude this section
with the proof of Proposition~\ref{prop:main1}.
\begin{proof}[Proof of Proposition~\ref{prop:main1}]

Let us take a small constant $\delta>0$ such that 
\begin{equation}
V(\sigma,\,0)-\delta>\max(V(m,\,0),V(s,\,0))\;.\label{eq:set energy threshold}
\end{equation}
Then, let 
\[
\mathcal{A}:=\{x\in\mathbb{R}^{2d}:V(x)\leq V(\sigma,\,0)-\delta\}\;.
\]

By~\eqref{eq:def h^*}, 
\[
\int_{\mathbb{R}^{2d}}h_{\mathcal{M}_{\epsilon},\,\mathcal{S}_{\epsilon}}^{*}(x)\mu_{\epsilon}(\mathrm{d}x)=\int_{\mathbb{R}^{2d}}h_{\mathcal{M}_{\epsilon},\,\mathcal{S}_{\epsilon}}(x)\mu_{\epsilon}(\mathrm{d}x)\;.
\]
Decompose the right hand side into 
\[
\int_{\mathcal{A}^{c}}h_{\mathcal{M}_{\epsilon},\,\mathcal{S}_{\epsilon}}(x)\mu_{\epsilon}(\mathrm{d}x)+\int_{\mathcal{A}\cap\mathcal{W}_{m}}h_{\mathcal{M}_{\epsilon},\,\mathcal{S}_{\epsilon}}(x)\mu_{\epsilon}(\mathrm{d}x)+\int_{\mathcal{A}\cap\mathcal{W}_{s}}h_{\mathcal{M}_{\epsilon},\,\mathcal{S}_{\epsilon}}(x)\mu_{\epsilon}(\mathrm{d}x)\;.
\]
Also by Proposition~\ref{prop:tight}, there exists a constant $C>0$
independent of $\epsilon$ such that 
\[
\int_{\mathcal{A}^{c}}h_{\mathcal{M}_{\epsilon},\,\mathcal{S}_{\epsilon}}(x)\mu_{\epsilon}(\mathrm{d}x)\leq\frac{C}{Z_{\epsilon}}\mathrm{e}^{-(V(\sigma,\,0)-\delta)/\epsilon}=o_{\epsilon}(1)\frac{1}{Z_{\epsilon}}(2\pi\epsilon)^{d}\frac{1}{\sqrt{\det{\mathbb{H}_{U}^{m}}}}e^{-V(m,\,0)/\epsilon}\;,
\]
according to~\eqref{eq:set energy threshold}. Moreover, the first
bound in Proposition~\ref{prop:ineq h_AB} ensures that 
\[
\int_{\mathcal{A}\cap\mathcal{W}_{m}}h_{\mathcal{M}_{\epsilon},\,\mathcal{S}_{\epsilon}}(x)\mu_{\epsilon}(\mathrm{d}x)=[1+o_{\epsilon}(1)]\mu_{\epsilon}(\mathcal{A}\cap\mathcal{W}_{m})=[1+o_{\epsilon}(1)]\frac{1}{Z_{\epsilon}}(2\pi\epsilon)^{d}\frac{1}{\sqrt{\det{\mathbb{H}_{U}^{m}}}}e^{-V(m,\,0)/\epsilon}\;,
\]
by the Laplace asymptotics on the function $\mathrm{e}^{-V(x)/\epsilon}$
in $\mathcal{A}\cap\mathcal{W}_{m}$.

Finally, the second bound in Proposition~\ref{prop:ineq h_AB} ensures
the existence of constants $C,\,\beta>0$ independent of $\epsilon$
such that 
\begin{align*}
\int_{\mathcal{A}\cap\mathcal{W}_{s}}h_{\mathcal{M}_{\epsilon},\,\mathcal{S}_{\epsilon}}(x)\mu_{\epsilon}(\mathrm{d}x) & \leq\frac{C}{\epsilon^{\beta}}\int_{\mathcal{A}\cap\mathcal{W}_{s}}\frac{1}{Z_{\epsilon}}\mathrm{e}^{(V(x)-V(\sigma,\,0))/\epsilon}\mathrm{e}^{-V(x)/\epsilon}\mathrm{d}x\\
 & \leq\frac{C|\mathcal{A}|}{\epsilon^{\beta}}\frac{1}{Z_{\epsilon}}\mathrm{e}^{-V(\sigma,\,0)/\epsilon}=o_{\epsilon}(1)\frac{1}{Z_{\epsilon}}(2\pi\epsilon)^{d}\frac{1}{\sqrt{\det{\mathbb{H}_{U}^{m}}}}e^{-V(m,\,0)/\epsilon}\;.
\end{align*}

All in all, this guarantees that

\[
\int_{\mathbb{R}^{2d}}h_{\mathcal{M}_{\epsilon},\,\mathcal{S}_{\epsilon}}^{*}(x)\mu_{\epsilon}(\mathrm{d}x)=[1+o_{\epsilon}(1)]\frac{1}{Z_{\epsilon}}(2\pi\epsilon)^{d}\frac{1}{\sqrt{\det{\mathbb{H}_{U}^{m}}}}e^{-V(m,\,0)/\epsilon}\;,
\]
which concludes the proof since $V(m,\,0)=U(m)$. 
\end{proof}

\section{Transition density of the process absorbed at a boundary}

\label{sec:density estimate}

Let us now delve into intermediate results aimed at proving Proposition~\ref{prop:bound exponential exit time} which proof is carried out in Section~\ref{sec:proof op Prop bound exponential exit time}. In this section, we provide an important identity involving the transition densities of the processes~\eqref{eq:sde} and~\eqref{eq:adjoint langevin} absorbed when exiting the following bounded, $C^{2}$, open set 
\begin{equation}
\mathcal{D}_{\epsilon}:=\mathcal{W}\cap\overline{\mathcal{M}_{\epsilon}}^{c}\cap\overline{\mathcal{S}_{\epsilon}}^{c}\;,\label{eq:def domain D}
\end{equation}
where $\mathcal{W}$ is defined in~\eqref{eq:def W}. The proofs mainly follow the scheme developed in~\cite[Sections 5 and 6]{LelRamRey}. 

\begin{prop}
\label{prop:property density} The process~\eqref{eq:sde}
(resp.~\eqref{eq:adjoint langevin}) starting from $x\in\mathcal{D}_{\epsilon}$
and absorbed at the boundary $\partial\mathcal{D}_{\epsilon}$ admits
a transition density $\mathrm{p}_{t}^{\epsilon,\,\mathcal{D}_{\epsilon}}(x,\,y)$
(resp. $\widetilde{\mathrm{p}}_{t}^{\epsilon,\,\mathcal{D_{\epsilon}}}(x,\,y)$),
i.e. for all measurable set $A\subset\mathcal{D}_{\epsilon}$, for
all $x\in\mathcal{D}_{\epsilon}$ and $t>0$,
\begin{align}
\mathbb{P}_{x}((q^{\epsilon}(t),\,p^{\epsilon}(t))\in A,\,\tau_{\partial\mathcal{D}_{\epsilon}}^{\epsilon}>t)&=\int_{A}\mathrm{p}_{t}^{\epsilon,\,\mathcal{D}_{\epsilon}}(x,\,y)\mathrm{d}y\,,\label{eq:killed density definition 1}\\
\mathbb{P}_{x}((\widetilde{q}^{\epsilon}(t),\,\widetilde{p}^{\epsilon}(t))\in A,\,\widetilde{\tau}_{\partial\mathcal{D_{\epsilon}}}^{\epsilon}>t)&=\int_{A}\widetilde{\mathrm{p}}_{t}^{\epsilon,\,\mathcal{D_{\epsilon}}}(x,\,y)\mathrm{d}y\;,\label{eq:killed density definition 2}
\end{align} 
where $\tau_{\partial\mathcal{D}_{\epsilon}}^{\epsilon}$ (resp. $\widetilde{\tau}_{\partial\mathcal{D_{\epsilon}}}^{\epsilon}$)
is the first hitting time of $\partial\mathcal{D}_{\epsilon}$ for the process
$(q^{\epsilon}(t),\,p^{\epsilon}(t))_{t\geq0}$ (resp. $(\widetilde{q}^{\epsilon}(t),\,\widetilde{p}^{\epsilon}(t))_{t\geq0}$). Furthermore, for all $t>0$, $x,\,y\in\mathcal{D}_{\epsilon}$, it holds that 
\begin{equation}
\mathrm{p}_{t}^{\epsilon,\,\mathcal{D}_{\epsilon}}(x,\,y)=\mathrm{e}^{d\gamma t}\widetilde{\mathrm{p}}_{t}^{\epsilon,\,\mathcal{D_{\epsilon}}}(y,\,x)\;.\label{eq:property density}
\end{equation}
\end{prop}

The proof of Proposition~\ref{prop:property density} is first obtained when the coefficients in~\eqref{eq:sde}
and~\eqref{eq:adjoint langevin} are regularized. Recall then the sequence $(F_{n})_{n\geq1}$ in $C^{\infty}(\mathbb{R}^{d})$
defined in~\eqref{eq:def F_n sigma_n} and denote this time by $(q^{\alpha,\,n,\,\epsilon}(t),\,p^{\alpha,\,n,\,\epsilon}(t))_{t\geq0}$
the process satisfying 
\[
\left\{ \begin{aligned} & \mathrm{d}q^{\alpha,\,n,\,\epsilon}(t)=p^{\alpha,\,n,\,\epsilon}(t)\mathrm{d}t+\sqrt{2\alpha\epsilon}\mathrm{d}\widetilde{B}(t)\;,\\
 & \mathrm{d}p^{\alpha,\,n,\,\epsilon}(t)=F_{n}(q^{\alpha,\,n,\,\epsilon}(t))\mathrm{d}t-\gamma p^{\alpha,\,n,\,\epsilon}(t)\mathrm{d}t+\sqrt{2\gamma\epsilon}\mathrm{d}B(t)\;.
\end{aligned}
\right.
\]
and by $(\widetilde{q}^{\alpha,\,n,\,\epsilon}(t),\,\widetilde{p}^{\alpha,\,n,\,\epsilon}(t))_{t\geq0}$
the process satisfying 
\begin{equation}
\left\{ \begin{aligned} & \mathrm{d}\widetilde{q}^{\alpha,\,n,\,\epsilon}(t)=-\widetilde{p}^{\alpha,\,n,\,\epsilon}(t)\mathrm{d}t+\sqrt{2\alpha\epsilon}\mathrm{d}\widetilde{B}(t)\;,\\
 & \mathrm{d}\widetilde{p}^{\alpha,\,n,\,\epsilon}(t)=-F_{n}(\widetilde{q}^{\alpha,\,n,\,\epsilon}(t))\mathrm{d}t+\gamma\widetilde{p}^{\alpha,\,n,\,\epsilon}(t)\mathrm{d}t+\sqrt{2\gamma\epsilon}\mathrm{d}B(t)\;.
\end{aligned}
\right.\label{eq:adjoint langevin regularized}
\end{equation}

Since $F_{n}$ is a smooth function and the set $\mathcal{D}_{\epsilon}$
is open, bounded and $C^{2}$, the existence of the transition densities
for the processes $(q^{\alpha,\,n,\,\epsilon}(t),\,p^{\alpha,\,n,\,\epsilon}(t))_{t\geq0}$
and $(\widetilde{q}^{\alpha,\,n,\,\epsilon}(t),\,\widetilde{p}^{\alpha,\,n,\,\epsilon}(t))_{t\geq0}$
absorbed at the boundary $\partial\mathcal{D}_{\epsilon}$ follows
from~\cite[Chapter 14, Section 4]{Friedman2}.
We shall denote them by $\mathrm{p}_{t}^{\alpha,\,n,\,\epsilon,\,\mathcal{D}_{\epsilon}}(x,\,y)$
and $\widetilde{\mathrm{p}}_{t}^{\alpha,\,n,\,\epsilon,\,\mathcal{D}_{\epsilon}}(x,\,y)$.
Additionally, these transition densities vanish continuously on $\mathbb{R}_{+}^{*}\times\partial(\mathcal{D}_{\epsilon}\times\mathcal{D}_{\epsilon})$.

Moreover, they satisfy the Kolmogorov backward and forward equation. Namely, if we define the differential operator 
\[
\mathcal{L}_{\alpha,\,n,\,\epsilon}=\langle p,\,\nabla_{q}\rangle+\langle F_{n}(q),\,\nabla_{p}\rangle-\gamma\langle p,\,\nabla_{p}\rangle+\alpha\epsilon\Delta_{q}+\gamma\epsilon\Delta_{p}\;,
\]
with adjoint generator in the flat canonical space 
\begin{equation}
(\mathcal{L}_{\alpha,\,n,\,\epsilon})^{\dagger}=-\langle p,\,\nabla_{q}\rangle-\langle F_{n}(q),\,\nabla_{p}\rangle+\gamma\mathrm{div}(p\cdot)+\alpha\epsilon\Delta_{q}+\gamma\epsilon\Delta_{p}\;,\label{eq:adjoint generator}
\end{equation}
then, we have 
\begin{equation}
\begin{aligned}
& \partial_{t}\mathrm{p}_{t}^{\alpha,\,n,\,\epsilon,\,\mathcal{D}_{\epsilon}}(x,\,y)=\mathcal{L}_{\alpha,\,n,\,\epsilon;\,x}\mathrm{p}_{t}^{\alpha,\,n,\,\epsilon,\,\mathcal{D}_{\epsilon}}(x,\;y)\text{ and}\\
& \partial_{t}(\mathrm{e}^{d\gamma t}\widetilde{\mathrm{p}}_{t}^{\alpha,\,n,\,\epsilon,\,\mathcal{D}_{\epsilon}}(x,\,y))=\mathrm{e}^{d\gamma t}\mathcal{L}_{\alpha,\,n,\,\epsilon;\,y}\widetilde{\mathrm{p}}_{t}^{\alpha,\,n,\,\epsilon,\,\mathcal{D}_{\epsilon}}(x,\,y)\;,
\end{aligned}
\label{eq:kolmogorov equations}
\end{equation} 
where $\mathcal{L}_{\alpha,\,n,\,\epsilon;\,z}$ corresponds to the
differentiation with respect to the variable $z$.

Before proving Proposition~\ref{prop:property density} we show the following lemma obtained in~\cite[Lemma 5.6]{LelRamRey} for cylindrical domains. The proof is adapted here for the non-cylindrical survival set $\mathcal{D}_\epsilon$. 
\begin{lem}
\label{lem:small time convergence integral}

Let $g\in C_{c}^{\infty}(\mathcal{D}_{\epsilon})$.
Then, for all $x\in\mathcal{D}_{\epsilon}$, 
\begin{equation}\label{eq:zero time convergence density}
    \int_{\mathcal{D}_{\epsilon}}\widetilde{\mathrm{p}}_{t}^{\alpha,\,n,\,\epsilon,\,\mathcal{D}_{\epsilon}}(y,\,x)g(y)\mathrm{d}y\underset{t\rightarrow0}{\longrightarrow}g(x)\;.
\end{equation}
\end{lem}

\begin{proof}[Proof of Lemma~\ref{lem:small time convergence integral}]
For $t>0$, let us define the following covariance matrix 
\[
\Sigma_{t}:=\begin{pmatrix}2\alpha\epsilon t\,I_{d} & 0\\
0 & 2\gamma\epsilon t\,I_{d}
\end{pmatrix}\;.
\]
Define the following transition density for $t>0$ and $x,\,y\in\mathbb{R}^{2d}$,
\[
\widehat{\mathrm{p}}_{t}^{\alpha,\,\epsilon}(x,\,y)=\frac{1}{(4\pi t\sqrt{\gamma\alpha\epsilon^{2}})^{d}}\mathrm{e}^{-\frac{1}{2}\langle x-y,\,\Sigma_{t}^{-1}(x-y)\rangle}\;.
\]
If we replace the transition density in~\eqref{eq:zero time convergence density} by $\widehat{\mathrm{p}}_{t}^{\alpha,\,\epsilon}(x,\,y)$, the convergence follows by an appropriate change of variable. Therefore, we consider here the difference $\widetilde{\mathrm{p}}_{t}^{\alpha,\,n,\,\epsilon,\,\mathcal{D}_{\epsilon}}(x,\,y)-\widehat{\mathrm{p}}_{t}^{\alpha,\,\epsilon}(x,\,y)$ and provide a mild formulation of its expression which is then used to prove the convergence.

Let us fix $x,\,y\in\mathcal{D}_{\epsilon}$ and let $\varphi\in C_{c}^{\infty}(\mathcal{D}_{\epsilon})$ with values in $[0,1]$ such that $\varphi(x)=\varphi(y)=1$. Following the idea developed in~\cite[Lemma 5.6]{LelRamRey}, we define the following
function for $t>0$, 
\begin{equation}\label{eq:def h_t(u)}
h_{t}:u\in(0,\,t)\mapsto\mathrm{e}^{d\gamma u}\int_{\mathcal{D}_{\epsilon}}\widetilde{\mathrm{p}}_{u}^{\alpha,\,n,\,\epsilon,\,\mathcal{D}_{\epsilon}}(x,\,z)\widehat{\mathrm{p}}_{t-u}^{\alpha,\,\epsilon}(y,\,z)\varphi(z)\mathrm{d}z\;.
\end{equation}
\textbf{Step 1:} Let us then show that for $t>0$,
\begin{equation}\label{eq:difference density expr}
\mathrm{e}^{d\gamma t}\widetilde{\mathrm{p}}_{t}^{\alpha,\,n,\,\epsilon,\,\mathcal{D}_{\epsilon}}(x,\,y)-\widehat{\mathrm{p}}_{t}^{\alpha,\,\epsilon}(x,\,y)=\int_{0}^{t}\frac{\mathrm{d}h_{t}(u)}{\mathrm{d}u}\mathrm{d}u\;.    
\end{equation}
It is sufficient to prove that
\begin{equation}\label{eq:double limit h_t}
    \lim_{u\rightarrow0}h_t(u)=\widehat{\mathrm{p}}_{t}^{\alpha,\,\epsilon}(x,\,y),\quad\lim_{u\rightarrow t}h_t(u)=\mathrm{e}^{d\gamma t}\widetilde{\mathrm{p}}_{t}^{\alpha,\,n,\,\epsilon,\,\mathcal{D}_{\epsilon}}(x,\,y)\;.
\end{equation}

Writing 
\[
h_{t}(u)=\mathrm{e}^{d\gamma u}\,\mathbb{E}_{x}\left[\widehat{\mathrm{p}}_{t-u}^{\alpha,\,\epsilon}(y,\,(\widetilde{q}^{\alpha,\,n,\,\epsilon}(u),\,\widetilde{p}^{\alpha,\,n,\,\epsilon}(u)))\varphi(\widetilde{q}^{\alpha,\,n,\,\epsilon}(u),\,\widetilde{p}^{\alpha,\,n,\,\epsilon}(u))\mathbf{1}_{\widetilde{\tau}_{\partial\mathcal{D_{\epsilon}}}^{\alpha,\,n,\,\epsilon}>u}\right]\;,
\]
where $\widetilde{\tau}_{\partial\mathcal{D_{\epsilon}}}^{\alpha,\,n,\,\epsilon}$
is the first hitting time of $\partial\mathcal{D}_{\epsilon}$ for
the process~\eqref{eq:adjoint langevin regularized}, the first limit then follows from the continuity of the trajectories of~\eqref{eq:adjoint langevin regularized} and the transition density
$\widehat{\mathrm{p}}^{\alpha,\,\epsilon}$ and from the condition $\varphi(x)=1$. Regarding the second limit, one can write
\[
\begin{aligned}
h_t(u)
&= \mathrm{e}^{d\gamma u}\,
\mathbb{E}\Bigl[
  \widetilde{\mathrm{p}}_{u}^{\alpha,\,n,\,\epsilon,\,\mathcal{D}_{\epsilon}}
  \Bigl(
    x,\,
    y + \bigl(
      \sqrt{2\alpha\epsilon}\,B(t-u),
      \sqrt{2\gamma\epsilon}\,\widetilde{B}(t-u)
    \bigr)
  \Bigr)
\\
&\qquad \qquad\quad\times
  \varphi\Bigl(
    y + \bigl(
      \sqrt{2\alpha\epsilon}\,B(t-u),
      \sqrt{2\gamma\epsilon}\,\widetilde{B}(t-u)
    \bigr)
  \Bigr)
\Bigr] \; .
\end{aligned}
\]
where $(B(t))_{t\geq0}$, $(\widetilde{B}(t))_{t\geq0}$ are independent Brownian motions in $\R^d$. The proof of~\eqref{eq:difference density expr} then follows from the dominated convergence theorem, when $u\rightarrow t$, using the Aronson Gaussian upper-bound~\cite{aronson1967bounds} satisfied by the transition density $\widetilde{\mathrm{p}}_{u}^{\alpha,\,n,\,\epsilon,\,\mathcal{D}_{\epsilon}}$. 
 
\textbf{Step 2:} Let $g\in C_{c}^{\infty}(\mathcal{D}_{\epsilon})$. Let us show the existence of a constant $C>0$ such that for $t>0$,
\begin{equation}\label{eq:diff density test function}
\left|\int_{\mathcal{D}_{\epsilon}}g(x)\widetilde{\mathrm{p}}_{t}^{\alpha,\,n,\,\epsilon,\,\mathcal{D}_{\epsilon}}(x,\,y)\mathrm{d}x-\int_{\mathcal{D}_{\epsilon}}g(x)\widehat{\mathrm{p}}_{t}^{\alpha,\,\epsilon}(y,\,x)\mathrm{d}x\right|\leq C\sqrt{t}\;.
\end{equation}
Then, since
\[
\int_{\mathcal{D}_{\epsilon}}g(x)\widehat{\mathrm{p}}_{t}^{\alpha,\,\epsilon}(y,\,x)\mathrm{d}x=\mathbb{E}\left[g(y+(\sqrt{2\alpha\epsilon}\,B(t),\sqrt{2\gamma\epsilon}\,\widetilde{B}(t)))\right]\underset{t\rightarrow0}{\longrightarrow}g(y)\;,
\] 
the proof of~\eqref{eq:zero time convergence density} follows from the inequality~\eqref{eq:diff density test function}.

Define now the following differential operator for $x=(q,p)\in\mathbb{R}^{2d}$,
\[
\mathcal{L}_{\alpha,\,\epsilon;\,x}=\alpha\epsilon\Delta_{q}+\gamma\epsilon\Delta_{p}\;.
\]
Notice that for all $t>0$, $x,\,y\in\mathbb{R}^{2d}$,
\begin{equation}
\partial_{t}\widehat{\mathrm{p}}_{t}^{\alpha,\,\epsilon}(x,\,y)=\mathcal{L}_{\alpha,\,\epsilon;\,x}\widehat{\mathrm{p}}_{t}^{\alpha,\,\epsilon}(x,\,y)=\mathcal{L}_{\alpha,\,\epsilon;\,y}\widehat{\mathrm{p}}_{t}^{\alpha,\,\epsilon}(x,\,y)\;.\label{eq:kolmogorov eq brownian}
\end{equation}

Consequently, by~\eqref{eq:kolmogorov equations},~\eqref{eq:difference density expr} 
and~\eqref{eq:kolmogorov eq brownian}, 
\begin{align*}
  \mathrm{e}^{d\gamma t}\widetilde{\mathrm{p}}_{t}^{\alpha,\,n,\,\epsilon,\,\mathcal{D}_{\epsilon}}(x,\,y)-\widehat{\mathrm{p}}_{t}^{\alpha,\,\epsilon}(x,\,y)
 & =\int_{0}^{t}\frac{\mathrm{d}h_{t}(u)}{\mathrm{d}u}\mathrm{d}u\\
 & =\int_{0}^{t}\mathrm{e}^{d\gamma u}\int_{\mathcal{D}_{\epsilon}}[\mathcal{L}_{\alpha,\,n,\,\epsilon;\,z}(\widetilde{\mathrm{p}}_{u}^{\alpha,\,n,\,\epsilon,\,\mathcal{D}_{\epsilon}}(x,\,z))\widehat{\mathrm{p}}_{t-u}^{\alpha,\,\epsilon}(y,\,z)\varphi(z)
 \\&\qquad \qquad \qquad \, 
 -\widetilde{\mathrm{p}}_{u}^{\alpha,\,n,\,\epsilon,\,\mathcal{D}_{\epsilon}}(x,\,z)\mathcal{L}_{\alpha,\,\epsilon;\,z}(\widehat{\mathrm{p}}_{t-u}^{\alpha,\,\epsilon}(y,\,z))\varphi(z)]\mathrm{d}z\,\mathrm{d}u\;.
\end{align*}
By integration by parts and using the adjoint generator~\eqref{eq:adjoint generator},
\begin{align*}
 & \mathrm{e}^{d\gamma t}\widetilde{\mathrm{p}}_{t}^{\alpha,\,n,\,\epsilon,\,\mathcal{D}_{\epsilon}}(x,\,y)-\widehat{\mathrm{p}}_{t}^{\alpha,\,\epsilon}(x,\,y)\\
 & =\int_{0}^{t}\mathrm{e}^{d\gamma u}\int_{\mathcal{D}_{\epsilon}}[\widetilde{\mathrm{p}}_{u}^{\alpha,\,n,\,\epsilon,\,\mathcal{D}_{\epsilon}}(x,\,z)\left((\mathcal{L}_{\alpha,\,n,\,\epsilon;\,z})^{\dagger}(\widehat{\mathrm{p}}_{t-u}^{\alpha,\,\epsilon}(y,\,z))-\mathcal{L}_{\alpha,\,\epsilon;\,z}(\widehat{\mathrm{p}}_{t-u}^{\alpha,\,\epsilon}(y,\,z)\right)]\varphi(z)\mathrm{d}z\,\mathrm{d}u\\
 & +\int_{0}^{t}\mathrm{e}^{d\gamma u}\int_{\mathcal{D}_{\epsilon}}\widetilde{\mathrm{p}}_{u}^{\alpha,\,n,\,\epsilon,\,\mathcal{D}_{\epsilon}}(x,\,z)[\langle\Sigma_{1}\nabla_{z}\widehat{\mathrm{p}}_{t-u}^{\alpha,\,\epsilon}(y,\,z),\,\nabla\varphi(z)\rangle+\widehat{\mathrm{p}}_{t-u}^{\alpha,\,\epsilon}(y,\,z)[(\mathcal{L}_{\alpha,\,n,\,\epsilon;\,z})^{\dagger}-d\gamma]\varphi(z)]\mathrm{d}z\,\mathrm{d}u\;.
\end{align*}
Furthermore, taking the gradient of $\widehat{\mathrm{p}}_{t-u}^{\alpha,\,\epsilon}(y,\,\cdot)$ ensures the existence of constants $C,\,\beta>0$
such that for all $u\in(0,\,t)$, $y,\,z\in\mathbb{R}^{2d}$, 
\[
\widehat{\mathrm{p}}_{t-u}^{\alpha,\,\epsilon}(y,\,z)+|\nabla_{z}\widehat{\mathrm{p}}_{t-u}^{\alpha,\,\epsilon}(y,\,z)|\leq\frac{C}{\sqrt{t-u}}\frac{\mathrm{e}^{-\frac{|y-z|^{2}}{2\beta (t-u)}}}{(2\pi\beta (t-u))^{d}}\;.
\]
By~\cite{aronson1967bounds}, up to increasing the constants
$C,\,\beta>0$, the density $\widetilde{\mathrm{p}}_{u}^{\alpha,\,n,\,\epsilon,\,\mathcal{D}_{\epsilon}}$ also
admits a similar Gaussian upper-bound. Hence, we get 
\[
\widetilde{\mathrm{p}}_{u}^{\alpha,\,n,\,\epsilon,\,\mathcal{D}_{\epsilon}}(x,\,z)\leq C\frac{\mathrm{e}^{-\frac{|x-z|^{2}}{2\beta u}}}{(2\pi\beta u)^{d}}\;.
\]
Therefore, since $(\mathcal{L}_{\alpha,\,n,\,\epsilon;\,z})^{\dagger}-\mathcal{L}_{\alpha,\,\epsilon;\,z}$
is only composed of first order derivatives, the above Gaussian upper-bounds and the boundedness of the set $\mathcal{D}_{\epsilon}$ provide a constant $C'>0$ which depends on $\Vert\nabla\varphi\Vert_\infty\lor\Vert\Delta\varphi\Vert_\infty$ such that 
\begin{align*}
 \mathrm{e}^{d\gamma t}\widetilde{\mathrm{p}}_{t}^{\alpha,\,n,\,\epsilon,\,\mathcal{D}_{\epsilon}}(x,\,y)-\widehat{\mathrm{p}}_{t}^{\alpha,\,\epsilon}(x,\,y)
 & \leq C'\int_{0}^{t}\frac{\mathrm{e}^{d\gamma r}}{\sqrt{t-r}}\int_{\mathcal{D}_{\epsilon}}\frac{\mathrm{e}^{-\frac{|z-x|^{2}}{2\beta r}}}{(2\pi\beta r)^{d}}\frac{\mathrm{e}^{-\frac{|y-z|^{2}}{2\beta(t-r)}}}{(2\pi\beta(t-r))^{d}}\mathrm{d}z\,\mathrm{d}r\\
 & \leq \;C'\frac{\mathrm{e}^{-\frac{|y-x|^{2}}{2\beta t}}}{(2\pi\beta t)^{d}}\int_{0}^{t}\frac{\mathrm{e}^{d\gamma r}}{\sqrt{t-r}}\mathrm{d}r\;
   \leq\;2C'\mathrm{e}^{d\gamma t}\sqrt{t}\,\frac{\mathrm{e}^{-\frac{|y-x|^{2}}{2\beta t}}}{(2\pi\beta t)^{d}}\;,
\end{align*}
by the Chapman-Kolmogorov relation.  Although the proof of the inequality above is done for fixed points $x,\,y\in\mathcal{D}_{\epsilon}$ and the definition of $\varphi$ depends on $x$ and $y$. One can always do the exact same reasoning for any points $x,\,y$ in a compact set $K\subset\mathcal{D}_{\epsilon}$ which contains the support of the test function $g$ in~\eqref{eq:zero time convergence density}. Then, the same reasoning works for any points $x,\,y\in K$ provided we take in~\eqref{eq:def h_t(u)} a test function $\varphi\in C_c^\infty(\mathcal{D}_\epsilon)$ equal to $1$ on $K$. Thus, integrating the inequality above against the test function $g$ concludes the proof of~\eqref{eq:zero time convergence density}.
\end{proof}
Following the scheme of proof developed in~\cite[Theorem 6.2]{LelRamRey}, we conclude this section with the proof of Proposition~\ref{prop:property density}.
\begin{proof}[Proof of Proposition~\ref{prop:property density}] \textbf{Step 1:} Let us first prove the equality for the regularized processes: for all $t>0$, $x,\,y\in\mathcal{D}_{\epsilon}$, 
\begin{equation}
\mathrm{p}_{t}^{\alpha,\,n,\,\epsilon,\,\mathcal{D}_{\epsilon}}(x,\,y)=\mathrm{e}^{d\gamma t}\widetilde{\mathrm{p}}_{t}^{\alpha,\,n,\,\epsilon,\,\mathcal{D}_{\epsilon}}(y,\,x)\;.\label{eq:density equality alpha,n}
\end{equation}
Let $g\in C_{c}^{\infty}(\mathcal{D}_{\epsilon})$ and define for $t>0$ and $x\in\mathcal{D}_{\epsilon}$,
\[
u(t,\,x)=\int_{\mathcal{D}_{\epsilon}}\mathrm{p}_{t}^{\alpha,\,n,\,\epsilon,\,\mathcal{D}_{\epsilon}}(x,\,y)g(y)\mathrm{d}y,\qquad v(t,\,x)=\mathrm{e}^{d\gamma t}\int_{\mathcal{D}_{\epsilon}}\widetilde{\mathrm{p}}_{t}^{\alpha,\,n,\,\epsilon,\,\mathcal{D}_{\epsilon}}(y,\,x)g(y)\mathrm{d}y\;.
\]
It is well known (see~\cite[Chapter 2, Theorem 3.7]{FriedmanParabolic})
that $u$ is the unique solution in $C^{\infty}(\mathbb{R}_{+}^{*}\times\mathcal{D}_{\epsilon})\cap C_{b}(\mathbb{R_{+}}\times\overline{\mathcal{D}_{\epsilon}})$
of 
\begin{equation}
\left\{ \begin{aligned}\partial_{t}u(t,\,x) & =\mathcal{L}_{\alpha,\;n,\;\epsilon}u(t,\;x) &  & t>0,\quad x\in\mathcal{D}_{\epsilon},\\
u(0,\,x) & =g(x) &  & x\in\mathcal{D}_{\epsilon},\\
u(t,\,x) & =0 &  & t>0,\quad x\in\partial\mathcal{D}_{\epsilon}\;,
\end{aligned}
\right.\label{eq:initial boundary value problem u}
\end{equation}
Therefore, the equality~\eqref{eq:density equality alpha,n} follows immediately if $v$ is also a solution in $C^{\infty}(\mathbb{R}_{+}^{*}\times\mathcal{D}_{\epsilon})\cap C_{b}(\mathbb{R_{+}}\times\overline{\mathcal{D}_{\epsilon}})$ of~\eqref{eq:initial boundary value problem u}.

The equality $\partial_{t}v(t,\,x)=\mathcal{L}_{\alpha,\,n,\,\epsilon}v(t,\,x)$ follows from~\eqref{eq:kolmogorov equations}. Additionally, the continuity of $\widetilde{\mathrm{p}}_{t}^{\alpha,\,n,\,\epsilon,\,\mathcal{D}_{\epsilon}}(y,\,x)$
on $\mathbb{R}_{+}^{*}\times\overline{\mathcal{D}_{\epsilon}}\times\overline{\mathcal{D}_{\epsilon}}$
ensures that $v$ is continuous on $\mathbb{R}_{+}^{*}\times\overline{\mathcal{D}_{\epsilon}}$
and vanishes at $\mathbb{R}_{+}^{*}\times\partial\mathcal{D}_{\epsilon}$. Finally, the continuity at $t=0$ follows from Lemma~\ref{lem:small time convergence integral}. Hence $v=u$ and the equality~\eqref{eq:density equality alpha,n} follows.
\smallskip

\noindent \textbf{Step 2:} In order to deduce~\eqref{eq:property density}, we would like to take the limits $\alpha\rightarrow0$, $n\rightarrow\infty$ in~\eqref{eq:density equality alpha,n}. In order to do that, let us first prove that for all non-negative $f,g\in C_{c}^{\infty}(\mathcal{D}_{\epsilon})$ one has that for all $t>0$, for all $x=(q,p)\in\mathcal{D}_{\epsilon}$, $\mathbb{P}_x$ almost-surely,
\begin{equation}
g(q^{\alpha,\,n,\,\epsilon}(t),\,p^{\alpha,\,n,\,\epsilon}(t))\mathbf{1}_{\tau_{\partial\mathcal{D}_{\epsilon}}^{\alpha,\,n,\,\epsilon}>t}\underset{\alpha\rightarrow0,\,n\rightarrow\infty}{\longrightarrow}g(q^{\epsilon}(t),\,p^{\epsilon}(t))\mathbf{1}_{\tau_{\partial\mathcal{D}_{\epsilon}}^{\epsilon}>t}\;,\label{eq:convergence expectation}
\end{equation}
and 
\begin{equation}\label{eq:convergence expectation 2}
f(\widetilde{q}^{\alpha,\,n,\,\epsilon}(t),\;\widetilde{p}^{\alpha,\,n,\,\epsilon}(t))\mathbf{1}_{\widetilde{\tau}_{\partial\mathcal{D_{\epsilon}}}^{\alpha,\,n,\,\epsilon}>t}\underset{\alpha\rightarrow0,\,n\rightarrow\infty}{\longrightarrow}f(\widetilde{q}^{\epsilon}(t),\,\widetilde{p}^{\epsilon}(t))\mathbf{1}_{\widetilde{\tau}_{\partial\mathcal{D_{\epsilon}}}^{\epsilon}>t}\;.
\end{equation}
Let us first control the distance between $(q^{\alpha,\,n,\,\epsilon}(t),\,p^{\alpha,\,n,\,\epsilon}(t))_{t\geq0}$
and $(q^{\epsilon}(t),\,p^{\epsilon}(t))_{t\geq0}$ using Gr\"onwall's inequality. Let $\mathcal{D}'$ be an open bounded set of $\mathbb{R}^{2d}$ such
that $\overline{\mathcal{D}_{\epsilon}}\subset\mathcal{D}'$. Let
$\tau$ be a positive and finite random variable. On the event $\{\tau_{\partial\mathcal{D}'}^{\alpha,\,n,\,\epsilon}\land\tau_{\partial\mathcal{D}'}^{\epsilon}>\tau\}$,
there exists a constant $C>0$ independent of $\tau,\,\alpha,\,n$
such that 
\begin{align}
&\sup_{t\in[0,\,\tau]}|(q^{\alpha,\,n,\,\epsilon}(t),\,p^{\alpha,\,n,\,\epsilon}(t))-(q^{\epsilon}(t),\,p^{\epsilon}(t))|\nonumber\\
&\leq\left(\tau\sup_{q:(q,\,p)\in\mathcal{D}'}|F_{n}(q)+\nabla U(q)|+\sqrt{2\alpha\epsilon}\sup_{s\in[0,\,\tau]}|B(s)|\right)\mathrm{e}^{C\tau}\;.\label{eq:Gronwall ineq regularized Langevin}
\end{align}

Notice that by~\eqref{eq:convergence compact F_n}, we have 
\[
\sup_{q:(q,\,p)\in\mathcal{D}'}|F_{n}(q)+\nabla U(q)|\underset{n\rightarrow\infty}{\longrightarrow}0\;.
\]
Using Gr\"onwall's inequality~\eqref{eq:Gronwall ineq regularized Langevin}
let us now prove that, $\mathbb{P}_x$ almost-surely, for all $t>0$,
\begin{equation}
\mathbf{1}_{\tau_{\partial\mathcal{D}_{\epsilon}}^{\alpha,\,n,\,\epsilon}>t}\underset{\alpha\rightarrow0,n\rightarrow\infty}{\longrightarrow}~\mathbf{1}_{\tau_{\partial\mathcal{D}_{\epsilon}}^{\epsilon}>t}\;.\label{eq:convergence indicatrices}
\end{equation}
Notice that 
\begin{equation}
\left|\mathbf{1}_{\tau_{\partial\mathcal{D}_{\epsilon}}^{\alpha,\,n,\,\epsilon}>t}-\mathbf{1}_{\tau_{\partial\mathcal{D}_{\epsilon}}^{\epsilon}>t}\right|\leq\mathbf{1}_{\tau_{\partial\mathcal{D}_{\epsilon}}^{\alpha,\,n,\,\epsilon}>t,\,\tau_{\partial\mathcal{D}_{\epsilon}}^{\epsilon}<t}+\mathbf{1}_{\tau_{\partial\mathcal{D}_{\epsilon}}^{\alpha,\,n,\,\epsilon}<t,\,\tau_{\partial\mathcal{D}_{\epsilon}}^{\epsilon}>t}\;,\label{eq:upperbound difference of indicatrices}
\end{equation}
since the events $\{\tau_{\partial\mathcal{D}_{\epsilon}}^{\epsilon}=t\}$
and $\{\tau_{\partial\mathcal{D}_{\epsilon}}^{\alpha,\,n,\,\epsilon}=t\}$
are of probability zero since the boundary $\partial\mathcal{D}_{\epsilon}$
has zero Lebesgue measure.

By~\cite[Proposition 2.8]{LelRamRey}, the process $(q^{\epsilon}(t),\,p^{\epsilon}(t))_{t\geq0}$ starting from $x\in\mathcal{D}_\epsilon$ almost-surely never reaches the boundary $\partial\mathcal{D}_{\epsilon}=\partial\mathcal{W}\cup\partial\mathcal{M}_{\epsilon}\cup\partial\mathcal{S}_{\epsilon}$
with zero velocity. Therefore, by Lemma~\ref{lem:small time asymptotics}, the process exits
$\mathcal{D}_{\epsilon}$ almost-surely immediately after hitting $\partial\mathcal{D}_{\epsilon}$ and on the event $\{\tau_{\partial\mathcal{D}_{\epsilon}}^{\epsilon}<t\}$, by the continuity of the trajectory of $(q^{\epsilon}(t),\,p^{\epsilon}(t))_{t\geq0}$, there exists $\delta>0$ satisfying $\tau_{\partial\mathcal{D}_{\epsilon}}^{\epsilon}+\delta<\tau_{\partial\mathcal{D}'}^{\epsilon}\land t$ such that
$(q^{\epsilon}(\tau_{\partial\mathcal{D}_{\epsilon}+\delta}^{\epsilon}),\,p^{\epsilon}(\tau_{\partial\mathcal{D}_{\epsilon}}^{\epsilon}+\delta))\notin\overline{\mathcal{D}_{\epsilon}}$. Taking
$\alpha$ small enough and $n\geq1$ large enough ensures by~\eqref{eq:Gronwall ineq regularized Langevin} that $(q^{\alpha,\,n,\,\epsilon}(\tau_{\partial\mathcal{D}_{\epsilon}}^{\epsilon}+\delta),\,p^{\alpha,\,n,\,\epsilon}(\tau_{\partial\mathcal{D}_{\epsilon}}^{\epsilon}+\delta))\notin\overline{\mathcal{D}_{\epsilon}}$,
which ensures that $\tau_{\partial\mathcal{D}_{\epsilon}}^{\alpha,\,n,\,\epsilon}\leq\tau_{\partial\mathcal{D}_{\epsilon}}^{\epsilon}+\delta<t$. Regarding the second term in the right-hand side
of~\eqref{eq:upperbound difference of indicatrices}, on the event $\{\tau_{\partial\mathcal{D}_{\epsilon}}^{\epsilon}>t\}$, the process $(q^{\epsilon}(r),\,p^{\epsilon}(r))_{r\in[0,\,t]}$ stays at a positive distance from $\partial\mathcal{D}_{\epsilon}$. Therefore, taking
$\alpha$ small and $n$ large enough yields by~\eqref{eq:Gronwall ineq regularized Langevin}
that $\tau_{\partial\mathcal{D}_{\epsilon}}^{\alpha,\,n,\,\epsilon}>t$ as well. The convergence~\eqref{eq:convergence indicatrices} then follows from the same argument as in Lemma~\ref{lem:conv h^*}. Combining with Gr\"onwall's inequality~\eqref{eq:Gronwall ineq regularized Langevin}, one obtains~\eqref{eq:convergence expectation}. The proof of~\eqref{eq:convergence expectation 2} is exactly similar.
\medskip

\noindent \textbf{Step 3:} Let us now conclude the proof of Proposition~\ref{prop:property density}. By~\eqref{eq:density equality alpha,n}, one has for any non-negative $f,g\in C^\infty_c(\mathcal{D}_\epsilon)$,
\begin{equation}
\iint_{\mathcal{D}_{\epsilon}\times\mathcal{D}_{\epsilon}}f(x)\mathrm{p}_{t}^{\alpha,\,n,\,\epsilon,\,\mathcal{D}_{\epsilon}}(x,\,y)g(y)\mathrm{d}x\mathrm{d}y=\mathrm{e}^{d\gamma t}\iint_{\mathcal{D}_{\epsilon}\times\mathcal{D}_{\epsilon}}f(x)\widetilde{\mathrm{p}}_{t}^{\alpha,\,n,\,\epsilon,\,\mathcal{D}_{\epsilon}}(y,\,x)g(y)\mathrm{d}x\mathrm{d}y\;.\label{eq:double integrale}
\end{equation}
Let us take the limit of the equality above when $\alpha\rightarrow0$ and $n\rightarrow\infty$. Notice first that the existence of the transition densities $\mathrm{p}_{t}^{\epsilon,\,\mathcal{D}_{\epsilon}}(x,\,y)$
and $\widetilde{\mathrm{p}}_{t}^{\epsilon,\,\mathcal{D_{\epsilon}}}(x,\,y)$
of the absorbed processes satisfying~\eqref{eq:killed density definition 1} and~\eqref{eq:killed density definition 2} is an immediate consequence of Radon-Nikodym's theorem. Moreover, by the Fubini permutation and the convergence~\eqref{eq:convergence expectation},
\begin{align*}
\iint_{\mathcal{D}_{\epsilon}\times\mathcal{D}_{\epsilon}}f(x)\mathrm{p}_{t}^{\alpha,\,n,\,\epsilon,\,\mathcal{D}_{\epsilon}}(x,\,y)g(y)\mathrm{d}x\mathrm{d}y & =\int_{\mathcal{D}_{\epsilon}}f(x)\mathbb{E}_{x}\left[g(q^{\alpha,\,n,\,\epsilon}(t),\,p^{\alpha,\,n,\,\epsilon}(t))\mathbf{1}_{\tau_{\partial\mathcal{D}_{\epsilon}}^{\alpha,\,n,\,\epsilon}>t}\right]\mathrm{d}x\\
 & \underset{\alpha\rightarrow0,\,n\rightarrow\infty}{\longrightarrow}\int_{\mathcal{D}_{\epsilon}}f(x)\mathbb{E}_{x}\left[g(q^{\epsilon}(t),\,p^{\epsilon}(t))\mathbf{1}_{\tau_{\partial\mathcal{D}_{\epsilon}}^{\epsilon}>t}\right]\mathrm{d}x\\
 & =\iint_{\mathcal{D}_{\epsilon}\times\mathcal{D}_{\epsilon}}f(x)\mathrm{p}_{t}^{\epsilon,\,\mathcal{D}_{\epsilon}}(x,\,y)g(y)\mathrm{d}x\mathrm{d}y\;,
\end{align*}
by definition of $\mathrm{p}_{t}^{\epsilon,\,\mathcal{D}_{\epsilon}}(x,\,y)$. Furthermore, by~\eqref{eq:convergence expectation 2},

\begin{align*}
\iint_{\mathcal{D}_{\epsilon}\times\mathcal{D}_{\epsilon}}f(x)\widetilde{\mathrm{p}}_{t}^{\alpha,\,n,\,\epsilon,\,\mathcal{D}_{\epsilon}}(y,\,x)g(y)\mathrm{d}x\mathrm{d}y & =\int_{\mathcal{D}_{\epsilon}}g(y)\mathbb{E}_{y}\left[f(\widetilde{q}^{\alpha,\,n,\,\epsilon}(t),\,\widetilde{p}^{\alpha,\,n,\,\epsilon}(t))\mathbf{1}_{\widetilde{\tau}_{\partial\mathcal{D_{\epsilon}}}^{\alpha,\,n,\,\epsilon}>t}\right]\mathrm{d}y\\
 & \underset{\alpha\rightarrow0,\,n\rightarrow\infty}{\longrightarrow}\int_{\mathcal{D}_{\epsilon}}g(y)\mathbb{E}_{y}\left[f(\widetilde{q}^{\epsilon}(t),\,\widetilde{p}^{\epsilon}(t))\mathbf{1}_{\widetilde{\tau}_{\partial\mathcal{D_{\epsilon}}}^{\epsilon}>t}\right]\mathrm{d}y\\
 & =\iint_{\mathcal{D}_{\epsilon}\times\mathcal{D}_{\epsilon}}f(x)\widetilde{\mathrm{p}}_{t}^{\epsilon,\,\mathcal{D_{\epsilon}}}(y,\,x)g(y)\mathrm{d}x\mathrm{d}y\;.
\end{align*}
As a result, taking the limits $\alpha\rightarrow0$, $n\rightarrow\infty$
in~\eqref{eq:double integrale} yields 
\begin{equation}\label{eq:equality density adjoint}
\iint_{\mathcal{D}_{\epsilon}\times\mathcal{D}_{\epsilon}}f(x)\mathrm{p}_{t}^{\epsilon,\,\mathcal{D}_{\epsilon}}(x,\,y)g(y)\mathrm{d}x\mathrm{d}y=\mathrm{e}^{d\gamma t}\iint_{\mathcal{D}_{\epsilon}\times\mathcal{D}_{\epsilon}}f(x)\widetilde{\mathrm{p}}_{t}^{\epsilon,\,\mathcal{D_{\epsilon}}}(y,\,x)g(y)\mathrm{d}x\mathrm{d}y\;.
\end{equation}
The equality~\eqref{eq:property density} then follows from the local integrability of $\widetilde{\mathrm{p}}_{t}^{\epsilon,\,\mathcal{D_{\epsilon}}}$
and $\mathrm{p}_{t}^{\epsilon,\,\mathcal{D}_{\epsilon}}$ in $\mathcal{D}_{\epsilon}\times\mathcal{D}_{\epsilon}$.
\end{proof}

\section{Lyapunov function and zero-noise Langevin process}

\label{sec:Lyapunov H} 
Let $z\in\{m\}\cup\{s\}$ be a local minimum of $U$. This section is dedicated to controlling, when $\epsilon\rightarrow0$, the first exit time probability from a neighbourhood of $z$, and the distance between the underdamped Langevin process~\eqref{eq:sde} and its zero-noise ($\epsilon=0$) counterpart. These results are essential for the proof of Proposition~\ref{prop:bound exponential exit time} carried out in the next section.

\subsection{First exit time probability} The proofs rely on the Lyapunov function $H$ defined in~\cite[Definition 4.1]{cutoff} and translated in the $q$-coordinate as follows
\begin{equation}
H(x):=\frac{1}{2}|p|^{2}+\frac{\gamma-\lambda}{2}\left\langle q-z,\,p\right\rangle +\frac{(\gamma-\lambda)^{2}}{4}|q-z|^{2}+U(q)-U(z)\;,\label{eq:def H}
\end{equation}
where $\lambda>0$ is a constant independent of $\epsilon$ satisfying~\eqref{eq:lambda}. 
\begin{lem}
\label{lem:lyapunov function} There exist $\rho,\,\alpha>0$ independent of $\epsilon$ such that for
all $x\in\mathrm{B}(z,\,\rho)$, 
\begin{equation}
\mathcal{L}_{\epsilon}H(x)\leq-\lambda H(x)+d\gamma\epsilon\;,\label{eq:lyapunov function H}
\end{equation}
and 
\[
\mathcal{L}_{\epsilon}\bigg(\mathrm{exp}\bigg(\frac{\alpha H(x)}{\epsilon}\bigg)\bigg)\leq\alpha d\gamma\,\mathrm{exp}\bigg(\frac{\alpha H(x)}{\epsilon}\bigg)\;.
\]
\end{lem}

\begin{proof}
Since $U$ is a Morse function, its Hessian matrix at the local minimum
$z$ admits only positive eigenvalues. Therefore, the
asymptotic expansion of $U$ in a neighbourhood of $z$ ensures the existence of constants $\rho,\,c>0$ such that for all $q\in\mathrm{B}(z,\,\rho)$,
\[
\langle\nabla U(q),\,q-z\rangle\geq c(|q-z|^{2}+U(q)-U(z))\;.
\]
Following the computation in~\cite[Lemma 4.3]{cutoff}, one deduces that for $x\in\mathrm{B}(z,\,\rho)$, 
\begin{align}
\mathcal{L}_{\epsilon}H(x) & \leq-\lambda H(x)+\gamma\epsilon\,\Delta_{p}H(x)\nonumber \\
 & \leq-\lambda H(x)+d\gamma\epsilon\;.\label{eq:ineq LH}
\end{align}
Consequently, for $\alpha>0$ and $x\in\mathrm{B}(z,\,\rho)$, 
\[
\mathcal{L}_{\epsilon}\bigg(\mathrm{exp}\bigg(\frac{\alpha H(x)}{\epsilon}\bigg)\bigg)=\mathrm{exp}\bigg(\frac{\alpha H(x)}{\epsilon}\bigg)\left[\frac{\alpha}{\epsilon}\mathcal{L}_{\epsilon}H(x)+\gamma\epsilon\frac{\alpha^{2}}{\epsilon^{2}}\left|p+\frac{\gamma-\lambda}{2}(q-z)\right|^{2}\right]\;.
\]

\noindent Moreover, as in~\cite[Lemma 4.2]{cutoff}, 
there exists a constant $C>0$ such that for all $x\in\mathbb{R}^{2d}$,
\[
\left|p+\frac{\gamma-\lambda}{2}(q-z)\right|^{2}\leq CH(x)\;.
\]
As a result, by~\eqref{eq:ineq LH}, 
\[
\mathcal{L}_{\epsilon}\bigg(\mathrm{exp}\bigg(\frac{\alpha H(x)}{\epsilon}\bigg)\bigg)\leq\mathrm{exp}\bigg(\frac{\alpha H(x)}{\epsilon}\bigg)\left[-\frac{\alpha}{\epsilon}\lambda H(x)+\alpha d\gamma+C\gamma\frac{\alpha^{2}}{\epsilon}H(x)\right]\;.
\]
Taking $\alpha\in(0,\lambda/C\gamma)$ independent of $\epsilon$ ensures that for $x\in\mathrm{B}(z,\,\rho)$, 
\[
\mathcal{L}_{\epsilon}\bigg(\mathrm{exp}\bigg(\frac{\alpha H(x)}{\epsilon}\bigg)\bigg)\leq\alpha d\gamma\,\mathrm{exp}\bigg(\frac{\alpha H(x)}{\epsilon}\bigg)\;,
\]
thus concluding the proof. 
\end{proof}
The Lyapunov control established in Lemma~\ref{lem:lyapunov function} allows to provide an upper-bound of the probability of exiting, before a time $t\geq0$, a neighbourhood of $z$ defined as follows
\begin{equation}
\mathcal{H}_{b}:=\{x\in\mathcal{W}_{z}:H(x)<b\}\;.\label{eq:def level set H}
\end{equation}
where $\mathcal{W}_z$ is defined in Section~\ref{sec:Estimate of numerator}.

\begin{prop}
\label{prop:control moments H} There exist constants $\delta,\,\alpha,\,\beta,\,C>0$
independent of $\epsilon$ such that for all $0<a\leq b\leq\delta$,
for all $x\in\mathcal{H}_{a}$ and $t\geq0$, 
\begin{equation}
\mathbb{P}_{x}(\tau_{\partial\mathcal{H}_{b}}^{\epsilon}\leq t)\leq\frac{C(1+t)}{\epsilon^{\beta}}\mathrm{exp}\bigg(-\frac{\alpha(b-a)}{\epsilon}\bigg)\;.\label{eq:proba exiting neighborhood z}
\end{equation}
\end{prop}

\begin{proof}[Proof of Proposition~\ref{prop:control moments H}]
Let $\rho,\,\alpha>0$ be as defined in Lemma~\ref{lem:lyapunov function}.
Since $H$ is continuous, positive and vanishes only at $(z,\,0)$
by~\cite[Lemma 4.2]{cutoff}, there exists $\delta>0$ small enough
such that for all $b\in(0,\,\delta)$, 
\[
\mathcal{H}_{b}\subset\mathrm{B}(z,\,\rho)\;.
\]

\textbf{Step 1:} Let $a\in(0,\,b]$. We first prove the existence
of constants $\beta,\,C>0$ independent of $\epsilon$ such that for
all $t\geq0$, for all $x\in\mathcal{H}_{a}$, 
\begin{equation}
\mathbb{E}_{x}\bigg[\mathrm{exp}\bigg(\frac{\alpha H(q^{\epsilon}(t),\,p^{\epsilon}(t))}{\epsilon}\bigg)\mathbf{1}_{\tau_{\partial\mathcal{H}_{b}}^{\epsilon}>t}\bigg]\leq C\left[1+\left(\frac{H(x)}{\epsilon}\right)^{\beta}\mathrm{exp}\bigg(\frac{\alpha H(x)}{\epsilon}\bigg)\right]\;.\label{eq:first ineq expec H}
\end{equation}

Let $x\in\mathcal{H}_{a}$. Apply It\^o's formula to the process
\[
t\geq0\mapsto\mathrm{exp}\bigg(\frac{\alpha}{\epsilon}H\big(q^{\epsilon}(t\land\tau_{\partial\mathcal{H}_{b}}^{\epsilon}),\,p^{\epsilon}(t\land\tau_{\partial\mathcal{H}_{b}}^{\epsilon})\big)\bigg)\mathrm{e}^{-\alpha d\gamma\big(t\land\tau_{\partial\mathcal{H}_{b}}^{\epsilon}\big)}\;.
\]
From Lemma~\ref{lem:lyapunov function}, we deduce that for all $t\geq0$, 
\[
\mathbb{E}_{x}\left[\mathrm{exp}\bigg(\frac{\alpha}{\epsilon}H\big(q^{\epsilon}(t\land\tau_{\partial\mathcal{H}_{b}}^{\epsilon}),\,p^{\epsilon}(t\land\tau_{\partial\mathcal{H}_{b}}^{\epsilon})\big)\bigg)\mathrm{e}^{-\alpha d\gamma\big(t\land\tau_{\partial\mathcal{H}_{b}}^{\epsilon}\big)}\right]\leq\mathrm{exp}\bigg(\frac{\alpha H(x)}{\epsilon}\bigg)\;,
\]
ensuring that 
\begin{equation}
\mathbb{E}_{x}\left[\mathrm{exp}\bigg(\frac{\alpha H(q^{\epsilon}(t),\,p^{\epsilon}(t))}{\epsilon}\bigg)\mathbf{1}_{\tau_{\partial\mathcal{H}_{b}}^{\epsilon}>t}\right]\leq\mathrm{e}^{\alpha d\gamma t}\mathrm{exp}\bigg(\frac{\alpha H(x)}{\epsilon}\bigg)\;.\label{eq:control exp moment H 1}
\end{equation}
Now let 
\[
T_{\epsilon}(x):=\frac{1}{\lambda}\log\left(\frac{\lambda H(x)}{d\gamma\epsilon}\right)\;.
\]
By~\eqref{eq:control exp moment H 1} there exist constants
$C,\,\beta>0$ independent of $\epsilon$ such that for all $t\in[0,\,T_{\epsilon}(x)]$,
\begin{equation}
\mathbb{E}_{x}\left[\mathrm{exp}\bigg(\frac{\alpha H(q^{\epsilon}(t),\,p^{\epsilon}(t))}{\epsilon}\bigg)\mathbf{1}_{\tau_{\partial\mathcal{H}_{b}}^{\epsilon}>t}\right]\leq C\left(\frac{H(x)}{\epsilon}\right)^{\beta}\mathrm{exp}\bigg(\frac{\alpha H(x)}{\epsilon}\bigg)\;.\label{eq:control exp moment H 2}
\end{equation}
Assume now that $t\geq T_{\epsilon}(x)$, then by definition of $T_{\epsilon}(x)$
one has 
\[
H(x)\mathrm{e}^{-\lambda t}\leq\frac{d\gamma\epsilon}{\lambda}\;.
\]
Therefore, following the proof of~\cite[Proposition 4.9]{cutoff}
and using~\eqref{eq:lyapunov function H},
one deduces that for all $n\geq0$ and $t\geq T_{\epsilon}(x)$, 
\begin{align*}
\mathbb{E}_{x}\left[\left(\frac{\alpha}{\epsilon}\right)^{n}H(q^{\epsilon}(t),\,p^{\epsilon}(t))^{n}\mathbf{1}_{\tau_{\partial\mathcal{H}_{b}}^{\epsilon}>t}\right] & \leq\alpha^{n}\left(\frac{H(x)\mathrm{e}^{-\lambda t}}{\epsilon}+\frac{d\gamma}{\lambda}\right)^{n}\leq\left(2\frac{\alpha d\gamma}{\lambda}\right)^{n}\;.
\end{align*}

\noindent Consequently, for all $t\geq T_{\epsilon}(x)$, 
\[
\mathbb{E}_{x}\left[\mathrm{exp}\bigg(\frac{\alpha H(q^{\epsilon}(t),\,p^{\epsilon}(t))}{\epsilon}\bigg)\mathbf{1}_{\tau_{\partial\mathcal{H}_{b}}^{\epsilon}>t}\right]\leq\mathrm{e}^{2\alpha d\gamma/\lambda}\;.
\]
Combining with~\eqref{eq:control exp moment H 2} concludes the proof
of~\eqref{eq:first ineq expec H}.

\textbf{Step 2:} Using the estimate~\eqref{eq:first ineq expec H}
let us now prove the inequality~\eqref{eq:proba exiting neighborhood z}. For this purpose, we apply It\^o's formula to the process 
\[
t\geq0\mapsto\mathrm{exp}\bigg(\frac{\alpha}{\epsilon}H\big(q^{\epsilon}(t\land\tau_{\partial\mathcal{H}_{b}}^{\epsilon}),\,p^{\epsilon}(t\land\tau_{\partial\mathcal{H}_{b}}^{\epsilon})\big)\bigg)\;.
\]
By Lemma~\ref{lem:lyapunov function},
\begin{align*}
 & \mathbb{E}_{x}\left[\mathrm{exp}\bigg(\frac{\alpha}{\epsilon}H\big(q^{\epsilon}(t\land\tau_{\partial\mathcal{H}_{b}}^{\epsilon}),\,p^{\epsilon}(t\land\tau_{\partial\mathcal{H}_{b}}^{\epsilon})\big)\bigg)\right]\\
 & \leq\mathrm{exp}\bigg(\frac{\alpha H(x)}{\epsilon}\bigg)+\alpha d\gamma\int_{0}^{t}\mathbb{E}_{x}\left[\mathrm{exp}\bigg(\frac{\alpha}{\epsilon}H\big(q^{\epsilon}(r),\,p^{\epsilon}(r)\big)\bigg)\mathbf{1}_{\tau_{\partial\mathcal{H}_{b}}^{\epsilon}>r}\right]\mathrm{d}r\\
 & \leq\mathrm{exp}\bigg(\frac{\alpha H(x)}{\epsilon}\bigg)+Cd\alpha\gamma t\left[1+\left(\frac{H(x)}{\epsilon}\right)^{\beta}\mathrm{exp}\bigg(\frac{\alpha H(x)}{\epsilon}\bigg)\right]\\
 & \leq\mathrm{exp}\bigg(\frac{\alpha a}{\epsilon}\bigg)+Cd\alpha\gamma t\left[1+\left(\frac{a}{\epsilon}\right)^{\beta}\mathrm{exp}\bigg(\frac{\alpha a}{\epsilon}\bigg)\right]
\end{align*}
since $x\in\mathcal{H}_{a}$. Moreover, since
\[
\mathbb{E}_{x}\left[\mathrm{exp}\bigg(\frac{\alpha}{\epsilon}H\big(q^{\epsilon}(t\land\tau_{\partial\mathcal{H}_{b}}^{\epsilon}),\,p^{\epsilon}(t\land\tau_{\partial\mathcal{H}_{b}}^{\epsilon})\big)\bigg)\right]\geq\mathrm{exp}\bigg(\frac{\alpha b}{\epsilon}\bigg)\mathbb{P}_{x}(\tau_{\partial\mathcal{H}_{b}}^{\epsilon}\leq t)\;,
\]
one deduces that 
\[
\mathbb{P}_{x}(\tau_{\partial\mathcal{H}_{b}}^{\epsilon}\leq t)\leq\mathrm{exp}\bigg(-\frac{\alpha(b-a)}{\epsilon}\bigg)+Cd\alpha\gamma t\left[\mathrm{exp}\bigg(-\frac{\alpha b}{\epsilon}\bigg)+\left(\frac{a}{\epsilon}\right)^{\beta}\mathrm{exp}\bigg(-\frac{\alpha(b-a)}{\epsilon}\bigg)\right]\;,
\]
which concludes the proof of~\eqref{eq:proba exiting neighborhood z}. 
\end{proof}
\subsection{Distance to the zero-noise underdamped Langevin process}
Let us now examine the distance, when $\epsilon\rightarrow0$, between the underdamped Langevin process~\eqref{eq:sde} and its zero-noise counterpart represented by
\begin{equation}
\left\{ \begin{aligned} & \mathrm{d}q_{t}=p_{t}\mathrm{d}t,\\
 & \mathrm{d}p_{t}=-\nabla U(q_{t})\mathrm{d}t-\gamma p_{t}\mathrm{d}t\;.
\end{aligned}
\right.\label{eq:deterministic Langevin}
\end{equation}
\medskip{}
 
 \noindent Recall below the notations $\mathcal{W}_m$ and $\mathcal{W}_s$ of Section~\ref{sec:Estimate of numerator}.
\begin{lem}
\label{lem:convergence deterministic Langevin} Let $K$ be a compact subset of $\overline{\mathcal{W}_z}$ such that $(\sigma,0)\notin K$. Then, for any $(q(0),\;p(0))\in K$,
\begin{equation}
\lim_{t\rightarrow\infty}(q(t),\,p(t))=(z,\,0)\;.\label{eq:long time limit deterministic}
\end{equation}
Furthermore, with $\rho$ defined as in Lemma~\ref{lem:lyapunov function},
there exist also constants $C,\,\lambda>0$ such that for all $(q(0),\,p(0))$
in $\mathrm{B}(z,\,\rho)$, 
\begin{equation}
|(q(t),\,p(t))-(z,\,0)|\leq C\mathrm{e}^{-\lambda t}.\label{eq:exponential convergence deterministic neighborhood}
\end{equation}
\end{lem}

\begin{proof}
We first prove the limit~\eqref{eq:long time limit deterministic}. Notice that for all $t\geq0$, 
\begin{equation}
\frac{\mathrm{d}}{\mathrm{d}t}\left(U(q(t))+\frac{1}{2}|p(t)|^{2}\right)=-\gamma|p(t)|^{2}\leq0\;.\label{eq:time derivative energy equality}
\end{equation}
In particular, it entails that $q(t)$ and $p(t)$ are bounded independently of the time $t\geq0$. And so are the time-derivatives below
\begin{equation}
\frac{\mathrm{d}}{\mathrm{d}t}|p(t)|^{2}=2\langle p(t),\,-\gamma p(t)-\nabla U(q(t))\rangle,\qquad\frac{\mathrm{d}}{\mathrm{d}t}U(q(t))=\langle p(t),\,\nabla U(q(t))\rangle\;.\label{eq:time derivatives}
\end{equation}
Moreover, by~\eqref{eq:time derivative energy equality}, $U(q(t))+|p(t)|^{2}/2$ is non-increasing in time and bounded
from below by $U(z)$. Therefore, it converges toward a limit $l\geq U(z)$. Assume that $l>U(z)$, then there exists a subsequence
of positive times $(t_{n})_{n\geq0}$ such that $t_{n}\underset{n\rightarrow\infty}{\longrightarrow}\infty$
and either 
\[
\lim_{n\rightarrow\infty}|p(t_{n})|^{2}>0\quad\text{or}\quad\lim_{n\rightarrow\infty}U(q(t_{n}))>U(z)\;.
\]

Assume now that $c_{1}:=\lim_{n\rightarrow\infty}|p(t_{n})|^{2}>0$.
Since the first time-derivative in~\eqref{eq:time derivatives} is
bounded then there exists a constant $\alpha>0$ independent of $n$
such that for all $n$ sufficiently large and $t\in[t_{n}-\alpha,\,t_{n}]$,
\begin{equation}
|p(t)|^{2}\geq\frac{c_{1}}{2}\;.\label{eq:lower bound velocity}
\end{equation}
Consequently, by~\eqref{eq:time derivative energy equality}, for
$n$ sufficiently large, 
\begin{align*}
U(q(t_{n}))+\frac{1}{2}|p(t_{n})|^{2} & =U(q(t_{n}-\alpha))+\frac{1}{2}|p(t_{n}-\alpha)|^{2}-\gamma\int_{t_{n}-\alpha}^{t_{n}}|p(r)|^{2}\mathrm{d}r\\
 & \leq U(q(t_{n}-\alpha))+\frac{1}{2}|p(t_{n}-\alpha)|^{2}-\frac{\gamma\alpha c_{1}}{2}\;.
\end{align*}
Taking the limit $n\rightarrow\infty$ above ensures that 
\[
l\leq l-\frac{\gamma\alpha c_1}{2}<l\;,
\]
which is a contradiction, hence $\lim_{t\rightarrow\infty}|p(t)|^{2}=0$.
It remains therefore to analyze the case $\lim_{n\rightarrow\infty}U(q(t_{n}))>U(z)$.
Let us show that $\lim_{n\rightarrow\infty}\nabla U(q(t_{n}))=0$
which shall ensure that necessarily $\lim_{n\rightarrow\infty}q(t_{n})=z$,
hence the proof of~\eqref{eq:long time limit deterministic}.

Assume that there exists a subsequence $(t_{n})_{n\geq0}$ such that
$\lim_{n\rightarrow\infty}\nabla U(q(t_{n}))=w$ where $w\neq~0$.
Then, using the boundedness of the time-derivative of $\nabla U(q(t))$,
since $U$ is $C^{2}$, along with the convergence of $p(t_{n})$
to zero, there exists $\alpha>0$ such that for all $n$ large enough
and $r\in[t_{n}-\alpha,\,t_{n}]$, 
\begin{equation}
\left|\nabla U(q(r))-w\right|+\big|p(r)\big|\leq\frac{\alpha|w|}{2(1+\alpha+\gamma\alpha)}\;.\label{eq:upperbound gradient U and velocity}
\end{equation}
It follows from~\eqref{eq:sde} that for all $n\geq0$, 
\[
p(t_{n})+\alpha w=p(t_{n}-\alpha)-\gamma\int_{t_{n}-\alpha}^{t_{n}}p(r)\mathrm{d}r-\int_{t_{n}-\alpha}^{t_{n}}\left(\nabla U(q(r))-w\right)\mathrm{d}r\;.
\]
Therefore, by~\eqref{eq:upperbound gradient U and velocity}, 
\[
\left|p(t_{n})+\alpha w\right|\leq\frac{\alpha|w|}{2(1+\alpha+\gamma\alpha)}(1+\alpha\gamma+\alpha)=\frac{\alpha|w|}{2}\;.
\]
Then, taking $n\rightarrow\infty$ ensures that 
\[
|\alpha w|\leq\frac{\alpha|w|}{2}\;,
\]
which is a contradiction with the fact that $w\neq0$, hence the proof
of~\eqref{eq:long time limit deterministic}.

Regarding the exponential control in~\eqref{eq:exponential convergence deterministic neighborhood},
its proof is identical to the proof of~\cite[Theorem 2.3]{cutoff}
and follows from the Lyapunov function in~\eqref{eq:lyapunov function H}. 
\end{proof}
We conclude this section with a control of the distance between the trajectories of~\eqref{eq:sde} and~\eqref{eq:deterministic Langevin} in times of
order $1/\epsilon^\lambda$ when $\epsilon\rightarrow0$. The estimate presented below was obtained in~\cite[Section 5]{cutoff} and involves the
Gaussian process $(Y(t))_{t\geq0}$ in $\mathbb{R}^{2d}$ satisfying
\begin{equation}
\textup{d}Y(t)=\mathbb{A}(q(t))Y(t)\textup{d}t+\mathbb{J}\,\textup{d}\widehat{B}(t),\quad Y(0)=0\;,\label{eq:def Y_t}
\end{equation}
where $(\widehat{B}(t))_{t\geq0}$ is a $\mathbb{R}^{2d}-$Brownian motion and (see~\cite[Equations 5.1, 5.3]{cutoff}) 
\[
\mathbb{A}(q):=\begin{pmatrix}\mathbb{O}_{d} & \mathbb{I}_{d}\\
-\mathbb{H}_{U}^{q} & -\gamma\mathbb{I}_{d}
\end{pmatrix},\qquad\mathbb{J}:=\begin{pmatrix}\mathbb{O}_{d} & \mathbb{O}_{d}\\
\mathbb{O}_{d} & \mathbb{I}_{d}
\end{pmatrix}\;.
\]
Therefore, $Y(t)\sim\mathcal{N}_{d}(0,\,\Sigma_{t}(x))$
where the covariance matrix $\Sigma_{t}(x)$ is defined in~\cite[Equation 6.2]{cutoff} and depends on the initial
condition $x$ of~\eqref{eq:deterministic Langevin}
through the dependency of $\mathbb{A}(q(t))$ in~\eqref{eq:def Y_t}.
The following lemma is an immediate
consequence of~\cite[Proposition 5.3 and Lemma 6.2]{cutoff}.
\begin{lem}
\label{lem:convergence X_T epsilon} Let $\rho,\,b>0$ be defined as in Lemma~\ref{lem:lyapunov function} and Proposition~\ref{prop:control moments H}. There exist constants $T>0$,
$\lambda\in(0,\,1/4)$ and $\beta>1$ independent of $\epsilon$ such
that for all $n\geq1$, there exists a constant $C_{n}>0$ such that
for all $t\in[T,\,1/\epsilon^{\lambda}]$, for all $x\in\mathrm{B}(z,\,\rho)$,
\begin{align*}
&\mathbb{E}_{x}\bigg[\bigg|\left(q^{\epsilon}\big(t\log\frac1\epsilon\big) ,\,p^{\epsilon}\big(t\log\frac1\epsilon\big)\right)-\left(q\big(t\log\frac1\epsilon\big),\,p\big(t\log\frac1\epsilon\big)\right)-\sqrt{2\epsilon\gamma}Y\big(t\log\frac1\epsilon\big)\bigg|^{n}\mathbf{1}_{\tau_{\partial\mathcal{H}_{b}}^{\epsilon}>t\log\frac1\epsilon}\bigg]\\
&\leq C_{n}\epsilon^{n\beta}\;,
\end{align*}
when $\epsilon>0$ is small enough. Furthermore, there exist constants
$C',\,\alpha>0$ and a symmetric positive definite matrix $\Sigma\in\R^{2d\times2d}$
such that 
\begin{equation}
\sup_{x\in\mathrm{B}(z,\,\rho)}\Vert\Sigma_{t}(x)-\Sigma\Vert_{F}\leq C'\mathrm{e}^{-\alpha t}\;.\label{eq:convergence covariance matrix}
\end{equation}
\end{lem}

\section{Proof of Proposition~\ref{prop:bound exponential exit time}.}

\label{sec:proof op Prop bound exponential exit time}

In this section, we first analyze the trajectory of the process~\eqref{eq:sde} starting from a point in $\mathcal{W}_m$ and provide some preliminary results regarding the first exit time of a neighbourhood of the saddle point $(\sigma,0)$ and the hitting time of $\mathcal{M}_\epsilon$. These results are then used in a second subsection to prove Proposition~\ref{prop:bound exponential exit time}. The proofs rely in particular on results obtained in Sections~\ref{sec:density estimate} and~\ref{sec:Lyapunov H}.

\begin{rem}\label{rem:extension to Ws}
The statements shown in this section remain true if $\mathcal{W}_{m}$
is replaced by $\mathcal{W}_{s}$ as the proofs can be transposed identically to this case.
\end{rem}

\subsection{Preliminary results}
This first lemma allows us to control the probability of remaining in a small neighbourhood of $(\sigma,0)$. We define for $\rho>0$,
\begin{equation}
Q_{\rho}:=\overline{\mathcal{W}_{m}}\setminus\mathrm{B}(\sigma,\,\rho)\;.\label{eq:def Q}
\end{equation}
\begin{lem}
\label{lem:exit ball sigma} 
There exist constants $C,\,\rho>0$ independent of $\epsilon$
such that for all $t>0$, 
\begin{equation}
\sup_{x\in\mathcal{W}_{m}}\mathbb{P}_{x}\big(\tau_{Q_{\rho}}^{\epsilon}>t\big)\leq\frac{C}{\sqrt{\epsilon t}}\;.\label{eq:sup proba exit Q}
\end{equation}
\end{lem}
\begin{proof}[Proof of Lemma~\ref{lem:exit ball sigma}]
It is sufficient to prove the inequality~\eqref{eq:sup proba exit Q} with a supremum over $x\in\mathcal{W}_m\cap\mathrm{B}(\sigma,\,\rho)$ since the probability vanishes on $Q_{\rho}$. Considering that $U$ is a Morse function, its Hessian matrix at the saddle point
$\sigma$ admits one negative eigenvalue $-\lambda_{1}$ and $d-1$
positive eigenvalues $\lambda_{2},\dots,\lambda_{d}$. We denote by $(e_{1},\dots,e_{d})$ the orthonormal basis associated with the eigenvalues $(\lambda_1,\dots,\lambda_d)$ and by $x_{i}=\langle x,\,e_{i}\rangle$ the
orthogonal projection of the vector $x$ on $e_{i}$. 

Computing the asymptotic expansion of $U(q)$ when $q$ is close to $\sigma$ one gets
\[
U(q)=U(\sigma)-\frac{\lambda_{1}}{2}(q_{1}-\sigma_{1})^{2}+\sum_{i=2}^{d}\frac{\lambda_{i}}{2}(q_{i}-\sigma_{i})^{2}+O(|q-\sigma|^{3})\;.
\]
As a result, if $U(q)<U(\sigma)$ and $|q-\sigma|\leq\delta$ for a sufficiently small $\delta$ then there exists a constant
$c_{1}>0$ such that 
\[
(q_{1}-\sigma_{1})^{2}\geq c_{1}\sum_{i=2}^{d}(q_{i}-\sigma_{i})^{2}\;,
\]
which ensures that $O(|q-\sigma|^{2})=O(|q_{1}-\sigma_{1}|^{2})$. Also,
up to changing the sign of $e_{1}$, one can assume that for
$(q,\;p)\in\mathcal{W}_m$ and $q\in\mathrm{B}(\sigma,\,\rho)$, then $q_{1}-\sigma_{1}>0$. Therefore, if we compute the asymptotics of $\langle\nabla U(q),\,e_{1}\rangle$
for $(q,\;p)\in\mathcal{W}_m$ and $q\in\mathrm{B}(\sigma,\,\rho)$, 
\begin{align*}
\langle\nabla U(q),\,e_{1}\rangle & =\langle\mathbb{H}_{U}^{\sigma}(q-\sigma),\,e_{1}\rangle+O(|q-\sigma|^{2})\\
 & =-\lambda_{1}(q_{1}-\sigma_{1})+O((q_{1}-\sigma_{1})^{2})\\
 & \leq0\;.
\end{align*}
Furthermore, by~\eqref{eq:sde},
it follows that 
\[
\mathrm{d}\left(q^\epsilon(t)+\frac{p^\epsilon(t)}{\gamma}\right)=-\frac{1}{\gamma}\nabla U(q^\epsilon(t))\mathrm{d}t+\sqrt{\frac{2\epsilon}{\gamma}}\mathrm{d}B(t)\;.
\]
As a result, for all $t\geq0$, 
\[
q^\epsilon(t)+\frac{p^\epsilon(t)}{\gamma}=q+\frac{p}{\gamma}-\frac{1}{\gamma}\int_{0}^{t}\nabla U(q^\epsilon(r))\mathrm{d}r+\sqrt{\frac{2\epsilon}{\gamma}}B(t)\;.
\]
Therefore, for $(q,p)\in\mathcal{W}_m\cap\mathrm{B}(\sigma,\,\rho)$, on the event $\{\tau_{Q_{\rho}}^{\epsilon}>t\}$, 
\[
\big\langle q^\epsilon(t)+\frac{p^\epsilon(t)}{\gamma},\,e_{1}\big\rangle\geq\big\langle q+\frac{p}{\gamma},\,e_{1}\big\rangle+\sqrt{\frac{2\epsilon}{\gamma}}\langle B(t),\,e_{1}\rangle\;.
\]
Since $\mathrm{B}(\sigma,\,\rho)$ is bounded, there exists a constant
$C>0$ independent of $\epsilon$ such that for $t>0$,
\begin{align*}
\mathbb{P}_{x}(\tau_{Q_{\rho}}^{\epsilon}>t) & \leq\mathbb{P}(\sqrt{\epsilon}|\langle B(t),\,e_{1}\rangle|\leq C)\\
 & \leq\mathbb{P}(|\langle B(1),\,e_{1}\rangle|\leq C/\sqrt{\epsilon t})  \leq\frac{2C}{\sqrt{2\pi\epsilon t}}\;,
\end{align*}
which concludes the proof. 
\end{proof}
Moreover, we can see that the time before entering a neighbourhood of $(m,0)$, for the zero-noise process~\eqref{eq:deterministic Langevin}, is a bounded function of the initial condition in $Q_{\rho}$.
\begin{lem}
\label{lem:hitting time deterministic} Let $\kappa>0$ and define  
\begin{equation}
\mathcal{S}:x\in\R^{2d}\mapsto\inf\{t\geq0:U(q(t))+|p(t)|^{2}/2\leq U(m)+\kappa,\quad(q_{0},\,p_{0})=x\}\;.\label{eq:def deterministic hitting time}
\end{equation}
Then, for $\rho>0$ defined as in Lemma~\ref{lem:exit ball sigma},
\begin{equation}
\sup_{x\in Q_{\rho}}\mathcal{S}(x)<\infty\;.\label{eq:control T_x in epsilon}
\end{equation}
\end{lem}

\begin{proof}[Proof of Lemma~\ref{lem:hitting time deterministic}]

It follows from Lemma~\ref{lem:convergence deterministic Langevin}
that for all $x\in Q_{\rho}$,
$\mathcal{S}(x)<\infty$. Therefore, if $\mathcal{S}$ is also continuous on the set $Q_{\rho}$ then the statement follows from the compactness of $Q_{\rho}$.

For $x\in Q_{\rho}$, we now denote by $(q(t;\,x),\,p(t;\,x))_{t\geq0}$ the trajectory
of~\eqref{eq:deterministic Langevin} satisfying $(q(0),\,p(0))=x$.
Notice that, by the continuity of the trajectories of~\eqref{eq:deterministic Langevin},
\begin{equation}
U(q(\mathcal{S}(x);\,x))+\frac{1}{2}|p(\mathcal{S}(x);\,x)|^{2}\leq U(m)+\kappa\;.\label{eq:def S_x}
\end{equation}
Moreover, for any $\theta>0$,
\begin{equation}
U(q(\mathcal{S}(x)+\theta;\,x))+\frac{1}{2}|p(\mathcal{S}(x)+\theta;\,x)|^{2}<U(m)+\kappa\;.\label{eq:decrease energy at S_x}
\end{equation}

\noindent Indeed, if this is not the case then for all $t\in[\mathcal{S}(x),\,\mathcal{S}(x)+\theta)$,
\begin{equation}\label{eq:time derivative hamiltonian}
\frac{\mathrm{d}}{\mathrm{d}t}\left(U(q(t;\,x))+\frac{1}{2}|p(t;\,x)|^{2}\right)=-\gamma|p(t;\,x)|^{2}=0\;.
\end{equation}
Therefore, by~\eqref{eq:deterministic Langevin}, for
all $t\in(\mathcal{S}(x),\,\mathcal{S}(x)+\theta)$, 
\[
p(t;\,x)=p(\mathcal{S}(x);\,x)-\gamma\int_{\mathcal{S}(x)}^{t}p(r;\,x)\mathrm{d}r-\int_{\mathcal{S}(x)}^{t}\nabla U(q(r;\,x))\mathrm{d}r\;.
\]
As a result,
\[
\frac{1}{t-\mathcal{S}(x)}\int_{\mathcal{S}(x)}^{t}\nabla U(q(r;\,x))\mathrm{d}r=0\;.
\]
Taking the limit when $t\rightarrow\mathcal{S}(x)$ ensures that
$\nabla U(q(\mathcal{S}(x);\,x))=0$ which implies that $q(\mathcal{S}(x);\,x)=m$. Thus, since $p(\mathcal{S}(x);\,x)=0$,
$$U(q(\mathcal{S}(x);\,x))+\frac{1}{2}|p(\mathcal{S}(x);\,x)|^{2}=U(m)\;.$$
Therefore, we deduce from~\eqref{eq:time derivative hamiltonian} that~\eqref{eq:decrease energy at S_x} is satisfied, hence a contradiction. As a result,~\eqref{eq:decrease energy at S_x} is satisfied for any $\theta>0$.

Notice that for $x\in Q_{\rho}$,
the trajectory $(q(t;\,x),\,p(t;\,x))_{t\geq0}$ is bounded in time since the energy $U(q(t;\,x))+|p(t;\,x)|^{2}/2$
is a non-increasing function of time. Therefore, if we take $x,\;y\in Q_{\rho}$, since $U$ is $C^{2}$ on $\R^{d}$, by Gr\"onwall's inequality there exists $C>0$ such that for all $t\geq0$, 
\[
\sup_{r\in[0,t]}|(q(r;\,x),\,p(r;\,x))-(q(r;\,y),\,p(r;\,y))|\leq\mathrm{e}^{Ct}|x-y|\;.
\]
Therefore, if $\mathcal{S}(y)>0$, then for any $\alpha\in(0,\mathcal{S}(y))$, if $|x-y|$ is small enough then 
\[
\inf_{t\in[0,\,\mathcal{S}(y)-\alpha]}U(q(t;\,x))+\frac{1}{2}|p(t;\,x)|^{2}>U(m)+\kappa,\quad U(q(\mathcal{S}(y)+\alpha;\,x))+\frac{1}{2}|p(\mathcal{S}(y)+\alpha;\,x)|^{2}<U(m)+\kappa\;,
\]
which ensures that $\mathcal{S}(y)-\alpha<\mathcal{S}(x)<\mathcal{S}(y)+\alpha$. Similarly, if $\mathcal{S}(y)=0$, then $0\leq\mathcal{S}(x)<\alpha$. Hence the convergence $\mathcal{S}(x)\underset{x\rightarrow y}{\longrightarrow}\mathcal{S}(y)$
which concludes the proof. 
\end{proof}

Using the result of Lemma~\ref{lem:hitting time deterministic}, we are now able to control the probability that the process~\eqref{eq:sde} starting from $x\in Q_\rho$ remains outside of a neighbourhood of $(m,0)$.
\begin{lem}
\label{lem:hitting ball minimum} For any $\delta>0$ independent
of $\epsilon$, there exists $\theta=\theta(\delta)>0$ such that for all $N\geq1$, there exists a constant $C=C(N,\,\delta)>0$ such that 
\[
\sup_{x\in Q_\rho}\mathbb{P}_{x}\big(\tau_{\mathrm{B}(m,\,\delta)}^{\epsilon}\land\tau_{\partial\mathcal{W}_{m}}^{\epsilon}>\theta\big)\leq C\epsilon^{N/2}\;,
\]
where $\rho$ is the radius defined in Lemma~\ref{lem:exit ball sigma}. 
\end{lem}

\begin{proof}[Proof of Lemma~\ref{lem:hitting ball minimum}]

Fix $\delta>0$ independent of $\epsilon$ and
let us take $\kappa=\kappa(\delta)>0$ small enough such that for all $x\in Q_\rho$,
\[
U(q)+\frac{1}{2}|p|^{2}\leq U(m)+\kappa\,\,\Longrightarrow\,\, x=(q,\,p)\in\mathrm{B}(m,\,\delta/2)\;.
\]
Then, by Lemma~\ref{lem:hitting time deterministic}, there exists $T=T(\delta)>0$
such that for all initial conditions $(q(0),\;p(0))\in Q_\rho$,
the process~\eqref{eq:deterministic Langevin} satisfies
\[
(q(T),\,p(T))\in\mathrm{B}(m,\,\delta/2)\;.
\]

\noindent Therefore, by the Markov inequality, for all $N\geq1$, 
\begin{align}
\mathbb{P}_{x}\big(\tau_{\mathrm{B}(m,\,\delta)}^{\epsilon}\land\tau_{\partial\mathcal{W}_{m}}^{\epsilon}>T\big) & \leq\mathbb{P}_{x}\big(\sup_{t\in[0,\,T]}\big|(q^{\epsilon}(t),\,p^{\epsilon}(t))-(q(t),\,p(t))\big|>\delta/2,\,\tau_{\partial\mathcal{W}_{m}}^{\epsilon}>T\big)\nonumber\\
 & \leq\frac{1}{(\delta/2)^{N}}\mathbb{E}_{x}\bigg[\sup_{t\in[0,\,T]}\big|(q^{\epsilon}(t),\,p^{\epsilon}(t))-(q(t),\,p(t))\big|^{N}\mathbf{1}_{\tau_{\partial\mathcal{W}_{m}}^{\epsilon}>T}\bigg]\label{eq:gronwall event no exit from W}\;.
\end{align} 
A standard application of the Gr\"onwall inequality on the event $\{\tau_{\partial\mathcal{W}_{m}}^{\epsilon}>T\}$ yields the existence
of a constant $C>0$ independent of $\epsilon$ such that 
\[
\sup_{t\in[0,\,T]}\big|(q^{\epsilon}(t),\,p^{\epsilon}(t))-(q(t),\,p(t))\big|\leq\sqrt{2\gamma\epsilon}\,\mathrm{e}^{CT}\sup_{t\in[0,\,T]}|B(t)|\;.
\]
\noindent Reinjecting into~\eqref{eq:gronwall event no exit from W} concludes the proof. 
\end{proof}

We conclude this subsection with the following lemma controlling the probability that the process~\eqref{eq:sde} does not enter $\mathcal{M}_\epsilon$ in times of order $\log(1/\epsilon)$, when starting from a neighbourhood of $(m,0)$. 

\begin{lem}
\label{lem:enter ball A_epsilon} There exist constants $c,\,\rho',\,\theta'>0$
independent of $\epsilon$ such that for $\epsilon$ small enough,
\[
\sup_{x\in\mathrm{B}(m,\,\rho')}\mathbb{P}_{x}(\tau_{\mathcal{M}_{\epsilon}}^{\epsilon}>\theta'\log(1/\epsilon))\leq1-c\epsilon^{d}\;.
\]
\end{lem}

\begin{proof}[Proof of Lemma~\ref{lem:enter ball A_epsilon}]

Let us fix $\delta>0$ small enough satisfying Proposition~\ref{prop:control moments H}
and such that $\mathcal{H}_{\delta}\subset\mathrm{B}(m,\,\rho)$
where $\rho$ is defined in Lemma~\ref{lem:lyapunov function}. Let
us take $T>0$ large enough independent of $\epsilon$ such that by~\eqref{eq:exponential convergence deterministic neighborhood}
in Lemma~\ref{lem:convergence deterministic Langevin}, for any $(q(0),\,p(0))=x\in\mathrm{B}(m,\,\rho)$,
\begin{equation}
\big|(q(T\log(1/\epsilon)),\,p(T\log(1/\epsilon)))-(m,0)\big|<\frac{\epsilon}{4}\;.\label{eq:asymptotics X_t}
\end{equation}

\noindent For $t\geq0$, let $Y(t)$ be the Gaussian vector satisfying~\eqref{eq:def Y_t} and define also
$$Z^\epsilon(t):=(q^\epsilon(t),p^\epsilon(t))-(q(t),p(t))-\sqrt{2\epsilon\gamma}\,Y(t)\,.$$

\noindent Consider now $x\in\mathcal{H}_{\delta/2}$ and let us look at the following probability 
\begin{align}
\mathbb{P}_{x}\bigg(\tau_{\mathcal{M}_{\epsilon}}^{\epsilon}>T\log(1/\epsilon)\bigg) & =1-\mathbb{P}_{x}\bigg(\tau_{\mathcal{M}_{\epsilon}}^{\epsilon}\leq T\log(1/\epsilon)\bigg)\nonumber \\
 & \leq1-\mathbb{P}_{x}\bigg(\tau_{\mathcal{M}_{\epsilon}}^{\epsilon}\leq T\log(1/\epsilon),\,\tau_{\partial\mathcal{H}_{\delta}}^{\epsilon}>T\log(1/\epsilon)\bigg)\;.\label{eq:upperbound hitting time A_eps}
\end{align}

By Lemma~\ref{lem:convergence X_T epsilon}, up to increasing $T$,
we have the existence of $\beta>1$ independent of $\epsilon$ such
that for all $m\geq1$ there exists $C_{m}>0$ such that for all $x\in\mathcal{H}_{\delta/2}$,
\begin{equation}\label{eq:control asymptotics epsilon Langevin}
\mathbb{E}_{x}\big[\big|Z^\epsilon(T\log(1/\epsilon))\big|^{m}\mathbf{1}_{\tau_{\partial\mathcal{H}_{\delta}}^{\epsilon}>T\log(1/\epsilon)}\big]\leq C_{m}\epsilon^{m\beta}\;.
\end{equation} 
Let us fix $m\geq1$ large enough such that $m(\beta-1)>d$. One has that
\begin{align*}
 & \mathbb{P}_{x}\bigg(\tau_{\mathcal{M}_{\epsilon}}^{\epsilon}\leq T\log(1/\epsilon),\,\tau_{\partial\mathcal{H}_{\delta}}^{\epsilon}>T\log(1/\epsilon)\bigg)\\
 & \geq\mathbb{P}_{x}\bigg((q^{\epsilon}(T\log(1/\epsilon)),\,p^{\epsilon}(T\log(1/\epsilon)))\in\mathcal{M}_{\epsilon},\,\tau_{\partial\mathcal{H}_{\delta}}^{\epsilon}>T\log(1/\epsilon)\bigg)\\
 & \geq\mathbb{P}_{x}\bigg(\sqrt{2\gamma\epsilon}\big|Y(T\log(1/\epsilon))\big|\leq\frac{\epsilon}{4},\,\big|Z^\epsilon(T\log(1/\epsilon))\big|\leq\frac{\epsilon}{2},\,\tau_{\partial\mathcal{H}_{\delta}}^{\epsilon}>T\log(1/\epsilon)\bigg)\;.
\end{align*}
using the triangle inequality with~\eqref{eq:asymptotics X_t}. Consequently, 
\begin{align}
 & \mathbb{P}_{x}\bigg(\tau_{\mathcal{M}_{\epsilon}}^{\epsilon}\leq T\log(1/\epsilon),\,\tau_{\partial\mathcal{H}_{\delta}}^{\epsilon}>T\log(1/\epsilon)\bigg)\nonumber \\
 & \geq\mathbb{P}_{x}\bigg(\sqrt{2\gamma\epsilon}\big|Y(T\log(1/\epsilon))\big|\leq\frac{\epsilon}{4}\bigg)-\mathbb{P}_{x}\bigg(\tau_{\partial\mathcal{H}_{\delta}}^{\epsilon}\leq T\log(1/\epsilon)\bigg)\\
 &\quad-\mathbb{P}_{x}\bigg(\big|Z^\epsilon(T\log(1/\epsilon))\big|>\frac{\epsilon}{2},\,\tau_{\partial\mathcal{H}_{\delta}}^{\epsilon}>T\log(1/\epsilon)\bigg)\nonumber \\ 
 & \geq\mathbb{P}_{x}\bigg(\sqrt{2\gamma\epsilon}\big|Y(T\log(1/\epsilon))\big|\leq\frac{\epsilon}{4}\bigg)-C\frac{(1+T\log(1/\epsilon))}{\epsilon^{\beta}}\mathrm{exp}\bigg(-\frac{\alpha\delta}{2\epsilon}\bigg)-2^{m}C_{m}\epsilon^{m(\beta-1)}\label{eq:lower bound proba hitting A_eps}
\end{align}
by Proposition~\ref{prop:control moments H} and by applying the Markov inequality along with the inequality~\eqref{eq:control asymptotics epsilon Langevin}.
Notice that since we took $m$ such that $m(\beta-1)>d$, the proof of the lemma follows if there exists a constant $c>0$
independent of $x$ and $\epsilon$ such that for $\epsilon$ small enough, 
\begin{equation}
\mathbb{P}_{x}\bigg(\big|Y(T\log(1/\epsilon))\big|\leq\frac{\sqrt{\epsilon}}{4\sqrt{2\gamma}}\bigg)\geq c\epsilon^{d}\;.\label{eq:lower bound Y_T}
\end{equation}
For this reason, let us take a symmetric matrix $\mathbb{B}_{T\log(1/\epsilon)}(x)$
such that $\mathbb{B}_{T\log(1/\epsilon)}^{2}(x)=\Sigma_{T\log(1/\epsilon)}(x)$
where $\Sigma_{t}(x)$ is the covariance matrix of $Y(t)$.
Then, taking a Gaussian vector $Z\sim\mathcal{N}_{2d}(0,\,\mathbb{I}_{2d})$,
it is easy to see that $Y(T\log(1/\epsilon))$ and $\mathbb{B}_{T\log(1/\epsilon)}(x)Z$
share the same law. Additionally, 
\[
\big|\mathbb{B}_{T\log(1/\epsilon)}(x)Z\big|\leq\sqrt{\big\Vert\Sigma_{T\log(1/\epsilon)}(x)\big\Vert_{\textup{F}}}|Z|\;,
\]
since $\mathbb{B}_{T\log(1/\epsilon)}^{2}(x)=\Sigma_{T\log(1/\epsilon)}(x)$.
Therefore, by~\eqref{eq:convergence covariance matrix},
there exists a constant $C>0$ independent of $x\in\mathcal{H}_{\delta/2}$
such that 
\[
|\mathbb{B}_{T\log(1/\epsilon)}(x)Z|\leq C|Z|\;.
\]
Consequently, 
\begin{align*}
\mathbb{P}_{x}\bigg(\big|Y(T\log(1/\epsilon))\big|\leq\frac{\sqrt{\epsilon}}{4\sqrt{2\gamma}}\bigg)&=\mathbb{P}\bigg(\big|\mathbb{B}_{T\log(1/\epsilon)}(x)Z\big|\leq\frac{\sqrt{\epsilon}}{4\sqrt{2\gamma}}\bigg)\\
&\geq\mathbb{P}\bigg(|Z|\leq\frac{\sqrt{\epsilon}}{4C\sqrt{2\gamma}}\bigg)\;.
\end{align*}
Let us define the constant $\beta=4C\sqrt{2\gamma}$. Since $|Z|\sim\chi(2d)$,
\begin{align*}
\mathbb{P}\bigg(|Z|\leq\frac{\sqrt{\epsilon}}{\beta}\bigg) & =\int_{0}^{\sqrt{\epsilon}/\beta}\frac{1}{2^{d-1}\Gamma(d)}x^{2d-1}\mathrm{e}^{-x^{2}/2}\mathrm{d}x\\
 & =\frac{1}{2^{d-1}\Gamma(d)\beta^{2d}}\frac{\epsilon^{d}}{2d}\mathrm{e}^{-\epsilon/2\beta^{2}}+\frac{1}{2^{d-1}\Gamma(d)}\int_{0}^{\sqrt{\epsilon}/\beta}\frac{x^{2d+1}}{2d}\mathrm{e}^{-x^{2}/2}\mathrm{d}x\\
 & \geq c\epsilon^{d}\;,
\end{align*}
for $\epsilon$ small enough, hence~\eqref{eq:lower bound Y_T}. Then, taking $\rho'>0$ small enough such that $\mathrm{B}(m,\rho')\subset\mathcal{H}_{\delta/2}$ concludes the proof of Lemma~\ref{lem:enter ball A_epsilon}.
\end{proof}

\subsection{Proof of Proposition~\ref{prop:bound exponential exit time}}

Thanks to the lemmas obtained in the previous subsection, we are able to provide a sharp control on the
tails of the exit time $\zeta^{\epsilon}$ defined in~\eqref{eq:def tau_epsilon}, which is then used to prove Proposition~\ref{prop:bound exponential exit time}.
\begin{prop}
\label{prop:exit time probability estimate} There exist constants
$C,\,\alpha,\,\beta>0$ independent of $\epsilon$, for $\epsilon$ small enough, such that for all $t\geq0$, for all $x\in\mathcal{W}_{m}\setminus\overline{\mathcal{M}_\epsilon}$, 
\begin{equation}\label{eq:control tail zeta}
\mathbb{P}_{x}\big(\zeta^{\epsilon}>t\big)\leq C\,\mathrm{exp}\big(-\alpha t\epsilon^{\beta}\big)\;,
\end{equation}
where $\zeta^{\epsilon}$ is defined in~\eqref{eq:def tau_epsilon}. 
\end{prop}

\begin{proof}[Proof of Proposition~\ref{prop:exit time probability estimate}]
Define the open set $\mathcal{D}_\epsilon:=\mathcal{W}_{m}\setminus\overline{\mathcal{M}_\epsilon}$. Starting from $x\in\mathcal{D}_{\epsilon}$, one has, $\mathbb{P}_x$ almost-surely, the equality $\zeta^{\epsilon}=\tau_{\partial\mathcal{D}_{\epsilon}}^{\epsilon}$. 

\textbf{Step 1:} Let us now take a constant $m>2d+1$ and let $T_{\epsilon}:=1/\epsilon^{m}$. We first prove the existence of a constant $c>0$ independent of $\epsilon$ such that for $\epsilon>0$ small enough,
\begin{equation}
\sup_{x\in\mathcal{D}_{\epsilon}}\mathbb{P}_{x}\big(\tau_{\partial\mathcal{D}_{\epsilon}}^{\epsilon}>T_{\epsilon}\big)\leq1-c\epsilon^{d}.\label{eq:sup proba M_1}
\end{equation}
Let $Q_{\rho}$ be the compact set defined in~\eqref{eq:def Q} where $\rho>0$
is given by Lemma~\ref{lem:exit ball sigma}. We first notice that for any $x\in\mathcal{D}_{\epsilon}$, 
\[
\mathbb{P}_{x}\big(\tau_{\partial\mathcal{D}_{\epsilon}}^{\epsilon}>T_{\epsilon}\big)\leq\mathbb{P}_{x}\big(\tau_{\partial\mathcal{D}_{\epsilon}}^{\epsilon}>T_{\epsilon},\,\tau_{Q_{\rho}}^{\epsilon}<T_{\epsilon}/4\big)+\mathbb{P}_{x}\big(\tau_{Q_{\rho}}^{\epsilon}\geq T_{\epsilon}/4\big)\;,
\]
By Lemma~\ref{lem:exit ball sigma},
there exists a constant $C>0$ independent of $\epsilon$ and $x\in\mathcal{D}_{\epsilon}$
such that 
\[
\mathbb{P}_{x}\big(\tau_{Q_{\rho}}^{\epsilon}\geq T_{\epsilon}/4\big)\leq C\epsilon^{(m-1)/2}\;.
\]
Furthermore, applying the strong Markov property at the stopping time
$\tau_{Q_{\rho}}^{\epsilon}$, 
\begin{equation}
\mathbb{P}_{x}\big(\tau_{\partial\mathcal{D}_{\epsilon}}^{\epsilon}>T_{\epsilon},\,\tau_{Q_{\rho}}^{\epsilon}<T_{\epsilon}/4\big)\leq\mathbb{E}_{x}\bigg[\mathbf{1}_{\tau_{\partial\mathcal{D}_{\epsilon}}^{\epsilon}>\tau_{Q_{\rho}}^{\epsilon}}\mathbb{P}_{\big(q^{\epsilon}(\tau_{Q_{\rho}}^{\epsilon}),\,p^{\epsilon}(\tau_{Q_{\rho}}^{\epsilon})\big)}\bigg(\tau_{\partial\mathcal{D}_{\epsilon}}^{\epsilon}>3T_{\epsilon}/4\bigg)\bigg]\;.\label{eq:probability first hitting time from M}
\end{equation}
Taking $\rho'>0$ as defined in Lemma~\ref{lem:enter ball A_epsilon}, we deduce from Lemma~\ref{lem:hitting ball minimum} the existence of
a constant $C_{m}>0$ such that for all $y\in Q_{\rho}$,
\begin{align*}
\mathbb{P}_{y}\bigg(\tau_{\partial\mathcal{D}_{\epsilon}}^{\epsilon}>3T_{\epsilon}/4\bigg) & \leq\mathbb{P}_{y}\bigg(\tau_{\partial\mathcal{D}_{\epsilon}}^{\epsilon}>3T_{\epsilon}/4,\,\tau_{\mathrm{B}(m,\,\rho')}^{\epsilon}<T_{\epsilon}/4\bigg)+\mathbb{P}_{y}\bigg(\tau_{\partial\mathcal{W}_{m}}^{\epsilon}\land\tau_{\mathrm{B}(m,\,\rho')}^{\epsilon}\geq T_{\epsilon}/4\bigg)\\
 & \leq\mathbb{P}_{y}\bigg(\tau_{\partial\mathcal{D}_{\epsilon}}^{\epsilon}>3T_{\epsilon}/4,\,\tau_{\mathrm{B}(m,\;\rho')}^{\epsilon}<T_{\epsilon}/4\bigg)+C_{m}\epsilon^{(m-1)/2}\;.
\end{align*}

Moreover, applying the strong Markov property at the hitting time
$\tau_{\mathrm{B}(m,\,\rho')}^{\epsilon}$, 
\begin{align*}
\mathbb{P}_{y}\bigg(\tau_{\partial\mathcal{D}_{\epsilon}}^{\epsilon}>3T_{\epsilon}/4,\,\tau_{\mathrm{B}(m,\,\rho')}^{\epsilon}<T_{\epsilon}/4\bigg) & \leq\mathbb{E}_{x}\bigg[\mathbf{1}_{\tau_{\partial\mathcal{D}_{\epsilon}}^{\epsilon}>\tau_{\mathrm{B}(m,\,\rho')}^{\epsilon}}\mathbb{P}_{\big(q^{\epsilon}(\tau_{\mathrm{B}(m,\,\rho')}^{\epsilon}),\,p^{\epsilon}(\tau_{\mathrm{B}(m,\,\rho')}^{\epsilon})\big)}\bigg(\tau_{\partial\mathcal{D}_{\epsilon}}^{\epsilon}>T_{\epsilon}/2\bigg)\bigg]\\
 & \leq\sup_{x\in\mathrm{B}(m,\,\rho')}\mathbb{P}_{x}\bigg(\tau_{\mathcal{M}_{\epsilon}}^{\epsilon}>T_{\epsilon}/2\bigg)\;.
\end{align*}
Therefore, by Lemma~\ref{lem:enter ball A_epsilon} and the definition of $T_{\epsilon}$, 
there exists a constant $c\in(0,\,1)$ independent of $\epsilon$
such that for $\epsilon$ small enough, 
\[
\mathbb{P}_{y}\bigg(\tau_{\partial\mathcal{D}_{\epsilon}}^{\epsilon}>3T_{\epsilon}/4,\,\tau_{\mathrm{B}(m,\,\rho')}^{\epsilon}<T_{\epsilon}/4\bigg)\leq1-c\epsilon^{d}\;.
\]
Combining all these inequalities ensures the existence of a constant
$C'>0$ independent of $\epsilon$ such that 
\[
\sup_{x\in\mathcal{D}_{\epsilon}}\mathbb{P}_{x}\big(\tau_{\partial\mathcal{D}_{\epsilon}}^{\epsilon}>T_{\epsilon}\big)\leq1-c\epsilon^{d}+C'\epsilon^{(m-1)/2}\;,
\]
which concludes the proof of~\eqref{eq:sup proba M_1}, for $\epsilon$
small enough, since $(m-1)/2>d$.

\textbf{Step 2:} Let us now deduce~\eqref{eq:control tail zeta} using~\eqref{eq:sup proba M_1}. For any $t\geq0$ and $x\in\mathcal{D}_{\epsilon}$, by the Markov property, 
\begin{align*}
\mathbb{P}_{x}\big(\tau_{\partial\mathcal{D}_{\epsilon}}^{\epsilon}>t\big) & \leq\mathbb{P}_{x}\bigg(\tau_{\partial\mathcal{D}_{\epsilon}}^{\epsilon}>\left\lfloor \frac{t}{T_{\epsilon}}\right\rfloor T_{\epsilon}\bigg)\\
 & \leq(1-c\epsilon^{d})^{\lfloor t/T_{\epsilon}\rfloor}\\
 &=\mathrm{exp}\bigg(\left\lfloor \frac{t}{T_{\epsilon}}\right\rfloor \log\left(1-c\epsilon^{d}\right)\bigg)\\
 & \leq\mathrm{exp}\bigg(-ct\epsilon^{m+d}+c\epsilon^{d}\bigg)\;,
\end{align*}
using the standard inequalities $\log(1-x)<-x$ for $x\in[0,\,1)$
and $\lfloor x\rfloor\geq x-1$ for $x\geq0$, hence~\eqref{eq:control tail zeta}. 
\end{proof}

Finally, we conclude this section with the proof of Proposition~\ref{prop:bound exponential exit time}.
\begin{proof}[Proof of Proposition~\ref{prop:bound exponential exit time}]
 We provide here the proof for the case $x\in\mathcal{W}_{m}$ but as it was emphasized on Remark~\ref{rem:extension to Ws} all the results in this section can be extended to the case $x\in\mathcal{W}_{s}$ using an identical proof.
 
 Let $\mathcal{D}_\epsilon:=\mathcal{W}_{m}\setminus\overline{\mathcal{M}_\epsilon}$ and recall the stopping times defined in~\eqref{eq:def tau_epsilon}. For $x\in\mathcal{D}_{\epsilon}$, one has that, $\mathbb{P}_x$ almost-surely, 
\[
\zeta^{\epsilon}=\tau_{\partial\mathcal{D}_{\epsilon}}^{\epsilon},\quad\widetilde{\zeta}^{\epsilon}=\tilde{\tau}_{\partial\mathcal{D}_{\epsilon}}^{\epsilon}\;,
\]
Consequently, for all $x\in\mathcal{D}_{\epsilon}$,
\begin{align*}
\mathbb{E}_{x}\left[\mathrm{exp}\big(d\gamma\widetilde{\zeta}^{\epsilon}\big)\right] & =\mathbb{E}_{x}\left[\mathrm{exp}\big(d\gamma\widetilde{\tau}_{\partial\mathcal{D}_{\epsilon}}^{\epsilon}\big)\right]\\
 & =1+d\gamma\mathbb{E}_{x}\left[\int_{0}^{\widetilde{\tau}_{\partial\mathcal{D}_{\epsilon}}^{\epsilon}}\mathrm{e}^{d\gamma r}\right]\mathrm{d}r\\
 & =1+d\gamma\int_{0}^{\infty}\mathrm{e}^{d\gamma r}\mathbb{P}_{x}(\widetilde{\tau}_{\partial\mathcal{D}_{\epsilon}}^{\epsilon}>r)\mathrm{d}r\\
 & \leq1+d\gamma\epsilon\mathrm{e}^{d\gamma\epsilon}+d\gamma\int_{\epsilon}^{\infty}\mathrm{e}^{d\gamma r}\mathbb{P}_{x}(\widetilde{\tau}_{\partial\mathcal{D}_{\epsilon}}^{\epsilon}>r)\mathrm{d}r\\
 & \leq1+d\gamma\epsilon\mathrm{e}^{d\gamma\epsilon}+d\gamma\int_{\epsilon}^{\infty}\int_{\mathcal{D}_{\epsilon}}\mathrm{e}^{d\gamma r}\widetilde{\mathrm{p}}_{r}^{\epsilon,\,\mathcal{D}_{\epsilon}}(x,\;x')\mathrm{d}x'\mathrm{d}r\;,
\end{align*}
by Proposition~\ref{prop:property density}. Additionally, by~\eqref{eq:property density}, 
\begin{align}
\mathbb{E}_{x}\left[\mathrm{exp}\big(d\gamma\widetilde{\zeta}^{\epsilon}\big)\right]\leq1+d\gamma\epsilon\mathrm{e}^{d\gamma\epsilon}+d\gamma\int_{\epsilon}^{\infty}\int_{\mathcal{D}_{\epsilon}}\mathrm{p}_{r}^{\epsilon,\,\mathcal{D}_{\epsilon}}(x',\,x)\mathrm{d}x'\mathrm{d}r\;.\label{eq:control exponential moment dgamma}
\end{align}
Furthermore, by Chapman-Kolomogorov's relation, for all $x,\,x'\in\mathcal{D}_{\epsilon}$
and for all $r>\epsilon$, 
\begin{equation}
\mathrm{p}_{r}^{\epsilon,\,\mathcal{D}_{\epsilon}}(x',\,x)=\mathbb{E}_{x'}\left[\mathbf{1}_{\tau_{\partial\mathcal{D}_{\epsilon}}^{\epsilon}>r-\epsilon}\mathrm{p}_{\epsilon}^{\epsilon,\,\mathcal{D}_{\epsilon}}((q^{\epsilon}(r-\epsilon),\,p^{\epsilon}(r-\epsilon)),\,x)\right]\;.\label{eq:chapman kolmogorov killed density M_epsilon}
\end{equation}
Besides, using the Gaussian upper-bound shown in~\cite[Theorem 2.19]{LelRamRey}
one has the existence of a constant $C_{1}>0$ independent of $\epsilon$
such that for all $x'',\,x\in\mathcal{D}_{\epsilon}$, 
\[
\mathrm{p}_{\epsilon}^{\epsilon,\,\mathcal{D}_{\epsilon}}(x'',\,x)\leq\frac{C_{1}}{\sqrt{(2\pi)^{2d}\left(\frac{\epsilon^{6}}{12}\phi(\gamma\epsilon)\right)^{d}}}\;,
\]
where $\phi$ is a positive continuous function defined in~\cite[Equation 74]{LelRamRey}
and satisfying $\phi(0)=1$. Therefore, there exists a constant $C_{2}>0$
independent of $\epsilon$ such that for all $\epsilon$ small enough and for all $x'',\,x\in\mathcal{D}_\epsilon$,
\[
\mathrm{p}_{\epsilon}^{\epsilon,\,\mathcal{D}_{\epsilon}}(x'',\,x)\leq\frac{C_{2}}{\epsilon^{3d}}\;.
\]

\noindent Reinjecting into~\eqref{eq:chapman kolmogorov killed density M_epsilon}
ensures that for all $x,\,x'\in\mathcal{D}_{\epsilon}$ and for all
$r>\epsilon$, 
\[
\mathrm{p}_{r}^{\epsilon,\,\mathcal{D}_{\epsilon}}(x',\,x)\leq\frac{C_{2}}{\epsilon^{3d}}\mathbb{P}_{x'}(\tau_{\partial\mathcal{D}_{\epsilon}}^{\epsilon}>r-\epsilon)\;.
\]
Moreover, by Proposition~\ref{prop:exit time probability estimate},
there exist constants $C_{3},\,\alpha,\,\beta>0$ independent of $\epsilon$
such that for all $x,\,x'\in\mathcal{D}_{\epsilon}$, for all $r>\epsilon$,
\[
\mathrm{p}_{r}^{\epsilon,\,\mathcal{D}_{\epsilon}}(x',\,x)\leq\frac{C_{3}}{\epsilon^{3d}}\mathrm{exp}\big(-\alpha(r-\epsilon)\epsilon^{\beta}\big)\;.
\]
Reinjecting
into~\eqref{eq:control exponential moment dgamma} ensures the existence of a constant $C_4>0$ such that 
\begin{align*}
\mathbb{E}_{x}\left[\mathrm{exp}\big(d\gamma\widetilde{\zeta}^{\epsilon}\big)\right] & \leq1+d\gamma\epsilon\mathrm{e}^{d\gamma\epsilon}+\frac{C_{4}}{\epsilon^{3d}}\int_{\epsilon}^{\infty}\mathrm{exp}\big(-\alpha(r-\epsilon)\epsilon^{\beta}\big)\mathrm{d}r\\
 &=1+d\gamma\epsilon\mathrm{e}^{d\gamma\epsilon}+\frac{C_{4}}{\alpha\epsilon^{3d+\beta}}\;,
\end{align*}
which concludes the proof of Proposition~\ref{prop:bound exponential exit time}. 
\end{proof}

\section{Proof of Proposition~\ref{prop:main2}}

\label{sec:Esitmate of dominator}

 This section is devoted to the proof of Proposition~\ref{prop:main2}. The quantity $\mathrm{Cap}_{\epsilon}(\mathcal{M}_{\epsilon},\,\mathcal{S}_{\epsilon})$ therein corresponds to the capacity between the sets $\mathcal{M}_{\epsilon}$ and $\mathcal{S}_{\epsilon}$. The computation of its asymptotic when $\epsilon\rightarrow0$ was done in the elliptic and reversible setting relying on the Dirichlet principle (see~\cite{BEGK,Bovier}) which is hardly tractable in non-reversible settings, see~\cite{LMS}.

In a recent work~\cite{JS}, the authors devised a new approach for the elliptic and non-reversible setting which avoids the use of the Dirichlet principle. It relies on a weak definition of the capacity similar to the one provided in Proposition~\eqref{prop:cap}. By using appropriate test functions, the authors are then able to obtain the asymptotic of the capacity when $\epsilon\rightarrow0$. However, their approach does not work in the current non-elliptic setting given that the computations therein rely on boundary regularity assumptions for the equilibrium potential function and also prior capacity bounds that are not available to us in the current non-elliptic setting. In this section, we adapt the reasoning in~\cite{JS} in a way that cleverly avoids these difficulties.

For simplicity in this section we will assume $\sigma=0$ and we shall
denote by $\mathbb{H}_{U}$ (resp. $\mathbb{H}_{V}$) the Hessian
matrix of $U$ (resp. $V$) at $\sigma$ (resp. at $(\sigma,\,0)\in\mathbb{R}^{2d}$).
In particular, by definition of $V$~\eqref{eq:ham}, one has that 
\begin{equation}
\mathbb{H}_{V}=\begin{pmatrix}\mathbb{H}_{U} & \mathbb{O}_{d}\\
\mathbb{O}_{d} & \mathbb{I}_{d}
\end{pmatrix}\in\mathbb{R}^{2d\times2d}\;.\label{eq:expression H_V}
\end{equation}

For $x=(q,\,p)\in\mathbb{R}^{2d}$, we shall use here the following
notation: 
\begin{equation}
(x_{1},\ldots,x_{d},\,x_{d+1},\;\ldots,x_{2d})=(q_{1},\ldots,q_{d},\,p_{1},\;\ldots,p_{d})\label{eq:notation x (q,p)}
\end{equation}

For $q,\,q'\in\mathbb{R}^{d}$, we shall also denote by 
\[
C^{q,\,q'}:=\left\{ \phi\in C([0,\,1],\mathbb{R}^{2d}):\phi(0)=(q,\,0)\text{ and }\phi(1)=(q',\,0)\right\} \;.
\]

\begin{rem}\label{rem:path across saddle point}
By Assumption~\ref{ass:U}, one has that
\[
\inf_{\phi\in C^{m,\,s}}\sup_{t\in[0,\,1]}V(\phi(t))=U(\sigma)\;.
\]
Besides, for any $\phi\in C^{m,\,s}$, 
\[
\sup_{t\in[0,\,1]}V(\phi(t))=U(\sigma)\Longrightarrow\quad\exists t_{0}\in[0,\,1]\text{ such that }\phi(t_{0})=(\sigma,\,0)\;.
\]
\end{rem}

\subsection{Spectrum at the saddle point}

Since $U$ is a Morse function on $\R^d$, $\mathbb{H}_{U}$ admits one negative
eigenvalue $-\lambda_{1}<0$ and positive eigenvalues $\lambda_{2},\ldots,\lambda_{d}>0$.
Therefore, by~\eqref{eq:expression H_V}, $\mathbb{H}_{V}$ admits
the following $2d$ eigenvalues 
\[
-\lambda_{1},\,\lambda_{2},\ldots,\;\lambda_{d},\,\lambda_{d+1},\;\ldots,\lambda_{2d}\;,
\]
where for all $i\in\llbracket d+1,\,2d\rrbracket$, $\lambda_{i}=1$.
We denote by $(e_{1},\ldots,e_{2d})$ an orthonormal diagonalisation
basis of $\mathbb{H}_{V}$.
\begin{lem}
\label{lem:eigenvalue mu} Let 
\begin{equation}
\mathbb{M}=\begin{pmatrix}\mathbb{O}_{d} & \mathbb{I}_{d}\\
-\mathbb{I}_{d} & \gamma\mathbb{I}_{d}
\end{pmatrix}\in\mathbb{R}^{2d\times2d}\;.\label{eq:def matrix M}
\end{equation}
Then, the matrix $\mathbb{H}_{V}\mathbb{M}$ admits a unique negative
eigenvalue $-\mu$. Furthermore, 
\begin{equation}
\mu=\frac{-\gamma+\sqrt{\gamma^{2}+4\lambda_{1}}}{2}\;.\label{eq:expression mu}
\end{equation}
Besides, the eigenvalue $\mu$ is associated to a unique eigenvector
$v=(v_{q},\,v_{p})\in\mathbb{R}^{d}\times\mathbb{R}^{d}$ satisfying
$|v_{p}|=1$ and $v_{q}=(\mu+\gamma)v_{p}$. 
\end{lem}

\begin{proof}
Let $-\mu<0$ be a negative eigenvalue of $\mathbb{H}_{V}\mathbb{M}$
with eigenvector $v:=(v_{q},\,v_{p})\in\mathbb{R}^{2d}$ then 
\[
\mathbb{H}_{U}v_{p}=-\mu v_{q}=-\mu(\mu+\gamma)v_{p}\;.
\]
In particular, this ensures that $v_{p}\neq0$ because $v=(v_{q},\,v_{p})\neq0$.
Since $\mathbb{H}_{U}$ admits a unique negative eigenvalue $-\lambda_{1}$,
one has that $\lambda_{1}=\mu(\mu+\gamma)$ which ensures~\eqref{eq:expression mu}.
Moreover, since $v_{p}\neq0$, we can rescale the eigenvector $v$
such that it satisfies $|v_{p}|=1$. 
\end{proof}
\begin{lem}
\label{lem:matrix equality} Let $v$ be the eigenvector defined in
Lemma~\ref{lem:eigenvalue mu}, then 
\[
-\frac{v_{1}^{2}}{\lambda_{1}}+\sum_{i=2}^{2d}\frac{v_{i}^{2}}{\lambda_{i}}=-\frac{\gamma}{\mu}<0\;,
\]
where $v_{i}=\langle v,\,e_{i}\rangle$. 
\end{lem}

\begin{proof}
It is equivalent to showing that 
\[
-\langle v,\,\mathbb{H}_{V}^{-1}v\rangle=\frac{\gamma}{\mu}\;.
\]
Moreover, 
\[
-\langle v,\,\mathbb{H}_{V}^{-1}v\rangle=-\langle v,\,\mathbb{M}(\mathbb{H}_{V}\mathbb{M})^{-1}v\rangle=\frac{1}{\mu}\langle v,\,\mathbb{M}v\rangle=\frac{\gamma}{\mu}|v_{p}|^{2}=\frac{\gamma}{\mu}\;.
\]
\end{proof}

\subsection{Partitioning the phase space}

In order to properly define the test function $f_{\epsilon}$ used in~\eqref{eq:cap} for the computation of the capacity $\mathrm{Cap}_{\epsilon}(\mathcal{M}_{\epsilon},\,\mathcal{S}_{\epsilon})$,
we shall partition the phase space and describe the behaviour of $f_{\epsilon}$ in each subspace.

First, let $K>0$ be a large constant independent of $\epsilon$, which shall be precised later. Let also $\delta=\delta(\epsilon)=\sqrt{\epsilon\log(1/\epsilon)}$ and define the box $\mathcal{K}_{\epsilon}$ centered around $(\sigma,\,0)$
\begin{equation}
\mathcal{K}_{\epsilon}:=\left[-\frac{K\delta}{\sqrt{\lambda_{1}}},\,\frac{K\delta}{\sqrt{\lambda_{1}}}\right]\times\prod_{i=2}^{2d}\left[-\frac{2K\delta}{\sqrt{\lambda_{i}}},\,\frac{2K\delta}{\sqrt{\lambda_{i}}}\right]\;.\label{eq:def box K}
\end{equation}
Let also 
\[
\partial\mathcal{K}_{\epsilon}^{\pm}=\partial\mathcal{K}_{\epsilon}\cap\left\{ x_{1}=\pm\frac{K\delta}{\sqrt{\lambda_{1}}}\right\}\;.
\]

\begin{lem}
\label{lem:energy upper boundary box} For all $x\in\partial\mathcal{K}_{\epsilon}\setminus(\partial\mathcal{K}_{\epsilon}^{+}\cup\partial\mathcal{K}_{\epsilon}^{-})$
and for $\epsilon$ small enough, 
\[
V(x)\geq U(\sigma)+\frac{5}{4}K^{2}\delta^{2}\;.
\]
\end{lem}

\begin{proof}
By Taylor's expansion, for all $x\in\partial\mathcal{K}_{\epsilon}\setminus(\partial\mathcal{K}_{\epsilon}^{+}\cup\partial\mathcal{K}_{\epsilon}^{-})$,
\begin{align*}
V(x) & =U(\sigma)-\frac{1}{2}\lambda_{1}|x_{1}|^{2}+\frac{1}{2}\sum_{i=2}^{2d}\lambda_{i}|x_{i}|^{2}+O(|x|^{3})\\
 & \geq U(\sigma)-\frac{1}{2}K^{2}\delta^{2}+2K^{2}\delta^{2}+O(\delta^{3})\;.
\end{align*}
Therefore, for $\epsilon$ small enough, 
\[
V(x)\geq U(\sigma)+\frac{5}{4}K^{2}\delta^{2}\;.
\]
\end{proof}
Let us now consider the following set 
\[
\mathcal{J}_{\epsilon}:=\{x\in\mathbb{R}^{2d}:V(x)<U(\sigma)+K^{2}\delta^{2}\}\;.
\]
Let us also denote by $\mathcal{J}_{\epsilon}^{m}$ (resp. $\mathcal{J}_{\epsilon}^{s}$)
the connected component of the set $\mathcal{J}_{\epsilon}\setminus\mathcal{K}_{\epsilon}$
containing the point $(m,\,0)$ (resp. $(s,\,0)$). The goal of this
subsection is to show that $\mathcal{J}_{\epsilon}^{m},\,\mathcal{J}_{\epsilon}^{s}$ and $\mathcal{K}_{\epsilon}\cap\mathcal{J}_{\epsilon}$ constitute a partition of $\mathcal{J}_{\epsilon}$.

Analogously to~\eqref{eq:def box K}, let us define the box below
for some $\beta>0$ independent of $\epsilon$,
\[
\mathcal{K}_{\beta}:=\left(-\frac{K\beta}{\sqrt{\lambda_{1}}},\,\frac{K\beta}{\sqrt{\lambda_{1}}}\right)\times\prod_{i=2}^{2d}\left(-\frac{2K\beta}{\sqrt{\lambda_{i}}},\,\frac{2K\beta}{\sqrt{\lambda_{i}}}\right)\;.
\]
Notice that $\mathcal{K}_{\epsilon}\subset\mathcal{K}_{\beta}$ for
$\epsilon$ small enough.
\begin{lem}
\label{lem:max energy path away from K beta} For all $\beta>0$, there
exists a constant $c_{\beta}>0$ such that for any $\phi\in C^{m,\,s}$ satisfying $\phi(t)\notin\mathcal{K}_{\beta}$ for all $t\in[0,\,1]$,
then 
\begin{equation}
\sup_{t\in[0,\,1]}V(\phi(t))>U(\sigma)+c_{\beta}\;.\label{eq:sup energy trajectory outside K beta}
\end{equation}
\end{lem}

\begin{proof}
If the assertion above is not true then by covering the bounded domain $\mathcal{J}_{\epsilon}\setminus\mathcal{K}_{\beta}$
with balls of radius $1/n$ and drawing a polynomial trajectory between
them one can build a set of $M$-Lipschitz continuous functions $\phi_{n}\in C^{m,\,s}$,
for some $M>0$ independent of $n\geq1$, such that for all $n\geq1$,
\[
\sup_{t\in[0,\,1]}V(\phi_{n}(t))\leq U(\sigma)+\frac{1}{n}
\]
and for all $t\in[0,\,1]$, $\phi_{n}(t)\notin\mathcal{K}_{\beta}$.
Moreover, the $\phi_{n}$ are bounded independently of $n$ since
they are all $M$-Lipschitz continuous on $[0,\,1]$. As a result,
by Arzela-Ascoli's theorem, up to taking an appropriate subsequence, one has the existence of $\phi\in C^{m,\,s}$
such that 
\[
\sup_{t\in[0,\,1]}|\phi_{n}(t)-\phi(t)|\underset{n\rightarrow\infty}{\longrightarrow}0\;.
\]
In particular, this ensures that for all $t\in[0,\,1]$, $\phi(t)\notin\mathcal{K}_{\beta}$.
Also, since $V$ is $C^{1}$ on $\mathbb{R}^{2d}$ and the functions
$\phi_{n}$ are uniformly bounded on $[0,\,1]$, one has the existence of
a constant $C>0$ independent of $n\geq1$ such that for all $t\in[0,\,1]$,
\[
|V(\phi_{n}(t))-V(\phi(t))|\leq C\sup_{t\in[0,\,1]}|\phi_{n}(t)-\phi(t)|\;.
\]
As a result, for all $t\in[0,\,1]$, 
\begin{align*}
V(\phi(t)) & \leq|V(\phi(t))-V(\phi_{n}(t))|+\sup_{t\in[0,\,1]}V(\phi_{n}(t))\\
 & \leq C\sup_{t\in[0,\,1]}|\phi_{n}(t)-\phi(t)|+U(\sigma)+\frac{1}{n}\underset{n\rightarrow\infty}{\longrightarrow}U(\sigma)\;.
\end{align*}
Therefore, $\sup_{t\in[0,\,1]}V(\phi(t))\leq U(\sigma)$. By Remark~\ref{rem:path across saddle point},
necessarily, 
\[
\sup_{t\in[0,\,1]}V(\phi(t))=U(\sigma)\;.
\]
Also, by Remark~\ref{rem:path across saddle point}, the path
$\phi$ necessarily attains $(\sigma,\,0)$ at a certain time which
is in contradiction with the fact that for all $t\in[0,\,1]$, $\phi(t)\notin\mathcal{K}_{\beta}$.
Hence the validity of Lemma~\ref{lem:max energy path away from K beta}. 
\end{proof}
In order to prove that $\mathcal{J}_{\epsilon}^{m},\mathcal{J}_{\epsilon}^{s}$
and $\mathcal{K}_{\epsilon}\cap\mathcal{J}_{\epsilon}$ constitute a partition of $\mathcal{J}_{\epsilon}$,
we first show the following disjointness result. 
\begin{lem}\label{lem:disjoint}
For $\epsilon$ sufficiently small, $\mathcal{J}_{\epsilon}^{m}\cap\mathcal{J}_{\epsilon}^{s}=\emptyset$. 
\end{lem}

\begin{proof}
Assume that there exists $x\in\mathcal{J}_{\epsilon}^{m}\cap\mathcal{J}_{\epsilon}^{s}$
then one can construct a path $\phi\in C^{m,\,s}$ with
values in $\mathcal{J}_{\epsilon}\setminus\mathcal{K}_{\epsilon}$. We fix $\beta>0$ independent of $\epsilon$. Necessarily $\phi$ enters $\mathcal{K}_{\beta}$, otherwise, by
Lemma~\ref{lem:max energy path away from K beta}, there exists a
constant $c_{\beta}>0$ such that 
\[
\sup_{t\in[0,\,1]}V(\phi(t))>U(\sigma)+c_{\beta}\;,
\]
which leads to a contradiction by taking $\epsilon$ small enough such that $K^{2}\delta^{2}<c_{\beta}$.

As a result, $\phi$ enters the set $\mathcal{K}_{\beta}$
at a time $T_{1}>0$ such that for all $t\leq T_{1}$, $\phi(t)\in\mathcal{J}_{\epsilon}^{m}$.
Following the same computation done in Lemma~\ref{lem:energy upper boundary box} for $\beta$ small, necessarily $\phi$ enters $\mathcal{K}_{\beta}$
through its boundary $\partial\mathcal{K}_{\beta}\cap\{x_{1}=\pm K\beta/\sqrt{\lambda_{1}}\}$
since the values of $\phi$ are assumed to be in $\mathcal{J}_{\epsilon}$.
Therefore, up to changing the sign of $e_{1}$ we can assume that
$\phi$ enters through the boundary $\partial\mathcal{K}_{\beta}\cap\{x_{1}=K\beta/\sqrt{\lambda_{1}}\}$.
Furthermore, $\phi$ never exits the set $\mathcal{K}_{\beta}$ through
its boundary $\partial\mathcal{K}_{\beta}\cap\{x_{1}=-K\beta/\sqrt{\lambda_{1}}\}$.
Otherwise, using Lemma~\ref{lem:energy upper boundary box} again, necessarily $\phi$ enters $\mathcal{K}_{\epsilon}$
which is in contradiction with the definition of $\phi$.

Let 
\[
T_{2}:=\inf\{T\in[0,\,1]:\phi(t)\notin\mathcal{K}_{\beta}\text{ for all }t\in(T,\,1]\}\;.
\]
We define $g$ as a continuous path on $[0,T_2]$ with values in $\mathcal{J}_{\epsilon}^{m}\setminus\mathcal{K}_{\beta}$ such that $g(0)=(m,\,0)$ and $g(T_2)=\phi(T_2)$. Let us now define 
\begin{align*}
\psi(t)=\begin{cases} 
g(t), & t\in[0,T_2]\\
\phi(t), & t\in[T_2,1]\;.
\end{cases}
\end{align*}
Then, $\psi\in\mathcal{C}^{m,\,s}$ and by construction,
$\psi(t)\notin\mathcal{K}_{\beta}$ for all $t\in[0,1]$, which leads to a contradiction by Lemma~\ref{lem:max energy path away from K beta} for $\epsilon$ sufficiently small. 
\end{proof}
\begin{lem}
\label{lem:partitioning} For $\epsilon$ sufficiently small, the
sets $\mathcal{J}_{\epsilon}^{m}$, $\mathcal{J}_{\epsilon}^{s}$
and $\mathcal{K}_{\epsilon}\cap\mathcal{J}_{\epsilon}$ constitute
a partition of $\mathcal{J}_{\epsilon}$. 
\end{lem}

\begin{proof}
The proof is now immediate from Lemma~\ref{lem:disjoint}.
\end{proof}

\subsection{Definition of the test function}

The idea of the proof of Proposition~\ref{prop:main2} is to consider
a suitable test function $f_{\epsilon}$ in~\eqref{eq:cap} for the computation of the
capacity $\mathrm{Cap}_{\epsilon}(\mathcal{M}_{\epsilon},\,\mathcal{S}_{\epsilon})$. The test function $f_{\epsilon}$ is such that $\mathcal{L}_{\epsilon}f_{\epsilon}\simeq0$ on the neighbourhood
$\mathcal{K}_{\epsilon}\cap\mathcal{J}_{\epsilon}$ of the saddle
point $(\sigma,\,0)$. The capacity $\mathrm{Cap}_{\epsilon}(\mathcal{M}_{\epsilon},\,\mathcal{S}_{\epsilon})$
is then shown to behave in its highest order as an integral over the
boundary $\partial\mathcal{K}_{\epsilon}$ for which we are able to
compute the asymptotic when $\epsilon$ goes to zero.

When $x=(q,\,p)\in\mathcal{K}_{\epsilon}\cap\mathcal{J}_{\epsilon}$,
the operator $\mathcal{L}_{\epsilon}$ defined in~\eqref{eq:gen}
behaves in the first order as the following operator: 
\begin{equation}
\widetilde{\mathcal{L}}_{\epsilon}=\langle p,\,\nabla_{q}\rangle-\langle\mathbb{H}_{U}q,\,\nabla_{p}\rangle-\gamma\langle p,\,\nabla_{p}\rangle+\gamma\epsilon\Delta_{p}\;.\label{eq:approx gen}
\end{equation}
Therefore, we consider the function 
\begin{equation}
j_{\epsilon}(x)=\sqrt{\frac{\mu}{2\pi\gamma\epsilon}}\int_{-\infty}^{\langle x,\,v\rangle}\exp\left(-\frac{\mu}{2\gamma\epsilon}t^{2}\right)\mathrm{d}t\;,\label{eq:definition j epsilon}
\end{equation}
where $\mu$ is defined in Lemma~\ref{lem:eigenvalue mu}. An easy computation shows that for all $x\in\mathbb{R}^{2d}$, 
\begin{equation}
\widetilde{\mathcal{L}}_{\epsilon}j_{\epsilon}(x)=\sqrt{\frac{\mu}{2\pi\gamma\epsilon}}\exp\left(-\frac{\mu}{2\gamma\epsilon}\langle x,\,v\rangle^{2}\right)\left[-\langle\mathbb{H}_{V}\mathbb{M}v,\,x\rangle-\mu\langle x,\,v\rangle\right]=0\;.\label{eq:j epsilon harmonic}
\end{equation}

The idea is then to define the test function $f_{\epsilon}=j_{\epsilon}$
on $\mathcal{K}_{\epsilon}\cap\mathcal{J}_{\epsilon}$ and extend
it smoothly on $\mathcal{J}_{\epsilon}$ such that $f_{\epsilon}=1$
on $\partial\mathcal{M}_{\epsilon}$ and $f_{\epsilon}=0$ on $\partial\mathcal{S}_{\epsilon}$
using the partitioning of Lemma~\ref{lem:partitioning}. In order to do that, let us take $\varphi\in C^{\infty}(\mathbb{R},\,[0,\,1])$
satisfying 
\begin{align*}
\varphi(t)=\begin{cases}
0, & t\leq0\\
1, & t\geq1\;.
\end{cases}
\end{align*}
In particular, the smoothness of $\varphi$ in $\mathbb{R}$ implies
that for all $n\geq1$, 
\begin{equation}
\frac{\mathrm{d}^{n}\varphi}{\mathrm{d}t^{n}}(0)=\frac{\mathrm{d}^{n}\varphi}{\mathrm{d}t^{n}}(1)=0\;.\label{eq:derivatives varphi}
\end{equation}
For $\theta>0$, we shall denote by $\varphi_{\theta}$ the function
defined for $x\in\mathbb{R}^{2d}$ as follows: 
\begin{equation}
\varphi_{\theta}(x)=\varphi\left(\frac{|x_{1}|-K\delta/\sqrt{\lambda_{1}}}{\theta}\right)\;.\label{eq:def varphi theta}
\end{equation}

Furthermore, in order to extend the function $f_{\epsilon}$ smoothly
outside of $\mathcal{J}_{\epsilon}$, we shall need the following lemma. 
\begin{lem}
\label{lem:g_epsilon definition} There exists a smooth function $g_{\epsilon}\in C^{\infty}(\mathbb{R}^{2d},\,[0,\,1])$
satisfying 
\begin{align*}
g_{\epsilon}(x)=\begin{cases}
1, & V(x)\leq U(\sigma)+\frac{K^{2}\delta^{2}}{2}\\
0, & V(x)\geq U(\sigma)+K^{2}\delta^{2}\;.
\end{cases}
\end{align*}
Besides, there exists a constant $C>0$ independent of $\epsilon$
such that 
\begin{equation}
\sup_{x\in\mathbb{R}^{2d}}\left(|\nabla g_{\epsilon}(x)|+|\nabla^{2}g_{\epsilon}(x)|\right)\leq\frac{C}{\delta^{6}}\;.\label{eq:derivative controls g}
\end{equation}
\end{lem}

\begin{proof}
Let $\Psi\in C_{c}^{\infty}(\mathbb{R}^{2d},\,\mathbb{R})$ be a non-negative
function with compact support in $\mathrm{B}(0,\,1)$ and integrating
to~$1$. For $x\in\mathbb{R}^{2d}$, let us define 
\begin{equation}
g_{\epsilon}(x)=\frac{1}{\delta^{6d}}\int_{\mathbb{R}^{2d}}\mathbf{1}_{V(x-y)\leq U(\sigma)+2K^{2}\delta^{2}/3}\Psi(y/\delta^{3})\mathrm{d}y\;.\label{eq:def g epsilon}
\end{equation}

As a mollified function it is clear that $g_{\epsilon}\in C^{\infty}(\mathbb{R}^{2d},\,[0,\,1])$.
Also, using the change of variable $z=(x-y)/\delta^3$ and using the fact that
$\Psi$ admits a compact support, one easily shows the derivative
controls satisfied by $g_{\epsilon}$ in~\eqref{eq:derivative controls g}.

Additionally, using Assumption~\ref{ass:nabla U} and following the proof of Proposition~\ref{prop:tight},
one shows that $V(x)\underset{|x|\rightarrow\infty}{\longrightarrow}\infty$.
In particular, there exists $M>0$ such that for all $|x|\geq M$,
$V(x)>U(\sigma)+2K^{2}\delta^{2}/3$. Therefore, for $\epsilon$ small
enough, if $|x|\geq M+1$ and $|y|\leq\delta^{3}$, then 
\[
V(x-y)>U(\sigma)+2K^{2}\delta^{2}/3\;,
\]
hence $g_{\epsilon}(x)=0$ for $|x|\geq M+1$ since $\Psi$ has support in the ball $\mathrm{B}(0,\,1)$.
Moreover, since $V$ is $C^{1}$ on $\mathbb{R}^{2d}$ it
is $C$-Lipschitz continuous on the ball $\mathrm{B}(0,\,M+2)$ for
some constant $C>0$.

Let us now take $x\in\mathrm{B}(0,\,M+1)$ such that $V(x)\leq U(\sigma)+\frac{K^{2}\delta^{2}}{2}$.
Then for $\epsilon$ small enough and $|y|\leq\delta^{3}$ one has
that 
\begin{align*}
V(x-y) & \leq|V(x-y)-V(x)|+V(x)\\
 & \leq C\delta^{3}+U(\sigma)+\frac{K^{2}\delta^{2}}{2}\\
 & \leq U(\sigma)+\frac{2K^{2}\delta^{2}}{3}\;.
\end{align*}
Therefore, reinjecting into~\eqref{eq:def g epsilon} ensures that
$g_{\epsilon}(x)=1$. Consider now the case $V(x)\geq U(\sigma)+K^{2}\delta^{2}$.
Then similarly, 
\begin{align*}
V(x-y) & \geq V(x)-|V(x-y)-V(x)|\\
 & \geq U(\sigma)+K^{2}\delta^{2}-C\delta^{3}\\
 & >U(\sigma)+\frac{2K^{2}\delta^{2}}{3}
\end{align*}
for $\epsilon$ small enough, hence $g_{\epsilon}(x)=0$ which concludes
the proof. 
\end{proof}
Let us conclude this subsection with the expression of the test function
$f_{\epsilon}$. In order to do that let us define the boxes 
\[
\mathcal{K}_{\epsilon,\,\theta}^{m}=\left[\frac{K\delta}{\sqrt{\lambda_{1}}},\,\frac{K\delta}{\sqrt{\lambda_{1}}}+\theta\right]\times\prod_{i=2}^{2d}\left[-\frac{2K\delta}{\sqrt{\lambda_{i}}},\,\frac{2K\delta}{\sqrt{\lambda_{i}}}\right]
\]
and 
\[
\mathcal{K}_{\epsilon,\,\theta}^{s}=\left[-\frac{K\delta}{\sqrt{\lambda_{1}}}-\theta,\,-\frac{K\delta}{\sqrt{\lambda_{1}}}\right]\times\prod_{i=2}^{2d}\left[-\frac{2K\delta}{\sqrt{\lambda_{i}}},\,\frac{2K\delta}{\sqrt{\lambda_{i}}}\right]\;.
\]
Now let 
\begin{align*}
f_{\epsilon,\,\theta}(x)=\begin{cases}
g_{\epsilon}(x)\left[(1-\varphi_{\theta}(x))j_{\epsilon}(x)+\varphi_{\theta}(x)\right], & x\in\mathcal{K}_{\epsilon,\theta}^{m}\cap\mathcal{J}_{\epsilon}\\
g_{\epsilon}(x)j_{\epsilon}(x), & x\in\mathcal{K}_{\epsilon}\cap\mathcal{J}_{\epsilon}\\
g_{\epsilon}(x)(1-\varphi_{\theta}(x))j_{\epsilon}(x), & x\in\mathcal{K}_{\epsilon,\theta}^{s}\cap\mathcal{J}_{\epsilon}\\
1, & x\in\mathcal{J}_{\epsilon}^{m}\\
0, & x\in\mathcal{J}_{\epsilon}^{s}\cup(\R^{2d}\setminus\mathcal{J}_{\epsilon})\;.
\end{cases}
\end{align*}

\begin{lem}
The function $f_{\epsilon,\,\theta}\in C_{c}^{\infty}(\overline{\mathcal{S}_{\epsilon}}^{c})$
and satisfies $f_{\epsilon,\,\theta}(x)=1$ for $x\in\partial\mathcal{M}_{\epsilon}$. 
\end{lem}

\begin{proof}
The smoothness of $f_{\epsilon,\,\theta}$ follows from the partitioning
of $\mathcal{J}_{\epsilon}$ in Lemma~\ref{lem:partitioning}, the
vanishing of the derivatives of $\varphi_{\theta}$ using~\eqref{eq:derivatives varphi}
and the smoothness of $g_{\epsilon}$ defined in Lemma~\ref{lem:g_epsilon definition}. 
\end{proof}

\subsection{Computation of $\mathrm{Cap}_{\epsilon}(\mathcal{M}_{\epsilon},\,\mathcal{S}_{\epsilon})$}

We show here that the capacity $\mathrm{Cap}_{\epsilon}(\mathcal{M}_{\epsilon},\,\mathcal{S}_{\epsilon})$
in~\eqref{eq:cap} can be computed
by choosing the test function $f_{\epsilon,\,\theta}$ and taking the limit of the integral when $\theta\rightarrow0$. The rest of this section is thus devoted to the proof of the theorem below which directly implies Proposition~\ref{prop:main2}.

Let 
\begin{equation}
\alpha_{\epsilon}=\frac{1}{Z_{\epsilon}}\frac{(2\pi\epsilon)^{d}}{2\pi}\frac{\mu}{\sqrt{|\det{\mathbb{H}_{U}}|}}e^{-U(\sigma)/\epsilon}\;,\label{eq:def alpha epsilon}
\end{equation}
where $\mu$ is defined in Lemma~\ref{lem:eigenvalue mu}. 
\begin{thm}
\label{thm:result prop 3.7} 
\[
\lim_{\theta\rightarrow0}\int_{\mathbb{R}^{2d}}h_{\mathcal{M}_{\epsilon},\,\mathcal{S}_{\epsilon}}^{*}(x)(-\mathcal{L}_{\epsilon}f_{\epsilon,\,\theta}(x))\mu_{\epsilon}(x)\mathrm{d}x=[1+o_{\epsilon}(1)]\alpha_{\epsilon}\;.
\]
\end{thm}

The proof of this theorem relies on the following propositions. 
\begin{prop}
\label{prop:control Lj_epsilon} 
\[
\int_{\mathcal{K}_{\epsilon}\cap\mathcal{J}_{\epsilon}}|\mathcal{L}_{\epsilon}j_{\epsilon}(x)|\mu_{\epsilon}(x)\mathrm{d}x=o_{\epsilon}(1)\alpha_{\epsilon}\;.
\]
\end{prop}

For any vector $x=(x_{1},\ldots,x_{2d})\in\mathbb{R}^{2d}$, let us
denote the vector 
\begin{equation}
\widetilde{x}=(x_{2},\ldots,x_{2d})\in\mathbb{R}^{2d-1}\label{eq:notation x tilde}
\end{equation}
such that for any $i\in\llbracket1,\,2d-1\rrbracket$, $\widetilde{x}_{i}=x_{i+1}$.
Also, let 
\begin{equation}
\mathcal{Q}_{\epsilon}^{\pm}:=\left\{ \widetilde{x}\in\prod_{i=2}^{2d}\left[-\frac{2K\delta}{\sqrt{\lambda_{i}}},\,\frac{2K\delta}{\sqrt{\lambda_{i}}}\right]:V\left(\pm\frac{K\delta}{\sqrt{\lambda_{1}}},\,\widetilde{x}\right)<U(\sigma)+K^{2}\delta^{2}\right\} \;.\label{eq:def Q epsilon}
\end{equation}

\begin{prop}
\label{prop:asymptotics boundary integral} One has that 
\begin{equation}
\frac{1}{Z_{\epsilon}}\int_{\mathcal{Q}_{\epsilon}^{+}}h_{\mathcal{M}_{\epsilon},\,\mathcal{S}_{\epsilon}}^{*}\left(\frac{K\delta}{\sqrt{\lambda_{1}}},\,\widetilde{x}\right)(-\widetilde{x}_{d})\left(1-j_{\epsilon}\left(\frac{K\delta}{\sqrt{\lambda_{1}}},\,\widetilde{x}\right)\right)\mathrm{exp}\left(-\frac{V\left(\frac{K\delta}{\sqrt{\lambda_{1}}},\,\widetilde{x}\right)}{\epsilon}\right)\mathrm{d}\widetilde{x}=[1+o_{\epsilon}(1)]\alpha_{\epsilon}\;.\label{eq:asymptotic expansion boundary K+}
\end{equation}
Also, 
\begin{equation}
\frac{1}{Z_{\epsilon}}\int_{\mathcal{Q}_{\epsilon}^{-}}h_{\mathcal{M}_{\epsilon},\,\mathcal{S}_{\epsilon}}^{*}\left(-\frac{K\delta}{\sqrt{\lambda_{1}}},\,\widetilde{x}\right)\widetilde{x}_{d}\,j_{\epsilon}\left(-\frac{K\delta}{\sqrt{\lambda_{1}}},\,\widetilde{x}\right)\mathrm{exp}\left(-\frac{V\left(-\frac{K\delta}{\sqrt{\lambda_{1}}},\,\widetilde{x}\right)}{\epsilon}\right)\mathrm{d}\widetilde{x}=o_{\epsilon}(1)\alpha_{\epsilon}\;.\label{eq:asymptotic expansion boundary K-}
\end{equation}
\end{prop}

Let us now prove Theorem~\ref{thm:result prop 3.7}. 
\begin{proof}
First notice that for any $x\in\mathbb{R}^{2d}$ such that $|x_{1}|>K\delta/\sqrt{\lambda_{1}}$,
one has that 
\[
1-\varphi_{\theta}(x)\underset{\theta\rightarrow0}{\longrightarrow}0\;,
\]
by definition of $\varphi_{\theta}$ in~\eqref{eq:def varphi theta}.
Therefore, by definition of $f_{\epsilon,\,\theta}$, one has that
\begin{align*}
 & \int_{\mathbb{R}^{2d}}h_{\mathcal{M}_{\epsilon},\,\mathcal{S}_{\epsilon}}^{*}(x)(-\mathcal{L}_{\epsilon}f_{\epsilon,\,\theta}(x))\mu_{\epsilon}(x)\mathrm{d}x\\
 & =\underbrace{\int_{\mathcal{K}_{\epsilon}\cap\mathcal{J}_{\epsilon}}h_{\mathcal{M}_{\epsilon},\,\mathcal{S}_{\epsilon}}^{*}(x)(-\mathcal{L}_{\epsilon}(g_{\epsilon}j_{\epsilon}))(x)\mu_{\epsilon}(x)\mathrm{d}x}_{=I_{\epsilon}^{1}}\;+\;\underbrace{\int_{\mathcal{K}_{\epsilon,\,\theta}^{m}\cap\mathcal{J}_{\epsilon}}h_{\mathcal{M}_{\epsilon},\,\mathcal{S}_{\epsilon}}^{*}(x)(-\mathcal{L}_{\epsilon}\varphi_{\theta}(x))(1-j_{\epsilon}(x))g_{\epsilon}(x)\mu_{\epsilon}(x)\mathrm{d}x}_{=I_{\epsilon,\,\theta}^{2}}\\
 & +\underbrace{\int_{\mathcal{K}_{\epsilon,\,\theta}^{s}\cap\mathcal{J}_{\epsilon}}h_{\mathcal{M}_{\epsilon},\,\mathcal{S}_{\epsilon}}^{*}(x)\mathcal{L}_{\epsilon}\varphi_{\theta}(x)j_{\epsilon}(x)g_{\epsilon}(x)\mu_{\epsilon}(x)\mathrm{d}x}_{=I_{\epsilon,\,\theta}^{3}}\; + \;o_{\theta}(1)\;.
\end{align*}
Let us first prove that 
\begin{equation}
I_{\epsilon}^{1}=o_{\epsilon}(1)\alpha_{\epsilon}\;.\label{eq:control integral I_1}
\end{equation}
Notice that by the definition~\eqref{eq:definition j epsilon}, $j_{\epsilon}(x)\in[0,\,1]$
for all $x\in\mathbb{R}^{2d}$. Moreover, there exists a constant
$C_{1}>0$ independent of $\epsilon$ such that 
\[
\sup_{x\in\mathbb{R}^{2d}}|\nabla j_{\epsilon}(x)|\leq\frac{C_{1}}{\sqrt{\epsilon}}\;.
\]
Additionally, by Lemma~\ref{lem:g_epsilon definition}, $g_{\epsilon}(x)\in[0,\,1]$ for all $x\in\mathbb{R}^{2d}$
and there exists a constant $C_{2}>0$ independent of $\epsilon$
such that 
\begin{equation}
\sup_{x\in\mathbb{R}^{2d}}\left(|\nabla g_{\epsilon}(x)|+|\nabla^{2}g_{\epsilon}(x)|\right)\leq\frac{C_{2}}{\delta^{6}}\;.\label{eq:control derivatives g epsilon}
\end{equation}
Besides, for $x\in\mathbb{R}^{2d}$, 
\[
\mathcal{L}_{\epsilon}(g_{\epsilon}j_{\epsilon})(x)=g_{\epsilon}(x)\mathcal{L}_{\epsilon}j_{\epsilon}(x)+j_{\epsilon}(x)\mathcal{L}_{\epsilon}g_{\epsilon}(x)+2\gamma\epsilon\langle\nabla_{p}j_{\epsilon}(x),\,\nabla_{p}g_{\epsilon}(x)\rangle\;.
\]
As a result, since $h_{\mathcal{M}_{\epsilon},\,\mathcal{S}_{\epsilon}}^{*}(x)\in[0,\,1]$
by~\eqref{eq:def h^*}, 
\[
|I_{\epsilon}^{1}|\leq\int_{\mathcal{K}_{\epsilon}\cap\mathcal{J}_{\epsilon}}|\mathcal{L}_{\epsilon}j_{\epsilon}(x)|\mu_{\epsilon}(x)\mathrm{d}x+\int_{\mathcal{K}_{\epsilon}\cap\mathcal{J}_{\epsilon}}|\mathcal{L}_{\epsilon}g_{\epsilon}(x)|\mu_{\epsilon}(x)\mathrm{d}x+2\gamma C_{1}\sqrt{\epsilon}\int_{\mathcal{K}_{\epsilon}\cap\mathcal{J}_{\epsilon}}|\nabla g_{\epsilon}(x)|\mu_{\epsilon}(x)\mathrm{d}x\;.
\]
Proposition~\ref{prop:control Lj_epsilon} ensures that the first
integral in the right-hand side of the inequality above is of order
$o_{\epsilon}(1)\alpha_{\epsilon}$. By Lemma~\ref{lem:g_epsilon definition}, $\nabla g_{\epsilon}(x)=0$
whenever $V(x)\leq U(\sigma)+K^{2}\delta^{2}/2$. Moreover, when $V(x)>U(\sigma)+K^{2}\delta^{2}/2$,
\[
\mathrm{e}^{-V(x)/\epsilon}\leq\mathrm{e}^{-U(\sigma)/\epsilon}\epsilon^{K^{2}/2}\,,
\]
since $\delta=\sqrt{\epsilon\log(1/\epsilon)}$. Thus, using the gradient control~\eqref{eq:control derivatives g epsilon} and taking the constant $K$ large enough and independent of $\epsilon$, we deduce~\eqref{eq:control integral I_1}.

Let us now show that 
\begin{equation}
\lim_{\theta\rightarrow0}I_{\epsilon,\,\theta}^{2}=[1+o_{\epsilon}(1)]\alpha_{\epsilon}\;.\label{eq:control integral I2}
\end{equation}
By definition of $\varphi_{\theta}$ in~\eqref{eq:def varphi theta}
one has that 
\[
I_{\epsilon,\,\theta}^{2}=\frac{1}{Z_{\epsilon}}\int_{\mathcal{K}_{\epsilon,\,\theta}^{m}\cap\mathcal{J}_{\epsilon}}h_{\mathcal{M}_{\epsilon},\;\mathcal{S}_{\epsilon}}^{*}(x)\frac{-x_{d+1}}{\theta}\frac{\mathrm{d}\varphi}{\mathrm{d}t}\left(\frac{x_{1}-K\delta/\sqrt{\lambda_{1}}}{\theta}\right)(1-j_{\epsilon}(x))g_{\epsilon}(x)\mathrm{e}^{-V(x)/\epsilon}\mathrm{d}x_{1}\ldots\mathrm{d}x_{2d}\;.
\]
Let us consider only the integral with respect to the first variable
$x_{1}$ and let us perform the change of variable $t=(x_{1}-K\delta/\sqrt{\lambda_{1}})/\theta$.
Thus, we obtain using the notation~\eqref{eq:notation x tilde},
\begin{align*}
 & \int_{0}^{1}h_{\mathcal{M}_{\epsilon},\,\mathcal{S}_{\epsilon}}^{*}\left(\frac{K\delta}{\sqrt{\lambda_{1}}}+t\theta,\,\widetilde{x}\right)(-\widetilde{x}_{d})\frac{\mathrm{d}\varphi}{\mathrm{d}t}(t)\left(1-j_{\epsilon}\left(\frac{K\delta}{\sqrt{\lambda_{1}}}+t\theta,\,\widetilde{x}\right)\right)g_{\epsilon}\left(\frac{K\delta}{\sqrt{\lambda_{1}}}+t\theta,\,\widetilde{x}\right)\mathrm{e}^{-V\left(\frac{K\delta}{\sqrt{\lambda_{1}}}+t\theta,\,\widetilde{x}\right)/\epsilon}\mathrm{d}t\\
 & =h_{\mathcal{M}_{\epsilon},\,\mathcal{S}_{\epsilon}}^{*}\left(\frac{K\delta}{\sqrt{\lambda_{1}}},\,\widetilde{x}\right)(-\widetilde{x}_{d})\left(1-j_{\epsilon}\left(\frac{K\delta}{\sqrt{\lambda_{1}}},\,\widetilde{x}\right)\right)g_{\epsilon}\left(\frac{K\delta}{\sqrt{\lambda_{1}}},\,\widetilde{x}\right)\mathrm{e}^{-V\left(\frac{K\delta}{\sqrt{\lambda_{1}}},\,\widetilde{x}\right)/\epsilon}\int_{0}^{1}\frac{\mathrm{d}\varphi}{\mathrm{d}t}(t)\mathrm{d}t+o_{\theta}(1)\;.
\end{align*}
Therefore, since 
\[
\int_{0}^{1}\frac{\mathrm{d}\varphi}{\mathrm{d}t}(t)\mathrm{d}t=\varphi(1)-\varphi(0)=1\;,
\]
one has using the notation~\eqref{eq:def Q epsilon}, 
\begin{align*}
 & \lim_{\theta\rightarrow0}I_{\epsilon,\,\theta}^{2}\\
 & =\int_{\mathcal{Q}_{\epsilon}^{+}}h_{\mathcal{M}_{\epsilon},\,\mathcal{S}_{\epsilon}}^{*}\left(\frac{K\delta}{\sqrt{\lambda_{1}}},\,\widetilde{x}\right)(-\widetilde{x}_{d})\left(1-j_{\epsilon}\left(\frac{K\delta}{\sqrt{\lambda_{1}}},\,\widetilde{x}\right)\right)\mathrm{exp}\left(-\frac{V\left(\frac{K\delta}{\sqrt{\lambda_{1}}},\,\widetilde{x}\right)}{\epsilon}\right)\mathrm{d}\widetilde{x}\\
 & +\int_{\mathcal{Q}_{\epsilon}^{+}}h_{\mathcal{M}_{\epsilon},\,\mathcal{S}_{\epsilon}}^{*}\left(\frac{K\delta}{\sqrt{\lambda_{1}}},\,\widetilde{x}\right)(-\widetilde{x}_{d})\left(1-j_{\epsilon}\left(\frac{K\delta}{\sqrt{\lambda_{1}}},\,\widetilde{x}\right)\right)\left(1-g_{\epsilon}\left(\frac{K\delta}{\sqrt{\lambda_{1}}},\,\widetilde{x}\right)\right)\mathrm{exp}\left(-\frac{V\left(\frac{K\delta}{\sqrt{\lambda_{1}}},\,\widetilde{x}\right)}{\epsilon}\right)\mathrm{d}\widetilde{x}\;.
\end{align*}
Using Proposition~\ref{prop:asymptotics boundary integral}, it remains
to show that the second integral above is of order $o_{\epsilon}(1)\alpha_{\epsilon}$.
By definition of $g_{\epsilon}$ in Lemma~\ref{lem:g_epsilon definition},
if $1-g_{\epsilon}(x)\neq0$ then $V(x)\geq U(\sigma)+K^{2}\delta^{2}/2$.
Therefore, taking again $K$ large enough ensures that the second integral is of order $o_{\epsilon}(1)\alpha_{\epsilon}$ which yields~\eqref{eq:control integral I2}.

Finally, it remains to show that 
\begin{equation}
\lim_{\theta\rightarrow0}I_{\epsilon,\,\theta}^{3}=o_{\epsilon}(1)\alpha_{\epsilon}\;.
\end{equation}
Similarly to the computation of $I_{\epsilon,\,\theta}^{2}$, one has that 
\begin{align*}
\lim_{\theta\rightarrow0}I_{\epsilon,\,\theta}^{3} & =\int_{\mathcal{Q}_{\epsilon}^{-}}h_{\mathcal{M}_{\epsilon},\,\mathcal{S}_{\epsilon}}^{*}\left(-\frac{K\delta}{\sqrt{\lambda_{1}}},\,\widetilde{x}\right)\widetilde{x}_{d}\,j_{\epsilon}\left(-\frac{K\delta}{\sqrt{\lambda_{1}}},\,\widetilde{x}\right)g_{\epsilon}\left(-\frac{K\delta}{\sqrt{\lambda_{1}}},\,\widetilde{x}\right)\mathrm{exp}\left(-\frac{V\left(-\frac{K\delta}{\sqrt{\lambda_{1}}},\,\widetilde{x}\right)}{\epsilon}\right)\mathrm{d}\tilde{x}\\
 & =\int_{\mathcal{Q}_{\epsilon}^{-}}h_{\mathcal{M}_{\epsilon},\,\mathcal{S}_{\epsilon}}^{*}\left(-\frac{K\delta}{\sqrt{\lambda_{1}}},\,\widetilde{x}\right)\widetilde{x}_{d}\,j_{\epsilon}\left(-\frac{K\delta}{\sqrt{\lambda_{1}}},\,\widetilde{x}\right)\mathrm{exp}\left(-\frac{V\left(-\frac{K\delta}{\sqrt{\lambda_{1}}},\,\widetilde{x}\right)}{\epsilon}\right)\mathrm{d}\widetilde{x}+o_{\epsilon}(1)\alpha_{\epsilon}\;.
\end{align*}
Proposition~\ref{prop:asymptotics boundary integral} then concludes
the proof of Theorem~\ref{thm:result prop 3.7}. 
\end{proof}
The rest of this section is devoted to the proof of Propositions~\ref{prop:control Lj_epsilon}
and~\ref{prop:asymptotics boundary integral}. Let us start with
Proposition~\ref{prop:control Lj_epsilon} which relies on the following
lemma. 
\begin{lem}
\label{lem:spectrum H + vv} The matrix $\mathbb{H}_{V}+\frac{\mu}{\gamma}v\otimes v^{\dagger}$
admits one zero eigenvalue of dimension 1 and all other eigenvalues
are positive. 
\end{lem}

\begin{proof}
The proof is exactly similar to the proof of~\cite[Lemma 8.2]{JS}
and mostly relies on Lemma~\ref{lem:matrix equality}. 
\end{proof}
Let us move on now to the proof of Proposition~\ref{prop:control Lj_epsilon} which is similar to the proof of~\cite[Proposition 8.5]{JS}. 
\begin{proof}[Proof of Proposition~\ref{prop:control Lj_epsilon}]
By~\eqref{eq:j epsilon harmonic}, for any $x\in\mathcal{K}_{\epsilon}$,
\begin{align*}
\mathcal{L}_{\epsilon}j_{\epsilon}(x) & =-\sqrt{\frac{\mu}{2\pi\gamma\epsilon}}\exp\left(-\frac{\mu}{2\gamma\epsilon}\langle x,\,v\rangle^{2}\right)\langle\nabla U(q)-\mathbb{H}_{U}q,\,v_p\rangle\\
 & =\frac{\delta^{2}}{\sqrt{\epsilon}}\exp\left(-\frac{\mu}{2\gamma\epsilon}\langle x,\,v\rangle^{2}\right)O_{\epsilon}(1)\;.
\end{align*}
Therefore, 
\begin{align*}
\int_{\mathcal{K}_{\epsilon}\cap\mathcal{J}_{\epsilon}}|\mathcal{L}_{\epsilon}j_{\epsilon}(x)|\mu_{\epsilon}(\mathrm{d}x) & =\frac{1}{Z_{\epsilon}}\frac{\delta^{2}}{\sqrt{\epsilon}}\int_{\mathcal{K}_{\epsilon}\cap\mathcal{J}_{\epsilon}}\exp\left(-\frac{\mu}{2\gamma\epsilon}\langle x,\,v\rangle^{2}\right)\mathrm{e}^{-V(x)/\epsilon}\mathrm{d}x\,O_{\epsilon}(1)\\
 & =\frac{1}{Z_{\epsilon}}\frac{\delta^{2}}{\sqrt{\epsilon}}\,\mathrm{e}^{-U(\sigma)/\epsilon}\int_{\mathcal{K}_{\epsilon}\cap\mathcal{J}_{\epsilon}}\mathrm{exp}\left(-\frac{1}{2\epsilon}\left\langle\left(\mathbb{H}_{V}+\frac{\mu}{\gamma}v\otimes v^{\dagger}\right)x,\;x\right\rangle \right)\mathrm{d}x\,O_{\epsilon}(1)\;.
\end{align*}
By Lemma~\ref{lem:spectrum H + vv}, let $\rho_{1}=0$ and $\rho_{2},\ldots,\rho_{2d}>0$
be the eigenvalues of $\mathbb{H}_{V}+\frac{\mu}{\gamma}v\otimes v^{\dagger}$.
Therefore, there exists $M>0$ independent of $\epsilon$
such that 
\begin{align*}
\int_{\mathcal{K}_{\epsilon}\cap\mathcal{J}_{\epsilon}}|\mathcal{L}_{\epsilon}j_{\epsilon}(x)|\mu_{\epsilon}(\mathrm{d}x) & \leq\frac{1}{Z_{\epsilon}}\frac{\delta^{2}}{\sqrt{\epsilon}}\mathrm{e}^{-U(\sigma)/\epsilon}\int_{\mathbb{R}^{2d-1}}\int_{-M\delta}^{M\delta}\mathrm{exp}\left(-\frac{1}{2\epsilon}\sum_{k=2}^{2d}\rho_{k}y_{k}^{2}\right)\mathrm{d}y_{1}\ldots\mathrm{d}y_{2d}\,O_{\epsilon}(1)\\
 & =\frac{1}{Z_{\epsilon}}\frac{\delta^{3}}{\epsilon}\mathrm{e}^{-U(\sigma)/\epsilon}\epsilon^{d}O_{\epsilon}(1)=o_{\epsilon}(1)\alpha_{\epsilon}\;,
\end{align*}
since $\delta^3/\epsilon=o_{\epsilon}(1)$. 
\end{proof}
In order to prove Proposition~\ref{prop:asymptotics boundary integral}
we need the following lemmas which follow the same ideas as~\cite[Lemma 8.10 and Proposition 8.6]{JS}
\begin{lem}
\label{lem:decomp boundary K+} There exists $a>0$ independent of
$\epsilon$ such that
\begin{equation}
x\in\partial\mathcal{K}_{\epsilon}^{+}\Longrightarrow\langle x,\,v\rangle\geq aK\delta\quad\text{ or }\quad V(x)\geq U(\sigma)+aK^{2}\delta^{2}\;.\label{eq:boundary K+}
\end{equation}
Moreover, 
\begin{equation}
x\in\partial\mathcal{K}_{\epsilon}^{-}\Longrightarrow\langle x,\,v\rangle\leq-aK\delta\quad\text{ or }\quad V(x)\geq U(\sigma)+aK^{2}\delta^{2}\;.\label{eq:boundary K-}
\end{equation}
\end{lem}

\begin{lem}
\label{lem:h close to 1 or 0 boundary} One has that 
\begin{equation}
\int_{\partial\mathcal{K}_{\epsilon}^{+}\cap\mathcal{J}_{\epsilon}}\left(1-h_{\mathcal{M}_{\epsilon},\,\mathcal{S}_{\epsilon}}^{*}\left(x\right)\right)\left(1-j_{\epsilon}\left(x\right)\right)\mu_{\epsilon}(\mathrm{d}x)=o_{\epsilon}(1)\alpha_{\epsilon}\;.\label{eq:first integral boundary}
\end{equation}
Also, 
\begin{equation}
\int_{\partial\mathcal{K}_{\epsilon}^{-}\cap\mathcal{J}_{\epsilon}}h_{\mathcal{M}_{\epsilon},\,\mathcal{S}_{\epsilon}}^{*}\left(x\right)j_{\epsilon}\left(x\right)\mu_{\epsilon}(\mathrm{d}x)=o_{\epsilon}(1)\alpha_{\epsilon}\;.\label{eq:second integral boundary}
\end{equation}
\end{lem}

\begin{proof}[Proof of Lemma~\ref{lem:decomp boundary K+}]
The proof of~\eqref{eq:boundary K+} is identical to the proof done
in~\cite[Lemma 8.10]{JS} and mostly relies on the identity given
in Lemma~\ref{lem:matrix equality}. Additionally, the proof of~\eqref{eq:boundary K-}
can be deduced from~\eqref{eq:boundary K+} by changing the sign
of $x_{1}$ and noticing that the asymptotic expansion of $V(x)$
around the saddle point satisfies 
\[
V(x)=U(\sigma)-\frac{1}{2}\lambda_{1}|x_{1}|^{2}+\frac{1}{2}\sum_{i=2}^{2d}\lambda_{i}|x_{i}|^{2}+O(|x|^{3})\;,
\]
which does not depend on the sign of $x_{1}$ at the highest order. 
\end{proof}
\begin{proof}[Proof of Lemma~\ref{lem:h close to 1 or 0 boundary}]
Let us first prove~\eqref{eq:first integral boundary}. In order
to do that, let us define the following sets: 
\[
C_{\epsilon}^{1}=\{x\in\mathbb{R}^{2d}:\langle x,\,v\rangle\geq aK\delta\},\qquad C_{\epsilon}^{2}=\{x\in\mathbb{R}^{2d}:V(x)\geq U(\sigma)+aK^{2}\delta^{2}\}\;.
\]
Then, by Lemma~\ref{lem:decomp boundary K+}, it is enough to consider
the integral in~\eqref{eq:first integral boundary} on the sets $C_{\epsilon}^{1}\cap\partial\mathcal{K}_{\epsilon}^{+}$
and $C_{\epsilon}^{2}\cap\partial\mathcal{K}_{\epsilon}^{+}$. Let
us start with the latter, then 
\[
\frac{1}{Z_\epsilon}\int_{C_{\epsilon}^{2}\cap\partial\mathcal{K}_{\epsilon}^{+}}\left(1-h_{\mathcal{M}_{\epsilon},\,\mathcal{S}_{\epsilon}}^{*}\left(x\right)\right)\left(1-j_{\epsilon}\left(x\right)\right)\mathrm{e}^{-V(x)/\epsilon}\mathrm{d}x\leq\frac{1}{Z_\epsilon}\mathrm{e}^{-U(\sigma)/\epsilon}\epsilon^{aK^{2}}|\partial\mathcal{K}_{\epsilon}^{+}|\;.
\]
Since the volume of $\partial\mathcal{K}_{\epsilon}^{+}$ can be bounded
independently of $\epsilon$, taking $K$ large enough ensures that the integral above is of order $o_{\epsilon}(1)\alpha_{\epsilon}$. Consider now the integral on the set $C_{\epsilon}^{1}\cap\partial\mathcal{K}_{\epsilon}^{+}$.
First, notice that for $x\in\partial\mathcal{K}_{\epsilon}^{+}$,
if $V(x)<U(\sigma)$ then necessarily $x\in\mathcal{W}_{m}$ and by
definition of $h_{\mathcal{M}_{\epsilon},\,\mathcal{S}_{\epsilon}}^{*}$
in~\eqref{eq:def h^*}, the estimate in Proposition~\ref{prop:ineq h_AB}
ensures the existence of constants $C,\,\beta>0$ independent of $\epsilon$
such that 
\[
\left|1-h_{\mathcal{M}_{\epsilon},\,\mathcal{S}_{\epsilon}}^{*}(x)\right|\leq\frac{C}{\epsilon^{\beta}}\mathrm{e}^{(V(x)-U(\sigma))/\epsilon}\;.
\]
Besides, since $h_{\mathcal{M}_{\epsilon},\,\mathcal{S}_{\epsilon}}^{*}(x)\in[0,\,1]$,
this inequality still remains valid if $V(x)-U(\sigma)\geq0$. Therefore,
\[
\int_{C_{\epsilon}^{1}\cap\partial\mathcal{K}_{\epsilon}^{+}}\left(1-h_{\mathcal{M}_{\epsilon},\,\mathcal{S}_{\epsilon}}^{*}\left(x\right)\right)\left(1-j_{\epsilon}\left(x\right)\right)\mathrm{e}^{-V(x)/\epsilon}\mathrm{d}x\leq\frac{C}{\epsilon^{\beta}}\mathrm{e}^{-U(\sigma)/\epsilon}\int_{C_{\epsilon}^{1}\cap\partial\mathcal{K}_{\epsilon}^{+}}\left(1-j_{\epsilon}\left(x\right)\right)\mathrm{d}x\;.
\]
Moreover, for $x\in C_{\epsilon}^{1}\cap\partial\mathcal{K}_{\epsilon}^{+}$,
\begin{align*}
1-j_{\epsilon}(x) & =\frac{1}{\sqrt{2\pi}}\int_{\sqrt{\frac{\mu}{\gamma\epsilon}}\langle x,\,v\rangle}^{\infty}e^{-t^{2}/2}dt\\
 & \leq e^{-\frac{\mu}{2\gamma\epsilon}\langle x,\,v\rangle^{2}}\\
 & \leq\epsilon^{a^{2}K^{2}\mu/2\gamma}\;,
\end{align*}
which ensures again by taking $K$ large enough that 
\[
\int_{C_{\epsilon}^{1}\cap\partial\mathcal{K}_{\epsilon}^{+}}\left(1-h_{\mathcal{M}_{\epsilon},\,\mathcal{S}_{\epsilon}}^{*}\left(x\right)\right)\left(1-j_{\epsilon}\left(x\right)\right)\mu_{\epsilon}(\mathrm{d}x)=o_{\epsilon}(1)\alpha_{\epsilon}\;.
\]
The proof of~\eqref{eq:second integral boundary} is exactly similar.
\end{proof}
Finally, the proof of Proposition~\ref{prop:asymptotics boundary integral} follows the steps of~\cite[Proposition 8.6 and Lemmas 8.7, 8.8 and 8.11]{JS}
\begin{proof}[Proof of Proposition~\ref{prop:asymptotics boundary integral}]
Let us start with the proof of~\eqref{eq:asymptotic expansion boundary K+}.
By Lemma~\ref{lem:h close to 1 or 0 boundary}, 
\begin{align*}
 & \frac{1}{Z_{\epsilon}}\left|\int_{\mathcal{Q}_{\epsilon}^{+}}\left(1-h_{\mathcal{M}_{\epsilon},\,\mathcal{S}_{\epsilon}}^{*}\left(\frac{K\delta}{\sqrt{\lambda_{1}}},\,\widetilde{x}\right)\right)(-\widetilde{x}_{d})\left(1-j_{\epsilon}\left(\frac{K\delta}{\sqrt{\lambda_{1}}},\,\widetilde{x}\right)\right)\mathrm{exp}\left(-\frac{V\left(\frac{K\delta}{\sqrt{\lambda_{1}}},\,\widetilde{x}\right)}{\epsilon}\right)\mathrm{d}\widetilde{x}\right|\\
 & \leq2K\delta\int_{\partial\mathcal{K}_{\epsilon}^{+}\cap\mathcal{J}_{\epsilon}}\left(1-h_{\mathcal{M}_{\epsilon},\,\mathcal{S}_{\epsilon}}^{*}\left(x\right)\right)\left(1-j_{\epsilon}\left(x\right)\right)\mu_{\epsilon}(\mathrm{d}x)=o_{\epsilon}(1)\alpha_{\epsilon}\;.
\end{align*}
Therefore, it remains to show that 
\[
\frac{1}{Z_{\epsilon}}\int_{\mathcal{Q}_{\epsilon}^{+}}(-\widetilde{x}_{d})\left(1-j_{\epsilon}\left(\frac{K\delta}{\sqrt{\lambda_{1}}},\;\widetilde{x}\right)\right)\mathrm{exp}\left(-\frac{V\left(\frac{K\delta}{\sqrt{\lambda_{1}}},\;\widetilde{x}\right)}{\epsilon}\right)\mathrm{d}\widetilde{x}=[1+o_{\epsilon}(1)]\alpha_{\epsilon}\;.
\]
Let $\widetilde{x}\in\mathcal{Q}_{\epsilon}^{+}$, then $x:=(K\delta/\sqrt{\lambda_{1}},\,\widetilde{x})\in\partial\mathcal{K}_{\epsilon}^{+}$
and by Lemma~\ref{lem:decomp boundary K+}, either 
\[
V(x)\geq U(\sigma)+aK^{2}\delta^{2},\quad\text{ or }\quad\langle x,v\rangle\geq aK\delta\;.
\]
If $V(x)\geq U(\sigma)+aK^{2}\delta^{2}$ then reinjecting in the
integral above and taking $K$ large enough provides us with a term
of order $o_{\epsilon}(1)\alpha_{\epsilon}$. Consider now the case
$\langle x,\,v\rangle\geq aK\delta$.

Let $\widetilde{\mathbb{H}}_{V}=(\mathbb{H}_{V})_{2\leq i,\,j\leq2d}\in\mathbb{R}^{(2d-1)\times(2d-1)}$
then the Taylor expansion of $V$ on $\partial\mathcal{K}_{\epsilon}^{+}$
ensures that 
\[
V\left(\frac{K\delta}{\sqrt{\lambda_{1}}},\,\widetilde{x}\right)=U(\sigma)-\frac{1}{2}K^{2}\delta^{2}+\frac{1}{2}\langle\widetilde{x},\,\widetilde{\mathbb{H}}_{V}\widetilde{x}\rangle+O(\delta^{3})\;.
\]
Moreover, standard estimates on the error function ensure that for
any $z>0$, 
\[
\frac{z}{z^{2}+1}e^{-z^{2}/2}\le\int_{z}^{\infty}e^{-t^{2}/2}dt\le\frac{1}{z}e^{-z^{2}/2}\;.
\]
Therefore, since $\langle x,\,v\rangle\geq aK\delta$, 
\begin{align}
1-j_{\epsilon}(x) & =\frac{1}{\sqrt{2\pi}}\int_{\sqrt{\frac{\mu}{\gamma\epsilon}}\langle x,\,v\rangle}^{\infty}e^{-t^{2}/2}dt\\
 & =[1+o_{\epsilon}(1)]\sqrt{\frac{\gamma\epsilon}{2\pi\mu}}\frac{1}{\langle x,\,v\rangle}e^{-\frac{\mu}{2\gamma\epsilon}\langle x,\,v\rangle^{2}}\;.
\end{align}
Consequently, 
\begin{align*}
 & \int_{\mathcal{Q}_{\epsilon}^{+}}(-\widetilde{x}_{d})\left(1-j_{\epsilon}\left(\frac{K\delta}{\sqrt{\lambda_{1}}},\,\widetilde{x}\right)\right)\mathrm{exp}\left(-\frac{V\left(\frac{K\delta}{\sqrt{\lambda_{1}}},\,\widetilde{x}\right)}{\epsilon}\right)\mathrm{d}\widetilde{x}\\
 & =[1+o_{\epsilon}(1)]\sqrt{\frac{\gamma\epsilon}{2\pi\mu}}\mathrm{e}^{-U(\sigma)/\epsilon}\mathrm{e}^{K^{2}\delta^{2}/2\epsilon}\int_{\mathcal{Q}_{\epsilon}^{+}}\frac{-\widetilde{x}_{d}}{\frac{K\delta}{\sqrt{\lambda_{1}}}v_{1}+\langle\widetilde{x},\,\widetilde{v}\rangle}\mathrm{e}^{-\frac{1}{2\epsilon}\left\langle \widetilde{x},\,\widetilde{\mathbb{H}}_{V}\widetilde{x}\right\rangle }\mathrm{e}^{-\frac{\mu}{2\gamma\epsilon}\left(\frac{K\delta}{\sqrt{\lambda_{1}}}v_{1}+\langle\widetilde{x},\,\widetilde{v}\rangle\right)^{2}}\mathrm{d}\widetilde{x}\;.
\end{align*}

Define now the vector $w\in\mathbb{R}^{2d-1}$ as follows: 
\[
\forall i\in\llbracket1,\,2d-1\rrbracket,\qquad w_{i}=K\delta\frac{\sqrt{\lambda}_{1}}{v_{1}}\frac{v_{i+1}}{\lambda_{i+1}}\;.
\]
We would like to perform the change of variables $y=\widetilde{x}+w\in\mathbb{R}^{2d-1}$
in the integral above. Therefore, let us rewrite the integrand in
term of the variable $y$. Namely, since $\lambda_{d+1}=1$, one has
that 
\begin{align*}
\widetilde{x}_{d} & =\langle y,\,e_{d}\rangle-K\delta\frac{\sqrt{\lambda_{1}}}{v_{1}}v_{d+1}\\
 & =\langle y,\,e_{d}\rangle-K\delta\frac{\sqrt{\lambda_{1}}}{\mu+\gamma}\;,
\end{align*}
since $v_{1}=(\mu+\gamma)v_{d+1}$ by Lemma~\ref{lem:eigenvalue mu}.
Moreover, using the notation from~\eqref{eq:notation x tilde} for
$\tilde{v}$ we have that 
\begin{align*}
\langle\widetilde{x},\,\widetilde{v}\rangle & =\langle y,\,\widetilde{v}\rangle-K\delta\frac{\sqrt{\lambda_{1}}}{v_{1}}\sum_{i=2}^{2d}\frac{v_{i}^{2}}{\lambda_{i}}\\
 & =\langle y,\,\widetilde{v}\rangle-\frac{K\delta}{\sqrt{\lambda_{1}}}v_{1}+K\delta\frac{\gamma\sqrt{\lambda_{1}}}{\mu v_{1}}
\end{align*}
using the identity in Lemma~\ref{lem:matrix equality}. Moreover,
by Lemma~\ref{lem:matrix equality} also, 
\[
\langle w,\,\widetilde{\mathbb{H}}_{V}w\rangle=K^{2}\delta^{2}\frac{\lambda_{1}}{v_{1}^{2}}\sum_{i=2}^{2d}\frac{v_{i}^{2}}{\lambda_{i}}=\left(1-\frac{\gamma\lambda_{1}}{\mu v_{1}^{2}}\right)K^{2}\delta^{2}\;.
\]
Finally, 
\[
\langle y,\,\widetilde{\mathbb{H}}_{V}w\rangle=K\delta\frac{\sqrt{\lambda_{1}}}{v_{1}}\langle y,\,\widetilde{v}\rangle\;.
\]

Using these computations, we have that 
\begin{align*}
 & -\frac{\mu}{2\gamma\epsilon}\Big(\langle\widetilde{x},\,\widetilde{v}\rangle+\frac{K\delta}{\sqrt{\lambda}_{1}}v_{1}\Big)^{2}-\frac{1}{2\epsilon}\langle\widetilde{x},\,\widetilde{\mathbb{H}}_{V}\widetilde{x}\rangle\\
 & =-\frac{\mu}{2\gamma\epsilon}\Big(\langle y,\,\widetilde{v}\rangle+K\delta\frac{\gamma\sqrt{\lambda}_{1}}{\mu v_{1}}\Big)^{2}-\frac{1}{2\epsilon}\left(\langle y,\,\widetilde{\mathbb{H}}_{V}y\rangle+\langle w,\,\widetilde{\mathbb{H}}_{V}w\rangle-2\langle y,\,\widetilde{\mathbb{H}}_{V}w\rangle\right)\\
 & =-\frac{1}{2\epsilon}\Big(\langle y,\,\widetilde{\mathbb{H}}_{V}y\rangle+\frac{\mu}{\gamma}\langle y,\,\widetilde{v}\rangle^2+K^{2}\delta^{2}\Big)
\end{align*}
Hence, 
\begin{align*}
 & \sqrt{\frac{\gamma\epsilon}{2\pi\mu}}\mathrm{e}^{-U(\sigma)/\epsilon}\mathrm{e}^{K^{2}\delta^{2}/2\epsilon}\int_{\mathcal{Q}_{\epsilon}^{+}}\frac{-\widetilde{x}_{d}}{\frac{K\delta}{\sqrt{\lambda_{1}}}v_{1}+\langle\widetilde{x},\,\widetilde{v}\rangle}\mathrm{e}^{-\frac{1}{2\epsilon}\left\langle \widetilde{x},\,\widetilde{\mathbb{H}}_{V}\widetilde{x}\right\rangle }\mathrm{e}^{-\frac{\mu}{2\gamma\epsilon}\left(\frac{K\delta}{\sqrt{\lambda_{1}}}v_{1}+\langle\widetilde{x},\,\widetilde{v}\rangle\right)^{2}}\mathrm{d}\tilde{x}\\
 & =\sqrt{\frac{\gamma\epsilon}{2\pi\mu}}\mathrm{e}^{-U(\sigma)/\epsilon}\int_{\mathcal{Q}_{\epsilon}^{+}+w}\frac{-\langle y,\,e_{d}\rangle+K\delta\frac{\sqrt{\lambda_{1}}}{\mu+\gamma}}{\langle y,\,\widetilde{v}\rangle+K\delta\frac{\gamma\sqrt{\lambda_{1}}}{\mu v_{1}}}\mathrm{e}^{-\frac{1}{2\epsilon}\left\langle y,\,\left(\widetilde{\mathbb{H}}_{V}+\frac{\mu}{\gamma}\widetilde{v}\otimes\widetilde{v}\right)y\right\rangle }\mathrm{d}y\;.
\end{align*}

Following the proof done in~\cite[Lemmas 8.8 and 8.11]{JS} one has
that 
\begin{align*}
  \int_{\mathcal{Q}_{\epsilon}^{+}+w}\frac{-\langle y,\,e_{d}\rangle+K\delta\frac{\sqrt{\lambda_{1}}}{\mu+\gamma}}{\langle y,\,\widetilde{v}\rangle+K\delta\frac{\gamma\sqrt{\lambda_{1}}}{\mu v_{1}}}\mathrm{e}^{-\frac{1}{2\epsilon}\left\langle y,\,\left(\widetilde{\mathbb{H}}_{V}+\frac{\mu}{\gamma}\widetilde{v}\otimes\widetilde{v}\right)y\right\rangle }\mathrm{d}y
 & =[1+o_{\epsilon}(1)]\frac{K\frac{\sqrt{\lambda_{1}}}{\mu+\gamma}}{K\frac{\gamma\sqrt{\lambda_{1}}}{\mu v_{1}}}\frac{(2\pi\epsilon)^{(2d-1)/2}}{\sqrt{\mathrm{det}\left(\widetilde{\mathbb{H}}_{V}+\frac{\mu}{\gamma}\widetilde{v}\otimes\widetilde{v}\right)}}\\
 & =[1+o_{\epsilon}(1)]\frac{\mu v_{1}}{\gamma(\mu+\gamma)}\frac{(2\pi\epsilon)^{(2d-1)/2}}{\sqrt{\mathrm{det}\left(\widetilde{\mathbb{H}}_{V}+\frac{\mu}{\gamma}\widetilde{v}\otimes\widetilde{v}\right)}}\;.
\end{align*}
Consequently, 
\begin{align*}
 & \int_{\mathcal{Q}_{\epsilon}^{+}}(-\widetilde{x}_{d})\left(1-j_{\epsilon}\left(\frac{K\delta}{\sqrt{\lambda_{1}}},\,\widetilde{x}\right)\right)\mathrm{exp}\left(-\frac{V\left(\frac{K\delta}{\sqrt{\lambda_{1}}},\,\widetilde{x}\right)}{\epsilon}\right)\mathrm{d}\widetilde{x}\\
 & =[1+o_{\epsilon}(1)]\sqrt{\frac{\gamma\epsilon}{2\pi\mu}}\mathrm{e}^{-U(\sigma)/\epsilon}\frac{\mu v_{1}}{\gamma(\mu+\gamma)}\frac{(2\pi\epsilon)^{(2d-1)/2}}{\sqrt{\mathrm{det}\left(\widetilde{\mathbb{H}}_{V}+\frac{\mu}{\gamma}\widetilde{v}\otimes\widetilde{v}\right)}}\;.
\end{align*}
By the well-known formula, e.g., \cite[Page 416]{Harville} one has that 
\begin{align*}
\det\left(\widetilde{\mathbb{H}}_{V}+\frac{\mu}{\gamma}\widetilde{v}\otimes\widetilde{v}\right) & =\left(1+\frac{\mu}{\gamma}\langle\widetilde{v},\,\widetilde{\mathbb{H}}_{V}^{-1}\widetilde{v}\rangle\right)\det{\widetilde{\mathbb{H}}_{V}}\\
 & =\frac{\mu v_{1}^{2}}{\gamma\lambda_{1}}\det{\widetilde{\mathbb{H}}_{V}}=\frac{\mu v_{1}^{2}}{\gamma\lambda_{1}^{2}}\left|\det{\mathbb{H}_{V}}\right|\;,
\end{align*}
by Lemma~\ref{lem:matrix equality}. Therefore, 
\begin{align*}
 & \frac{1}{Z_{\epsilon}}\int_{\mathcal{Q}_{\epsilon}^{+}}(-\widetilde{x}_{d})\left(1-j_{\epsilon}\left(\frac{K\delta}{\sqrt{\lambda_{1}}},\,\widetilde{x}\right)\right)\mathrm{exp}\left(-\frac{V\left(\frac{K\delta}{\sqrt{\lambda_{1}}},\,\widetilde{x}\right)}{\epsilon}\right)\mathrm{d}\widetilde{x}\\
 & =[1+o_{\epsilon}(1)]\frac{1}{Z_{\epsilon}}\mathrm{e}^{-U(\sigma)/\epsilon}\frac{1}{2\pi}(2\pi\epsilon)^{d}\frac{\lambda_{1}}{\mu+\gamma}=[1+o_{\epsilon}(1)]\alpha_{\epsilon}\;,
\end{align*}
since $\lambda_{1}=\mu(\mu+\gamma)$ by Lemma~\ref{lem:eigenvalue mu}
and $\det{\mathbb{H}_{U}}=\det{\mathbb{H}_{V}}$. 
\end{proof}

\medskip

\noindent\textbf{Acknowledgement.} S. Lee, M. Ramil, and I. Seo were supported by the National Research Foundation of Korea (NRF) grant funded by the Korean government (MEST) No. 2023R1A2C100517311 and Samsung Science and Technology Foundation (Project Number SSTF-BA1901-03). S. Lee and I. Seo were also supported by a NRF grant funded by the Korean government No. 2016K2A9A2A1300381525. I. Seo was supported by the NRF grant funded by the Korean government No. 2019R1A6A1A10073437 and a Seoul National University Research Grant in 2023. I. Seo thanks KIAS for their support as a visiting scholar from March 2024 to February 2025. 

\bibliographystyle{plain}
\bibliography{biblio}
 
\end{document}